\documentclass[9pt]{amsart}

\usepackage{amssymb}
\usepackage{bm}
\usepackage{graphicx}
\usepackage[centertags]{amsmath}
\usepackage{amsfonts}
\usepackage{amsthm}
\usepackage{graphicx}
\usepackage{slashed}
\usepackage{float}

\usepackage[usenames,dvipsnames,svgnames,table]{xcolor}
\usepackage[colorlinks=true]{hyperref}
\hypersetup{linkcolor=BrickRed, urlcolor=green, citecolor=blue, linktoc=page}

\numberwithin{equation}{section}

\newtheorem*{proposition*}{Proposition}
\newtheorem*{theorem*}{Theorem}
\newtheorem*{conjecture*}{Conjecture}
\newtheorem*{claim*}{Claim}
\newtheorem*{lemma*}{Lemma}
\newtheorem*{corollary*}{Corollary}
\newtheorem{theorem}{Theorem}[section]
\newtheorem{proposition}[theorem]{Proposition}

\newtheorem{lemma}[theorem]{Lemma}

\newtheorem*{definition*}{Definition}

\newtheorem*{assumption*}{\mathcal{A}ssumption}

\newtheorem*{remark*}{Remark}
\newtheorem{remark}{Remark}[section]

\newcommand{\mcC}{\mathcal{C}}

\newcommand{\mcR}{\mathcal{R}}
\newcommand{\mcN}{\mathcal{N}}

\numberwithin{equation}{section}
\begin{document}
\title[Nonlinear scalar perturbations of extremal Reissner--Nordstr\"{o}m spacetimes]{Nonlinear scalar perturbations of \protect \\ extremal Reissner--Nordstr\"{o}m spacetimes}
\author{Y. Angelopoulos, S. Aretakis, and D. Gajic}

\maketitle

\begin{abstract}
We present the first rigorous study of nonlinear wave equations on extremal black hole spacetimes without any symmetry assumptions on the solution. Specifically, we prove global existence with asymptotic blow-up for solutions to nonlinear wave equations satisfying the null condition on extremal Reissner--Nordstr\"{o}m backgrounds. This result shows that the extremal horizon instability persists in model nonlinear theories. Our proof crucially relies on a new vector field method that allows us to obtain almost sharp decay estimates. 
\end{abstract}

\section{Introduction}\label{intro}

\subsection{Introduction}
\label{sec:Introduction}

 Extremal black holes are characterized by the vanishing of the surface gravity (or, equivalently, of the Hawking temperature) of the event horizon. Special examples are the maximally charged extremal Reissner--Nordstr\"{o}m family (ERN) and the maximally rotating extremal Kerr family (EK). Extremal black holes are of interest from both theoretical and practical points of view. Indeed, extremal black holes saturate various geometric inequalities \cite{SD08}, have interesting uniqueness properties \cite{pauextremalluci}, and, moreover, are of interest in supersymmetry and string theory \cite{stromingerextremalentropy}.  Furthermore, abundant astronomical evidence suggests that many stellar and supermassive black holes are near-extremal \cite{rees2005, brenneman-spin}. As we shall discuss in detail below, unlike  sub-extremal RN and Kerr black holes, ERN and EK exhibit various horizon instability properties. It has recently been found that these properties could potentially serve as an observational signature for extremal black holes by far away observers \cite{extremal-prl}.

In view of the instabilities present already in the linear theory, understanding the full dynamics of extremal black holes is an important and challenging problem. In this article we investigate the global behavior of nonlinear scalar perturbations of extremal black holes. Specifically, we consider  nonlinear wave equations of the form:
\begin{equation}\label{nw}
\left\{\begin{aligned}
       \Box_{g_M} \psi = A (x, \psi ) \cdot g^{\alpha \beta} \cdot  \partial_{\alpha}  \psi  \cdot \partial_{\beta} \psi + \mathcal{O} ( |\psi |^k , |T\psi |^k ) \mbox{  $k \geq 3$},\\
       \psi |_{\Sigma_{\tau_0} } = \epsilon f , \quad n_{\Sigma^{int}_{\tau_0}} \psi |_{\Sigma^{int}_{\tau_0} } = \epsilon h, \
       \end{aligned} \right.
\end{equation}
where $g_M$ is the metric of an extremal Reissner--Nordstr\"{o}m spacetime with mass $M$,  $x$ denotes a spacetime variable, $\Sigma_{\tau_0}$ is a Cauchy spacelike-null hypersurface (i.e. a hypersurface that is spacelike from the horizon till $\{ r = R \}$ for some $R > M$, and null from $\{ r= R \}$ till null infinity) with $\Sigma^{int}_{\tau_0}$ its spacelike part. We will show that for sufficiently small initial data (i.e.~for small $\epsilon$) solutions of \eqref{nw} are unique and exist globally in the domain of outer communications $\mathcal{M}$ up to and including the event horizon.

Our motivation for considering such a model is the study of the stability or instability of extremal Reissner--Nordstr\"{o}m black hole spacetimes in the context of the Einstein--Maxwell equations, where part of the problem consists of dealing with nonlinearities of the form studied in this article. Hence, the present work provides the first step in understanding the fully nonlinear dynamics of extremal black holes without any symmetry assumptions.  We introduce new techniques that we believe will be relevant for the study of that problem. 

The main difficulties, discussed in detail below, arise from the slow non-integrable decay of solutions $\psi$ and the fact that certain derivatives of $\psi$ grow in time. This necessitates the development of a new physical space method that yields maximum decay (for the quantities that do decay) in order to compensate for the growth of the other quantities. Before we provide an overview of these difficulties and their resolution we present some relevant background for sub-extremal and extremal black holes.

\subsection{Linear and nonlinear waves on sub-extremal black holes}

The sub-extremal black hole stability problem is currently one of the most actively studied problems in general relativity. Important developments have been presented by various research teams during the past two decades. Stability results for the linear wave equation on subextremal Kerr backgrounds were obtained in the seminal works of Dafermos and Rodnianski \cite{redshift}, \cite{dr7} (see also the lecture notes \cite{lecturesMD} and the subsequent work with Shlapentokh-Rothman \cite{tria}), which moreover introduced a mathematical interpretation of the celebrated \textit{redshift effect}, allowing the authors to obtain \textit{non-degenerate} integrated local energy decay estimates up to and including the event horizon.  Furthermore, using weighted estimates at infinity introduced in \cite{newmethod} (and extensively studied in \cite{moschidis1}), the authors of \cite{tria} were able to show polynomial decay in time for the solution and its derivatives of all orders. Precise inverse polynomial time asymptotics were rigorously derived in \cite{paper2}. For further results see also \cite{blu1}, \cite{tataru1}, \cite{tataru2}, \cite{volker1}. Global existence and uniqueness of solutions of \eqref{nw} with small initial data on sub-extremal black hole backgrounds were proved by Luk \cite{luknullcondition}. See also \cite{blu0}, \cite{blu3}, \cite{dafrodeasy}, \cite{lindbladmet}. The major difficulty encountered in \cite{luknullcondition} was the loss of time derivatives due to the trapping effect of the photon sphere. We also refer to the work of Yang \cite{shiwu} which can be used to give an alternative proof of the results of \cite{luknullcondition} (using, however, the techniques of \cite{luknullcondition} to deal with the loss of derivatives at the photon sphere). In the recent breakthrough of Dafermos, Rodnianski and Holzegel decay was derived for the system of linearized gravity around the Schwarzschild spacetime \cite{Dafermos2016}. See also \cite{dhr-teukolsky-kerr, anderssonkerr} for works on linearized gravity around Kerr, and \cite{elena1} for the problem of linearized gravity on sub-extremal Reissner--Nordstr\"{o}m. 
Finally, we refer to the impressive recent work by Klainerman and Szeftel on the fully nonlinear stability of the Schwarzschild spacetime \cite{klainerman17} in polarized axial symmetry. It is worth noting that the latter work, among other things, makes use of the techniques introduced in \cite{paper1} which are useful for deriving improved decay results and which play a crucial role in the present paper.

\subsection{Related works on linear and nonlinear waves on extremal black holes}
\subsubsection{Linear waves on extremal black holes}
The study of linear waves on extremal black holes was initiated by the second author in \cite{aretakis1, aretakis2, aretakis3, aretakis4, aretakis2012} where it was shown that, in contrast to the case of sub-extremal backgrounds, first-order transversal derivatives of generic scalar perturbations on extremal Reissner--Nordstr\"{o}m \textbf{do not decay in time} along the event horizon. Higher-order derivatives in fact \textbf{blow up} along the event horizon. The source of these instabilities is the degeneracy of the redshift effect at the event horizon and a hierarchy of conserved charges along the event horizon. Subsequent work \cite{aag1} showed that generic solutions to the wave equation \textbf{do not satisfy a non-degenerate Morawetz estimate} up to and including the event horizon. The latter work, in particular, makes apparent that new techniques are needed in addressing the global evolution of nonlinear wave equations on such backgrounds. Precise asymptotics were derived in \cite{paper4} where it was in fact shown that solutions to the wave equation \textbf{decay non-integrably} in time. For extremal Kerr spacetimes we refer to the works \cite{hm2012, zimmerman1, harveyeffective}. Extentions of these instabilities have been presented in various settings \cite{ori2013,sela, harvey2013, godazgar17, khanna17, cvetic2018}. For an extensive list of references we refer to  \cite{aretakisbrief}.

\subsubsection{Nonlinear waves on extremal black holes}

The study of nonlinear wave equations satisfying the null condition on extremal black holes was initiated by the first author in \cite{yannis1} in the context of \textbf{spherical symmetry}. It was shown that solutions of \eqref{nw}, with smooth spherically symmetric $f$ and $h$,  are \textit{globally smooth} and \textit{unique} in $\mathcal{M}$.  It was further shown that, in analogy to the linear case, the derivatives of the solution that are transversal to the horizon do not decay along the event horizon, while higher-order transversal derivatives diverge to infinity along the event horizon. Other nonlinearities were studied in \cite{aretakis4, aretakisglue, aretakiselliptic, bizon-extremal-nonlinear}. Numerical simulations of the fully non-linear evolution in the context of spherical symmetry were carried out in \cite{harvey2013}. The results of \cite{harvey2013} are in complete agreement with the results of the present paper.

\subsection{The main theorems}

\subsubsection{Notation}\label{notation}
First, we introduce some notation in order to rigorously state the theorems proven in this article. We start by recording the basics of the geometry of extremal Reissner--Nordstr\"{o}m black hole spacetimes. The domain of outer communications up to and including the event horizon of an extremal Reissner--Nordstr\"{o}m spacetime with mass $M > 0$ can be given by the following $4$-dimensional Lorentzian manifold-with-boundary:
$$ \mathcal{M} \doteq \mathbb{R} \times [M, \infty) \times \mathbb{S}^2 , $$
with metric
$$ g_M = - D dv^2 + 2 dv dr + r^2 \gamma_{\mathbb{S}^2} , $$
in the \textit{ingoing} Eddington--Finkelstein coordinates $(v,r,\omega) \in \mathbb{R} \times [M, \infty) \times \mathbb{S}^2$ where 
$$ D \doteq D(r) = \left( 1- \frac{M}{r} \right)^2 , $$
and $\gamma_{\mathbb{S}^2}$ is the standard metric on the $2$-sphere $\mathbb{S}^2$ (in the rest of the document we will use $g$ or $g_M$ for the metric, with raised indices indicating the inverse of the metric). We also consider the \textit{double null} coordinates $(u,v)$ for $v$ as before and $u \doteq v - 2r_{*}$ for $r_{*} (r) = r-M - \frac{M^2}{r-M} + 2M \log \left( 1 - \frac{M}{r} \right)$ the so-called tortoise coordinate. The metric takes the following form in double null coordinates:
$$ g_M = - D du dv + r^2 \gamma_{\mathbb{S}^2} . $$
We denote the future event horizon at $r=M$ by $\mathcal{H}^{+} \doteq \{ r = M \}$, and future null infinity by $\mathcal{I}^{+}$ which is where the null hypersurfaces $\{ u = \tau \}$ terminate as $v \rightarrow \infty$, for any $\tau$.

In the $(v,r,\omega )$ coordinates we denote $T \doteq \partial_v$, $Y \doteq \partial_r$, and in the $(u,v,\omega )$ coordinates we denote $L \doteq \partial_v$, $\underline{L} \doteq \partial_u$. We also have that
$$ L = T + \frac{1}{2} D Y , \quad \underline{L} = -\frac{1}{2} D Y . $$
For $\omega \in \mathbb{S}^2$ we have that the corresponding volume form is given by $d\omega = \sin \theta d\theta d\varphi$ for $\theta$ and $\varphi$ the standard coordinates on the sphere, and the covariant derivative on $\mathbb{S}^2$ is given by $\nabla_{\mathbb{S}^2}$ and the corresponding Laplacian by $\Delta_{\mathbb{S}^2}$. We also use the notation $\slashed{\nabla} = \frac{1}{r} \nabla_{\mathbb{S}^2}$ and $\slashed{\Delta} = \frac{1}{r^2} \Delta_{\mathbb{S}^2}$, and furthermore we define the three Killing vector fields $\Omega_i$, $i \in \{1,2,3\}$, associated to the spherical symmetry of an extremal Reissner--Nordstr\"{o}m spacetime as follows:
$$ \Omega_1 = \sin \varphi \partial_{\theta} + \cot \theta \cos \varphi \partial_{\varphi} , \quad \Omega_2 = - \cos \varphi \partial_{\theta} + \cot \theta \sin \varphi \partial_{\varphi} , \quad \Omega_3 = \partial_{\varphi} , $$
and finally the vector fields $\Omega$ by:
$$ \Omega^m = \Omega_1^{m_1} \Omega_2^{m_2} \Omega_3^{m_3} , $$
for any $m \in \mathbb{N}$ and $( m_1 , m_2 , m_3 ) \in \mathbb{N}_0^3$ where $m_1 + m_2 + m_3 = m$. 

Now we define the null-spacelike-null hypersurfaces $\Sigma_{\tau}$ by setting first
$$ \Sigma_{\tau_0} \doteq \{ v = v_{\Sigma_{\tau_0}} (r) \} , $$
for $v_{\Sigma_{\tau_0}} : [ M , \infty) \rightarrow \mathbb{R}$ given by:
$$ v_{\Sigma_{\tau_0}} (r) = v_0 + \int_M^r G (r' ) \, dr' ,$$
for some $v_0 \in \mathbb{R}_{> 0}$ and $G$ a non-negative function on $[M , \infty)$ satisfying:
$$ G (r) \geq 1 , \quad G(r) - \frac{2}{D} > 0 ,\quad G(r) - \frac{2}{D} = \mathcal{O} ( r^{-1-\delta} ) , $$
for some $\delta > 0$. We further impose the following symmetry condition: if $\left( t = \frac{u+v}{2} , r_{*} , \omega \right) \in \Sigma_{\tau_0}$, then $\left( t = \frac{u+v}{2} , - r_{*} , \omega \right) \in \Sigma_{\tau_0}$. Now $\Sigma_{\tau}$ can be defined by $\Sigma_{\tau} \doteq f_{\tau} ( \Sigma_{\tau_0} )$ for $f_{\tau}$ the flow of $T$. We will work in the spacetime region $\mcR = J^{+} ( \Sigma_{\tau_0} )$. 

\subsubsection{Statement of the theorems}
\label{stattheorem}

In the current article we show the following result for small-data solutions of equation \eqref{nw}:
\begin{theorem}\label{thm:main}
Let $(\mathcal{M} , g_M )$ be the domain of outer communications of an extremal Reissner--Nordstr\"{o}m spacetime up to and including the future event horizon with mass $M > 0$, and consider the nonlinear wave equation
\begin{equation}\label{eq:thm}
 \Box_{g_M} \psi = A(u,v,\omega,\psi ) \cdot g^{\alpha \beta} \cdot ( \partial_{\alpha} \psi ) \cdot ( \partial_{\beta} \psi ) + \mathcal{O} (| \psi |^k , | T\psi |^k ) , \mbox{  $k \geq 3$}, 
\end{equation}  
in $\mathcal{M}$ up to and including the future event horizon. Here $A$ denotes a function that depends smoothly on the coordinates $(u,v,\omega )$ and the solution $\psi$ (where $(u,v)$ are null coordinates and $\omega$ the angular coordinate in $\mathcal{M}$) that is bounded along with all its derivatives. For equation \eqref{eq:thm} we consider smooth and compactly supported initial data $\epsilon f$ and $\epsilon h$ given as
$$ \psi_{ID} \doteq \left( \psi\Big{|}_{\Sigma_{\tau_0}} , n_{\Sigma_{\tau_0}} \psi \Big{|}_{\Sigma^{int}_{\tau_0}} \right) = ( \epsilon f , \epsilon h ) , $$
on an initial null-spacelike-null hypersurface $\Sigma_{\tau_0}$ where $\Sigma^{int}_{\tau_0}$ denotes its spacelike part, where $f$ and $h$ additionally satisfy:
\begingroup
\allowdisplaybreaks
\begin{align*}
E^{\tau_0} [ f ]  + \sum_{\substack{k \leq 5 \\ l \leq 5} }\int_{\Sigma_{\tau_0}}  J^T [ \Omega^k T^{l-1} f ] \cdot \textbf{n}_{\Sigma_0} \, d\mu_{\Sigma_0} + \| f \|_{H^s_{0}} + \| h \|_{\widetilde{H}^{s-1}_{0}}  < \infty ,
\end{align*}
\endgroup
for some $s > 5/2$ (where the norms $E^{\tau} [ \bullet ]$, $H^s_{\tau}$ and $\widetilde{H}^s_{\tau}$ for $\tau \geq \tau_0$, are defined in Appendix \ref{norm:add}). Then there exists a $\Delta > 0$ such that for all $0 \leq \epsilon \leq \Delta$, equation \eqref{eq:thm} with data $(\epsilon f , \epsilon g ) $ as above admits a \textit{unique}, \textit{global} and smooth solution $\psi$ in $\mathcal{M}$ with finite $E^{\tau} [ \psi ]$ and $\| \psi \|_{H^s_{\tau}}$ norms for any $\tau \in [\tau_0 , \infty )$.
\end{theorem} 
Our main result establishes the global existence and uniqueness for solutions of \eqref{nw} for small enough, smooth and compactly supported data given on a null-spacelike-null hypersurface (see section \ref{notation} for the precise definition) that crosses the event horizon. Note that data of this type, which are compactly supported at infinity, but non-zero close to and on the horizon, are the most interesting from a physical point of view, they can be used to model \textit{local} perturbations of the extremal Reissner--Nordstr\"{o}m black hole spacetimes, and in the physics literature they represent \textit{outgoing radiation}. Moreover, and in sharp contrast to the sub-extremal case, our solution exhibits \textit{non-decay} along the event horizon for the derivative that is transversal to the horizon, and \textit{growth} for the second such derivative. In particular the qualitative behaviour of solutions established in Theorem \ref{thm:main} is described by the following result:

\begin{theorem}\label{thm:qual}
Under the conditions of Theorem \ref{thm:main}, with $\psi$ a solution of \eqref{nw} given by Theorem \ref{thm:main}, we have that:
\begin{equation}\label{thm:psidec}
| \psi | ( v , r , \omega ) \lesssim \frac{\epsilon}{v^{1-\delta_1 / 2}} \mbox{  close to $\mathcal{H}^{+}$,} 
\end{equation}
\begin{equation}\label{thm:psideco}
| \Omega \psi | ( v , r , \omega ) \lesssim \frac{\epsilon}{v^{1+\delta_2 / 2}} \mbox{  close to $\mathcal{H}^{+}$,} 
\end{equation}
\begin{equation}\label{thm:psidect}
| T \psi | ( v , r , \omega ) \lesssim \frac{\epsilon}{v^{1+\delta_2 / 2}} \mbox{  close to $\mathcal{H}^{+}$,} 
\end{equation}
\begin{equation}\label{thm:psidecy}
| Y \psi | ( v , r , \omega ) \lesssim \epsilon \mbox{  close to $\mathcal{H}^{+}$, and  } | Y \psi  ( v , M , \omega ) - Y \psi  ( v_0 , M , \omega ) | \simeq \epsilon^2 \mbox{  on $\mathcal{H}^{+}$,}
\end{equation}
and
\begin{equation}\label{thm:psidecyy}
| Y^2 \psi | ( v , M , \omega ) \simeq \epsilon v \mbox{  on $\mathcal{H}^{+}$,} 
\end{equation}
for $0 < \delta_1 \ll 1$, $0 < \delta_2 \ll 1$ small enough, for $\epsilon$ as in Theorem \ref{thm:main}, where we work with the $(v,r,\omega)$ ingoing Eddington--Finkelstein coordinates close to the horizon $\mathcal{H}^{+}$ at $r=M$.
\end{theorem}

\subsection{The main difficulties}
\label{themaindiff}

The growth and non-decay of derivatives of the linear flow is one of the main obstacles in proving global existence for \eqref{nw}. Moreover there is an additional difficulty originating from the quadratic terms near the event horizon in the extremal case. This difficulty can be illustrated by considering the transformed problem in a neighborhood of null infinity via the Couch--Torrence conformal isometry\footnote{This transformation maps the event horizon to null infinity and vice versa. See also Appendix \ref{app:iso}.}. Indeed, applying this conformal transformation to a solution $\psi_H$ of the following equation restricted close to the event horizon
\begin{equation}\label{eq:hor}
\Box_{g_M} \psi_H = g^{\alpha \beta} \cdot \partial_{\alpha} \psi_H \cdot \partial_{\beta} \psi_H ,
\end{equation}
yields a solution $\psi_I$ of the following equation restricted to a neighborhood of null infinity:
\begin{equation}\label{eq:inf}
\Box_{g_M} \psi_I = \frac{1}{r} ( L \phi_I ) \cdot ( \underline{L} \phi_I ) + \frac{1}{r^2} \phi_I \cdot ( L \phi_I ) - \frac{1}{r^3} \phi_I \cdot (\underline{L} \phi_I ) + \frac{1}{r^3} \phi_I^2 + \frac{1}{r} | \slashed{\nabla} \phi_I |^2 , 
\end{equation}
where $\phi_{I,H} = r\psi_{I,H} $ and $L = \partial_v$, $\underline{L} = \partial_u$ the standard double null coordinates. On the other hand, the classical null form (as defined in \eqref{eq:inf}) would take the following form 
\begin{equation}\label{null:classical}
 ( L \psi_I ) \cdot ( \underline{L} \psi_I ) + | \slashed{\nabla} \psi_I |^2 . 
 \end{equation}
Note that, in view of the bounds $|L \psi|\leq C r^{-2},  |\underline{L} \psi|\leq C r^{-1}$ and $|\Omega \psi_I|\leq C r^{-1}$, the expression \eqref{null:classical} decays in $r$ towards null infinity like $r^{-3}$. On the other hand, in order to obtain the same decay rate in $r$ for the right hand side of \eqref{eq:inf} we need to derive the following improved bounds:
$|\phi_I|\leq C$, $|\Omega \phi_I|\leq C$ and $|\underline{L} \phi_I|\leq C$ and moreover  $|L\phi_I |\leq Cr^{-2}$.

Hence, we see that merely obtaining the needed $r$-decay would require one to show stronger estimates than those needed in the sub-extremal case. Such estimates have not been shown in previous nonlinear works on asymptotically flat settings. On the other hand, deriving mere boundedness of the transversal derivative $Y\psi$ at the event horizon in the extremal case (which is required by the continuation criterion for \eqref{nw}) corresponds (again via the Couch--Torrence transformation) to bounding pointwise the $r$-weighted derivative $r^2 L(r\psi)$ in a neighborhood of null infinity. It is important to emphasize that in the sub-extremal case the horizon and null infinity are not conformally related and hence one can show pointwise boundedness and decay for the transversal derivatives at the event horizon relatively easily using the redshift effect. See for instance  \cite{luknullcondition} and \cite{shiwu} for a demonstration of this in nonlinear settings.  The method of \cite{luknullcondition} and \cite{shiwu} breaks down  for linear fields on extremal black holes. In fact, they break down even if one considers a strongly degenerate nonlinearity on extremal backgrounds such as 
\begin{equation}\label{eq:sqrtd}
 \Box_{g_M} \psi = \left(1-\frac{M}{r}\right) \cdot g^{\alpha \beta} \cdot \partial_{\alpha} \psi \cdot \partial_{\beta} \psi.
 \end{equation}
Global existence and uniqueness of solutions to \eqref{eq:sqrtd} was proved in \cite{areangel6} using the inhomogeneous energy estimates of \cite{aretakis1} and a novel $(r-M)$-weighted commuted (with $Y$) estimate in order to bound $\left(1-\frac{M}{r}\right) Y\psi$ (which is required for the continuation criterion). On the other hand, in the context of spherical symmetry,  \cite{yannis1} overcame the extremal difficulties by a delicate use of the method of characteristics in combination with the weak decay of the solution, something that is not enough outside spherical symmetry.

\subsection{Overview and method of proof}\label{sec:method}

The classical local existence and uniqueness  of solutions of \eqref{eq:thm} for data as in Theorem \ref{thm:main} can be upgraded to global existence and uniqueness in $\mathcal{M}$  provided one verifies the following continuation criteria (stated schematically here):
$$ | Y \psi | \lesssim \epsilon, \quad | T\psi | \lesssim \epsilon, \quad | \Omega \psi | \lesssim \epsilon\ \  \mbox{ everywhere in $\mathcal{M}$. } $$
Here we used the vector fields corresponding to the ingoing Eddington--Finkelstein coordinates (see Section \ref{notation}). The rest of the article is hence devoted to verifying the aforementioned continuation criteria. This is done through a bootstrap argument. First, we state the energy estimates we are going to use in Section \ref{energy}. In Section \ref{as:boot}, we state our bootstrap assumptions. In Section \ref{dec}, we use the bootstrap assumptions of Section \ref{as:boot} to derive energy and pointwise boundedness and decay estimates (and hence conditionally verifying the continuation criteria). In Section \ref{sec:GrowthEstimates} we show a growth estimate for the second transversal derivative close to the horizon, and a new $v$-weighted estimate for $L \phi$ close to the horizon. Then in Section \ref{end:boot}, we improve the bootstrap assumptions of Section \ref{as:boot}, thereby closing the bootstrap argument. The results of the aforementioned sections establish also estimates \eqref{thm:psidec}, \eqref{thm:psideco}, \eqref{thm:psidect} and the first estimate from \eqref{thm:psidecy}. Finally in Section \ref{insta} we demonstrate non-decay for $Y \psi$ and growth for $Y^2 \psi$ on the horizon establishing the second estimate of \eqref{thm:psidecy} and estimate \eqref{thm:psidecyy}, while in Section \ref{non:other} we discuss how our methods can be adapted to  weighted nonlinearities on sub-extremal black holes.

Our proof relies on several novel techniques which we summarize below:

\bigskip

1. \textbf{(Improved Morawetz estimate)} We prove an improved Morawetz estimate that optimizes the $(r-M)$-weights at the horizon. Schematically, it has the following form:
$$ \int_{\mathcal{A}} (r-M)^{1+\delta} | \partial \psi |^2 \lesssim \int_{\Sigma} J^T [ \psi ] + \mbox{  inhomogeneous terms} , $$
for any $\delta > 0$, and where $\mathcal{A}$ is a spacetime region close to the event horizon not intersecting the photon sphere (see section \ref{energy} for precise definitions). Our proof has certain similarities with Alinhac's method of ghost weights (see \cite{ali}). Note  that our improvement is decoupled from the trapping effect at the photon sphere $\{ r=2M \}$ since the $(r-M)$-weights that we introduce are optimized only at the horizon, while the region close to the photon sphere is treated separately using the inhomogeneous versions of estimates first introduced in \cite{aretakis1}.

\bigskip

2. \textbf{(Angular decomposition)} We split solutions into a spherically symmetric part and a remainder supported on angular frequencies greater or equal to 1 as follows
\[\psi=\psi_0+\psi_{\geq 1}.\] Even though $\psi_0$ and $\psi_{\geq 1}$ are coupled via the  nonlinearity, we are still able to derive sharp decay results for each of them. 

\bigskip

3. \textbf{(Horizon-localized and infinity-localized weighted hierarchies)}  We establish various $(r-M)^{-p}$-weighted  and  $r^p$-weighted hierarchies of estimates which schematically take the following form:
$$ \int_{\mathcal{N}^H} (r-M)^{-p} ( \underline{L} \phi )^2 + \int_{\mathcal{A}} (r-M)^{-p+1} ( \underline{L} \phi )^2 \lesssim \int_{\mathcal{N}^H_0} (r-M)^{-p} ( \underline{L} \phi )^2 + \mbox{  error terms} , $$
and
$$ \int_{\mathcal{N}^I} r^p ( L \phi )^2 + \int_{\mathcal{B}} r^{p-1} ( L\phi )^2 \lesssim \int_{\mathcal{N}^I_0} r^p ( L \phi )^2 + \mbox{  error terms} , $$
where $\phi = r \psi$, and where  $\mathcal{N}^H$, $\mathcal{N}^I$ are null hypersurfaces intersecting  the event horizon and null infinity, respectively, and $\mathcal{A}$ and $\mathcal{B}$ are appropriate spacetime neighborhoods of the event horizon and null infinity, respectively. Such estimates were presented for the linear wave equation on extremal Reissner--Nordstr\"{o}m in \cite{paper4}.

We can show almost-sharp decay  for $\psi_0$ by using the full range of $p$, namely for  $p \in (0,3)$ for the uncommuted estimates and $p \in (0,1)$ for the $(r-M)^{-2} \underline{L}$ --commuted estimates. For the non-spherically symmetric part $\psi_{\geq 1}$ we can prove integrable decay by using an extended range for $p$ for both the uncommuted and the commuted hierarchies. The resulting estimates allow us to show integrability  for  $\psi_{\geq 1}|_{\mathcal{H}^{+}}$  along the event horizon and the radiation field $\phi_{\geq 1}|_{\mathcal{I}^{+}} \doteq r \psi_{\geq 1}|_{\mathcal{I}^{+}}$ along null infinity.

We only apply these weighted hierarchies when considering the higher order derivatives $T^k\psi$ where $k=1,2,3,4,5$.   We use the same range of weighted estimates for $T\psi$ as for $\psi_{\geq 1}$, and then we appropriately restrict $p$ to smaller ranges for $T^k \psi$, $k \in \{2,3,4,5\}$. Note that we need to commute with $T$ multiple times due to the presence of the trapping effect at the photon sphere $\{ r = 2M \}$. The progressively restricted range of $p$ in both the $(r-M)^{-p}$-weighted estimates and the $r^p$-weighted estimates for $T^k \psi$ implies slower decay for these time derivatives. This is a version of the \textit{top order energy technique}. 

The ranges of $p$ for the $(r-M)^{-p}$-weighted estimates close to the horizon and for the $r^p$-weighted estimates close to infinity for $\psi$ and $T\psi$ are summarized in the following table for $\delta_1 , \delta_2 > 0$:
\begin{table}[H]
%\begin{center}
   % \begin{tabular}{ l | c  }
    %\hline
   %asymptotics for $\psi$ & origin of the coefficient \\\Xcline{1-2}{0.05cm} 

    %$-4I_{0}[\psi]\cdot \frac{1}{\tau^2}$ & $I_{0}[\psi]\neq 0$ unique obstruction to inverting $T$ \\ \hline
   %$8I_{0}^{(1)}[\psi]\cdot \frac{1}{\tau^3}$ &  $I_{0}^{(1)}[\psi]\neq 0$ unique obstruction to inverting $T^2$   \\     \hline
  %\end{tabular}
%\end{center}
%\end{table}

\begin{center}
\begin{tabular}{  c | c | c | c | c  } 
\mbox{  Multiplier \ / \ Commutator  } & $(r-M)^{-p} \underline{L}$  \ / \  $\mbox{none}$ &  $r^p L$  \ / \  $ \mbox{none} $ & $(r-M)^{-p} \underline{L}$  \ / \  $(r-M)^{-2} \underline{L} $ & $r^p L$  \ / \  $ r^2 L  $ \\ \hline \hline
$\psi_0$ & $p \in (0,3 - \delta_1 ]$ & $p \in (0, 3 - \delta_1 ]$ & $p \in (0, 1-\delta_1 ]$ & $p \in (0, 1-\delta_1 ]$ \\ 
\hline
$\psi_{\geq 1}$ & $p \in (0,2)$ & $p \in (0,2)$ & $p \in (0, 1+\delta_2 ]$ & $p \in (0, 1+\delta_2 ]$ \\ 
\hline
$T \psi$ & $p \in (0,2)$ & $p \in (0,2)$ & $p \in (0, 1+\delta_2 ]$ & $p \in (0, 1+\delta_2 ] $ \\ 
\end{tabular}
\end{center}
\end{table}
\vspace{-0.2cm}
 It is worth noticing that this is the first nonlinear small-data problem where such an extended range for the $r^p$-weighted estimates is needed in a neighborhood of null infinity.

\bigskip

4. \textbf{(The method of characteristics for $\psi_0$)} The above energy hierarchies allow us to verify the continuation criteria for $\partial_r\psi_{\geq 1}$, $\Omega \psi$ and $T\psi$. For the spherically symmetric derivative $Y \psi_0$, however, we need to use the method of characteristics (this is done in Section \ref{char}) as in \cite{yannis1}. Indeed, if we were to use the energy method then we would need to apply the $(r-M)^{-p}$-weighted commuted estimate for $p=1$. However, it was shown in \cite{aag1} that such an estimate does not hold even in the linear case.  

\bigskip

5. \textbf{($v$-weighted $L^2_{v,\omega} L^{\infty}_u$ estimates)} The bootstrap assumptions cannot be closed using purely the weighted energy hierarchies since this would require to use a range for $p$ that is longer than allowed. For example, consider the following nonlinear term
$$ L \phi_{\geq 1} \cdot \underline{L} \left( \frac{2r}{D} \underline{L} \phi_0 \right). $$
One would ideally want to estimate the $L$ derivative in $L^{\infty}$ and use the commuted $(r-M)^{-p}$-weighted estimates for the second factor with  $p=1+\delta_2$. This is however, not possible since in this case we can only take $p<1$. For this purpose we prove new $v$-weighted $L^2_{v,\omega} L^{\infty}_u$ estimates bounding, for example, quantities such as the following one
$$ \int_{v_0}^{\infty} \int_{\mathbb{S}^2} \sup_{u \in [U , u_{R}(v)]} ( L T^k \Omega^m \phi )^2 \cdot v^{1+\delta} \, d\omega dv , $$
where $k\in\{0.1,2,3\}$, $m\in\{0, 1,2,3,4,5\}$, $\delta>0$, $M < R \leq r_0$ and $r_0 < 2M$ and where $u_{R}(v)$ is such that $r\left(u_R(v),v\right)=R$. The proof of such estimates involves a very delicate use of the bootstrap assumptions as well as the structure of the equation. Note that the  loss of two angular derivatives, introduced by using the wave equation, is overcome by  appropriately integrating by parts on the sphere. The aforementioned estimate can be seen as a weighted Strichartz-type estimate, which, in contrast with other settings, it is proven through physical space energy methods. See also Section \ref{aux} for the details.

\bigskip

6. \textbf{(Growth estimates)} Finally we derive growth estimates for $Y^2\psi$ along the event horizon. More generally, we establish upper and lower bounds for $Y^2 \psi$ \textit{in a region close to the horizon}. The latter bounds are necessary, because in order to recover certain bootstraps assumptions we need to estimate in $L^{\infty}$ the second derivative that is transversal to the horizon of the spherically symmetric part of the solution $\partial_r\partial_r  \psi_0$. Specifically, working in double null coordinates (with respect to which $Y^2 \sim \frac{2r}{D} \underline{L} \left( \frac{2r}{D} \underline{L} \right)$) we show that close to the horizon we have that:
$$ \left| (r-M)^{1-\delta} \frac{2r}{D} \underline{L} \left( \frac{2r}{D} \underline{L} \phi \right) \right| \lesssim \epsilon v^{\delta} , $$
where $\delta \in (0,1]$. The proof of such estimates uses an appropriate version the method of characteristics where we allow for a loss of angular derivatives. These techniques provide new results for the linear flow as well.

\begin{remark}
If we consider data that are \textit{supported away from the event horizon}, then the proof of Theorem \ref{thm:main} can be simplified. There is no need to separate the solution in its spherically symmetric and non-spherically symmetric parts, and there is also no need for the extra estimates described in points 5 and 6 above. This is because we can apply commuted $(r-M)^{-p}$-weighted hierarchy for $\psi$ with $p \in (0,1+\delta]$ for some $\delta > 0$ which yields integrable decay for $\psi$ close to the horizon and boundedness for $\partial_r\psi$. However, the physically relevant case is that of outgoing perturbations with initial support crossing the event horizon. 
\end{remark}

\subsection{Relation with impulsive gravitational wave spacetimes}

It is worth comparing the current work with the construction of  impulsive gravitational wave local spacetimes  by Luk and Rodnianski  \cite{iwaves1, iwaves2}. We will argue that our methods can potentially be used to provide a global study of such spacetimes. 

The impulsive gravitational wave spacetimes are solutions of the Einstein vacuum equations with a delta singularity for the Riemann curvature tensor. Specifically, the authors of \cite{iwaves1}, \cite{iwaves2} considered characteristic initial data  on two null intersecting hypersurfaces $H_{u_0}$ and $\underline{H}_{\underline{u}_0}$ 
such that on $H_{u_0}$ the Riemann curvature has a delta singularity. Optical functions  $u$ and $\underline{u}$ are dynamically constructed with $u$ being ingoing and $\underline{u}$ outgoing-- similar to $u$ and $v$ respectively in the present paper-- with corresponding renormalized null vector fields $e_3$ and $e_4$ that are complemented by the spacelike vector fields $e_A$ and $e_B$ for the angular directions. The level sets $H_u$ and $\underline{H}_{\underline{u}}$ are then null hypersurfaces of constant $u$ and constant $\underline{u}$ coordinates respectively.  In \cite{iwaves1} solutions of $R_{\mu \nu} = 0$ are constructed in the region $u_0 \leq u \leq u_0 + I$, $\underline{u}_0 \leq \underline{u} \leq \underline{u}_0 + \epsilon$ with $\epsilon > 0$ small enough and $I$ finite such that on $H_{u_0}$ the Riemann curvature component $\alpha_{AB} \doteq R ( e_A , e_4 , e_B , e_4 )$ has a delta singularity on $H_{u_0} \cap \{ \underline{u} = \underline{u}_0 + \frac{\epsilon}{2}\}$. Note that the second fundamental form $\chi_{AB} = g (D_A e_4 , e_B )$ has a jump discontinuity on $H_{u_0} \cap \{ \underline{u} = \underline{u}_0 + \frac{\epsilon}{2}\}$ which is propagated along the hypersurfaces $\underline{H}_{\underline{u}_0 + \frac{\epsilon}{2}}$. The metric is smooth away from the singular hypersurface. On the other hand, in \cite{iwaves2}, delta singularities are placed on both $H_{u_0} \cap \{ \underline{u} = \underline{u}_0 + \frac{\epsilon}{2} \}$ (for $\alpha_{AB}$ again) and on $\underline{H}_{\underline{u}_0} \cap \{ u = u_0 + \frac{\epsilon}{2}\}$ (where now $\underline{\alpha}_{AB} \doteq R ( e_A , e_3 , e_B , e_3 )$ has a delta singularity and $\underline{\chi}_{AB} = g (D_A e_3 , e_B )$ has a jump discontinuity) and a local solution of $R_{\mu \nu } = 0$ is constructed in $u_0 \leq u \leq u_0 + \epsilon$, $\underline{u}_0 \leq \underline{u} \leq \underline{u}_0 + \epsilon$ for $\epsilon > 0$ small enough, with the singularity for $\alpha$ propagating along $\underline{H}_{\underline{u}_0 +\frac{\epsilon}{2}}$ and the singularity for $\underline{\alpha}$ propagating along $H_{u_0 + \frac{\epsilon}{2}}$, while the solution is smooth elsewhere.  

To draw some analogies with the problem of the current paper, a  nonlinear model scalar problem is to consider an equation of the form \eqref{nw} with $\epsilon$ not necessarily small (i.e. no small data) on the Minkowski spacetime with data given on two intersecting null hypersurfaces $H_{u_0}$ and $\underline{H}_{\underline{u}_0}$ (with $u$ and $\underline{u}$ the standard double null coordinates) where we assume that  $\partial_u ( r\psi)$ has a jump discontinuity on $\underline{H}_{\underline{u}_0 + \frac{\delta}{2}}$ and that $\partial_{\underline{u}} (r \psi)$ has a jump discontinuity on  $H_{u_0 + \frac{\delta}{2}}$ for some $\delta$ that is small enough. The discontinuities for $\partial_u (r \psi)$ and $\partial_{\underline{u}} ( r\psi)$ will propagate along $H_{u_0 + \frac{\delta}{2}}$ and $\underline{H}_{\underline{u}_0 + \frac{\delta}{2}}$ respectively, while the second derivatives $\partial_{uu}^2 ( r\psi)$ and $\partial_{\underline{u} \underline{u}}^2 ( r\psi ) $ will have delta singularities on these hypersurfaces. Note that the analogies with the fully nonlinear problem for the Einstein equations are at the following level:
$$ \psi \rightsquigarrow g , \quad \partial_u ( r\psi )  \rightsquigarrow \underline{\chi} , \quad \partial_{uu}^2 ( r\psi ) \rightsquigarrow \underline{\alpha}, \quad \partial_{\underline{u}} ( r\psi) \rightsquigarrow \chi , \quad \partial_{\underline{u} \underline{u}}^2 (r\psi ) \rightsquigarrow
\alpha . $$

In our case, the event horizon plays the role of the singular hypersurface (analogous to $H_{u_0 + \frac{\delta}{2}}$ in the aforementioned problem -- note that it is a constant $u$ hypersurface for $u = - \infty$) where the second transversal derivative $\partial_{rr}^2 \psi$ (corresponding to $\partial_{uu}^2 ( r\psi)$ in the problem above, and to the Riemann curvature component $\underline{\alpha}$ in the fully nonlinear problem of \cite{iwaves2}) does \textit{not} have a delta singularity, but exhibits asymptotic blow up. Yet, at the level of techniques, the two problems seem to have a further connection, as one key ingredient of our proof is the weighted estimate described at point 5 of the previous section. This is an $L^2_v L^{\infty}_u L^2 ( \mathbb{S}^2 )$ estimate with $v$-weights for $\partial_v ( r\psi )$ which is a quantity that corresponds to $\partial_{\underline{u}} ( r\psi)$ in the aforementioned problem, and to $\partial_{\underline{u}} g \approx \chi$ in the fully nonlinear problem. From the statement of Theorem 3 in pages 29-30 of \cite{iwaves2} and from the use of the $\mathcal{O}_{i,2}, i \leq 2$ norms from section 2.7 of \cite{iwaves2}, we see that $\partial_{\underline{u}} g$ and $\chi$ are bounded in $L^2_{\underline{u}} L^{\infty}_u L^2 (S)$ and this is a key ingredient in the proof of the main result of \cite{iwaves2} as well. It should be noted that the norms in \cite{iwaves2} are not weighted in $\underline{u}$, but this is only because the problem is local and not global (yet weighted versions of these norms analogous to the ones used in the current paper can be used if instead of a local construction of impulsive gravitational wave spacetimes one attempts to do a semi-global construction of impulsive gravitational wave spacetimes - this construction will be established in an upcoming work of the first author).   
\section{Energy inequalities}\label{energy}
In this section, as well as in the one that follows, we prove certain $L^2$ estimates for general solutions of the equation:
\begin{equation}\label{nwf}
\Box_g \psi = F .
\end{equation}
We define the following regions for any given $\tau_1$, $\tau_2$ with $\tau_1 < \tau_2$:
\begin{equation}\label{mcala}
\mathcal{A}_{\tau_1}^{\tau_2} \doteq \mcR (\tau_1 , \tau_2 ) \cap \{ M \leq r \leq r_0 < 2M \} , 
\end{equation}
\begin{equation}\label{mcalb}
\mathcal{B}_{\tau_1}^{\tau_2} \doteq \mcR (\tau_1 , \tau_2 ) \cap \{ 2M < r_1 \leq r \leq \infty \} , 
\end{equation}
 and
\begin{equation}\label{mcalc}
\mcC_{\tau_1}^{\tau_2} \doteq \mcR (\tau_1 , \tau_2 ) \cap \{ r_0 \leq r \leq r_1 \} ,
\end{equation}
for some fixed $r_0$ and $r_1$, and where $\mcR (\tau_1 , \tau_2 ) = \bigcup_{\tau \in [ \tau_1 , \tau_2 ]} \Sigma_{\tau}$ for $\Sigma_{\tau}$ a null-spacelike-null hypersurface that crosses the event horizon (for the precise definition see section \ref{notation}). We also have the following hypersurfaces
$$ \mcN^H_{\tau} \doteq \Sigma_{\tau} \cap \{ M \leq r \leq r_0 \} , \quad \mcN^I_{\tau} \doteq \Sigma_{\tau} \cap \{ r_1 \leq r \leq \infty \} , $$
and we note that
$$ \mathcal{A}_{\tau_1}^{\tau_2} \doteq \bigcup_{\tau \in [ \tau_1 , \tau_2 ]} \mcN_{\tau}^H , \quad \mathcal{B}_{\tau_1}^{\tau_2} \doteq \bigcup_{\tau \in [ \tau_1 , \tau_2 ]} \mcN_{\tau}^I . $$ 
We will derive $(r-M)^{-p}$-weighted estimates over the hypersurfaces $\mcN^H$ and the spacetime region $\mathcal{A}$, and $r^p$-weigthed estimates over the hypersurfaces $\mcN^I$ and the spacetime region $\mathcal{B}$.
\begin{figure}[H]
\centering
\includegraphics[width=4.5cm]{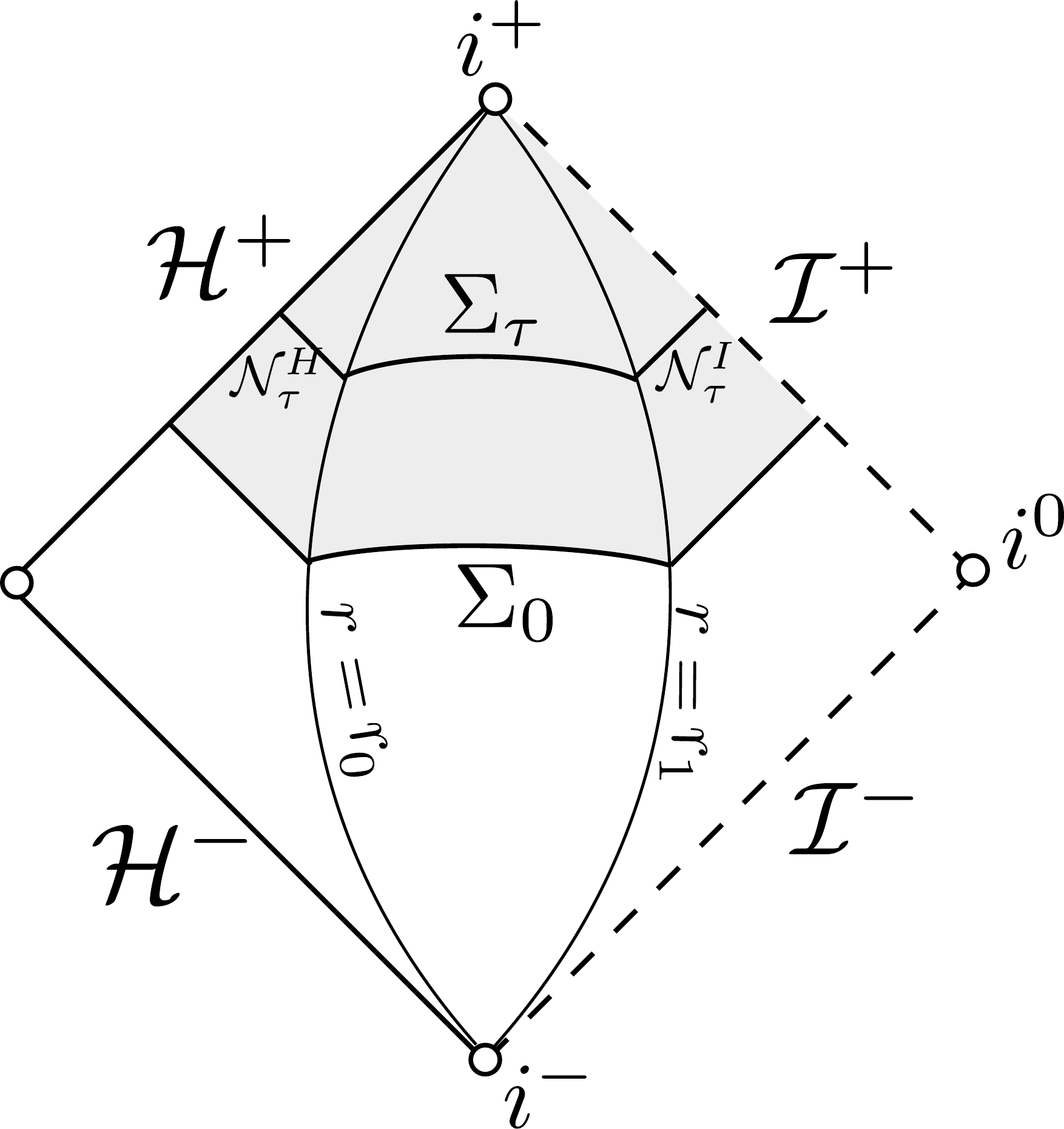}
\end{figure}
Recall that the energy-momentum tensor for the linear wave equation has the form:
$$ \textbf{T}_{\alpha \beta}[\psi ]=\partial_{\alpha}\psi\cdot \partial_{\beta}\psi-\frac{1}{2}g_{\alpha \beta}\partial^{\gamma}\psi\cdot \partial_{\gamma}\psi  , $$
and an energy current is defined as:
$$ ( J^{V_1}[V_2 \psi] )_{\alpha}=\textbf{T}_{\alpha \beta}[V_2\psi]\cdot V_1^{\beta} , $$
for vector fields $V_1$, $V_2$.

\subsection{Morawetz estimates within and outside of spherical symmetry}
First we record a Morawetz estimate for the spherically symmetric part of a solution $\psi$ of \eqref{nwf}, which we denote by 
$$\psi_0 = \frac{1}{4\pi} \int_{\mathbb{S}^2} \psi \, d\omega , $$ 
and which satisfies the equation
\begin{equation}\label{nwf0}
\Box_{g_M} \psi_0 = F_0 ,
\end{equation}
where 
$$ F_0 = \frac{1}{4\pi} \int_{\mathbb{S}^2} F \, d\omega . $$
We have that:
\begin{proposition}\label{prop:mor0}
Let $\psi_0$ be the spherically symmetric part of a solution $\psi$ of \eqref{nwf} which satisfies equation \eqref{nwf0}. For any $\tau_1$, $\tau_2$ with $\tau_1 < \tau_2$ and any $l \in \mathbb{N}$ we have that
\begin{equation}\label{est:mor0}
\begin{split}
\int_{\Sigma_{\tau_2}} J^T [T^l \psi_0 ] \cdot \textbf{n}_{\Sigma}\, d\mu_{\Sigma} + \int_{\mcR_{\tau_1}^{\tau_2}} \Big( \frac{(T T^l \psi_0 )^2}{r^{1+\eta}} + & D^{2} \frac{(YT^l \psi_0)^2}{r^{1+\eta}} + D \frac{( T^l \psi_0 )^2}{r^4}  \Big) \, d\mu_{\mcR} \\ \lesssim & \int_{\Sigma_{\tau_1}} J^T [T^l \psi_0 ] \cdot \textbf{n}_{\Sigma}\, d\mu_{\Sigma}  + \int_{\mcR_{\tau_1}^{\tau_2}} r^{1+\eta} | T^l F_0 |^2 \, d\mu_{\mcR} ,
\end{split}
\end{equation}
for any $\eta > 0$.
\end{proposition}

We now consider the non-spherically symmetric part of a solution of \eqref{nwf}:
$$ \psi_{\geq 1} \doteq \psi - \psi_0 , $$
which in turn satisfies the equation
\begin{equation}\label{nwf1}
\Box_g \psi_{\geq 1} = F_{\geq 1} .
\end{equation}
The difference with the analogous estimates for the spherically symmetric part $\psi_0$ of $\psi$ comes from the trapping effect of the photon sphere (at $r = 2M$) which results in the loss of one or two $T$ derivatives. We state the Morawetz estimate for $\psi_{\geq 1}$ that is supported away from the photon sphere (see \cite{areangel6} for a reference), which has the following form:
\begin{proposition}\label{prop:mor1}
Let $\psi_{\geq 1}$ be the non-spherically symmetric part of a solution $\psi$ of \eqref{nwf} which satisfies equation \eqref{nwf0}. For any $\tau_1$, $\tau_2$ with $\tau_1 < \tau_2$ and any $l,k \in \mathbb{N}$ we have that 
\begin{equation}\label{est:mor1}
\begin{split}
\int_{\Sigma_{\tau_2}} J^T & [\Omega^k T^l \psi_{\geq 1} ]  \cdot \textbf{n}_{\Sigma} d\mu_{\Sigma}  \\ + &  \int_{\mcR (\tau_1 , \tau_2 )} \chi_{(\mcC_{\tau_1}^{\tau_2} )^c} \left( \frac{(T \Omega^k T^l \psi_{\geq 1} )^2 }{r^{1+\eta}} + D^{5/2} \frac{ (Y \Omega^k T^l \psi_{\geq 1} )^2}{r^{1+\eta}} + \sqrt{D} \frac{|\Omega^k T^l \slashed{\nabla} \psi_{\geq 1} |^2}{r}  + D \frac{( \Omega^k T^l \psi )^2}{r^4} \right) \,   d\mu_{\mcR} 
 \\ \lesssim & \int_{\Sigma_{\tau_1}} J^T [\Omega^k T^l \psi_{\geq 1} ] \cdot \textbf{n}_{\Sigma} d\mu_{\Sigma} +   \int_{\mcR_{\tau_1}^{\tau_2}} r^{1+\eta} |\Omega^k T^l F_{\geq 1} |^2 d\mu_{\mcR} \\ & + \int_{\mcC_{\tau_1}^{\tau_2}} |\Omega^k T^{l+1} F_{\geq 1}|^2 d\mu_{\mcC} + \sup_{\tau' \in [\tau_1, \tau_2]} \int_{\Sigma_{\tau'} \cap \mcC_{\tau_1}^{\tau_2}} |\Omega^k T^l F_{\geq 1} |^2 d\mu_{\mcC}, 
 \end{split}
 \end{equation}
where $\mcC$ was defined in \eqref{mcalc}, and where $\chi_{(\mcC_{\tau_1}^{\tau_2} )^{c}}$ is a smooth function that is equal to 1 on the complement of $\mcC$ and 0 around the photon sphere.
\end{proposition}
Next we state two versions of the Morawetz estimate with support on the photon sphere (which can be found in \cite{areangel6}):
\begin{proposition}\label{prop:mor2}
Let $\psi_{\geq 1}$ be the non-spherically symmetric part of a solution $\psi$ of \eqref{nwf} which satisfies equation \eqref{nwf0}. For any $\tau_1$, $\tau_2$ with $\tau_1 < \tau_2$ and any $l,k \in \mathbb{N}$ we have that
\begin{equation}\label{est:mor2}
\begin{split}
\int_{\mcR_{\tau_1}^{\tau_2}} \Big( \frac{(T \Omega^k  T^{l}\psi_{\geq 1} )^2 }{r^{1+\eta}} + &  D^{5/2} \frac{ (Y \Omega^k  T^l \psi_{\geq 1} )^2}{r^{1+\eta}} + \sqrt{D}\frac{|\Omega^k T^l \slashed{\nabla} \psi_{\geq 1} |^2}{r}  + D \frac{(\Omega^k T^l \psi )^2}{r^4} \Big) \,  d\mu_{\mcR} \\ \lesssim &  \sum_{m = l}^{l+1} \left( \int_{\Sigma_{\tau_1}} J^T_{\mu} [\Omega^k T^m \psi_{\geq 1} ] \cdot \textbf{n}_{\Sigma} d\mu_{\Sigma} +  \int_{\mcR_{\tau_1}^{\tau_2}} r^{1+\eta} |\Omega^k T^m F_{\geq 1} |^2 d\mu_{\mcR} \right) \\ & +  \int_{\mathcal{C}_{\tau_1}^{\tau_2}} |\Omega^k T^{l+2} F_{\geq 1} |^2 d\mu_{\mcC} + \sum_{m=l}^{l+1} \sup_{\tau' \in [\tau_1, \tau_2]} \int_{\Sigma_{\tau'} \cap \mcC_{\tau_1}^{\tau_2}} |\Omega^k T^{m}   F_{\geq 1} |^2 d\mu_{\Sigma_{\mcC}}  , 
\end{split}
\end{equation}
and
\begin{equation}\label{est:mor3}
\begin{split}
\int_{\mcR_{\tau_1}^{\tau_2}} \Big( \frac{(T\Omega^k  T^{l}\psi_{\geq 1} )^2 }{r^{1+\eta}} + & D^{5/2} \frac{ (Y\Omega^k  T^l \psi_{\geq 1} )^2}{r^{1+\eta}} + \sqrt{D} \frac{|\Omega^k T^l \slashed{\nabla}  \psi_{\geq 1} |^2}{r} + D \frac{(\Omega^k T^l \psi )^2}{r^4} \Big) \, d\mu_{\mcR} \\ \lesssim &  
\sum_{m = l}^{l+1} \int_{\Sigma_{\tau_1}} J^T [\Omega^k T^m \psi_{\geq 1} ] \cdot \textbf{n}_{\Sigma} d\mu_{\Sigma} + \int_{\mcR_{\tau_1}^{\tau_2}}  r^{1+\eta} |\Omega^k T^l F_{\geq 1}|^2 d\mu_{\mcR} \\ & +   \left( \int_{\tau_1}^{\tau_2} \left( \int_{\Sigma_{\tau'} \cap \mathcal{A}_{\tau_1}^{\tau_2}} |\Omega^k T^{l+1}  F_{\geq 1}|^2 d\mu \right)^{1/2} d\tau' \right)^2  , 
\end{split}
\end{equation}
for any $\eta > 0$, where $\mcC$ was defined in \eqref{mcalc}.
\end{proposition}

\begin{remark}\label{rem:inhommor}
We note that the inhomogeneous terms of the above estimates come from a term of the form
$$ \left| \int_{\mathcal{R}} ( \Omega^k T^l F ) \cdot ( \mathcal{X} \Omega^k T^l \psi )  \, d\mu_{\mathcal{R}} \right|  $$
where $\mathcal{X}$ is the Morawetz multiplier vector field (which close to the horizon roughly has the form $\mathcal{X} \sim T + D \cdot Y$), after applying Cauchy-Schwarz to it and absorbing certain terms in the left hand side. In the following Section we will improve the weights (in terms of $D$) on these terms.
\end{remark}

Finally we state a basic estimate that allows to bound the $T$-flux without any loss of derivatives:
\begin{equation}\label{basic}
\begin{split}
\int_{\Sigma_{\tau_2}} J^T [\Omega^k T^l \psi_{\geq 1} ] \cdot \textbf{n}_{\Sigma} d\mu_{\Sigma} & \lesssim \int_{\Sigma_{\tau_1}} J^T [\Omega^k T^l \psi_{\geq 1} ] \cdot \textbf{n}_{\Sigma} d\mu_{\Sigma} \\ & + \int_{\mcR_{\tau_1}^{\tau_2}} r^{1+\delta} |\Omega^k T^l F_{\geq 1} |^2 d\mu_{\mcR} + \left( \int_{\tau_1}^{\tau_2} \left( \int_{\Sigma_{\tau} \cap \mathcal{C}_{\tau_1}^{\tau_2}} | \Omega^k T^l F |^2 \, d\mu_{\Sigma_{\tau}} \right)^{1/2} \, d\tau \right)^2 ,
\end{split}
\end{equation}
for any $\tau_1 < \tau_2$ and any $\delta > 0$.

\subsection{An improved Morawetz estimate}
We will need to improve the weights close to the horizon on the aforementioned Morawetz estimates.
\begin{proposition}\label{prop:morimproved}
Let $\psi$ be a solution of the equation \eqref{nwf}. Then for any $\tau_1$, $\tau_2$ with $\tau_1 < \tau_2$, $\phi = r\psi$, any $l, k \in \mathbb{N}$, and any $\delta  > 0$ small enough we have that:
\begin{equation}\label{mor:improved}
\begin{split}
\int_{\Sigma_{\tau_2}} J^T [\Omega^k T^l \psi ] & \cdot \textbf{n}_{\Sigma} \,d\mu_{\Sigma} + \int_{\mathcal{B}_{\tau_1}^{\tau_2} } \left[ r^{-1-\delta} ( L \Omega^k T^l \phi )^2 +r^{-1-\delta} ( \underline{L} \Omega^k T^l \phi )^2 + r^{-1}|\Omega^k T^l \slashed{\nabla} \phi |^2\right] \,  d\omega dv du \\ & +  \int_{\mathcal{A}_{\tau_1}^{\tau_2} } \left[ (r-M)^{1+\delta} ( L \Omega^k T^l \phi )^2 + (r-M)^{1+\delta} ( \underline{L} \Omega^k T^l \phi )^2 + (r-M)^3 |\Omega^k T^l \slashed{\nabla} \phi |^2\right] \,  d\omega du dv
 \\ \lesssim & \int_{\Sigma_{\tau_1}} J^T [\Omega^k T^l \psi ] \cdot \textbf{n}_{\Sigma} d\mu_{\Sigma} +   \int_{\mathcal{A}_{\tau_1}^{\tau_2}} \frac{1}{(r-M)^{1+\delta}} D^2 |\Omega^k T^l F|^2 \, d\omega du dv\\ & + \int_{\mcC_{\tau_1}^{\tau_2}} |\Omega^k T^{l+1} F |^2 \, d\mu_{\mcC} + \sup_{\tau' \in [\tau_1, \tau_2]} \int_{\Sigma_{\tau'} \cap \mcC_{\tau_1}^{\tau_2}} |\Omega^k T^l F |^2 \, d\mu_{\mcC}, 
\end{split}
\end{equation}
\end{proposition}
\begin{proof}
For simplicity we look at the case $k = l =0$ as both the $\Omega$ and the $T$ operators commute with the wave operator. We will also ignore the bulk term of the first line in \eqref{mor:improved} as we have the optimal $r$ weights at infinity by the previous Morawetz estimates.  We show how to improve only the weights close to the horizon. First we will improve the weight in front of the $\underline{L}$ derivative. We define the function
$$ f (r ) = -\dfrac{1}{ \eta \cdot \left[ \log (r-M)^{-1} \right]^{\eta}} , $$
for some $\eta > 0$ that is small enough. Now we integrate by parts and use the equation for the following integral:
$$ 2 \int_{\mathcal{A}_{\tau_1}^{\tau_2}} e^f ( L \underline{L} \phi ) \cdot ( \underline{L} \phi ) \, d\omega du dv $$
and this gives us the following equality:
\begin{equation*}
\begin{split}
\int_{\mcN_{\tau_2}^H} e^f ( \underline{L} \phi )^2 \, d\omega du & + \int_{\mathcal{A}_{\tau_1}^{\tau_2}} \frac{D \cdot e^f}{2(r-M) \cdot \left[ \log (r-M)^{-1} \right]^{1+\eta}} ( \underline{L} \phi )^2 \, d\omega du dv \\ = & \int_{\mcN_{\tau_1}^H} e^f ( \underline{L} \phi )^2 \, d\omega du  \\ & + \int_{\mathcal{A}_{\tau_1}^{\tau_2}} \left( -D \cdot D' \cdot e^f - D^2 \cdot e^f \cdot f' + \frac{D^2 \cdot e^f}{2r} \right) | \slashed{\nabla} \phi |^2 \, d\omega du dv \\ & + \left. \int \frac{D}{4} \cdot e^f | \slashed{\nabla} \phi |^2 \, d\omega dv \right|_{r =R} \\ & + \int_{\mathcal{A}_{\tau_1}^{\tau_2}} \mathcal{O} ( (r-M)^3 ) \cdot e^f  \phi \cdot ( \underline{L} \phi ) \, d\omega du dv \\ & - \frac{1}{2} \int_{\mathcal{A}_{\tau_1}^{\tau_2}} e^f ( \underline{L} \phi ) \cdot D r F \, d\omega du dv .
\end{split}
\end{equation*}
In the second term of the right-hand side above involving the angular derivatives, we note that the first term is the dominant one. The term with the angular derivatives on $r=R$ can be bounded by the left hand side of the Morawetz estimates provided by Propositions \ref{prop:mor0} and \ref{prop:mor1}. The fourth term can be handled by Cauchy-Schwarz and by using the zeroth order term of the standard Morawetz estimates \eqref{est:mor2} and \eqref{est:mor3}, and both of the terms can be absorbed by the bulk term of the left hand side. It should be noted that when we use the standard Morawetz estimates of Propositions  \ref{prop:mor0} and \ref{prop:mor1} we work with the inhomogeneous term that was mentioned in Remark \eqref{rem:inhommor} and we apply Cauchy-Schwarz to it with the better weights (in terms of $D$) that are available now from the left hand side of the last equality (otherwise we would get no improvement in terms of $D$-weights in our inhomogeneous terms).

Finally noticing that due to the definition of $\mathcal{A}$ we have that in the integrated region:
$$ c \leq e^f \leq C \mbox{  for some constants $c , C$, } $$
we have that
\begin{equation*}
\begin{split}
\int_{\mcN_{\tau_2}^H} ( \underline{L} \phi )^2 \, d\omega du & + \int_{\mathcal{A}_{\tau_1}^{\tau_2}} \left( \frac{r-M}{\left[ \log (r-M)^{-1} \right]^{1+\eta}} ( \underline{L} \phi )^2 + (r-M)^3 |\slashed{\nabla} \phi |^2 \right) \, d\omega du dv \\ \lesssim & \int_{\mcN_{\tau_1}^H} ( \underline{L} \phi )^2 \, d\omega du + \left| \left. \int D \cdot e^f | \slashed{\nabla} \phi |^2 \, d\omega dv  \right|_{r =R} \right| \\ & + \left| \int_{\mathcal{A}_{\tau_1}^{\tau_2}} e^f ( \underline{L} \phi ) \cdot D r F \, d\omega du dv \right| .
\end{split}
\end{equation*}
Note now that for any $\delta > 0$ we have that:
$$ (r-M)^{\delta} \lesssim \frac{1}{\left[ \log (r-M)^{-1} \right]^{1+\eta}} $$
which implies that for any $\delta > 0$ we have that:
\begin{equation}\label{morimp:1}
\begin{split}
\int_{\mcN_{\tau_2}^H} ( \underline{L} \phi )^2 \, d\omega du & + \int_{\mathcal{A}_{\tau_1}^{\tau_2}} \left( (r-M)^{1+\delta} ( \underline{L} \phi )^2 + (r-M)^3 |\slashed{\nabla} \phi |^2 \right) \, d\omega du dv \\ \lesssim & \int_{\mcN_{\tau_1}^H} ( \underline{L} \phi )^2 \, d\omega du + \left| \left. \int D \cdot e^f | \slashed{\nabla} \phi |^2 \, d\omega dv  \right|_{r =R} \right| \\ & + \left| \int_{\mathcal{A}_{\tau_1}^{\tau_2}} e^f ( \underline{L} \phi ) \cdot D r F \, d\omega du dv \right| .
\end{split}
\end{equation}
On the other hand we integrate by parts and we use the equation for the quantity:
$$ - \int_{\mathcal{A}_{\tau_1}^{\tau_2}}  \underline{L} \left[ (r-M)^{\delta} ( L \phi )^2 \right] \, d\omega du dv $$
and we get that:
\begin{equation}\label{morimp:2}
\begin{split}
\frac{\delta}{2} \int_{\mathcal{A}_{\tau_1}^{\tau_2}} & \frac{(r-M)^{1+\delta}}{r^2} ( L \phi )^2 \, d\omega du dv  + \int_{\mathcal{A}_{\tau_1}^{\tau_2}} \mathcal{O} (( r-M)^{3+\delta}) | \slashed{\nabla} \phi |^2 \, d\omega du dv \\ & + \int_{\mcN_{\tau_2}^H} \frac{(r-M)^{2+\delta }}{4r^2} | \slashed{\nabla} \phi |^2 \, d\omega du  =  - \left. \int (r-M)^{\delta} ( L \phi )^2 \, d\omega dv \right|_{r=R} \\ & + \int_{\mcN_{\tau_1}^H} \frac{(r-M)^{2+\delta}}{4r^2} | \slashed{\nabla} \phi |^2 \, d\omega du   + \int_{\mathcal{A}_{\tau_1}^{\tau_2}} \mathcal{O} (( r-M)^{3+\delta}) \phi \cdot (L \phi ) \, d\omega du dv \\ & + \frac{1}{2} \int_{\mathcal{A}_{\tau_1}^{\tau_2}} (r-M)^{\delta} ( L \phi ) \cdot D r F \, d\omega du dv   .
\end{split}
\end{equation}
Combining \eqref{morimp:1} and \eqref{morimp:2}, and noticing that the bulk term with the angular derivatives in the left hand side of \eqref{morimp:2} can be absorbed from the similar term of \eqref{morimp:1} (as it has a bigger $(r-M)$ weight), that the flux terms with the angular derivatives of \eqref{morimp:2} can be absorbed by the $T$-fluxes (i.e. the $T$-energies) of \eqref{morimp:1}, that the term on $r=R$ of \eqref{morimp:2} can be bounded by the left hand sides of the Morawetz estimate \eqref{est:mor3}, that the term before the last term of the right hand side of \eqref{morimp:2} can be absorbed by the left hand side of \eqref{morimp:2} and \eqref{morimp:1} after applying Cauchy-Schwarz while the one before last (after applying again Cauchy-Schwarz) can be absorbed by the left hand side and the zeroth order term of the standard Morawetz, and finally that for the terms of the left hand side of the Morawetz estimate \eqref{est:mor3} we have that due to the $(r-M)$-weights on the horizon:
\begin{equation*}
\begin{split}
\int_{\mathcal{A}_{\tau_1}^{\tau_2}} \Big( (r-M)^{1+\delta} (L\phi)^2 + &  (r-M)^{1+\delta} ( \underline{L} \phi )^2 + (r-M)^3 |\slashed{\nabla} \phi |^2 \Big) \, d\omega du dv \\ & - \int_{\mathcal{A}_{\tau_1}^{\tau_2}} \left( (r-M)^2 (T \psi)^2 + (r-M)^7 ( Y \psi )^2 + (r-M)^3 | \slashed{\nabla} \psi |^2 \right) \, r^2 d\omega du dv \\ \gtrsim &   \int_{\mathcal{A}_{\tau_1}^{\tau_2}} \Big( (r-M)^{1+\delta} (L\phi)^2 +   (r-M)^{1+\delta} ( \underline{L} \phi )^2 + (r-M)^3 |\slashed{\nabla} \phi |^2 \Big) \, d\omega du dv ,
\end{split}
\end{equation*}
we get that
\begin{equation*}
\begin{split}
\int_{\Sigma_{\tau_2}} J^T [\Omega^k T^l \psi ] & \cdot \textbf{n}_{\Sigma} \,d\mu_{\Sigma} \\ & +  \int_{\mathcal{A}_{\tau_1}^{\tau_2} } \left[ (r-M)^{1+\delta} ( L \Omega^k T^l \phi )^2 + (r-M)^{1+\delta} ( \underline{L} \Omega^k T^l \phi )^2 + (r-M)^3 |\Omega^k T^l \slashed{\nabla} \phi |^2\right] \,  d\omega du dv
 \\ \lesssim & \int_{\Sigma_{\tau_1}} J^T [\Omega^k T^l \psi ] \cdot \textbf{n}_{\Sigma} d\mu_{\Sigma}  \\ & + \int_{\mcC_{\tau_1}^{\tau_2}} |\Omega^k T^{l+1} F |^2 \, d\mu_{\mcC} + \sup_{\tau' \in [\tau_1, \tau_2]} \int_{\Sigma_{\tau'} \cap \mcC_{\tau_1}^{\tau_2}} |\Omega^k T^l F |^2 \, d\mu_{\mcC} \\ & +  \int_{\mathcal{A}_{\tau_1}^{\tau_2}}  ( |\underline{L} \phi | + | L \phi | ) \cdot D r |F | \, d\omega du dv .
\end{split}
\end{equation*}
We get the desired estimate after applying Cauchy-Schwarz to the last term. 
 
\end{proof}

With the improved Morawetz estimate that we just showed we can also improve the estimate \eqref{basic} for the $T$-flux that does not lose any derivative and conclude that:
\begin{equation}\label{basicimp}
\begin{split}
\int_{\Sigma_{\tau_2}} J^T [\Omega^k T^l \psi ] & \cdot \textbf{n}_{\Sigma} d\mu_{\Sigma}+  \int_{\mathcal{B}_{\tau_1}^{\tau_2} } \left[ r^{-1-\eta} ( L \Omega^k T^l \phi )^2 +r^{-1-\eta} ( \underline{L} \Omega^k T^l \phi )^2 + r^{-1}|\Omega^k T^l \slashed{\nabla} \phi |^2\right] \,  d\omega dv du  \\ & +  \int_{\mathcal{A}_{\tau_1}^{\tau_2} } \left[ (r-M)^{1+\delta} ( L \Omega^k T^l \phi )^2 + (r-M)^{1+\delta} ( \underline{L} \Omega^k T^l \phi )^2 + (r-M)^3 |\Omega^k T^l \slashed{\nabla} \phi |^2\right] \,  d\omega du dv  \\ \lesssim & \int_{\Sigma_{\tau_1}} J^T [\Omega^k T^l \psi ] \cdot \textbf{n}_{\Sigma} d\mu_{\Sigma} + \left( \int_{\tau_1}^{\tau_2} \left( \int_{\Sigma_{\tau} \cap \mathcal{C}_{\tau_1}^{\tau_2}} | \Omega^k T^l F  |^2 \, d\mu_{\Sigma} \right)^{1/2} \, d\tau \right)^2\\ & +  \int_{\tau_1}^{\tau_2} \int_{\mcN_{v}^H} \frac{1}{(r-M)^{-1-\eta}} D^2 |\Omega^k T^l F |^2 d\omega dv du+\int_{\tau_1}^{\tau_2} \int_{\mcN_{u}^I} r^{1+\eta} |\Omega^k T^l F |^2 d\omega dv du  ,
\end{split}
\end{equation}
for any $\tau_1 < \tau_2$ and any $\eta > 0$.

\subsection{$(r-M)^{-p}$-weighted estimates and $r^p$-weighted estimates}
From \cite{paper4} we have the following $(r-M)^{-p}$-weigthed estimates at the horizon and the $r^p$-weighted estimates at infinity that can be summarized in the following propositions that will be used to show decay for these weighted energies later in the article. We define the following quantities:
$$ \Phi_0^H \doteq \frac{2r}{D} \underline{L} \phi_0 , \quad \Phi_{\geq 1}^H \doteq \frac{2r}{D} \underline{L} \phi_{\geq 1} , \quad \Phi_0^I \doteq \frac{2r^2}{D} L \phi_0, \quad \Phi_{\geq 1}^I \doteq \frac{2r^2}{D} L \phi_{\geq 1} . $$
We have the following proposition for $\psi_0$ (with $\phi_0 = r \psi_0$):
\begin{proposition}\label{prop:rp05}
Let $\psi_0$ be the spherically symmetric part of a solution $\psi$ of \eqref{nwf} which satisfies equation \eqref{nwf0}, and let $\phi \doteq r \psi_0$. For any $\tau$, $\tau_1$, $\tau_2$ with $\tau_1 < \tau_2$, $p\in (0,3)$, and any $l \in \mathbb{N}$ for the quantities
\begin{equation}\label{def:enp0I}
I_{0,l}^p (\tau ) = \int_{\mcN_{\tau}^H} (r-M)^{-p}  (  \underline{L} T^l \phi_0 )^2 \, d\omega du + \int_{\Sigma_{\tau} \setminus ( \mcN_{\tau}^H \cup \mcN_{\tau}^I )} J^T [ T^l \psi_0 ] \cdot \textbf{n}_{\Sigma_{\tau}} \, d\mu_{\Sigma_{\tau}}  +  \int_{\mcN_{\tau}^I} r^p (  L T^l \phi_0 )^2 \, d\omega dv
\end{equation}
we have that
\begin{equation}\label{est:rp05}
\begin{split}
I_{0,l}^p (\tau_2 ) + \int_{\tau_1}^{\tau_2} & I_{0,l}^{p-1} (\tau' ) \, d\tau' \lesssim I_{0,l}^p (\tau_1 ) +\int_{\tau_1}^{\tau_2} \int_{\mcN_v^H} \frac{1}{(r-M)^{1+\eta}} D^2 | T^l F_0 |^2 \, d\omega du dv+ \int_{\tau_1}^{\tau_2} \int_{\mcN_u^I} r^{1+\eta} | T^l F_0 |^2 \, d\omega dv du  \\ + &  \left| \int_{\tau_1}^{\tau_2} \int_{\mcN_u^I} (r-M)^{-p} ( \underline{L} T^l \phi_0 ) \cdot \frac{Dr}{4} (T^l F_0 ) \, d\omega dv du \right| + \left| \int_{\tau_1}^{\tau_2} \int_{\mcN_v^H} r^p ( L T^l \phi_0 ) \cdot (T^l F_0 ) \, d\omega du dv \right| \\ + &  \int_{\mcC_{\tau_1}^{\tau_2}} | T^{l} F_0 |^2 \, d\mu_{\mcC} ,
\end{split}
\end{equation}
for any $\eta > 0$.
\end{proposition} 
We also have the following proposition for $\Phi_0^H$ and $\Phi_0^I$:
\begin{proposition}\label{prop:rp06}
Let $\psi_0$ be the spherically symmetric part of a solution $\psi$ of \eqref{nwf} which satisfies equation \eqref{nwf0}, and let $\phi \doteq r \psi_0$. For any $\tau$, $\tau_1$, $\tau_2$ with $\tau_1 < \tau_2$, $p \in (0,1)$, and any $l \in \mathbb{N}$ for the quantities
\begin{equation}\label{def:enp0II}
II_{0,l}^p (\tau ) = \int_{\mcN_{\tau}^H} (r-M)^{-p}  (  \underline{L} T^l \Phi_0^H )^2 \, d\omega du + \int_{\Sigma_{\tau} \setminus ( \mcN_{\tau}^H \cup \mcN_{\tau}^I )} J^T [ T^{l+1} \psi_0 ] \cdot \textbf{n}_{\Sigma_{\tau}} \, d\mu_{\Sigma_{\tau}} +   \int_{\mcN_{\tau}^I} r^p  (  L T^l \Phi_0^I )^2 \, d\omega dv 
\end{equation}
we have that
\begin{equation}\label{est:rp06}
\begin{split}
II_{0,l}^p (\tau_2 ) + \int_{\tau_1}^{\tau_2} II_{0,l}^{p-1} (\tau' ) \, d\tau' & \lesssim II_{0,l}^p (\tau_1 ) + \int_{\tau_1}^{\tau_2} \int_{\mcN_v^H} \frac{1}{(r-M)^{1+\eta}} D^2 | T^{l+1} F_0 |^2 \, d\omega du dv \\ & + \int_{\tau_1}^{\tau_2} \int_{\mcN_u^I} r^{1+\eta} | T^{l+1} F_0 |^2 \, d\omega dv du\\ & + \left| \int_{\tau_1}^{\tau_2} \int_{\mcN_v^H}  (r-M)^{-p} (\underline{L} T^l \Phi_0^H ) \cdot \underline{L} ( r^2 T^l F_0 ) \, d\omega du dv \right| \\ & + \left| \int_{\tau_1}^{\tau_2} \int_{\mcN_u^I}  r^{p}  ( L T^l \Phi_0^I ) \cdot L ( r^3 T^l F_0 ) \, d\omega dv du \right| \\ & + \int_{\mcC_{\tau_1}^{\tau_2}} | T^{l+1} F_0 |^2 \, d\mu_{\mcC} .
\end{split}
\end{equation}
\end{proposition}
Analogously we have the following for $\phi_{\geq 1}$:
\begin{proposition}\label{prop:rp07}
Let $\psi_{\geq 1}$ be the non-spherically symmetric part of a solution $\psi$ of \eqref{nwf} which satisfies equation \eqref{nwf1}, and let $\phi \doteq r \psi_0$. For any $\tau$, $\tau_1$, $\tau_2$ with $\tau_1 < \tau_2$, $p\in (0,2)$, any $k \leq 5$, any $\eta > 0$, and any $l \in \mathbb{N}$ for the quantities
\begin{equation}\label{def:enp1}
I_{\geq 1 ,k,l}^p (\tau ) = \int_{\mcN_{\tau}^H} (r-M)^{-p}  (  \underline{L} \Omega^k T^l \phi_{\geq 1} )^2 \, d\omega du + \int_{\Sigma_{\tau} \setminus ( \mcN_{\tau}^H \cup \mcN_{\tau}^I ) } J^T [ \Omega^k T^l \psi_{\geq 1} ] \cdot \textbf{n}_{\Sigma} \, d\mu_{\Sigma} + \int_{\mcN_{\tau}^I} r^p  (  L \Omega^k T^l \phi_{\geq 1} )^2 \, d\omega dv 
\end{equation}
we have that
\begin{equation}\label{est:rp07}
\begin{split}
I_{\geq 1 ,k,l}^p (\tau_2 ) + & \int_{\tau_1}^{\tau_2} I_{\geq 1,k,l}^{p-1} (\tau' ) \, d\tau' + (2-p) \int_{\tau_1}^{\tau_2} \int_{\mcN_{v}^H} (r-M)^{-p+3} | \slashed{\nabla} \Omega^k T^l \phi_{\geq 1} |^2 \, d\omega du dv  \\ \lesssim &  I_{\geq 1 ,k,l}^p (\tau_1 ) + \int_{\tau_1}^{\tau_2} \int_{\mcN_v^H} \frac{1}{(r-M)^{1+\eta}} D^2 | \Omega^k T^l F_{\geq 1} |^2 \, d\omega du dv\\ + &  \left| \int_{\tau_1}^{\tau_2} \int_{\mcN_v^H} (r-M)^{-p} ( \underline{L} \Omega^k T^l \phi_{\geq 1} ) \cdot \frac{Dr}{4} (\Omega^k T^l F_{\geq 1} ) \, d\omega du dv \right| \\ + &  \int_{\tau_1}^{\tau_2} \int_{\mcN_u^I} r^{1+\eta} | \Omega^k T^l F_{\geq 1} |^2 \, d\omega dv du + \int_{\mathcal{C}_{\tau_1}^{\tau_2}} | \Omega^k T^{l+2} F_{\geq 1} |^2 \, d\mu_{\mathcal{C}} + \sup_{\tau' \in [\tau_1, \tau_2]} \int_{\Sigma_{\tau'} \cap \mcC_{\tau_1}^{\tau_2}} |\Omega^k T^{l+1} F_{\geq 1} |^2 \, d\mu_{\mcC} .
\end{split}
\end{equation}
\end{proposition} 
Finally we have the following for $\Phi_{\geq 1}^H$ and $\Phi_{\geq 1}^I$:
\begin{proposition}\label{prop:rp08}
Let $\psi_{\geq 1}$ be the non-spherically symmetric part of a solution $\psi$ of \eqref{nwf} which satisfies equation \eqref{nwf1}, and let $\phi \doteq r \psi_0$. For any $\tau$, $\tau_1$, $\tau_2$ with $\tau_1 < \tau_2$, $p \in (0,2]$ and any $l \in \mathbb{N}$ for the quantities
\begin{equation}\label{def:enp11}
II_{\geq 1 , k, l}^p (\tau ) = \int_{\mcN_{\tau}^H} (r-M)^{-p}  (  \underline{L} \Omega^k T^l \Phi_{\geq 1}^H )^2 \, d\omega du + \int_{\Sigma_{\tau} \setminus ( \mcN_{\tau}^H \cup \mcN_{\tau}^I ) } J^T [ \Omega^k T^{l+1} \psi_{\geq 1} ] \cdot \textbf{n}_{\Sigma} \, d\mu_{\Sigma} +   \int_{\mcN_{\tau}^I} r^p  (  L \Omega^k T^l \Phi_{\geq 1}^I )^2 \, d\omega dv
\end{equation}
we have that
\begin{equation}\label{est:rp08}
\begin{split}
II_{\geq 1 ,k,l}^p (\tau_2 ) + & \int_{\tau_1}^{\tau_2} II_{\geq 1 ,k,l}^{p-1} (\tau' ) \, d\tau' + \int_{\tau_1}^{\tau_2} \int_{\mcN_{v}^H} (r-M)^{-p+3} | \slashed{\nabla} \Omega^k T^l \Phi_{\geq 1}^H |^2 \, d\omega du dv \\ \lesssim & II_{\geq 1 ,k,l}^p (\tau_1 ) + \int_{\tau_1}^{\tau_2} \int_{\mcN_v^H} \frac{1}{(r-M)^{1+\eta}} D^2 | \Omega^k T^{l+1} F_{\geq 1} |^2 \, d\omega du dv+ \int_{\tau_1}^{\tau_2} \int_{\mcN_u^I} r^{1+\eta} | \Omega^k T^{l+1} F_{\geq 1} |^2 \, d\omega dv du \\ & + \left| \int_{\tau_1}^{\tau_2} \int_{\mcN_v^H}  (r-M)^{-p} (\underline{L} \Omega^k T^{l} \Phi_{\geq 1}^H ) \cdot \underline{L} ( r^2 \Omega^k T^{l+1} F_{\geq 1} ) \, d\omega du dv \right| \\ & + \int_{\tau_1}^{\tau_2} \int_{\Sigma_{\tau} \setminus \mcN_{\tau}^H} r^{1+\eta} | \Omega^k T^{l+1} F_{\geq 1} |^2 \, d\omega dv du + \int_{\mathcal{C}_{\tau_1}^{\tau_2}} | \Omega^k T^{l+3} F_{\geq 1} |^2 \, d\mu_{\mcC} \\ & + \sup_{\tau' \in [\tau_1, \tau_2]} \int_{\Sigma_{\tau'} \cap \mcC_{\tau_1}^{\tau_2}} |\Omega^k T^{l+2} F_{\geq 1} |^2 \, d\mu_{\Sigma_{\tau'}} .
\end{split}
\end{equation}
\end{proposition}
\begin{remark}
The estimates of the last Proposition \ref{prop:rp08} hold also without the need to restrict to higher angular frequencies if $l \geq 1$.
\end{remark}

We note that we also have separate $(r-M)^{-p}$ and $r^p$ weighted estimates for any $\tau_1$, $\tau_2$ with $\tau_1 < \tau_2$, any $l \in \mathbb{N}$ and any $\eta > 0$. For $\phi_0$ close to the horizon we have that:
\begin{equation}\label{est:rp01}
\begin{split}
\int_{\mcN_{\tau_2}^H} & (r-M)^{-p}  (  \underline{L} T^l \phi_0 )^2 \, d\omega du +  \int_{\tau_1}^{\tau_2} \int_{\mcN_v^H}  (r-M)^{-p+1} (  \underline{L} T^l \phi_0 )^2 \, d\omega du dv \\ \lesssim &  \int_{\mcN_{\tau_1}^H} (r-M)^{-p} ( \underline{L} T^l \phi_0 )^2 \, d\omega du + \int_{\Sigma_{\tau_1}} J^T [T^{l} \psi_0 ] \cdot \textbf{n}_{\Sigma_{\tau_1}} \, d\mu_{\Sigma_{\tau_1}} \\ &  + \left| \int_{\tau_1}^{\tau_2} \int_{\mcN_v^H}  (r-M)^{-p} (\underline{L} T^l \phi_0 ) \cdot \frac{Dr}{4}  ( T^l F_0 ) \, d\omega du dv \right| \\ & + \int_{\tau_1}^{\tau_2} \int_{\mcN_v^H} \frac{1}{(r-M)^{1+\eta}} D^2 | T^l F_0 |^2 \, d\omega du dv+ \int_{\tau_1}^{\tau_2} \int_{ \mcN_u^I} r^{1+\eta} | T^l F_0 |^2 \, d\omega dv du  ,
\end{split}
\end{equation}
for $p \in (0,3)$. For $\Phi_0^H$ close to the horizon we have that:
\begin{equation}\label{est:rp02}
\begin{split}
\int_{\mcN_{\tau_2}^H} & (r-M)^{-p}  (  \underline{L} T^l \Phi_0^H )^2 \, d\omega du +  \int_{\tau_1}^{\tau_2} \int_{\mcN_v^H}  (r-M)^{-p+1} (  \underline{L} T^l \Phi_0^H )^2 \, d\omega du dv \\ \lesssim &  \int_{\mcN_{\tau_1}^H} (r-M)^{-p} ( \underline{L} T^l \Phi_0^H )^2 \, d\omega du + \int_{\mcN_{\tau_1}^H} (r-M)^{-p-2} ( \underline{L} T^l \phi_0 )^2 \, d\omega du + \int_{\Sigma_{\tau_1}} J^T [T^{l+1} \psi_0 ] \cdot \textbf{n}_{\Sigma} \, d\mu_{\Sigma} \\ & +  \left| \int_{\tau_1}^{\tau_2} \int_{\mcN_v^H} (r-M)^{-p-2} ( \underline{L} T^l \phi_0 ) \cdot  ( D T^l F_0 ) \, d\omega du dv  \right| + \left| \int_{\tau_1}^{\tau_2} \int_{\mcN_v^H}  (r-M)^{-p} (\underline{L} T^l \Phi_0^H ) \cdot \underline{L} ( r^2 T^l F_0 ) \, d\omega du dv \right| \\ & + \sum_{m=l}^{l+1} \left( \int_{\tau_1}^{\tau_2} \int_{\mcN_v^H} \frac{1}{(r-M)^{1+\eta}} D^2 | T^m F_0 |^2 \, d\omega du dv+ \int_{\tau_1}^{\tau_2} \int_{\mcN_u^I} r^{1+\eta} | T^m F_0 |^2 \, d\omega du dv \right) \\ & + \sum_{m=l}^{l+1} \int_{\mathcal{C}_{\tau_1}^{\tau_2}} | T^m F_0 |^2 \, d\mu_{\mcC}  ,
\end{split}
\end{equation}
for $p \in (0,1)$. For $\psi_0$ at infinity we have the following estimates:
\begin{equation}\label{est:rp03}
\begin{split}
\int_{\mcN_{\tau_2}^I} r^p & (  L T^l \phi_0 )^2 \, d\omega dv +  \int_{\tau_1}^{\tau_2} \int_{\mcN_u^I}  r^{p-1} (  L T^l \phi_0 )^2 \, d\omega dv du \\ \lesssim &  \int_{\mcN_{\tau_1}^I} r^p ( L T^l \phi_0 )^2 \, d\omega du + \int_{\Sigma_{\tau_1}} J^T [T^l \psi_0 ] \cdot \textbf{n}_{\Sigma} \, d\mu_{\Sigma} \\ & + \left| \int_{\tau_1}^{\tau_2} \int_{\mcN_u^I} r^{p} | ( L T^l \phi_0 ) \cdot ( T^l F_0 ) \, d\omega dv du \right| \\ & + \int_{\tau_1}^{\tau_2} \int_{N_v^H} \frac{1}{(r-M)^{1+\eta}} D^2 | T^l F_0 |^2 \, d\omega du dv + \int_{\tau_1}^{\tau_2} \int_{\mcN_u^I} r^{1+\eta} | T^l F_0 |^2 \, d\omega dv du + \int_{\mcC_{\tau_1}^{\tau_2}} | T^l F_0 |^2 \, d\mu_{\mcC} ,
\end{split}
\end{equation}
for $p \in (0,3)$. For $\Phi_0^I$ at infinity we have that:
\begin{equation}\label{est:rp04}
\begin{split}
\int_{\mcN_{\tau_2}^I} r^p & (  L T^l \Phi_0^I )^2 \, d\omega dv +  \int_{\tau_1}^{\tau_2} \int_{\mcN_u^I}  r^{p-1} (  L T^l \Phi_0^I )^2 \, d\omega dv du \\ \lesssim &  \int_{\mcN_{\tau_1}^I} r^p ( L T^l \Phi_0^I )^2 \, d\omega du + \int_{\mcN_{\tau_1}^I} r^{p+2} ( L T^l \phi_0 )^2 \, d\omega du + \int_{\Sigma_{\tau_1}} J^T [T^{l+1} \psi_0 ] \cdot \textbf{n}_{\Sigma} \, d\mu_{\Sigma} \\ & +\left| \int_{\tau_1}^{\tau_2} \int_{\mcN_u^I} r^{p+2} | ( L T^l \phi_0 ) \cdot ( T^l F_0 ) \, d\omega dv du \right| + \left| \int_{\tau_1}^{\tau_2} \int_{\mcN_u^I} r^p ( L T^l \Phi_0^I ) \cdot ( L ( r^3 T^l F_0 ) \, d\omega dv du \right|  \\ & +\sum_{m=l}^{l+1} \left( \int_{\tau_1}^{\tau_2} \int_{\mcN_v^H} \frac{1}{(r-M)^{1+\eta}} D^2 | T^m F_0 |^2 \, d\omega du dv+ \int_{\tau_1}^{\tau_2} \int_{\mcN_u^I} r^{1+\eta} | T^m F_0 |^2 \, d\omega dv du \right) \\ & + \sum_{m=l}^{l+1} \int_{\mcC_{\tau_1}^{\tau_2}} | T^m F_0 |^2 \, d\mu_{\mcC} ,
\end{split}
\end{equation}
for $p \in (0,1)$. 

Analogous estimates hold for $\phi_{\geq 1}$ close to the horizon (as \eqref{est:rp01} for $p\in (0,2)$) and close to infinity (as \eqref{est:rp03} for $p \in (0,2)$), for $\Phi_{\geq 1}^H$ close to the horizon (as \eqref{est:rp02} for $p \in (0,2)$ without the uncommuted terms with weight $(r-M)^{-p-2}$), and for $\Phi_{\geq 1}^I$ close to infinity (as \eqref{est:rp04} for $p \in (0,2)$ without the uncommuted terms with weight $r^{p+2}$), with extra terms coming from the trapping on the photon sphere (losing either one or no derivatives as we do not consider integrated quantities over the photon sphere).

\section{Bootstrap assumptions}\label{as:boot}
From this section and on we will assume that our nonlinearity has the form given in equation \eqref{nw} without the cubic and higher order terms (that are easier to deal with), hence we will use all the estimates that were presented in all the sections so far for
\begin{equation*}
F \doteq A (u,v,\omega, \psi ) \cdot g^{\alpha \beta} \cdot  \partial_{\alpha}  \psi  \cdot \partial_{\beta} \psi ,
\end{equation*}
for $A$ as defined before in the statement of Theorem \ref{thm:main}. We set $F_0 \doteq F_{\ell = 0}$ and $F_{\geq 1} \doteq F_{\ell \geq 1}$. We also set
\begin{equation*}
F^c \doteq  g^{\alpha \beta} \cdot  \partial_{\alpha}  \psi  \cdot \partial_{\beta} \psi .
\end{equation*}

We will assume the following estimates in all the remaining section of the paper for $C$ a constant, $E_0$ the initial energy as defined in Appendix \ref{norm:add}, $\delta_1 , \delta_2, \beta_0 > 0$ small enough, for some $\epsilon > 0$, for $0 < \beta < \delta_2$ and $0 < \delta \leq \delta_2$, and for $\mathcal{L}$ being any ``linear" term among the ones that show up on the left hand side of inequality \eqref{dec:en}-\eqref{dec:mor5} (so when we want to show these estimates the ``linear" $\mathcal{L}$ terms from the inhomogeneities can be absorbed in the left hand side of inequality that is used). After examining their implications, we will verify their validity through a bootstrap argument (the letters used below roughly correspond to the number of $T$ derivatives).
\begin{equation}\label{A1}
 \int_{\mcR_{\tau_1}^{\tau_2}} r^{2} | F_0 |^2 \, d\mu_{\mcR}  \lesssim  \frac{C E_0 \epsilon^2}{ (1+\tau_1 )^{3-\delta_1}} \tag{\textbf{A1}},
 \end{equation} 
\begin{equation}\label{A2}
\int_{\tau_1}^{\tau_2} \int_{\mcN_v^H} (r-M)^{-p-1} D^2 | F_0 |^2 \, d\omega du dv \lesssim \frac{C E_0 \epsilon^2}{ (1+\tau_1 )^{3-\delta_1-p}} \mbox{  for $p\in (0,2-\delta_1 ]$} \tag{\textbf{A2}},
\end{equation}
\begin{equation}\label{A3}
\left| \int_{\tau_1}^{\tau_2} \int_{\mcN_v^H} (r-M)^{-p} ( \underline{L} \phi_0 ) \cdot  ( D F_0 )  \, d\omega du dv \right| \lesssim \frac{C E_0 \epsilon^2}{ (1+\tau_1 )^{3-\delta_1-p}} \mbox{  for $p\in (2,3-\delta_1 ]$} \tag{\textbf{A3}},
\end{equation}
\begin{equation}\label{A4}
\left| \int_{\tau_1}^{\tau_2} \int_{\mcN_v^H}  (r-M)^{-p} (\underline{L}  \Phi_0^H ) \cdot \underline{L} ( r^2  F_0 ) \, d\omega du dv \right| \lesssim \beta_0 \mathcal{L} + \frac{C E_0 \epsilon^2}{ (1+\tau_1 )^{1-\delta_1-p}} \mbox{  for $p \in (0,1-\delta_1]$} \tag{\textbf{A4}} ,
\end{equation}
\begin{equation}\label{A5}
\begin{split}
 \int_{\mcR_{\tau_1}^{\tau_2}   \setminus \mathcal{A}_{\tau_1}^{\tau_2}} &  r^{p+1} | F_0 |^2 \, d\mu_{\mcR}  \lesssim  \frac{C E_0 \epsilon^2}{ (1+\tau_1 )^{3-\delta_1-p}} \mbox{  for $p \in (1,2-\delta_1 ]$, and } \\ & \left( \int_{\tau_1}^{\tau_2} \left( \int_{\mcN_u^I} r^p | F_0 |^2 \, d\omega dv \right)^{1/2} du \right)^2 \lesssim \frac{C E_0 \epsilon^2}{ (1+\tau_1 )^{3-\delta_1-p}} \mbox{  for $p \in (2,3-\delta_1 ]$} 
\end{split} 
\tag{\textbf{A5}},
\end{equation} 
\begin{equation}\label{A6}
\left| \int_{\tau_1}^{\tau_2} \int_{\mcN_u^I} r^p ( L \Phi_0^I ) \cdot  L ( r^3  F_0 ) \, d\omega dv du \right| \lesssim \beta_0 \mathcal{L} + \frac{C E_0 \epsilon^2}{ (1+\tau_1 )^{1-\delta_1-p}} \mbox{  for $p \in (0,1-\delta_1 ]$} \tag{\textbf{A6}},
\end{equation}
\begin{equation}\label{B1}
 \sum_{k \leq 5} \left( \int_{\mathcal{A}_{\tau_1}^{\tau_2}} (r-M)^{-1-\delta} D^2 |\Omega^k F_{\geq 1} |^2 \, d\omega du dv + \int_{\mcR_{\tau_1}^{\tau_2} \setminus \mathcal{A}_{\tau_1}^{\tau_2}} r^{2} | \Omega^k F_{\geq 1} |^2 \, d\mu_{\mcR} \right) \lesssim  \frac{C E_0 \epsilon^2}{ (1+\tau_1 )^{3+\delta_2}} \tag{\textbf{B1}},
 \end{equation} 
 \begin{equation}\label{B2}
\sum_{k \leq 5} \left| \int_{\tau_1}^{\tau_2} \int_{\mcN_v^H}  (r-M)^{-p} ( \underline{L} \Omega^k \phi_{\geq 1} ) \cdot D  ( \Omega^k F_{\geq 1} ) \, d\omega du dv \right|  \lesssim \beta_0 \mathcal{L} + \frac{C E_0 \epsilon^2}{ (1+\tau_1 )^{3+\delta_2-p}} \mbox{  for $p \in (0,2-\delta_1 ]$} \tag{\textbf{B2}},
\end{equation}
\begin{equation}\label{B3}
\sum_{k \leq 5} \left| \int_{\tau_1}^{\tau_2} \int_{\mcN_v^H}  (r-M)^{-p} (\underline{L}  \Omega^k \Phi_{\geq 1}^H ) \cdot \underline{L} ( r^2  \Omega^k F_{\geq 1} ) \, d\omega du dv \right| \lesssim \beta_0 \mathcal{L} + \frac{C E_0 \epsilon^2}{ (1+\tau_1 )^{1+\delta_2-p}}  \mbox{  for $p \in (0,1+\delta_2]$} \tag{\textbf{B3}} ,
\end{equation}
\begin{equation}\label{B4}
 \sum_{k \leq 5} \int_{\mcR_{\tau_1}^{\tau_2} \setminus \mathcal{A}_{\tau_1}^{\tau_2}} r^{p+1} | \Omega^k F_{\geq 1} |^2 \, d\mu_{\mcR}  \lesssim  \frac{C E_0 \epsilon^2}{ (1+\tau_1 )^{3+\delta_2-p}} \mbox{  for $p \in (1,2 - \delta_1 ]$} \tag{\textbf{B4}},
\end{equation} 
\begin{equation}\label{B5}
\sum_{k \leq 5} \left| \int_{\tau_1}^{\tau_2} \int_{\mcN_u^I} r^p ( L \Omega^k \Phi_{\geq 1}^I ) \cdot ( L ( r^3 \Omega^k F_{\geq 1} ) \, d\omega dv du \right| \lesssim \beta_0 \mathcal{L} + \frac{C E_0 \epsilon^2}{ (1+\tau_1 )^{1+\delta_2-p}} \mbox{  for $p \in (0,1+\delta_2 ]$} \tag{\textbf{B5}},
\end{equation}
\begin{equation}\label{C1}
\sum_{k \leq 5} \left( \int_{\mathcal{A}_{\tau_1}^{\tau_2}} (r-M)^{-1-\delta} D^2 | \Omega^k T F |^2 \, d\omega du dv +   \int_{\mcR_{\tau_1}^{\tau_2} \setminus \mathcal{A}_{\tau_1}^{\tau_2}} r^{2} | \Omega^k T F |^2 \, d\mu_{\mcR} \right)  \lesssim  \frac{C E_0 \epsilon^2}{ (1+\tau_1 )^{3+\delta_2}} \tag{\textbf{C1}},
\end{equation}
\begin{equation}\label{C2}
\sum_{k \leq 5} \left| \int_{\tau_1}^{\tau_2} \int_{\mcN_v^H}  (r-M)^{-p} ( \underline{L} \Omega^k T \phi ) \cdot D  ( \Omega^k T F ) \, d\omega du dv \right|  \lesssim \beta_0 \mathcal{L} + \frac{C E_0 \epsilon^2}{ (1+\tau_1 )^{3+\delta_2-p}} \mbox{  for $p \in (0,2 - \delta_1 ]$} \tag{\textbf{C2}},
\end{equation}
\begin{equation}\label{C3}
\sum_{k \leq 5} \left| \int_{\tau_1}^{\tau_2} \int_{\mcN_v^H}  (r-M)^{-p} (\underline{L}  \Omega^k T \Phi^H ) \cdot \underline{L} ( r^2  \Omega^k T F ) \, d\omega du dv \right| \lesssim \beta_0 \mathcal{L} +  \frac{C E_0 \epsilon^2}{ (1+\tau_1 )^{1+\delta_2-p}}  \mbox{  for $p \in (0,1+\delta_2]$} \tag{\textbf{C3}} ,
\end{equation}
\begin{equation}\label{C4}
 \sum_{k \leq 5} \int_{\mcR_{\tau_1}^{\tau_2} \setminus \mathcal{A}_{\tau_1}^{\tau_2}} r^{p+1} |  \Omega^k T F |^2 \, d\mu_{\mcR}  \leq  \frac{C E_0 \epsilon^2}{ (1+\tau_1 )^{3+\delta_2-p}} \mbox{  for $p \in (1,2 - \delta_1 ]$} \tag{\textbf{C4}},
\end{equation} 
\begin{equation}\label{C5}
\sum_{k \leq 5} \left| \int_{\tau_1}^{\tau_2} \int_{\mcN_u^I} r^p ( L \Omega^k T \Phi^I ) \cdot ( L ( r^3 \Omega^k T F ) \, d\omega dv du \right| \lesssim \beta_0 \mathcal{L} + \frac{C E_0 \epsilon^2}{ (1+\tau_2 )^{1+\delta_2-p}} \mbox{  for $p \in (0,1+\delta_2 ]$} \tag{\textbf{C5}},
\end{equation}
\begin{equation}\label{D1}
\sum_{k \leq 5} \left( \int_{\mathcal{A}_{\tau_1}^{\tau_2}} (r-M)^{-1-\delta} D^2 | \Omega^k T^2 F |^2 \, d\omega du dv +   \int_{\mcR_{\tau_1}^{\tau_2} \setminus \mathcal{A}_{\tau_1}^{\tau_2}} r^{2} | \Omega^k T^2 F |^2 \, d\mu_{\mcR} \right)  \lesssim  \frac{C E_0 \epsilon^2}{ (1+\tau_1 )^{2+\delta_2}} \tag{\textbf{D1}},
\end{equation}
\begin{equation}\label{D2}
\sum_{k \leq 5} \left| \int_{\tau_1}^{\tau_2} \int_{\mcN_v^H} (r-M)^{-p} ( \underline{L} \Omega^k T^2 \phi ) \cdot D (\Omega^k T^2 F ) \, d\omega du dv \right| \lesssim \beta_0 \mathcal{L} +  \frac{C E_0 \epsilon^2}{ (1+\tau_1 )^{2+\delta_2}} \mbox{  for $p \in (0,2-\delta_1 ]$} \tag{\textbf{D2}},
\end{equation}
\begin{equation}\label{D3}
\sum_{k \leq 5} \left| \int_{\tau_1}^{\tau_2} \int_{\mcN_v^H}  (r-M)^{-p} (\underline{L}  \Omega^k T^2 \Phi^H ) \cdot \underline{L} ( r^2  \Omega^k T^2 F ) \, d\omega du dv \right| \lesssim \beta_0 \mathcal{L} +  \frac{C E_0 \epsilon^2}{ (1+\tau_1 )^{\delta_2-p}}  \mbox{  for $p \in (0,\delta_2]$} \tag{\textbf{D3}} ,
\end{equation}
\begin{equation}\label{D4}
 \sum_{k \leq 5} \int_{\mcR_{\tau_1}^{\tau_2} \setminus \mathcal{A}_{\tau_1}^{\tau_2}} r^{p+1} |  \Omega^k T^2 F |^2 \, d\mu_{\mcR}  \lesssim  \frac{C E_0 \epsilon^2}{ (1+\tau_1 )^{2+\delta_2-p}} \mbox{  for $p \in (1,2 - \delta_1 ]$} \tag{\textbf{D4}},
\end{equation} 
\begin{equation}\label{D5}
\sum_{k \leq 5} \left| \int_{\tau_1}^{\tau_2} \int_{\mcN_u^I} r^p ( L \Omega^k T^2 \Phi^I ) \cdot ( L ( r^3 \Omega^k T^2 F ) \, d\omega dv du \right| \lesssim \beta_0 \mathcal{L} + \frac{C E_0 \epsilon^2}{ (1+\tau_2 )^{\delta_2-p}} \mbox{  for $p \in (0,\delta_2 ]$} \tag{\textbf{D5}},
\end{equation}
\begin{equation}\label{D6}
\sum_{k \leq 5} \int_{\mcC_{\tau_1}^{\tau_2}} |  \Omega^k T^2 F |^2 \, d\mu_{\mcC} \lesssim \frac{C E_0 \epsilon^2}{ (1+\tau_1 )^{3+\delta_2}} \tag{\textbf{D6}},
\end{equation}
\begin{equation}\label{E1}
\sum_{k \leq 5} \left( \int_{\mathcal{A}_{\tau_1}^{\tau_2}} (r-M)^{-1-\delta} D^2 | \Omega^k T^3 F |^2 \,  d\omega du dv +  \int_{\mcR_{\tau_1}^{\tau_2} \setminus \mathcal{A}_{\tau_1}^{\tau_2}} r^{2} | \Omega^k T^3 F |^2 \, d\mu_{\mcR} \right) \lesssim  \frac{C E_0 \epsilon^2}{ (1+\tau_1 )^{1+\delta_2}} \tag{\textbf{E1}},
\end{equation}
\begin{equation}\label{E2}
\sum_{k \leq 5} \left| \int_{\tau_1}^{\tau_2} \int_{\mcN_v^H} (r-M)^{-p} ( \underline{L} \Omega^k T^3 \phi ) \cdot D (\Omega^k T^3 F ) \, d\omega du dv \right| \lesssim \beta_0 \mathcal{L} +  \frac{C E_0 \epsilon^2}{ (1+\tau_1 )^{1+\delta_2}} \mbox{  for $p \in (0,1+\delta_2 ]$} \tag{\textbf{E2}},
\end{equation}
\begin{equation}\label{E3}
\sum_{k \leq 5} \int_{\mathcal{A}_{\tau_1}^{\tau_2}} (r-M)^{-2} D^2 | \Omega^k T^3 F |^2 \cdot v^{1+\beta} \, d\omega du dv \lesssim C^2 E_0^2 \epsilon^4 \tag{\textbf{E3}}, 
\end{equation}
\begin{equation}\label{E4}
 \sum_{k \leq 5} \int_{\mcR_{\tau_1}^{\tau_2}} r^{p+1} |  \Omega^k T^3 F |^2 \, d\mu_{\mcR}  \lesssim  \frac{C E_0 \epsilon^2}{ (1+\tau_1 )^{1+\delta_2-p}} \mbox{  for $p \in (0,1+\delta_2]$} \tag{\textbf{E4}},
\end{equation} 
\begin{equation}\label{E5}
\sum_{k \leq 5} \int_{\mcC_{\tau_1}^{\tau_2}} | \Omega^k T^3 F |^2 \, d\mu_{\mcC} \lesssim \frac{C E_0 \epsilon^2}{ (1+\tau_1 )^{2+\delta_2}} \tag{\textbf{E5}},
\end{equation}
\begin{equation}\label{E6}
\sum_{k \leq 5} \left( \int_{\tau_1}^{\tau_2} \left( \int_{\Sigma_{\tau} \cap \left( \mcR_{\tau_1}^{\tau_2} \setminus \mathcal{A}_{\tau_1}^{\tau_2} \right)} r^2 | \Omega^k T^3 F |^2 \, d\mu_{\Sigma} \right)^{1/2} d\tau \right)^2 \lesssim C E_0 \epsilon^2  \tag{\textbf{E6}},
\end{equation}
\begin{equation}\label{F1}
\sum_{k \leq 5} \left( \int_{\mathcal{A}_{\tau_1}^{\tau_2}} (r-M)^{-1-\delta} D^2 | \Omega^k T^4 F|^2 \, d\omega du dv + \int_{\mcR_{\tau_1}^{\tau_2} \setminus \mathcal{A}_{\tau_1}^{\tau_2}} r^{2} |  \Omega^k T^4 F |^2 \, d\mu_{\mathcal{R}} \right)  \lesssim  \frac{C E_0 \epsilon^2}{1+\tau_1} \tag{\textbf{F1}},
\end{equation}
\begin{equation}\label{F2}
\sum_{k \leq 5} \left| \int_{\tau_1}^{\tau_2} \int_{\mcN_v^H} (r-M)^{-p} ( \underline{L} \Omega^k T^4 \phi ) \cdot D ( \Omega^k T^4 F ) \, d\omega du dv \right| \lesssim  \frac{C E_0 \epsilon^2}{(1+\tau_1 )^{1 - p}} \mbox{  for $p \in (0,1]$}  \tag{\textbf{F2}},
\end{equation}
\begin{equation}\label{F3}
\sum_{k \leq 5} \int_{\mathcal{A}_{\tau_1}^{\tau_2}} (r-M)^{-1-\delta} D^2 | \Omega^k T^4 F |^2 \cdot v^{1+\beta} \, d\omega du dv \lesssim C E_0 \epsilon^2 \tag{\textbf{F3}},
\end{equation}
\begin{equation}\label{F4}
\sum_{k \leq 5} \int_{\mcR_{\tau_1}^{\tau_2} \setminus \mathcal{A}_{\tau_1}^{\tau_2}} r^{p+1} | \Omega^k T^4 F |^2 \, d\mu_{\mcR} \lesssim \frac{C E_0 \epsilon^2}{(1+\tau_1 )^{1-p}} \mbox{  for $p \in (0,1]$} \tag{\textbf{F4}},
\end{equation}
\begin{equation}\label{F5}
\sum_{k \leq 5} \int_{\mathcal{C}_{\tau_1}^{\tau_2}} | \Omega^k T^4 F |^2 \, d\mu_{\mathcal{C}} \lesssim \frac{C E_0 \epsilon^2}{(1+\tau_1 )^2} \tag{\textbf{F5}},
\end{equation}
\begin{equation}\label{F6}
\sum_{k \leq 5} \left( \int_{\tau_1}^{\tau_2} \left( \int_{\Sigma_{\tau} \cap \left( \mcR_{\tau_1}^{\tau_2} \setminus \mathcal{A}_{\tau_1}^{\tau_2} \right)} r^{1+\delta} | \Omega^k T^4 F |^2 \, d\mu_{\Sigma} \right)^{1/2} d\tau \right)^2 \lesssim C E_0 \epsilon^2  \tag{\textbf{F6}},
\end{equation}
\begin{equation}\label{G1}
\begin{split}
\sum_{k \leq 5} & \int_{\mathcal{A}_{\tau_1}^{\tau_2}}  (r-M)^{-1-\delta} D^2 |  \Omega^k T^5 F |^2 \cdot v^{1+\beta} \, d\omega du dv \\ & + \left( \int_{\tau_1}^{\tau_2} \left( \int_{\Sigma_{\tau} \setminus ( \mcN_{\tau}^H \cup \mcN_{\tau}^I )} | \Omega^k T^5 F |^2 d\mu_{\Sigma} \right)^{1/2} d\tau \right)^2 \\ & + \int_{\tau_1}^{\tau_2} \int_{\mcN_{\tau}^I} r^{1+\delta} | \Omega^k T^5 F |^2 \, d\mu_{\mcN^I} d\tau  \\ & +  \left( \int_{\tau_1}^{\tau_2} \left( \int_{\Sigma_{\tau} \cap \left( \mcR_{\tau_1}^{\tau_2} \setminus \mathcal{A}_{\tau_1}^{\tau_2} \right)} r^{1+\delta} | \Omega^k T^5 F |^2 \, d\mu_{\Sigma} \right)^{1/2} d\tau \right)^2 \lesssim  C E_0 \epsilon^2 ,
\end{split}
\tag{\textbf{G1}} 
\end{equation}
\begin{equation}\label{G2}
\sum_{k \leq 5} \int_{\mathcal{C}_{\tau_1}^{\tau_2}} | \Omega^k T^5 F |^2 \, d\mu_{\mathcal{C}} \lesssim \frac{C E_0 \epsilon^2}{1+\tau_1} \tag{\textbf{G2}}.
\end{equation}

\section{Decay and boundedness estimates}\label{dec}
\subsection{Energy decay estimates}
First, we will derive the decay estimates for the various energies restricted to the spherically symmetric part $\psi_0$ of a solution $\psi$ of \eqref{nw}. We will apply the bootstrap assumptions of Section \ref{as:boot} together with the energy inequalities of Section \eqref{energy}.
\begin{lemma}
Let $\psi_0$ be the spherically symmetric part of a solution $\psi$ of \eqref{nw} for which the assumptions of Section \ref{as:boot} are satisfied. Then for all $\tau \geq \tau_0$ we have that
\begin{equation}\label{dec:en0}
\int_{\Sigma_{\tau}} J^T [\psi_0 ] \cdot \textbf{n}_{\Sigma} \, d\mu_{\Sigma} \lesssim \frac{E_0 \epsilon^2}{(1+\tau)^{3-\delta_1}} ,
\end{equation}
for $\epsilon$ and $\delta_1$ as in the bootstrap assumptions of section \ref{as:boot}.
\end{lemma}
\begin{proof}
We will omit several details of the proof as they are quite standard. We first note that by the bootstrap assumptions \eqref{A3}, \eqref{C1}-\eqref{C5}, we have that the quantity
$$ \int_{\tau_1}^{\tau_2} II_{0,0}^{\delta_1} (\tau) \, d\tau  $$
(see Proposition \ref{prop:rp06}) is bounded by $C E_0 \epsilon^2$ for all $\tau_1$, $\tau_2$ with $\tau_1 < \tau_2$. Note that the above quantity contains 
$$ \int_{\tau_1}^{\tau_2} \int_{\mcN_u^I} r^{-\delta_1} (L \Phi_0^I )^2 \, d\omega dv du $$
and so the later is also bounded by $C E_0 \epsilon^2$ for all $\tau_1$, $\tau_2$ with $\tau_1 < \tau_2$. By a standard argument we have that
$$ II_{0,0}^{\delta_1} (\tau^1_n ) \lesssim \frac{C E_0 \epsilon^2}{1+\tau^1_n } , $$
over a dyadic sequence $\{ \tau^1_n \}$, and that
$$ \int_{\mcN_{\tau^2_n}^I} r^{-\delta_1} (L \Phi_0^I )^2 \, d\omega dv \lesssim \frac{C E_0 \epsilon^2}{1+\tau^2_n } , $$
over another dyadic sequence $\{ \tau^2_n \}$. It's easy then to see that we have that
$$ II_{0,0}^{\delta_1} (\tau_n ) + \int_{\mcN_{\tau_n}^I} r^{-\delta_1} (L \Phi_0^I )^2 \, d\omega dv \lesssim \frac{C E_0 \epsilon^2}{1+\tau_n } , $$
for $\{ \tau_n \} = \{ \tau^1_n \} \cup \{ \tau^2_n \}$. By Hardy's inequality \eqref{hardy} we have that for all $\tau$
$$ I_{0,0}^{2-\delta_1} (\tau) \lesssim II_{0,0}^{\delta_1} (\tau) + \int_{\Sigma_{\tau}} J^T [\psi_0 ] \cdot \textbf{n}_{\Sigma} \, d\mu_{\Sigma} , $$
and that
$$ \int_{\mcN_{\tau}^I} r^{2-\delta_1} (L \phi_0 )^2 \, d\omega dv\lesssim \int_{N_{\tau}^I} r^{-\delta_1} (L \Phi_0^I )^2 \, d\omega dv + \int_{\Sigma_{\tau}} J^T [\psi_0 ] \cdot \textbf{n}_{\Sigma} \, d\mu_{\Sigma} , $$
by using the decay over $\{ \tau_n \}$ for $II_{0,0}^{\delta_1}$ and $\int_{\mcN^I} r^{-\delta_1} (L \Phi_0^I )^2 \, d\omega dv$, estimates \eqref{est:rp05} and \eqref{est:rp03} for $l=0$, and the bootstrap estimates \eqref{A2} and \eqref{A4}, we can get that over another dyadic sequence $\{ \lambda_n \}$ we have that
$$  I_{0,0}^{2-\delta_1} (\lambda_n ) + \int_{\mcN_{\lambda_n}^I} r^{2-\delta_1} (L \phi_0 )^2 \, d\omega dv\lesssim  \frac{C E_0 \epsilon^2}{1+\lambda_n } , $$
where we also used that by the bootstrap assumptions \eqref{A1}-\eqref{A6} we can actually show decay of rate $\tau^{-2}$ for the $T$-flux of $\psi_0$ (for details on this see Theorem 21 of Section 6 of \cite{yannis1}, the situation is analogous to the one in the present paper). By a standard argument and using estimates \eqref{est:rp02} and again \eqref{est:rp03} we have that
$$ I_{0,0}^{2-\delta_1} (\tau ) \lesssim  \frac{C E_0 \epsilon^2}{1+\tau } \mbox{  for all $\tau$,} $$
and
$$ \int_{\mcN_{\tau}^I} r^{2-\delta_1} (L \phi_0 )^2 \, d\omega dv\lesssim   \frac{C E_0 \epsilon^2}{1+\tau } \mbox{  for all $\tau$.} $$
Arguing in the same way we can now show that 
$$ I_{0,0}^1 (\tau )  \lesssim \frac{C E_0 \epsilon^2}{( 1+\tau )^{2-\delta_1} } \mbox{  for all $\tau$,} $$
from which, using moreover \eqref{est:rp05} and \eqref{est:rp01}, it easily follows that
$$ \int_{\Sigma_{\tau}} J^T [\psi_0 ] \cdot \textbf{n}_{\Sigma} \, d\mu_{\Sigma} \lesssim \frac{C E_0 \epsilon^2}{( 1+\tau )^{3-\delta_1} } \mbox{  for all $\tau$.} $$
by using also \eqref{mor:improved} for $\psi = \psi_0$ (which does not lose any derivatives on the photon sphere). 

\end{proof}

\begin{remark}
Note that by the proof of the above Lemma, we get the following hierarchy of energy decay estimates under the condition that the assumptions of section \ref{as:boot} are satisfied:
\begin{equation}\label{dec:en0p}
\int_{\mcN_{\tau}^H} (r-M)^{-p}  (  \underline{L} \phi_0 )^2 \, d\omega du \lesssim \frac{C E_0 \epsilon^2}{( 1+\tau )^{3-\delta_1 - p} } \mbox{  for all $\tau$ and for all $p \in (0,3-\delta_1 ]$,} 
\end{equation}
\begin{equation}\label{dec:en0p1}
\int_{\mcN_{\tau}^H} (r-M)^{-p}  (  \underline{L} \Phi_0^H )^2 \, d\omega du  \lesssim \frac{C E_0 \epsilon^2}{( 1+\tau )^{1-\delta_1 - p} } \mbox{  for all $\tau$ and for all $p \in (0,1-\delta_1]$,} 
\end{equation}
\begin{equation}\label{dec:en0p2}
\int_{\mcN_{\tau}^I} r^p  (  L\phi_0 )^2 \, d\omega dv \lesssim \frac{C E_0 \epsilon^2}{( 1+\tau )^{3-\delta_1 - p} } \mbox{  for all $\tau$ and for all $p \in (0,3-\delta_1 ]$,} 
\end{equation}
and
\begin{equation}\label{dec:en0p3}
\int_{\mcN_{\tau}^I} r^p  (  L \Phi_0^I )^2 \, d\omega du  \lesssim \frac{C E_0 \epsilon^2}{( 1+\tau )^{1-\delta_1 - p} } \mbox{  for all $\tau$ and for all $p \in (0,1-\delta_1]$,} 
\end{equation}
where the range of $p \in (0 , 1-\delta_1 )$ in \eqref{dec:en0p1} and \eqref{dec:en0p3}, and the range of $p \in [1-\delta_1 , 2]$ in \eqref{dec:en0p} and \eqref{dec:en0p2} can be obtained through interpolation.
\end{remark}

For the non-spherically symmetric part $\psi_{\geq 1}$ of a solution $\psi$ of \eqref{nw} we have the following energy decay estimate arguing as in the case of the spherically symmetric part for which we argue as in the proof of the previous estimate Lemma and where we use the corresponding energy decay estimates.

\begin{lemma}
Let $\psi_{\geq 1}$ be the non-spherically symmetric part of a solution $\psi$ of \eqref{nw} for which the assumptions of section \ref{as:boot} are satisfied. Then for all $\tau$ we have that
\begin{equation}\label{dec:en1}
\sum_{k=1}^5 \int_{\Sigma_{\tau}} J^T [\Omega^k \psi_{\geq 1} ] \cdot \textbf{n}_{\Sigma} \, d\mu_{\Sigma} \lesssim \frac{C E_0 \epsilon^2}{(1+\tau)^{3+\delta_2}} ,
\end{equation}
for $\epsilon$ and $\delta_2$ as in the bootstrap assumptions of section \ref{as:boot}.
\end{lemma}
\begin{proof}
The proof of the above Lemma follows the same lines as the previous Lemma using now the corresponding bootstrap assumptions from section \ref{as:boot}.
\end{proof}

Combining the previous two lemmas we have that:
\begin{equation}\label{dec:en}
\sum_{k=1}^5 \int_{\Sigma_{\tau}} J^T [\Omega^k \psi ] \cdot \textbf{n}_{\Sigma} \, d\mu_{\Sigma} \lesssim \frac{C E_0 \epsilon^2}{(1+\tau)^{3-\delta_1}} .
\end{equation}
Similarly we have the following estimates after commuting with $T$ derivatives:
\begin{equation}\label{dec:ent}
\sum_{k=1}^5 \int_{\Sigma_{\tau}} J^T [\Omega^k T \psi ] \cdot \textbf{n}_{\Sigma} \, d\mu_{\Sigma} \lesssim \frac{C E_0 \epsilon^2}{(1+\tau)^{3+\delta_2}} ,
\end{equation}
\begin{equation}\label{dec:entt}
\sum_{k=1}^5 \int_{\Sigma_{\tau}} J^T [\Omega^k T^2 \psi ] \cdot \textbf{n}_{\Sigma} \, d\mu_{\Sigma} \lesssim \frac{C E_0 \epsilon^2}{(1+\tau)^{2+\delta_2}} ,
\end{equation}
\begin{equation}\label{dec:enttt}
\sum_{k=1}^5 \int_{\Sigma_{\tau}} J^T [\Omega^k T^3 \psi ] \cdot \textbf{n}_{\Sigma} \, d\mu_{\Sigma} \lesssim \frac{C E_0 \epsilon^2}{(1+\tau)^{1+\delta_2}} ,
\end{equation}
\begin{equation}\label{dec:entttt}
\sum_{k=1}^5 \int_{\Sigma_{\tau}} J^T [\Omega^k T^4 \psi ] \cdot \textbf{n}_{\Sigma} \, d\mu_{\Sigma} \lesssim \frac{C E_0 \epsilon^2}{1+\tau} ,
\end{equation}
and
\begin{equation}\label{dec:enttttt}
\sum_{k=1}^5 \int_{\Sigma_{\tau}} J^T [\Omega^k T^5 \psi ] \cdot \textbf{n}_{\Sigma} \, d\mu_{\Sigma} \lesssim C E_0 \epsilon^2 ,
\end{equation}

Note that by the proof of the previous Lemma, we get the following hierarchy of energy decay estimates under the condition that the assumptions of section \ref{as:boot} are satisfied:
\begin{equation}\label{dec:en1p}
\int_{\mcN_{\tau}^H} (r-M)^{-p}  (  \underline{L} \phi_{\geq 1} )^2 \, d\omega du \lesssim \frac{C E_0 \epsilon^2}{( 1+\tau )^{3+\delta_2 - p} } \mbox{  for all $\tau$ and for all $p \in (0,3-\delta_1 ]$,} 
\end{equation}
\begin{equation}\label{dec:en1p1}
\int_{\mcN_{\tau}^H} (r-M)^{-p}  (  \underline{L} \Phi_{\geq 1}^H )^2 \, d\omega du  \lesssim \frac{C E_0 \epsilon^2}{( 1+\tau )^{1+\delta_2 - p} } \mbox{  for all $\tau$ and for all $p \in (0,1+\delta_2]$,} 
\end{equation}
\begin{equation}\label{dec:en1p2}
\int_{\mcN_{\tau}^I} r^p  (  L\phi_{\geq 1} )^2 \, d\omega dv \lesssim \frac{C E_0 \epsilon^2}{( 1+\tau )^{3+\delta_2 - p} } \mbox{  for all $\tau$ and for all $p \in (0,2]$,} 
\end{equation}
and
\begin{equation}\label{dec:en1p3}
\int_{\mcN_{\tau}^I} r^p  (  L \Phi_{\geq 1}^I )^2 \, d\omega du  \lesssim \frac{C E_0 \epsilon^2}{( 1+\tau )^{1+\delta_2 - p} } \mbox{  for all $\tau$ and for all $p \in (0,1+\delta_2]$.} 
\end{equation}
Similarly after commuting with $T$ derivatives we also have the following estimates:
\begin{equation}\label{dec:enp}
\int_{\mcN_{\tau}^H} (r-M)^{-p}  (  \underline{L} \phi )^2 \, d\omega du \lesssim \frac{C E_0 \epsilon^2}{( 1+\tau )^{3-\delta_1 - p} } \mbox{  for all $\tau$ and for all $p \in (0,3-\delta_1 ]$,} 
\end{equation}
\begin{equation}\label{dec:enp1}
\int_{\mcN_{\tau}^H} (r-M)^{-p}  (  \underline{L} \Phi^H )^2 \, d\omega du  \lesssim \frac{C E_0 \epsilon^2}{( 1+\tau )^{1-\delta_1 - p} } \mbox{  for all $\tau$ and for all $p \in (0,1-\delta_1]$,} 
\end{equation}
\begin{equation}\label{dec:enp2}
\int_{\mcN_{\tau}^I} r^p  (  L\phi )^2 \, d\omega dv \lesssim \frac{C E_0 \epsilon^2}{( 1+\tau )^{3-\delta_1 - p} } \mbox{  for all $\tau$ and for all $p \in (0,3-\delta_1 ]$,} 
\end{equation}
\begin{equation}\label{dec:enp3}
\int_{\mcN_{\tau}^I} r^p  (  L \Phi^I )^2 \, d\omega du  \lesssim \frac{C E_0 \epsilon^2}{( 1+\tau )^{1-\delta_1 - p} } \mbox{  for all $\tau$ and for all $p \in (0,1-\delta_1]$,} 
\end{equation}
\begin{equation}\label{dec:entp}
\int_{\mcN_{\tau}^H} (r-M)^{-p}  (  \underline{L} T\phi )^2 \, d\omega du \lesssim \frac{C E_0 \epsilon^2}{( 1+\tau )^{3+\delta_2 - p} } \mbox{  for all $\tau$ and for all $p \in (0,2]$,} 
\end{equation}
\begin{equation}\label{dec:entp1}
\int_{\mcN_{\tau}^H} (r-M)^{-p}  (  \underline{L} T\Phi^H )^2 \, d\omega du  \lesssim \frac{C E_0 \epsilon^2}{( 1+\tau )^{1+\delta_2 - p} } \mbox{  for all $\tau$ and for all $p \in (0,1+\delta_2]$,} 
\end{equation}
\begin{equation}\label{dec:entp2}
\int_{\mcN_{\tau}^I} r^p  (  LT\phi )^2 \, d\omega dv \lesssim \frac{C E_0 \epsilon^2}{( 1+\tau )^{3+\delta_2 - p} } \mbox{  for all $\tau$ and for all $p \in (0,2]$,} 
\end{equation}
\begin{equation}\label{dec:entp3}
\int_{\mcN_{\tau}^I} r^p  (  L T\Phi^I )^2 \, d\omega du  \lesssim \frac{C E_0 \epsilon^2}{( 1+\tau )^{1+\delta_2 - p} } \mbox{  for all $\tau$ and for all $p \in (0,1+\delta_2]$,} 
\end{equation}
\begin{equation}\label{dec:enttp}
\int_{\mcN_{\tau}^H} (r-M)^{-p}  (  \underline{L} T^2 \phi )^2 \, d\omega du \lesssim \frac{C E_0 \epsilon^2}{( 1+\tau )^{2+\delta_2 - p} } \mbox{  for all $\tau$ and for all $p \in (0,2]$,} 
\end{equation}
\begin{equation}\label{dec:enttp1}
\int_{\mcN_{\tau}^H} (r-M)^{-p}  (  \underline{L} T^2 \Phi^H )^2 \, d\omega du  \lesssim \frac{C E_0 \epsilon^2}{( 1+\tau )^{\delta_2 - p} } \mbox{  for all $\tau$ and for all $p \in (0,\delta_2]$,} 
\end{equation}
\begin{equation}\label{dec:enttp2}
\int_{\mcN_{\tau}^I} r^p  (  LT^2 \phi )^2 \, d\omega dv \lesssim \frac{C E_0 \epsilon^2}{( 1+\tau )^{2+\delta_2 - p} } \mbox{  for all $\tau$ and for all $p \in (0,2]$,} 
\end{equation}
\begin{equation}\label{dec:enttp3}
\int_{\mcN_{\tau}^I} r^p  (  L T^2 \Phi^I )^2 \, d\omega du  \lesssim \frac{C E_0 \epsilon^2}{( 1+\tau )^{\delta_2 - p} } \mbox{  for all $\tau$ and for all $p \in (0,\delta_2]$,} 
\end{equation}
\begin{equation}\label{dec:entttp}
\int_{\mcN_{\tau}^H} (r-M)^{-p}  (  \underline{L} T^3 \phi )^2 \, d\omega du \lesssim \frac{C E_0 \epsilon^2}{( 1+\tau )^{1+\delta_2 - p} } \mbox{  for all $\tau$ and for all $p \in (0,1+\delta_2 ]$,} 
\end{equation}
\begin{equation}\label{dec:entttp1}
\int_{\mcN_{\tau}^H} (r-M)^{-2}  (  \underline{L} T^3 \phi )^2 \, d\omega du \lesssim C E_0 \epsilon^2 ,
\end{equation}
\begin{equation}\label{dec:entttp2}
\int_{\mcN_{\tau}^I} r^p  (  LT^3 \phi )^2 \, d\omega dv \lesssim \frac{C E_0 \epsilon^2}{( 1+\tau )^{1+\delta_2 - p} } \mbox{  for all $\tau$ and for all $p \in (0,1+\delta_2]$,} 
\end{equation}
\begin{equation}\label{dec:entttp3}
\int_{\mcN_{\tau}^I} r^2  (  LT^3 \phi )^2 \, d\omega dv \lesssim C E_0 \epsilon^2 ,
\end{equation}
\begin{equation}\label{dec:enttttp}
\int_{\mcN_{\tau}^H} (r-M)^{-p}  (  \underline{L} T^4 \phi )^2 \, d\omega du \lesssim \frac{C E_0 \epsilon^2}{( 1+\tau )^{1- p} } \mbox{  for all $\tau$ and for all $p \in (0,1 ]$,} 
\end{equation}
\begin{equation}\label{dec:enttttp1}
\int_{\mcN_{\tau}^H} (r-M)^{-1-\delta}  (  \underline{L} T^4 \phi )^2 \, d\omega du \lesssim C E_0 \epsilon^2 \mbox{  for all $\delta \leq \delta_2$},
\end{equation}
\begin{equation}\label{dec:enttttp2}
\int_{\mcN_{\tau}^I} r^p  (  LT^4 \phi )^2 \, d\omega dv \lesssim \frac{C E_0 \epsilon^2}{( 1+\tau )^{1 - p} } \mbox{  for all $\tau$ and for all $p \in (0,1]$,} 
\end{equation}
\begin{equation}\label{dec:enttttp3}
\int_{\mcN_{\tau}^I} r^{1+\delta}  (  LT^4 \phi )^2 \, d\omega dv \lesssim C E_0 \epsilon^2 \mbox{  for all $\delta \leq \delta_2$},
\end{equation}
\begin{equation}\label{dec:entttttp}
\int_{\mcN_{\tau}^H} (r-M)^{-1-\delta}  (  \underline{L} T^5 \phi )^2 \, d\omega du \lesssim C E_0 \epsilon^2 \mbox{  for all $\delta \leq \delta_2$} , 
\end{equation}
\begin{equation}\label{dec:entttttp1}
\int_{\mcN_{\tau}^I} r^{1+\delta}  (  LT^5 \phi )^2 \, d\omega dv \lesssim C E_0 \epsilon^2 \mbox{  for all $\delta \leq \delta_2$}.
\end{equation}

Also as it is evident from the proof of estimates \eqref{dec:en0}and \eqref{dec:en1}, we also get the following decay estimates under the condition that the assumptions of section \ref{as:boot} are satisfied: 
\begin{equation}\label{dec:mor01}
\begin{split}
\int_{\mathcal{A}_{\tau_1}^{\tau_2} } & \left[ (r-M)^{1+\delta} ( L \Omega^k  \phi_0 )^2 +  (r-M)^{1+\delta} ( \underline{L} \Omega^k \phi_0 )^2 \right] \,  d\omega du dv
\\ & +  \int_{\mcR_{\tau_1}^{\tau_2} \setminus ( \mathcal{A}_{\tau_1}^{\tau_2} \cup \mathcal{B}_{\tau_1}^{\tau_2} )} \left(  (  \Omega^k T \psi_0 )^2 + (Y \Omega^k  \psi_0 )^2 \right) \, d\mu_{\mcR} \\ & + \int_{\mathcal{B}_{\tau_1}^{\tau_2}} \left( \frac{( L \Omega^k  \phi_0 )^2}{r^{1+\delta}} + \frac{( \underline{L} \Omega^k \phi_0 )^2}{r^{1+\delta}}  \right) \, d\omega dv du  \lesssim  \frac{C E_0 \epsilon^2}{(1+\tau)^{3-\delta_1}} ,
\end{split}
\end{equation}
and
\begin{equation}\label{dec:mor11}
\begin{split}
 \int_{\mathcal{A}_{\tau_1}^{\tau_2} } & \left[ (r-M)^{1+\delta} ( L \Omega^k  \phi_{\geq 1} )^2 +  (r-M)^{1+\delta} ( \underline{L} \Omega^k \phi_{\geq 1} )^2 + (r-M)^3 |\Omega^k  \slashed{\nabla} \phi_{\geq 1} |^2\right] \,  d\omega du dv
\\ & +  \int_{\mcR_{\tau_1}^{\tau_2} \setminus ( \mathcal{A}_{\tau_1}^{\tau_2} \cup \mathcal{B}_{\tau_1}^{\tau_2} )} \left(  (  \Omega^k T \psi_{\geq 1} )^2 + (Y \Omega^k  \psi_{\geq 1} )^2+ | \slashed{\nabla} \Omega^k \psi_{\geq 1} |^2 \right) \, d\mu_{\mcR} \\ & + \int_{\mathcal{B}_{\tau_1}^{\tau_2}} \left( \frac{( L \Omega^k  \phi_{\geq 1} )^2}{r^{1+\delta}} + \frac{( \underline{L} \Omega^k \phi_{\geq 1} )^2}{r^{1+\delta}} + \frac{| \slashed{\nabla} \phi_{\geq 1} |^2}{r} \right) \, d\omega dv du  \lesssim  \frac{C E_0 \epsilon^2}{(1+\tau)^{3+\delta_2}} ,
\end{split}
\end{equation}
for any $\tau_1 < \tau_2$, for $\epsilon$ and $\delta_1$, $\delta_2$ as in the bootstrap assumptions of section \ref{as:boot}, for any $k \leq 5$, and for any $\delta > 0$. We also have that:
\begin{equation}\label{dec:mor}
\begin{split}
 \int_{\mathcal{A}_{\tau_1}^{\tau_2} } & \left[ (r-M)^{1+\delta} ( L \Omega^k  \phi )^2 +  (r-M)^{1+\delta} ( \underline{L} \Omega^k  \phi )^2 + (r-M)^3 |\Omega^k  \slashed{\nabla} \phi |^2\right] \,  d\omega du dv
\\ & +  \int_{\mcR_{\tau_1}^{\tau_2} \setminus ( \mathcal{A}_{\tau_1}^{\tau_2} \cup \mathcal{B}_{\tau_1}^{\tau_2} )} \left(  ( T \Omega^k \psi )^2 + (Y \Omega^k \psi )^2+ | \slashed{\nabla} \Omega^k \psi |^2 \right) \, d\mu_{\mcR} \\ & + \int_{\mathcal{B}_{\tau_1}^{\tau_2}} \left( \frac{( L \Omega^k \phi )^2}{r^{1+\delta}} + \frac{( \underline{L} \Omega^k \phi )^2}{r^{1+\delta}} + \frac{| \slashed{\nabla} \phi |^2}{r} \right) \, d\omega dv du  \lesssim  \frac{C E_0 \epsilon^2}{(1+\tau)^{3-\delta_1}} .
\end{split}
\end{equation}
Similarly we have after commuting with $T$ derivatives the following estimates:
\begin{equation}\label{dec:mor1}
\begin{split}
 \int_{\mathcal{A}_{\tau_1}^{\tau_2} } & \left[ (r-M)^{1+\delta} ( L \Omega^k T \phi )^2 +  (r-M)^{1+\delta} ( \underline{L} \Omega^k T \phi )^2 + (r-M)^3 |\Omega^k T \slashed{\nabla} \phi |^2\right] \,  d\omega du dv
\\ & +  \int_{\mcR_{\tau_1}^{\tau_2} \setminus ( \mathcal{A}_{\tau_1}^{\tau_2} \cup \mathcal{B}_{\tau_1}^{\tau_2} )} \left(  (  \Omega^k T^2 \psi )^2 + (Y \Omega^k T \psi )^2+ | \slashed{\nabla} \Omega^k T \psi |^2 \right) \, d\mu_{\mcR} \\ & + \int_{\mathcal{B}_{\tau_1}^{\tau_2}} \left( \frac{( L \Omega^k T \phi )^2}{r^{1+\delta}} + \frac{( \underline{L} \Omega^k T\phi )^2}{r^{1+\delta}} + \frac{| \slashed{\nabla} T \phi |^2}{r} \right) \, d\omega dv du  \lesssim  \frac{C E_0 \epsilon^2}{(1+\tau)^{3+\delta_2}} ,
\end{split}
\end{equation}
\begin{equation}\label{dec:mor2}
\begin{split}
 \int_{\mathcal{A}_{\tau_1}^{\tau_2} } & \left[ (r-M)^{1+\delta} ( L \Omega^k T^2 \phi )^2 +  (r-M)^{1+\delta} ( \underline{L} \Omega^k T^2 \phi )^2 + (r-M)^3 |\Omega^k T^2 \slashed{\nabla} \phi |^2\right] \,  d\omega du dv
\\ & +  \int_{\mcR_{\tau_1}^{\tau_2} \setminus ( \mathcal{A}_{\tau_1}^{\tau_2} \cup \mathcal{B}_{\tau_1}^{\tau_2} )} \left(  (  \Omega^k T^3 \psi )^2 + (Y \Omega^k T^2 \psi )^2+ | \slashed{\nabla} \Omega^k T^2 \psi |^2 \right) \, d\mu_{\mcR} \\ & + \int_{\mathcal{B}_{\tau_1}^{\tau_2}} \left( \frac{( L \Omega^k T^2 \phi )^2}{r^{1+\delta}} + \frac{( \underline{L} \Omega^k T^2 \phi )^2}{r^{1+\delta}} + \frac{| \slashed{\nabla} T^2 \phi |^2}{r} \right) \, d\omega dv du  \lesssim  \frac{C E_0 \epsilon^2}{(1+\tau)^{2+\delta_2}} ,
\end{split}
\end{equation}
\begin{equation}\label{dec:mor3}
\begin{split}
 \int_{\mathcal{A}_{\tau_1}^{\tau_2} } & \left[ (r-M)^{1+\delta} ( L \Omega^k T^3 \phi )^2 +  (r-M)^{1+\delta} ( \underline{L} \Omega^k T^3 \phi )^2 + (r-M)^3 |\Omega^k T^3 \slashed{\nabla} \phi |^2\right] \,  d\omega du dv
\\ & +  \int_{\mcR_{\tau_1}^{\tau_2} \setminus ( \mathcal{A}_{\tau_1}^{\tau_2} \cup \mathcal{B}_{\tau_1}^{\tau_2} )} \left(  (  \Omega^k T^4 \psi )^2 + (Y \Omega^k T^3 \psi )^2+ | \slashed{\nabla} \Omega^k T^3 \psi |^2 \right) \, d\mu_{\mcR} \\ & + \int_{\mathcal{B}_{\tau_1}^{\tau_2}} \left( \frac{( L \Omega^k T^3 \phi )^2}{r^{1+\delta}} +\frac{( \underline{L} \Omega^k T^3 \phi )^2}{r^{1+\delta}} + \frac{| \slashed{\nabla} T^3 \phi |^2}{r} \right) \, d\omega dv du  \lesssim  \frac{C E_0 \epsilon^2}{(1+\tau)^{1+\delta_2}} ,
\end{split}
\end{equation}
\begin{equation}\label{dec:mor4}
\begin{split}
 \int_{\mathcal{A}_{\tau_1}^{\tau_2} } & \left[ (r-M)^{1+\delta} ( L \Omega^k T^4 \phi )^2 +  (r-M)^{1+\delta} ( \underline{L} \Omega^k T^4 \phi )^2 + (r-M)^3 |\Omega^k T^4 \slashed{\nabla} \phi |^2\right] \,  d\omega du dv
\\ & +  \int_{\mcR_{\tau_1}^{\tau_2} \setminus ( \mathcal{A}_{\tau_1}^{\tau_2} \cup \mathcal{B}_{\tau_1}^{\tau_2} )} \left(  (  \Omega^k T^5 \psi )^2 + (Y \Omega^k T^4 \psi )^2+ | \slashed{\nabla} \Omega^k T^4 \psi |^2 \right) \, d\mu_{\mcR} \\ & + \int_{\mathcal{B}_{\tau_1}^{\tau_2}} \left( \frac{( L \Omega^k T^4 \phi )^2}{r^{1+\delta}} + \frac{( \underline{L} \Omega^k T^4 \phi )^2}{r^{1+\delta}} + \frac{| \slashed{\nabla} T^4 \phi |^2}{r} \right) \, d\omega dv du  \lesssim  \frac{C E_0 \epsilon^2}{1+\tau} ,
\end{split}
\end{equation}
\begin{equation}\label{dec:mor5}
\begin{split}
 \int_{\mathcal{A}_{\tau_1}^{\tau_2} } & \left[ (r-M)^{1+\delta} ( L \Omega^k T^5 \phi )^2 +  (r-M)^{1+\delta} ( \underline{L} \Omega^k T^5 \phi )^2 + (r-M)^3 |\Omega^k T^5 \slashed{\nabla} \phi |^2\right] \,  d\omega du dv
\\ & +  \int_{\mcR_{\tau_1}^{\tau_2} \setminus ( \mathcal{A}_{\tau_1}^{\tau_2} \cup \mathcal{B}_{\tau_1}^{\tau_2} )} \left(  (  \Omega^k T^6 \psi )^2 + (Y \Omega^k T^5 \psi )^2+ | \slashed{\nabla} \Omega^k T^5 \psi |^2 \right) \, d\mu_{\mcR} \\ & + \int_{\mathcal{B}_{\tau_1}^{\tau_2}} \left( \frac{( L \Omega^k T^5 \phi )^2}{r^{1+\delta}} + \frac{( \underline{L} \Omega^k T^5 \phi )^2}{r^{1+\delta}} + \frac{| \slashed{\nabla} T^5 \phi |^2}{r} \right) \, d\omega dv du  \lesssim  C E_0 \epsilon^2.
\end{split}
\end{equation}

\subsection{Pointwise decay estimates}
First we state the decay estimates satisfied by various quantities away from the horizon.
\begin{lemma}
Let $\psi$ be a solution of the equation \eqref{nw} for which the assumptions of section \ref{as:boot} are satisfied. Then we have that
\begin{equation}\label{dec:psi0a}
r^{1/2} | \psi_0 | (\tau , r )  \lesssim \frac{C_r E_0^{1/2} \epsilon}{\tau^{3/2-\delta_1/2}} ,
\end{equation}
\begin{equation}\label{dec:psi1a}
\sum_{k=0}^3 r^{1/2} | \Omega^k \psi_{\geq 1} | (\tau , r ,\omega) \lesssim  \frac{C_r^{1/2} E_0^{1/2} \epsilon}{\tau^{3/2+\delta_2/2}} ,
\end{equation}
\begin{equation}\label{dec:psi1a1}
\sum_{k=4}^5 r \int_{\mathbb{S}^2} ( \Omega^k \psi_{\geq 1} )^2 (\tau,r,\omega) \, d\omega \lesssim  \frac{C_r E_0 \epsilon^2}{v^{3+\delta_2}} ,
\end{equation}
for all $\tau$, for $\epsilon$ and $\delta_1$ as in the bootstrap assumptions of section \ref{as:boot}, and for any $r > M$ where the constant constant $C_r$ diverges to infinity as $r \rightarrow M$ or as $r \rightarrow \infty$.
\end{lemma}
\begin{proof}
The proof follows from a standard application of the fundamental theorem of calculus. We demonstrate it only for $\psi_0$ (since the argument for estimate \eqref{dec:psi1a} and \eqref{dec:psi1a1} is almost identical) where we use a coordinate system $(\rho ,\omega)$ on $\Sigma_{\tau}$ for any $\tau$:
\begin{equation*}
\begin{split}
\psi_0^2 (\tau , r) & \lesssim \frac{1}{r} \cdot  \int_r^{\infty} ( \partial_{\rho} \psi_0 )^2 \, \rho^2 d\rho \lesssim  \frac{1}{r} \cdot  \int_r^{\infty} D ( \partial_{\rho} \psi_0 )^2 \, \rho^2 d\rho \\ \Rightarrow & r \cdot \psi_0^2 (\tau , r) \lesssim C_r^2 \int_{\Sigma_{\tau}} J^T [\psi_0 ] \cdot \textbf{n}_{\Sigma} \, d\mu_{\Sigma} \\ \Rightarrow & r \cdot \psi_0^2 (\tau , r) \lesssim \frac{C_r^2 E_0 \epsilon^2}{\tau^{3-\delta_1}} ,
\end{split}
\end{equation*}
where we used the fact $r > M$ in order to present the $D$ factor in the integral in the first line, and in the last line we used the decay estimate \eqref{dec:en0}.
\end{proof}

Using the previous estimates we have the following decay estimates close to the horizon for the spherically symmetric part of \eqref{nw}.
\begin{lemma}
Let $\psi_0$ be the spherically symmetric part of a solution $\psi$ of \eqref{nw} for which the assumptions of section \ref{as:boot} are satisfied. Then for all $(u,v) \in \mathcal{A}_{\tau_0}^{\infty} / SO(2)$ we have that
\begin{equation}\label{dec:psi0}
| \phi_0 | (u,v) \lesssim \frac{C^{1/2} E_0^{1/2} \epsilon}{v^{1-\delta_1/2}} ,
\end{equation}
for $\epsilon$ and $\delta_1$ as in the bootstrap assumptions of section \ref{as:boot}, and for all $(u,v) \in \mathcal{B}_{\tau_0}^{\infty}$ that
\begin{equation}\label{dec:psi01}
| \phi_0 | (u,v) \lesssim \frac{C^{1/2} E_0^{1/2} \epsilon}{u^{1-\delta_1/2}} .
\end{equation}
\end{lemma}
\begin{proof}
We apply the fundamental theorem of calculus to $\psi_0^2$ and we have that
\begin{equation*}
\begin{split}
\phi_0^2 (u,v) \lesssim \phi_0^2 (u_0 , v) +& \int_{\mcN_v^H} | \phi_0 | \cdot | \underline{L} \phi_0 | \, du \lesssim \psi_0^2 (u_0 , v)  \\ & + \int_{\mcN_v^H} \phi_0^2 \, du+ \int_{\mcN_v^H} | \phi_0 | \cdot | \underline{L} \phi_0 | \, du \\ \lesssim &  \phi_0^2 (u_0 , v)  + \int_{\Sigma_v } J^T [\psi_0 ] \cdot \textbf{n}_{\Sigma} \, d\mu_{\Sigma} \\ &+ \left( \int_{\mcN_v^H} \phi_0^2 \, D du \right)^{1/2} \cdot \left( \int_{\mcN_v^H} \frac{1}{D} (\underline{L} \phi_0 )^2 \, du \right)^{1/2} \\ \lesssim & \psi_0^2 (u_0 , v)  + \int_{\Sigma_v } J^T [\psi_0 ] \cdot \textbf{n}_{\Sigma} \, d\mu_{\Sigma} \\ &+ \left( \int_{\Sigma_v} J^T [\psi_0 ] \cdot \textbf{n}_{\Sigma}  \, d\mu_{\Sigma} \right)^{1/2} \cdot \left( \int_{\mcN_v^H} (r-M)^{-2} (\underline{L} \phi_0 )^2 \, du \right)^{1/2} \\ \lesssim & \frac{E_0 \epsilon^2}{v^{3-\delta_1}} + \frac{E_0^{1/2} \epsilon}{v^{3/2-\delta_1/2}} \cdot \frac{E_0^{1/2} \epsilon}{v^{1/2-\delta_1/2}} \lesssim \frac{E_0\epsilon^2}{v^{2-\delta_1}} ,
\end{split}
\end{equation*}
where we used the decay estimate in the interior \eqref{dec:psi0a}, the energy decay estimate \eqref{dec:en0} and the energy decay estimate \eqref{dec:en0p}. The estimate at infinity follows in a similar way using estimates \eqref{dec:en0p2}.
\end{proof}

For the non-spherically symmetric part we have the following decay estimates:
\begin{lemma}
Let $\psi_{\geq 1}$ be the non-spherically symmetric part of a solution $\psi$ of \eqref{nw} for which the assumptions of section \ref{as:boot} are satisfied. Then for all $k\leq 5$ and for all $(u,v) \in \mathcal{A}_{\tau_0}^{\infty} / SO(2)$ we have that
\begin{equation}\label{dec:psi1}
 \int_{\mathbb{S}^2} ( \Omega^k \phi_{\geq 1} )^2 (u,v,\omega) \, d\omega \lesssim \frac{C E_0 \epsilon^2}{v^{2+\delta_2}} ,
\end{equation}
for $\delta_2$ as in the bootstrap assumptions of section \ref{as:boot}, and for all $(u,v) \in \mathcal{B}_{\tau_0}^{\infty}$ that:
\begin{equation}\label{dec:psi11}
 \int_{\mathbb{S}^2} ( \Omega^k \phi_{\geq 1} )^2 (u,v,\omega) \, d\omega \lesssim \frac{C E_0 \epsilon^2}{u^{2+\delta_2}} .
\end{equation}
\end{lemma}

Combining the previous Lemmas gives us also the following estimates for any $k \leq 5$:
\begin{equation}\label{dec:psi}
 \int_{\mathbb{S}^2} ( \Omega^k \phi )^2 (u,v,\omega) \, d\omega \lesssim \frac{C E_0 \epsilon^2}{v^{2-\delta_1}} \mbox{  for all $(u,v) \in \mathcal{A}_{\tau_0}^{\infty} / SO(2)$} ,
\end{equation}
\begin{equation}\label{dec:psiinf}
 \int_{\mathbb{S}^2} ( \Omega^k \phi )^2 (u,v,\omega) \, d\omega \lesssim \frac{C E_0 \epsilon^2}{u^{2-\delta_1}} \mbox{  for all $(u,v) \in \mathcal{B}_{\tau_0}^{\infty} / SO(2)$} ,
\end{equation}
\begin{equation}\label{dec:psiint}
 r \int_{\mathbb{S}^2} ( \Omega^k \psi )^2 (\tau ,r,\omega) \, d\omega \lesssim \frac{C_r E_0 \epsilon^2}{\tau^{3-\delta_1}} \mbox{  for any $\tau$ and $M < r < \infty$} .
\end{equation}

After commuting with $T$ derivatives we have the following decay estimates for any $k\leq 5$: 
\begin{equation}\label{dec:tpsi}
 \int_{\mathbb{S}^2} ( \Omega^k T\phi )^2 (u,v,\omega) \, d\omega \lesssim \frac{C E_0 \epsilon^2}{v^{2+\delta_2}} \mbox{  for all $(u,v) \in \mathcal{A}_{\tau_0}^{\infty} / SO(2)$} ,
\end{equation}
\begin{equation}\label{dec:tpsiinf}
 \int_{\mathbb{S}^2} ( \Omega^k T\phi )^2 (u,v,\omega) \, d\omega \lesssim \frac{C E_0 \epsilon^2}{u^{2+\delta_2}} \mbox{  for all $(u,v) \in \mathcal{B}_{\tau_0}^{\infty} / SO(2)$} ,
\end{equation}
\begin{equation}\label{dec:tpsiint}
 r \int_{\mathbb{S}^2} ( \Omega^k T \psi )^2 (\tau ,r,\omega) \, d\omega \lesssim \frac{C_r E_0 \epsilon^2}{\tau^{3+\delta_2}} \mbox{  for any $\tau$ and $M < r < \infty$} ,
\end{equation}
\begin{equation}\label{dec:ttpsi}
 \int_{\mathbb{S}^2} ( \Omega^k T^2 \phi )^2 (u,v,\omega) \, d\omega \lesssim \frac{C E_0 \epsilon^2}{v^{1+\delta_2}} \mbox{  for all $(u,v) \in \mathcal{A}_{\tau_0}^{\infty} / SO(2)$} ,
\end{equation}
\begin{equation}\label{dec:ttpsiinf}
 \int_{\mathbb{S}^2} ( \Omega^k T^2 \phi )^2 (u,v,\omega) \, d\omega \lesssim \frac{C E_0 \epsilon^2}{u^{1+\delta_2}} \mbox{  for all $(u,v) \in \mathcal{B}_{\tau_0}^{\infty} / SO(2)$} ,
\end{equation}
\begin{equation}\label{dec:ttpsiint}
 r \int_{\mathbb{S}^2} ( \Omega^k T^2 \psi )^2 (\tau ,r,\omega) \, d\omega \lesssim \frac{C_r E_0 \epsilon^2}{\tau^{2+\delta_2}} \mbox{  for any $\tau$ and $M < r < \infty$} ,
\end{equation}
\begin{equation}\label{dec:tttpsi}
 \int_{\mathbb{S}^2} ( \Omega^k T^3 \phi )^2 (u,v,\omega) \, d\omega \lesssim \frac{C E_0 \epsilon^2}{v^{1/2 +\delta_2}} \mbox{  for all $(u,v) \in \mathcal{A}_{\tau_0}^{\infty} / SO(2)$} ,
\end{equation}
\begin{equation}\label{dec:tttpsiinf}
 \int_{\mathbb{S}^2} ( \Omega^k T^3 \phi )^2 (u,v,\omega) \, d\omega \lesssim \frac{C E_0 \epsilon^2}{u^{1/2 +\delta_2}} \mbox{  for all $(u,v) \in \mathcal{B}_{\tau_0}^{\infty} / SO(2)$} ,
\end{equation}
\begin{equation}\label{dec:tttpsiint}
 r \int_{\mathbb{S}^2} ( \Omega^k T^3 \psi )^2 (\tau ,r,\omega) \, d\omega \lesssim \frac{C_r E_0 \epsilon^2}{\tau^{1+\delta_2}} \mbox{  for any $\tau$ and $M < r < \infty$} ,
\end{equation}
\begin{equation}\label{dec:ttttpsi}
 \int_{\mathbb{S}^2} ( \Omega^k T^4 \phi )^2 (u,v,\omega) \, d\omega \lesssim \frac{C E_0 \epsilon^2}{v^{1/4+\delta_2}} \mbox{  for all $(u,v) \in \mathcal{A}_{\tau_0}^{\infty} / SO(2)$} ,
\end{equation}
\begin{equation}\label{dec:ttttpsiinf}
 \int_{\mathbb{S}^2} ( \Omega^k T^4 \phi )^2 (u,v,\omega) \, d\omega \lesssim \frac{C E_0 \epsilon^2}{u^{1/4+\delta_2}} \mbox{  for all $(u,v) \in \mathcal{B}_{\tau_0}^{\infty} / SO(2)$} ,
\end{equation}
\begin{equation}\label{dec:ttttpsiint}
 r \int_{\mathbb{S}^2} ( \Omega^k T^4  \psi )^2 (\tau ,r,\omega) \, d\omega \lesssim \frac{C_r E_0 \epsilon^2}{\tau} \mbox{  for any $\tau$ and $M < r < \infty$} ,
\end{equation}
\begin{equation}\label{dec:tttttpsi}
 \int_{\mathbb{S}^2} ( \Omega^k T^5 \phi )^2 (u,v,\omega) \, d\omega \lesssim C E_0 \epsilon^2 \mbox{  for all $(u,v) \in \mathcal{A}_{\tau_0}^{\infty} / SO(2)$} ,
\end{equation}
\begin{equation}\label{dec:tttttpsiinf}
 \int_{\mathbb{S}^2} ( \Omega^k T^5 \phi )^2 (u,v,\omega) \, d\omega \lesssim C E_0 \epsilon^2 \mbox{  for all $(u,v) \in \mathcal{B}_{\tau_0}^{\infty} / SO(2)$} ,
\end{equation}
\begin{equation}\label{dec:tttttpsiint}
 r \int_{\mathbb{S}^2} ( \Omega^k T^5  \psi )^2 (\tau ,r,\omega) \, d\omega \lesssim C_r E_0 \epsilon^2 \mbox{  for any $\tau$ and $M < r < \infty$} .
\end{equation}

\subsection{Boundedness of $\partial_r \psi$}\label{char}

In this section we will show that the transversal to the horizon $\partial_r$ derivative is bounded within spherical symmetry close to the horizon, while it decays (at a slow rate) outside spherical symmetry or after commuting with $T$ (again close to the horizon -- away it decays with a much better rate via the use of the elliptic estimates \eqref{est:elliptic}). For the boundedness result we will use the method of characteristics (following a similar approach to \cite{yannis1}) while for the decay results we will use the $(r-M)^{-p}$-weighted energy hierarchies. We also note that the following boundedness estimate involves a bootstrap argument under the assumptions of Section \ref{as:boot} (the same is done for the growth estimates of Section \ref{sec:GrowthEstimates}). The later assumptions are then verified through another bootstrap argument (in Section \ref{end:boot}) and the choice of the final smallness constant is the minimum of the constants involved in the aforementioned bootstrap arguments.

\begin{theorem}\label{thm:boundy0}
Let $\psi$ be a solution of \eqref{nw} with the corresponding initial data, and assume that the bootstrap assumptions of Section \ref{as:boot} hold true. Then there exists some $\epsilon' > 0$ such that for all $0 < \epsilon < \epsilon'$ we have for all $(u,v,\omega ) \in \mathcal{A}_{\tau_0}^{\infty}$ that:
\begin{equation}\label{est:boundy0}
\left| \frac{2r}{D} \underline{L} \phi \right| (u,v,\omega) \lesssim C \sqrt{E_0} \epsilon .
\end{equation}
\end{theorem}

\begin{proof}
The proof will follow a standard bootstrap argument. We note that equation \eqref{eq:thm} has the following form in double null coordinates for $\phi = r \cdot \psi$ close to the horizon
\begin{equation}\label{eq:dn}
\begin{split}
L \underline{L} \phi = & \mathcal{O} ( (r-M)^2 ) \slashed{\Delta} \phi + \mathcal{O} ((r-M)^3 ) \phi  \\ & +  \frac{Dr}{4} A \cdot g^{\alpha \beta} \cdot \partial_{\alpha} \psi \cdot \partial_{\beta} \psi .
\end{split}
\end{equation}
For the quantity 
$$h(u,v,\omega) = \left( \frac{2r}{D} \underline{L} \phi \right) (u,v,\omega) $$
the above equation \eqref{eq:dn} gives us the following equation
\begin{equation}\label{eq:yder}
\begin{split}
L \left( \frac{2r}{D} \underline{L} \phi \right) = &  \frac{1}{2r} \slashed{\Delta} \phi + \left( \frac{D'}{2} + \frac{D}{2r} \right) \cdot \left( \frac{2r}{D} \underline{L} \phi \right) \\ & + \mathcal{O} ( (r-M)) \cdot \phi + \frac{r^2}{2} A(x,\psi ) \cdot g^{\alpha \beta} \cdot \partial_{\alpha} \psi \cdot \partial_{\beta} \psi \\ = & \frac{1}{2r} \slashed{\Delta} \phi + \left( \frac{D'}{2} + \frac{D}{2r} \right) \cdot \left( \frac{2r}{D} \underline{L} \phi \right) + \mathcal{O} ( (r-M)) \cdot \phi \\ & + A \cdot \frac{1}{r} \cdot (L\phi ) \cdot \left( \frac{2r}{D} \underline{L} \phi \right) - A \cdot \frac{D}{2r^2} \cdot \phi \cdot \left( \frac{2r}{D} \underline{L} \phi \right) \\ & + A \cdot \frac{1}{r} \cdot \phi \cdot (L\phi) - A \cdot \frac{D}{2r^2} \cdot \phi^2 + \frac{A}{2} \cdot | \slashed{\nabla} \phi |^2 .
\end{split}
\end{equation}
The last equation is of the form
\begin{equation}\label{eq:y}
L h + (r-M) h \simeq F_h ,
\end{equation}
where 
\begin{equation*}
F_h =  \frac{1}{2r} \slashed{\Delta} \phi + \mathcal{O} ( (r-M)) \cdot \phi  + A \cdot \frac{1}{r} \cdot (L\phi ) \cdot \left( \frac{2r}{D} \underline{L} \phi \right) - A \cdot \frac{D}{2r^2} \cdot \phi \cdot \left( \frac{2r}{D} \underline{L} \phi \right)  + A \cdot \frac{1}{r} \cdot \phi \cdot (L\phi) - A \cdot \frac{D}{2r^2} \cdot \phi^2 + \frac{A}{2} \cdot | \slashed{\nabla} \phi |^2 .
\end{equation*}
We integrate the previous equation \eqref{eq:y} in the $v$ direction and we have that
\begin{equation*}
\begin{split}
\left( \frac{2r}{D} \underline{L} \phi \right) & (u,v,\omega) \simeq  \left( \frac{2r}{D} \underline{L} \phi \right) (u,v_0 ,\omega) + \frac{1}{D(u,v)} \int_{v_0}^v D \cdot \slashed{\Delta} \phi \, dv' \\ & + \frac{1}{D(u,v)} \int_{v_0}^v D \cdot \mathcal{O} ((r-M)) \cdot \phi \, dv'  + \frac{1}{D(u,v)} \int_{v_0}^v D \cdot \frac{A}{r} \cdot (L\phi ) \cdot \left( \frac{2r}{D} \underline{L} \phi \right) \, dv' \\ & - \frac{1}{D(u,v)} \int_{v_0}^v D \cdot A\cdot \frac{D}{2r^2} \cdot \phi \cdot \left( \frac{2r}{D} \underline{L} \phi \right) \, dv'  + \frac{1}{D(u,v)} \int_{v_0}^v D \cdot \frac{A }{r} \cdot \phi \cdot (L\phi) \, dv' \\ & - \frac{1}{D(u,v)} \int_{v_0}^v D \cdot A\cdot \frac{D}{2r^2} \cdot \phi^2 \, dv'  + \frac{1}{D(u,v)} \int_{v_0}^v D \cdot \frac{A}{2} \cdot | \slashed{\nabla} \phi |^2 \, dv' .
\end{split}
\end{equation*}
For the second term we have that
\begin{equation*}
\frac{1}{D(u,v)} \int_{v_0}^v D \cdot \slashed{\Delta} \phi \, dv' \leq  \int_{v_0}^v  \slashed{\Delta} \phi \, dv'  \lesssim  \int_{v_0}^v \frac{C^{1/2} E_0^{1/2} \epsilon}{(v' )^{1+\delta_2 / 2}} \, dv' \lesssim C^{1/2} E_0^{1/2} \epsilon ,
\end{equation*}
by using the pointwise decay estimates \eqref{dec:psi1}. For the third term we have that
\begin{equation*}
\begin{split}
\frac{1}{D(u,v)} \int_{v_0}^v D \cdot \mathcal{O} ((r-M)) \cdot \phi \, dv' \lesssim & \frac{1}{D(u,v)} \int_{v_0}^v D \cdot \mathcal{O} ((r-M)) \cdot \frac{C^{1/2} E_0^{1/2} \epsilon}{(v' )^{1-\delta_1 / 2}} \, dv' \\ \lesssim & C^{1/2} E_0^{1/2} \epsilon \frac{1}{D(u,v)} \int_{v_0}^v D \cdot D' \, dv' \simeq C^{1/2} E_0^{1/2} \epsilon \frac{1}{D(u,v)} \int_{v_0}^v L (D ) \, dv' \\ \lesssim & C^{1/2} E_0^{1/2} \epsilon \frac{1}{D(u,v)} \cdot \left. D(u,v' ) \right|_{v_0}^v \lesssim C^{1/2} E_0^{1/2} \epsilon ,
\end{split}
\end{equation*}
where we used the pointwise decay estimate \eqref{dec:psi}. For the fifth term we have that
\begin{equation*}
\begin{split}
\frac{1}{D(u,v)} \int_{v_0}^v D \cdot A \cdot \frac{D}{2r^2} \cdot & \phi \cdot \left( \frac{2r}{D} \underline{L} \phi \right) \, dv' \lesssim  E_0 \epsilon^2 \frac{1}{D(u,v)} \int_{v_0}^v D^2 \cdot \frac{1}{(v' )^{1-\delta_1 / 2}} \, dv' \\ \lesssim & E_0 \epsilon^2 \frac{1}{D(u,v)} \int_{v_0}^v D^2 \, dv' \lesssim  E_0 \epsilon^2 \frac{1}{D(u,v)} \int_{v_0}^v L (D ) \, dv' \\ \lesssim & E_0 \epsilon^2 \frac{1}{D(u,v)} \cdot \left. D(u,v' ) \right|_{v_0}^v \lesssim E_0 \epsilon^2 , 
\end{split}
\end{equation*}
where we used the pointwise decay estimate \eqref{dec:psi}. For the sixth term we integrate by parts and we have that
\begin{equation*}
\begin{split}
\frac{1}{D(u,v)} \int_{v_0}^v & D \cdot \frac{A(u,v,\psi) }{r}  \cdot \phi \cdot (L\phi) \, dv' \\ = & \frac{1}{2D(u,v)} \cdot  \left[ D \cdot \frac{A}{r} \phi^2 \right] (u,v' ,\omega) \Bigg{|}_{v_0}^v  - \frac{1}{2D(u,v)} \int_{v_0}^v L \left( D \cdot \frac{A}{r} \right) \cdot \phi^2 \, dv' \\ \lesssim & C E_0 \epsilon^2 + \int_{v_0}^v \phi^2 \, dv'  \lesssim  C E_0 \epsilon^2 + \int_{v_0}^v \frac{C E_0 \epsilon^2}{(v' )^{2-\delta_1}} \, dv'  \lesssim  C E_0 \epsilon^2 ,
\end{split}
\end{equation*}
where we used pointwise decay estimate \eqref{dec:psi} and the smallness of $\delta_1$, and the fact 
$$ L \left( D \cdot \frac{A}{r} \right) \simeq (r-M)^3 . $$
The seventh and eighth terms can be treated similarly. Finally for the fourth term we have that
\begin{equation*}
\begin{split}
\frac{1}{D(u,v)} \int_{v_0}^v D \cdot & \frac{A}{r} \cdot (L\phi ) \cdot  \left( \frac{2r}{D} \underline{L} \phi \right) \, dv' \\ = &  \frac{1}{D(u,v)} \cdot \left. \left[ D \cdot \frac{A}{r} \cdot \phi \cdot \left( \frac{2r}{D} \underline{L} \phi \right) \right] (u,v' , \omega) \right|_{v_0}^v  - \frac{1}{D(u,v)} \int_{v_0}^v L \left( \frac{D \cdot A}{r} \right) \cdot \phi \cdot \left( \frac{2r}{D} \underline{L} \phi \right) \, dv' \\ & -  \frac{1}{D(u,v)} \int_{v_0}^v D \cdot \frac{A}{r} \cdot \phi \cdot L \left( \frac{2r}{D} \underline{L} \phi \right) \, dv' ,
\end{split}
\end{equation*}
and we note that the term 
$$ - \frac{1}{D(u,v)} \int_{v_0}^v L \left( \frac{D \cdot A}{r} \right) \cdot \phi \cdot \left( \frac{2r}{D} \underline{L} \phi \right) \, dv' $$
is of size $\lesssim C E_0 \epsilon^2$ as it is of higher order in $D$ due to the $L$ derivative hitting $\frac{D \cdot A}{r}$, while for the last term we use the equation and we have that
\begin{equation}\label{boundy:aux}
\begin{split}
 -  \frac{1}{D(u,v)} \int_{v_0}^v D & \cdot \frac{A}{r} \cdot \phi \cdot  L \left( \frac{2r}{D} \underline{L} \phi \right) \, dv' \\ = &  -  \frac{1}{D(u,v)} \int_{v_0}^v D \cdot \frac{A}{r} \cdot \phi \cdot \slashed{\Delta} \phi \, dv'  -  \frac{1}{D(u,v)} \int_{v_0}^v D \cdot \frac{A}{r} \cdot \phi \cdot \mathcal{O} ((r-M)) \left(\frac{2r}{D} \underline{L} \phi \right) \, dv' \\ & -  \frac{1}{D(u,v)} \int_{v_0}^v D \cdot \frac{A^2 }{r^2} \cdot \mathcal{O} ((r-M)) \phi^2 \, dv'  -  \frac{1}{D(u,v)} \int_{v_0}^v D \cdot \frac{A}{r} \cdot \phi \cdot (L\phi ) \cdot \left( \frac{2r}{D} \underline{L} \phi \right) \, dv' \\ & +  \frac{1}{D(u,v)} \int_{v_0}^v D^2 \cdot \frac{A^2 }{2r^3} \cdot \phi^2 \cdot \left( \frac{2r}{D} \underline{L} \phi \right) \, dv'  -  \frac{1}{D(u,v)} \int_{v_0}^v D \cdot \frac{A^2 }{r^2} \cdot \phi^2 \cdot (L \phi ) \, dv' \\ & +  \frac{1}{D(u,v)} \int_{v_0}^v D^2 \cdot \frac{A^2 }{2r^3} \cdot \phi^3 \, dv' -  \frac{1}{D(u,v)} \int_{v_0}^v D \cdot \frac{A^2 }{2r} \cdot \phi \cdot | \slashed{\nabla} \phi |^2 \, dv' .
\end{split}
\end{equation}
We note that all the terms apart from the second one and the fourth one are integrable in $v$ as they can be bounded by
$$ \lesssim \int_{v_0}^v \frac{C E_0 \epsilon^2}{(v' )^{2-\delta_1}} \, dv' \mbox{  or  } \lesssim \int_{v_0}^v \frac{C^{3/2} E_0^{3/2} \epsilon^2}{(v' )^{2-\delta_1}} \, dv' $$
so they are of size $\lesssim C E_0 \epsilon^2 + C^{3/2} E_0^{3/2} \epsilon^2$ as $\delta_1$ is small enough, by using the pointwise decay estimates \eqref{dec:psi}, \eqref{dec:tpsi}, and the bootstrap assumption. For the second term we note that we have that
\begin{equation*}
\begin{split}
\frac{1}{D(u,v)} \int_{v_0}^v D \cdot & \frac{A(x,\psi)}{r} \cdot \phi \cdot \mathcal{O} ((r-M))  \left(\frac{2r}{D} \underline{L} \phi \right) \, dv' \\ \lesssim & \frac{1}{D(u,v)} \int_{v_0}^v D \cdot \mathcal{O} ((r-M)) \cdot \frac{C E_0 \epsilon^2}{(v' )^{1-\delta_1 / 2}} \, dv' \lesssim  \epsilon^2 \frac{1}{D(u,v)} \int_{v_0}^v D \cdot \mathcal{O} ((r-M)) \, dv' \\ \simeq & \epsilon^2 \frac{1}{D(u,v)} \int_{v_0}^v D \cdot D' \, dv'  \lesssim  C E_0 \epsilon^2 \epsilon^2 \frac{1}{D(u,v)} \int_{v_0}^v L ( D ) \, dv' \lesssim C E_0 \epsilon^2 ,
\end{split}
\end{equation*}
where now we took advantage of the $(r-M)$ factors inside the integral, after using the pointwise decay estimate \eqref{dec:psi} and the bootstrap assumption.

For the fourth term of \eqref{boundy:aux} we integrate by parts with respect to $L$ and use again the equation. It is easy to check that all the terms can be bounded by
$$ \lesssim C^{3/2} E_0^{3/2} \epsilon^3 + C^2 E_0^2 \epsilon^4 $$
after using the pointwise decay estimate \eqref{dec:psi} and the bootstrap assumption.

Gathering together all the above estimates we note that we got a contribution of size $ \lesssim C^{1/2} E_0^{1/2} \epsilon $ from all the linear terms, while from all the nonlinear terms we got a contribution of size $ \lesssim C E_0 \epsilon^2 + C^{3/2} E_0^{3/2} \epsilon^3 + C^2 E_0^2 \epsilon^4$, and this suffices in order to close the bootstrap argument and prove estimate \eqref{est:boundy} if we choose $\epsilon < \epsilon'$ for $\epsilon'$ small enough.
\end{proof}

For the restriction of $\psi$ to higher angular frequencies and for $T\psi$ we have that:
\begin{theorem}
Let $\psi$ be a solution of \eqref{nw} with the corresponding initial data, and assume that the assumptions of Section \eqref{as:boot} hold true. We have that
\begin{equation}\label{est:boundy1}
\sum_{k \leq 5} \int_{\mathbb{S}^2} \left(  \frac{2r}{D} \underline{L} \Omega^k  \phi_{\geq 1} \right)^2 (u,v,\omega) \, d\omega \lesssim \frac{C E_0 \epsilon^2}{v^{\delta_2}} ,
\end{equation}
and
\begin{equation}\label{est:boundyt}
\sum_{k \leq 5} \int_{\mathbb{S}^2} \left( \frac{2r}{D} \underline{L} \Omega^k T \phi \right)^2 (u,v,\omega) \, d\omega \lesssim \frac{C E_0 \epsilon^2}{v^{\delta_2}} ,
\end{equation}
for all $(u,v) \in \mathcal{A}_{\tau_0}^{\infty} / SO(3)$.
\end{theorem}
The proof of the above Theorem follows easily from the decay estimates \eqref{dec:en1p1} and \eqref{dec:entp1}, and an application of the fundamental theorem of calculus.

A combination of estimates \eqref{est:boundy0} and \eqref{est:boundy1} gives us the following:
\begin{equation}\label{est:boundy}
\sum_{k \leq 5} \int_{\mathbb{S}^2} \left(  \frac{2r}{D} \underline{L} \Omega^k  \phi \right)^2 (u,v,\omega) \, d\omega \lesssim C E_0 \epsilon^2 ,
\end{equation}
for all $(u,v) \in \mathcal{A}_{\tau_0}^{\infty} / SO(3)$ and for all $\epsilon < \epsilon'$ (for $\epsilon'$ as in Theorem \ref{thm:boundy0}).

Finally we record the following two auxiliary estimates for the transversal derivative to the horizon with added $(r-M)^q$-weights where $q \in [1/2 + \delta_1 , 3/2 + \delta_1/ 2 )$, and where $l \leq 5$:
\begin{equation}\label{est:auxl}
\int_{\mathbb{S}^2} (r-M)^q  \left( \frac{2r}{D} \underline{L} \Omega^l \phi \right)^2 (u,v,\omega )  \, d\omega \lesssim  \frac{C E_0 \epsilon^2}{v^{q-\delta_1 }} ,
\end{equation}
and
\begin{equation}\label{est:auxlt}
\int_{\mathbb{S}^2} (r-M)^q \left( \frac{2r}{D} \underline{L} \Omega^l T \phi \right)^2 (u,v,\omega )  \, d\omega \lesssim  \frac{C E_0 \epsilon^2}{v^{q - \delta_1 / 2 +\delta_2 /2}} ,
\end{equation}
for $(u,v) \in \mathcal{A}_{\tau_0}^{\infty} / SO(3)$, which follow from the use of the $(r-M)^{-p}$-weighted estimates.

\section{Growth estimates}
\label{sec:GrowthEstimates}

\subsection{Growth for $\partial_r^2 \psi$}\label{growth}
In \cite{aretakis2} it was shown that a linear wave on an extremal Reissner--Nordstr\"{o}m spacetime behaves as follows on the event horizon:
\begin{equation}\label{y2hor}
\left| \frac{2r}{D} \underline{L} \left( \frac{2r}{D} \underline{L} \phi \right) \right| (v,\omega) \Bigg{|}_{\mathcal{H}^{+}} \simeq   v . 
\end{equation}
Here our goal is to obtain an upper bound for the second transversal derivative of a nonlinear wave $\psi$ satisfying \eqref{nw} with small data in a neighbourhood of the horizon. We have the following:
\begin{theorem}\label{thm:boundyy}
Let $\psi$ be a solution of \eqref{nw} with corresponding initial data, and assume that the bootstrap assumptions of Section \ref{as:boot} hold true. Then for all $0 < \epsilon < \epsilon''$ where $\epsilon'' > 0$ is small enough, we have that
\begin{equation}\label{est:boundyy}
\left| D^{1/2-\delta /2} \frac{2r}{D} \underline{L} \left( \frac{2r}{D} \underline{L} \phi \right) \right| (u,v,\omega) \lesssim C^{1/2} E_0^{1/2} \epsilon v^{\delta} ,
\end{equation}
for all $(u,v,\omega ) \in \mathcal{A}_{\tau_0}^{\infty}$ and any $\delta \in (0,1]$.
\end{theorem}
\begin{remark}
Note that in our case as well we have an estimate of the form \eqref{y2hor} on the horizon (which corresponds to the case $\delta = 1$ in the aforementioned Theorem). This shows that for $\delta = 1$ estimate \eqref{est:boundyy} is sharp in the region $\mathcal{A}$. 
\end{remark}
\begin{proof}
Using equation \eqref{null:eq} we have that:
\begin{equation}\label{null:eqyy}
\begin{split}
4 L & \left( D^{1/2-\delta/2} \frac{2r}{D}  \underline{L} \left( \frac{2r}{D} \underline{L} \phi \right) \right) =  \Bigg( \frac{-4(r-M)}{M^2}D^{1/2-\delta/2} - \frac{2(1+\delta)M (r-M)}{r^3} D^{1/2-\delta/2} \\ & \quad + \mathcal{O} ( (r-M)^2 ) D^{1/2-\delta/2} \Bigg) \frac{2r}{D} \underline{L} \left( \frac{2r}{D} \underline{L} \phi \right)  + \frac{2}{r} D^{1/2-\delta/2} \slashed{\Delta} \left( \frac{2r}{D} \underline{L} \phi \right) \\ & + \left( \frac{4r^3}{M^4} D^{1/2-\delta/2} + \mathcal{O} ((r-M)) D^{1/2-\delta/2} \right) \frac{2r}{D} \underline{L} \phi  + \left( -\frac{2r^3}{M^2} D^{1/2-\delta/2} + \mathcal{O} (( r-M)) D^{1/2-\delta/2} \right) \slashed{\Delta} \phi \\ & + \left( - \frac{4r^3}{M^2} D^{1/2-\delta/2} + \mathcal{O} ((r-M)) D^{1/2-\delta/2} \right) \phi   +  D^{1/2-\delta/2} \frac{2r}{D} \underline{L} \left( \frac{A \cdot r}{2} \cdot g^{\alpha \beta} \cdot \partial_{\alpha} \psi \cdot \partial_{\beta} \psi \right)   .
\end{split}
\end{equation}
We examine the last nonlinear term and we have that using \eqref{null:nonlin}
\begin{equation*}
\begin{split}
D^{1/2-\delta/2} \frac{2r}{D} \underline{L} & \left( \frac{A \cdot r}{2} \cdot g^{\alpha \beta} \cdot \partial_{\alpha} \psi \cdot \partial_{\beta} \psi \right) =  -  D^{1/2-\delta/2} \frac{2r}{D} \underline{L} \left( \frac{A \cdot r}{2} \right) \cdot  g^{\alpha \beta} \cdot \partial_{\alpha} \psi \cdot \partial_{\beta} \psi  \\ & + A \cdot D^{1/2-\delta/2} \frac{r^2}{D} \underline{L} \left[ \frac{2}{r^3} L \phi \cdot \left( \frac{2r}{D}\underline{L} \phi \right) \right]  - A \cdot D^{1/2-\delta/2} \frac{r^2}{D} \underline{L} \left[ \frac{D}{r^4} \phi \cdot \left( \frac{2r}{D} \underline{L} \phi \right) \right] \\ & + A \cdot D^{1/2-\delta/2} \frac{r^2}{D} \underline{L}   \left( \frac{2}{r^3} \phi \cdot L\phi \right)  - A \cdot D^{1/2-\delta/2} \frac{r^2}{D} \underline{L}   \left( \frac{D}{r^4} \phi^2 \right) \\ & + A \cdot D^{1/2-\delta/2} \frac{r^2}{D} \underline{L} ( | \slashed{\nabla} \phi |^2 )  \doteq  f_{H;n}^1 + f_{H;n}^2 + f_{H;n}^3 + f_{H;n}^4 + f_{H;n}^5 + f_{H;n}^6 .
\end{split}
\end{equation*}
We examine each term separately. We leave the first term as is for now. For the second term we compute that
\begin{equation*}
\begin{split}
f_{H;n}^2 = & A \cdot D^{1/2-\delta/2} \frac{r^2}{D} \underline{L} \left[ \frac{2}{r^3} L \phi \cdot \left( \frac{2r}{D}\underline{L} \phi \right) \right]  =  A \cdot D^{1/2-\delta/2} \cdot \left( \frac{6}{r^2} L\phi \right) \cdot \left( \frac{2r}{D} \underline{L} \phi \right) \\ & + A \cdot D^{1/2-\delta/2} \cdot \left( \frac{1}{r^2} L\phi \right) \cdot \left[ \frac{2r}{D} \underline{L} \left( \frac{2r}{D} \underline{L} \phi \right) \right]  -  A \cdot D^{1/2 - \delta/2} \cdot \frac{D'}{2r^2}  \phi \cdot \left( \frac{2r}{D} \underline{L} \phi \right) \\ & + A \cdot D^{1/2 - \delta/2} \cdot \frac{1}{2r}   \cdot ( \slashed{\Delta} \phi )\cdot \left( \frac{2r}{D} \underline{L} \phi \right) +A \cdot D^{1/2-\delta/2} \frac{1}{2r} g^{\alpha \beta} \cdot \partial_{\alpha} \psi \cdot \partial_{\beta} \psi \cdot \left( \frac{2r}{D} \underline{L} \phi \right) .
\end{split}
\end{equation*} 
For the third term we compute that
\begin{equation*}
\begin{split}
f_{H;n}^3 = & - A \cdot D^{1/2-\delta/2} \frac{r^2}{D} \underline{L} \left[ \frac{D}{r^4} \phi \cdot \left( \frac{2r}{D} \underline{L} \phi \right) \right]  =  - A \cdot D^{1/2-\delta/2} \left( \frac{2D}{r^3} - \frac{D'}{r^2} \right) \phi \cdot \left( \frac{2r}{D} \underline{L} \phi \right) \\ & - A \cdot D^{1/2-\delta/2} \frac{D}{2r^3} \left( \frac{2r}{D} \underline{L} \phi \right)^2  - A \cdot D^{1/2-\delta/2} \frac{D}{2r^3} \phi \cdot \left[ \frac{2r}{D} \underline{L} \left( \frac{2r}{D} \underline{L} \phi \right) \right] .
\end{split}
\end{equation*}
For the fourth term we compute that
\begin{equation*}
\begin{split}
f_{H;n}^4 = & A \cdot D^{1/2-\delta/2} \frac{r^2}{D} \underline{L}   \left( \frac{2}{r^3} \phi \cdot L\phi \right) =  A \cdot D^{1/2-\delta/2} \frac{6}{r^2} \phi \cdot L\phi \\ & + A \cdot D^{1/2-\delta/2} \cdot \left( \frac{1}{r^2}  L\phi \right) \cdot \left( \frac{2r}{D} \underline{L} \phi \right) -  A \cdot D^{1/2 - \delta/2} \cdot \frac{D'}{2r^2}  \phi ^2 \\ & + A \cdot D^{1/2 - \delta/2} \cdot \frac{1}{2r}  \phi \cdot ( \slashed{\Delta} \phi ) + A \cdot D^{1/2-\delta/2} \frac{1}{r} \phi \cdot g^{\alpha \beta} \cdot \partial_{\alpha} \psi \cdot \partial_{\beta} \psi . 
\end{split}
\end{equation*}
For the fifth term we compute that
\begin{equation*}
\begin{split}
 f_{H;n}^5 = & - A \cdot D^{1/2-\delta/2} \frac{r^2}{D} \underline{L}   \left( \frac{D}{r^4} \phi^2 \right) =  - A \cdot D^{1/2-\delta/2} \left( - \frac{D'}{r^2} + \frac{2D}{r^3} \right) \cdot \phi^2 \\ & - A \cdot D^{1/2-\delta/2} \frac{D}{r^3} \phi \cdot \left( \frac{2r}{D} \underline{L} \phi \right) .
\end{split}
\end{equation*}
For the sixth term we compute that
\begin{equation*}
\begin{split}
f_{H;n}^6 = & A \cdot D^{1/2-\delta/2} \frac{r^2}{D} \underline{L} ( |\slashed{\nabla} \phi |^2 ) =  A \cdot D^{1/2-\delta/2} 2Dr^2 | \slashed{\nabla} \phi |^2 \\ & + A \cdot D^{1/2-\delta/2} \frac{1}{r} \left\langle  \nabla_{\mathbb{S}^2} \left( \frac{2r}{D} \underline{L} \phi \right) , \nabla_{\mathbb{S}^2} \phi \right\rangle .
\end{split}
\end{equation*}
We get the following equation:
\begin{equation}\label{eq:nlw2}
L f_{H;\delta} + \frac{3+\delta}{2} (r-M) f_{H;\delta} \simeq g_{H;n,\delta} ,
\end{equation}
where 
$$ f_{H;\delta} = D^{1/2-\delta/2} \frac{r}{2D} \underline{L} \left( \frac{2r}{D} \underline{L} \phi \right) ,$$
and where
$$ g_{H;n,\delta} \doteq g_{H;\delta} + \frac{1}{4} \left( f_{H;n}^1 + f_{H;n}^2 + f_{H;n}^3 + f_{H;n}^4 + f_{H;n}^5 + f_{H;n}^6 \right), $$
for $g_{H;\delta}$ given by:
\begin{equation*}
\begin{split}
g_{H;\delta} (u,v,\omega) = & \frac{2}{r} D^{1/2-\delta/2} \slashed{\Delta} \left( \frac{2r}{D} \underline{L} \phi \right)  + \left( \frac{4r^3}{M^4} D^{1/2-\delta/2} + \mathcal{O} ((r-M)) D^{1/2-\delta/2} \right) \frac{2r}{D} \underline{L} \phi \\ & + \left( -\frac{2r^3}{M^2} D^{1/2-\delta/2} + \mathcal{O} (( r-M)) D^{1/2-\delta/2} \right) \slashed{\Delta} \phi  + \left( - \frac{4r^3}{M^2} D^{1/2-\delta/2} + \mathcal{O} ((r-M)) D^{1/2-\delta/2} \right) \phi .
\end{split}
\end{equation*}
Using the boundedness of $A$ we have that:
\begingroup
\allowdisplaybreaks
\begin{align*}
f_{H;n}^2 + & f_{H;n}^3 + f_{H;n}^4 + f_{H;n}^5 + f_{H;n}^6 \simeq  D^{1/2-\delta/2} \cdot (L\phi ) \cdot \left[ \frac{2r}{D} \underline{L} \left( \frac{2r}{D} \underline{L} \phi \right) \right] \\ & + D^{1/2-\delta/2}  \cdot D \phi \cdot \left[ \frac{2r}{D} \underline{L} \left( \frac{2r}{D} \underline{L} \phi \right) \right]   + D^{1/2-\delta/2} \cdot D \left( \frac{2r}{D} \underline{L} \phi \right)^2 \\ & + D^{1/2-\delta/2} \cdot (L\phi ) \cdot \left( \frac{2r}{D} \underline{L} \phi \right)  + D^{1/2-\delta/2} \cdot (\slashed{\Delta} \phi ) \cdot \left( \frac{2r}{D} \underline{L} \phi \right)+ D^{1/2-\delta/2} \cdot D' \phi \cdot \left( \frac{2r}{D} \underline{L} \phi \right) \\ & + D^{1/2-\delta/2} \cdot \phi \cdot (L\phi )  +D^{1/2-\delta/2} \cdot \phi \cdot (\slashed{\Delta} \phi ) + D^{1/2-\delta/2} \cdot D' \phi^2 \\ & +D^{1/2-\delta/2} 2Dr^2 | \slashed{\nabla} \phi |^2  + D^{1/2-\delta/2} \frac{1}{r} \left\langle  \nabla_{\mathbb{S}^2} \left( \frac{2r}{D} \underline{L} \phi \right) , \nabla_{\mathbb{S}^2} \phi \right\rangle\\ & + D^{1/2-\delta/2} \cdot \left( \frac{2r}{D} \underline{L} \phi \right) \cdot g^{\alpha \beta} \cdot \partial_{\alpha} \psi \cdot \partial_{\beta} \psi  + D^{1/2-\delta/2} \cdot \phi \cdot g^{\alpha \beta} \cdot \partial_{\alpha} \psi \cdot \partial_{\beta} \psi .
\end{align*}
\endgroup
We solve \eqref{eq:nlw2}, which gives us that
\begin{equation*}
f_{H;\delta} (u,v,\omega) \simeq \frac{D^{\frac{3+\delta}{2}} (u,v_0)}{D^{\frac{3+\delta}{2}} (u,v)} f_{H;\delta} (u,v_0 , \omega ) + \frac{1}{D^{\frac{3+\delta}{2}} (u,v)} \int_{v_0}^v D^{\frac{3+\delta}{2}} (u,v' ) \cdot g_{H;n,\delta} \, dv' . 
\end{equation*}
The term
$$ \frac{D^{\frac{3+\delta}{2}} (u,v_0)}{D^{\frac{3+\delta}{2}} (u,v)} f_{H;\delta} (u,v_0 , \omega ) $$
is bounded by $\lesssim C^{1/2} E_0^{1/2} \epsilon$ by the properties of $D$ and our initial data assumptions. From the inhomogeneous part we first we look at the contribution of the linear terms coming from $g_{H;\delta}$ and more specifically the second one of them (which turns out to be the worst) we break the integral into the regions $v_0 \leq v' \leq \frac{u}{2}$ and $\frac{u}{2} \leq v' \leq v$ (noting that if $v \leq \frac{u}{2}$ then the second integral is just 0) and we have that:
\begingroup
\allowdisplaybreaks
\begin{align*}
\frac{1}{D^{\frac{3+\delta}{2}} (u,v)} \int_{v_0}^v D^2 (u,v') \cdot &  \left( \frac{2r}{D} \underline{L} \phi \right) \, dv' =  \frac{1}{D^{\frac{3+\delta}{2}} (u,v)} \int_{v_0}^{\min ( v,u/2 ) } D^2  \cdot \left( \frac{2r}{D} \underline{L} \phi \right) \, dv' \\ & + \frac{1}{D^{\frac{3+\delta}{2}} (u,v)} \int_{\min(v,u/2 )}^v D^2  \cdot  \left( \frac{2r}{D} \underline{L} \phi \right) \, dv' .
\end{align*}
\endgroup
In the region $v_0 \leq v' \leq u/2$ we note that we have $(r-M) \lesssim \frac{1}{v}$ which implies after using the boundedness estimate \eqref{est:boundy} that we have the following for the first term of the above expression:
\begin{equation*}
\begin{split}
\frac{1}{D^{\frac{3+\delta}{2}} (u,v)} \int_{v_0}^{\min (v,u/2 )} D^2 \cdot & \left( \frac{2r}{D} \underline{L} \phi \right) \, dv' \lesssim C^{1/2} E_0^{1/2} \epsilon \int_{v_0}^{\min (v,u/2)} D^{\frac{1-\delta}{2}} \, dv \\ \lesssim &  C^{1/2} E_0^{1/2} \epsilon \int_{v_0}^{u/2} \frac{1}{(v' )^{1-\delta}} \, dv' \lesssim C^{1/2} E_0^{1/2} \epsilon v^{\delta} .
\end{split}
\end{equation*}
For the second term of the previous expression we now work in the region $u/2 \leq v' \leq v$ where it holds that $(r-M) \gtrsim \frac{1}{v}$ which implies after using again the boundedness estimate \eqref{est:boundy} that:
\begin{equation*}
\begin{split}
\frac{1}{D^{\frac{3+\delta}{2}} (u,v)} \int_{\min(v,u/2)}^{v} D^2  \cdot & \left( \frac{2r}{D} \underline{L} \phi \right) \, dv' \lesssim C^{1/2} E_0^{1/2} \epsilon \int_{\min(v,u/2)}^{v} D^{\frac{1-\delta}{2}} \, dv \\ \lesssim &  C^{1/2} E_0^{1/2} \epsilon \frac{1}{(r-M)^{\delta}} \Bigg{|}_{\min(v,u/2)}^v \lesssim C^{1/2} E_0^{1/2} \epsilon v^{\delta} .
\end{split}
\end{equation*}
All the other terms coming from $g_{H;\delta}$ can be treated similarly (noting that they behave better than the above after using the pointwise decay estimates \eqref{dec:psi} for the terms involving $\slashed{\Delta} \phi$ and $\phi$, and the decay estimates \eqref{est:boundy1} for the term involving $\slashed{\Delta} \left( \frac{2r}{D} \underline{L} \phi \right)$). All the terms of $f_{H;n}^1$ can be treated also in a similar manner and it is easy to see that they are better and of size $\lesssim C E_0 \epsilon^2 v^{\delta}$. Now we consider the second term of the sum $f_{H;n}^2 + f_{H;n}^3 + f_{H;n}^4 + f_{H;n}^5 + f_{H;n}^6$ and we have that
\begingroup
\allowdisplaybreaks
\begin{align*}
\frac{1}{D^{\frac{3+\delta}{2}} (u,v)} & \int_{v_0}^v D^{\frac{3+\delta}{2}}  \cdot D \phi \cdot  D^{\frac{1-\delta}{2}} \left[ \frac{2r}{D} \underline{L} \left( \frac{2r}{D} \underline{L} \phi \right) \right] \, dv' \\ \lesssim &  \frac{1}{D^{\frac{3+\delta}{2}} (u,v)} \int_{v_0}^v D^{\frac{3+\delta}{2}} \cdot D^{1/2} \cdot \frac{C E_0 \epsilon^2}{(v' )^{1-\sigma-\delta'}} \, dv' \lesssim   C E_0 \epsilon^2 \frac{1}{D^{\frac{3+\delta}{2}} (u,v)} \int_{v_0}^v D^{\frac{3+\delta}{2}} \cdot D^{1/2} \, dv' \\ \lesssim &  C E_0 \epsilon^2 \frac{1}{D^{\frac{3+\delta}{2}} (u,v)} \int_{v_0}^v L ( D^{\frac{3+\delta}{2}} ) \, dv' \lesssim  C E_0 \epsilon^2,
\end{align*}
\endgroup
where we used the pointwise decay estimates \eqref{dec:psi} and the fact that
$$ D^{1/2} D^{1/2 -\delta /2 } \left| \frac{2r}{D} \underline{L} \left( \frac{2r}{D} \underline{L} \phi \right) \right| (u,v,\omega ) \lesssim C^{1/2} E_0^{1/2} v^{\delta'} , $$
for any $\delta' > 0$ and for all $\delta \in (0,1]$ (due to the extra $D^{1/2}$ weight) by our bootstrap assumptions. The previous estimate follows by choosing $\delta'$ small enough so that $\delta' + \delta_1 < 1$. For the third term of the sum  $f_{H;n}^2 + f_{H;n}^3 + f_{H;n}^4 + f_{H;n}^5 + f_{H;n}^6$ we have that
\begingroup
\allowdisplaybreaks
\begin{align*}
\frac{1}{D^{\frac{3+\delta}{2}} (u,v)} & \int_{v_0}^v D^{\frac{3+\delta}{2}}  \cdot  D^{\frac{1-\delta}{2}} D \left( \frac{2r}{D} \underline{L} \phi \right)^2 \, dv'  \\ \lesssim & \frac{1}{D^{\frac{3+\delta}{2}} (u,v)} \int_{v_0}^v D^{\frac{3+\delta}{2}}  \cdot D^{\frac{1-\delta}{2}} D^{\frac{1}{2}} \left( \frac{2r}{D} \underline{L} \phi \right)^2 \, dv'  \lesssim  C E_0 \epsilon^2 \frac{1}{D^{\frac{3+\delta}{2}} (u,v)} \int_{v_0}^v D^{\frac{3+\delta}{2}}  \cdot D^{\frac{1}{2}} \, dv'  \, \\ \lesssim & C E_0 \epsilon^2 \frac{1}{D^{\frac{3+\delta}{2}} (u,v)} \int_{v_0}^v L ( D^{\frac{3+\delta}{2}} ) \, dv' \lesssim  C E_0 \epsilon^2,
\end{align*}
\endgroup
where we used the boundedness estimate \eqref{est:boundy}. For the fourth term of the sum $f_{H;n}^2 + f_{H;n}^3 + f_{H;n}^4 + f_{H;n}^5 + f_{H;n}^6$ we integrate by parts and we have that:
\begingroup
\allowdisplaybreaks
\begin{align*}
\frac{1}{D^{\frac{3+\delta}{2}} (u,v)} & \int_{v_0}^v D^{\frac{3+\delta}{2}}   \cdot D^{\frac{1-\delta}{2}} (L\phi ) \cdot  \left( \frac{2r}{D} \underline{L} \phi \right) \, dv'  =  \frac{1}{D^{\frac{3+\delta}{2}} (u,v)} \left[ D^{\frac{3+\delta}{2}}  \phi (u,v,\omega)  \cdot  \left( \frac{2r}{D} \underline{L} \phi  \right) (u,v,\omega) \right]_{v_0}^v  \\ & - \frac{1}{D^{\frac{3+\delta}{2}} (u,v)} \int_{v_0}^v L \left( D^{\frac{3+\delta}{2}} \cdot D^{\frac{1-\delta}{2}} \right) \phi   \cdot  \left( \frac{2r}{D} \underline{L} \phi \right) \, dv'  - \frac{1}{D^{\frac{3+\delta}{2}} (u,v)} \int_{v_0}^v D^{\frac{3+\delta}{2}} \cdot D^{\frac{1-\delta}{2}} \phi   \cdot  L \left( \frac{2r}{D} \underline{L} \phi \right) \, dv' .
\end{align*}
\endgroup
The last of term of this last expression is better or similar to the rest by using the equation \eqref{eq:yder}. The first term is just bounded by $C E_0 \epsilon^2$, while the second term can be bounded by $C E_0 \epsilon^2$ as well by using the boundedness of $\phi$ and $\frac{2r}{D} \underline{L} \phi$ and the presence of the $D^{5/2}$ weight in the integral. For the fifth term of the sum $f_{H;n}^2 + f_{H;n}^3 + f_{H;n}^4 + f_{H;n}^5 + f_{H;n}^6$ we have that
\begingroup
\allowdisplaybreaks
\begin{align*}
\frac{1}{D^{\frac{3+\delta}{2}} (u,v)} \int_{v_0}^v & D^{\frac{3+\delta}{2}} \cdot D^{\frac{1-\delta}{2}} D' \phi \cdot  \left( \frac{2r}{D} \underline{L} \phi \right) \, dv' \lesssim  \frac{1}{D^{\frac{3+\delta}{2}} (u,v)} \int_{v_0}^v D^{\frac{3+\delta}{2}} \cdot D^{\frac{1}{2}} \frac{C E_0 \epsilon^2}{(v' )^{1-\delta_1 / 2}} \, dv' \\ \lesssim &  C E_0 \epsilon^2 \frac{1}{D^{\frac{3+\delta}{2}} (u,v)} \int_{v_0}^v D^{\frac{3+\delta}{2}}  \cdot D^{\frac{1}{2}} \, dv' \lesssim  C E_0 \epsilon^2 \frac{1}{D^{\frac{3+\delta}{2}} (u,v)} \int_{v_0}^v L ( D^{\frac{3+\delta}{2}} ) \, dv' \lesssim  C E_0 \epsilon^2. 
\end{align*}
\endgroup
For the sixth term of the sum $f_{H;n}^2 + f_{H;n}^3 + f_{H;n}^4 + f_{H;n}^5 + f_{H;n}^6$ we integrate by parts and we have that
\begingroup
\allowdisplaybreaks
\begin{align*}
\frac{1}{D^{\frac{3+\delta}{2}} (u,v)} \int_{v_0}^v & D^{\frac{3+\delta}{2}} \cdot   D^{\frac{1-\delta}{2}} \phi \cdot (L \phi ) \, dv' =  \frac{1}{2D^{\frac{3+\delta}{2}} (u,v)} \left[ D^2 \phi^2 (u,v,\omega) \right]_{v_0}^v \\ & -\frac{1}{2D^{\frac{3+\delta}{2}} (u,v)} \int_{v_0}^v L (D^2 ) \phi^2 \, dv' .
\end{align*}
\endgroup
The first term is bounded by $C E_0 \epsilon^2$ by the properties of $D$ and the boundedness of $\phi$, while the second term is also bounded by $C E_0 \epsilon^2$ again by the boundedness of $\phi$ and the presence of the $D^{5/2}$ inside the integral. For the seventh term we have that
\begin{equation*}
\frac{1}{D^{\frac{3+\delta}{2}} (u,v)} \int_{v_0}^v D^{\frac{3+\delta}{2}} (u,v' )  \cdot D^{\frac{1-\delta}{2}} D' \phi^2 \, dv' \lesssim \int_{v_0}^v \frac{\epsilon^2}{(v' )^{2-2\sigma}} \, dv'  \lesssim  \epsilon^2 ,
\end{equation*}
again for $\sigma$ small enough so that the term in the last integral is integrable. The eight term of the sum $f_{H;n}^2 + f_{H;n}^3 + f_{H;n}^4 + f_{H;n}^5 + f_{H;n}^6$ can be treated similarly to the term of the last expression. For the ninth term of the sum $f_{H;n}^2 + f_{H;n}^3 + f_{H;n}^4 + f_{H;n}^5 + f_{H;n}^6$ we have that
\begin{equation*}
\begin{split}
\frac{1}{D^{\frac{3+\delta}{2}} (u,v)} \int_{v_0}^v D^{\frac{3+\delta}{2}} (u,v' )  \cdot & D^{1/2-\delta/2} \cdot \frac{1}{r} \left\langle  \nabla_{\mathbb{S}^2} \left( \frac{2r}{D} \underline{L} \phi \right) , \nabla_{\mathbb{S}^2} \phi \right\rangle \, dv' \\ \lesssim & \frac{1}{D^{\frac{3+\delta}{2}} (u,v)} \int_{v_0}^v D^{\frac{3+\delta}{2}} (u,v' )  \cdot  D^{1/2-\delta/2} \cdot \left| \nabla_{\mathbb{S}^2} \left( \frac{2r}{D} \underline{L} \phi \right) \right| \cdot | \nabla_{\mathbb{S}^2} \phi | \, dv' \\ \lesssim & \int_{v_0}^v \frac{C E_0 \epsilon^2}{(v' )^{1+\delta_2}} \, dv' \lesssim C E_0 \epsilon^2 ,
\end{split}
\end{equation*}
where we used the pointwise decay estimate \eqref{dec:psi1}. Note that the last two terms of the sum $f_{H;n}^2 + f_{H;n}^3 + f_{H;n}^4 + f_{H;n}^5 + f_{H;n}^6$ are of cubic nature, so they can be treated similarly and they are bounded by $\lesssim C^{3/2} E_0^{3/2} \epsilon^3 v^{\delta}$. Finally for the first term of the sum $f_{H;n}^2 + f_{H;n}^3 + f_{H;n}^4 + f_{H;n}^5 + f_{H;n}^6$ we integrate by parts and we have that
\begingroup
\allowdisplaybreaks
\begin{align*}
\frac{1}{D^{\frac{3+\delta}{2}} (u,v)} & \int_{v_0}^v D^{\frac{3+\delta}{2}} (u,v' )  \cdot  D^{1/2-\delta/2} \cdot (L\phi ) \cdot  \left[ \frac{2r}{D} \underline{L} \left( \frac{2r}{D} \underline{L} \phi \right) \right] \, dv' \\ \simeq & \frac{1}{D^{\frac{3+\delta}{2}} (u,v)} \cdot \left. D^{\frac{3+\delta}{2}} \cdot  \phi  \cdot  D^{1/2-\delta/2} \left[ \frac{2r}{D} \underline{L} \left( \frac{2r}{D} \underline{L} \phi \right) \right] \right|_{v_0}^v \\ & - \frac{1}{D^{\frac{3+\delta}{2}} (u,v)} \int_{v_0}^v D^{\frac{3+\delta}{2}}  \mathcal{O} ((r-M)) \phi \cdot D^{\frac{1-\delta}{2}} \left[ \frac{2r}{D} \underline{L} \left( \frac{2r}{D} \underline{L} \phi \right) \right] \, dv' \\ & - \frac{1}{D^{\frac{3+\delta}{2}} (u,v)} \int_{v_0}^v D^{\frac{3+\delta}{2}} \cdot D^{\frac{1-\delta}{2}} \phi \cdot \slashed{\Delta} \left( \frac{2r}{D} \underline{L} \phi \right) \, dv' - \frac{1}{D^{\frac{3+\delta}{2}} (u,v)} \int_{v_0}^v D^{\frac{3+\delta}{2}} \cdot D^{\frac{1-\delta}{2}} \phi \cdot \left( \frac{2r}{D} \underline{L} \phi \right) \, dv' \\ & - \frac{1}{D^{\frac{3+\delta}{2}} (u,v)} \int_{v_0}^v D^{\frac{3+\delta}{2}} \cdot D^{\frac{1-\delta}{2}} \phi \cdot \slashed{\Delta} \phi \, dv'  - \frac{1}{D^{\frac{3+\delta}{2}} (u,v)} \int_{v_0}^v D^{\frac{3+\delta}{2}} \cdot D^{\frac{1-\delta}{2}} \phi^2 \, dv' \\ &    - \frac{1}{D^{\frac{3+\delta}{2}} (u,v)} \int_{v_0}^v D^{\frac{3+\delta}{2}} \cdot D^{\frac{1-\delta}{2}} \phi \cdot \frac{2r}{D} \underline{L} \left( \frac{r}{2} \cdot g^{\alpha \beta} \cdot \partial_{\alpha} \psi \cdot \partial_{\beta} \psi \right) \, dv' .
\end{align*}
\endgroup
The last term of the last expression is of cubic nature and can be treated similarly to the rest, in the end it can be bounded by $C^{3/2} E_0^{3/2} \epsilon^3 v^{\delta}$. The first term due to the pointwise decay estimate \eqref{dec:psi}, the properties of $D$ and the bootstrap assumption is of size $\lesssim C E_0 \epsilon^2 v^{\delta}$. The second term can be bounded by $C E_0 \epsilon^2 v^{\delta}$ by using the pointwise decay estimate \eqref{dec:psi}, the bootstrap assumption, and the presence of the term $D^{\frac{3+\delta}{2}} D^{1/2}$ inside the integral. For the fourth term we have for $\delta = 1$ that
\begin{equation*}
\frac{1}{D^{2} (u,v)} \int_{v_0}^v D^{2} |\phi | \cdot  \left| \frac{2r}{D} \underline{L} \phi \right| \, dv'  \lesssim   \int_{v_0}^v \frac{C E_0 \epsilon^2}{v^{1-\delta_1}} \, dv'  \lesssim  C E_0 \epsilon^2 v^{\delta_1} \lesssim C E_0 \epsilon^2 v^{\delta} ,
\end{equation*}
where we used the pointwise decay estimate \eqref{dec:psi} and $\delta_1 \ll 1$. On the other hand for $\delta = 0$ (which is outside the range of the Theorem) we have that
\begin{equation*}
\begin{split}
\frac{1}{D^{3/2} (u,v)} \int_{v_0}^v D^{2} |\phi | \cdot & \left| \frac{2r}{D} \underline{L} \phi \right| \, dv'  \lesssim   \frac{1}{D^{3/2} (u,v)} \int_{v_0}^v D^{2} \frac{C E_0 \epsilon^2}{v^{1-\delta_1}} \, dv' \\ \lesssim & C E_0 \epsilon^2 \frac{1}{D^{3/2} (u,v)} \int_{v_0}^v D^{2} \, dv'  \lesssim  C E_0 \epsilon^2 ,
\end{split}
\end{equation*}
where we used again the pointwise decay estimate \eqref{dec:psi}. The rest of the $\delta$ range follows by interpolation. The third, fifth and sixth terms can be treated similarly (and as can be easily seen they admit better bounds). Similarly we can treat the term $D^{1/2-\delta/2} (\slashed{\Delta} \phi ) \left( \frac{2r}{D} \underline{L} \phi \right)$.

In the very end, we note that we were able to bound $D^{1/2-\delta /2} \left| \frac{2r}{D} \underline{L} \left( \frac{2r}{D} \underline{L} \phi \right) \right| (u,v,\omega)$ by
$$ \lesssim C^{1/2} E_0^{1/2} \epsilon v^{\delta} + C E_0 \epsilon^2 v^{\delta} + C^{3/2} E_0^{3/2} \epsilon^3 v^{\delta} , $$
and for $\epsilon$ small enough this is bounded by $C^{1/2} E_0^{1/2} \epsilon v^{\delta}$. 

The Theorem now follows by gathering together all previous estimates.
\end{proof}

\subsection{Some auxiliary estimates}\label{aux}
From our previous pointwise estimates it is clear that our energy estimates are not enough to conclude that $T^m\psi$ for $m \in \{ 0,1,2,3 \}$ apart from the case of $m=0$, are integrable in $v$ close to the horizon. Nevertheless we will show some weighted boundedness estimates in $v$ for $T^m \psi$, $m \in \{ 0 ,1,2,3 \}$. 

\begin{theorem}\label{thm:aux}
Let $\psi$ be a solution of the equation \eqref{nw}. Under the bootstrap assumptions of section \eqref{as:boot} and for all $0 < \epsilon < \epsilon_0$ for $\epsilon_0$ small enough we have that
\begin{equation}\label{est:aux1}
\int_{v_0}^V \int_{\mathbb{S}^2} \sup_{u \in [U , u_{R,n}]} ( L T^m \Omega^k \phi )^2 \cdot v^{1+\delta} \, d\omega dv \lesssim C E_0 \epsilon^2 ,
\end{equation}
for all $V$, for any $0 < \delta < 2-\delta_1$ if $m=0$, for any $0 < \delta < 2+\delta_2$ if $m=1$, for any $0 < \delta < 1+\delta_2$ if $m=2$, for any $0 < \delta < \delta_2$ if $m=3$, for any  $(u_{R,n} , v ) \in \mathcal{A}_{\tau_0}^{\infty} \setminus SO(3)$ where $v \in [v_0 , V]$ and any $(U,V)  \in \mathcal{A}_{\tau_0}^{\infty} \setminus SO(3)$ (where $u_{R,n}$ is on the hypersurface $r=R$ for any $M < R \leq r_0$), for $m\in \{0,1,2,3\}$, and for any $k \leq 5$. 
\end{theorem}

\begin{proof}
The proof will be done through a standard bootstrap argument. First we use the fundamental theorem of calculus for a dyadic sequence $\{ v_n \}$ and we have that
\begin{equation}\label{eq:ftc}
\begin{split}
\left. \int_{v_n}^{v_{n+1}} \int_{\mathbb{S}^2}\Bigg|_{r\leq R}  ( L \Omega^k T^m \phi )^2 \cdot v^{1+\delta} \, d\omega dv \right|_u = & \left. \int_{v_n}^{v_{n+1}} \int_{\mathbb{S}^2}  ( L \Omega^k T^m \phi )^2 \cdot v^{1+\delta} \, d\omega dv \right|_{r=R} \\ & + \int_{v_n}^{v_{n+1}} \int_{u_{R,n}}^u \int_{\mathbb{S}^2} 2 ( \underline{L} L \Omega^k T^m \phi ) \cdot (L T^m \phi ) \cdot v^{1+\delta} \, d\omega du' dv .
\end{split}
\end{equation}
We use again equation \eqref{eq:dn} after commuting with $\Omega^k$ for $k \in \{0,1,2,3,4,5\}$ and we have that
\begin{equation*}
\begin{split}
2 L \underline{L} ( \Omega^k T^m \phi ) = & \mathcal{O} ( (r-M)^2 ) \slashed{\Delta} ( \Omega^k T^m\phi ) +  \mathcal{O} ((r-M)^3 ) (\Omega^k T^m \phi ) \\ &  +  \frac{Dr}{2} \Omega^k T^m \left[ A \cdot  g^{\alpha \beta} \cdot \partial_{\alpha} \psi \cdot \partial_{\beta} \psi \right] ,
\end{split}
\end{equation*}
and as
\begingroup
\allowdisplaybreaks
\begin{align*}
\frac{Dr}{2} \Omega^k T^m & \left[ A(u,v,\psi ) \cdot g^{\alpha \beta} \cdot \partial_{\alpha} \psi \cdot \partial_{\beta} \psi \right] =  \sum_{\substack{ k_1 + k_2 = k \\ m_1 + m_2 = m}} A \cdot \frac{D}{r^2} \cdot (L \Omega^{k_1} T^{m_1} \phi ) \cdot \left( \frac{2r}{D} \underline{L} \Omega^{k_2} T^{m_2} \phi \right)  \\ & - \sum_{\substack{ k_1 + k_2 = k \\ m_1 + m_2 = m}} A \cdot \frac{D^2}{2r^3} \cdot ( \Omega^{k_1} T^{m_1} \phi ) \cdot \left( \frac{2r}{D} \underline{L} \Omega^{k_2} T^{m_2} \phi \right)  + \sum_{\substack{ k_1 + k_2 = k \\ m_1 + m_2 = m}} A \cdot \frac{D}{r^2} \cdot ( \Omega^{k_1} T^{m_1} \phi ) \cdot ( L \Omega^{k_2} T^{m_2} \phi )  \\ & - \sum_{\substack{ k_1 + k_2 = k \\ m_1 + m_2 = m}}A \cdot \frac{D^2}{2r^3} \cdot ( \Omega^{k_1} T^{m_1} \phi ) \cdot (  \Omega^{k_2} T^{m_2} \phi )  + \sum_{\substack{ k_1 + k_2 = k \\ m_1 + m_2 = m}} A \cdot \frac{D}{r} \cdot \langle \slashed{\nabla} \Omega^{k_1} T^{m_1} \phi , \slashed{\nabla} \Omega^{k_2} T^{m_2} \phi \rangle \\ & + \sum_{\substack{ k_1 + k_2 = k \\ m_1 + m_2 = m \\ k_1 + m_1 > 0}} ( \Omega^{k_1} T^{m_1} A  ) \cdot \frac{Dr}{2} \cdot ( \Omega^{k_2} T^{m_2} F^c ) ,
\end{align*}
\endgroup
we have that
\begin{equation}\label{eq:dn_c}
\begin{split}
2 L & \underline{L} ( \Omega^k T^m \phi ) =  \mathcal{O} ( (r-M)^2 ) \slashed{\Delta} ( \Omega^k T^m \phi )  + \mathcal{O} ((r-M)^3 ) (\Omega^k T^m \phi ) \\ & +\sum_{\substack{ k_1 + k_2 = k \\ m_1 + m_2 = m}} A \cdot \frac{D}{r^2} \cdot (L \Omega^{k_1} T^{m_1} \phi ) \cdot \left( \frac{2r}{D} \underline{L} \Omega^{k_2} T^{m_2} \phi \right)  - \sum_{k_1 + k_2 = k} A\cdot \frac{D^2}{2r^3} \cdot ( \Omega^{k_1} T^{m_1} \phi ) \cdot \left( \frac{2r}{D} \underline{L} \Omega^{k_2} T^{m_2} \phi \right) \\ & + \sum_{\substack{ k_1 + k_2 = k \\ m_1 + m_2 = m}} A \cdot \frac{D}{r^2} \cdot ( \Omega^{k_1} T^{m_1} \phi ) \cdot ( L \Omega^{k_2} T^{m_2} \phi )   - \sum_{\substack{ k_1 + k_2 = k \\ m_1 + m_2 = m}} A \cdot \frac{D^2}{2r^3} \cdot ( \Omega^{k_1} T^{m_1} \phi ) \cdot (  \Omega^{k_2} T^{m_2} \phi ) \\ & + \sum_{\substack{ k_1 + k_2 = k \\ m_1 + m_2 = m}} A \cdot \frac{D}{r} \cdot \langle \slashed{\nabla} \Omega^{k_1} T^{m_1} \phi , \slashed{\nabla} \Omega^{k_2} T^{m_2} \phi \rangle  +  \sum_{\substack{ k_1 + k_2 = k \\ m_1 + m_2 = m \\ k_1 + m_1 > 0}} ( \Omega^{k_1} T^{m_1} A  ) \cdot \frac{Dr}{2} \cdot ( \Omega^{k_2} T^{m_2} F^c ) .
\end{split}
\end{equation}
Going back to equation \eqref{eq:ftc} we have that
\begingroup
\allowdisplaybreaks
\begin{align*}
 \int_{v_n}^{v_{n+1}} \int_{\mathbb{S}^2} ( L & \Omega^k T^m \phi )^2  \cdot  v^{1+\delta} \, d\omega dv \Bigg|_u  =  \int_{v_n}^{v_{n+1}} \int_{\mathbb{S}^2}  ( L \Omega^k T^m \phi )^2 \cdot v^{1+\delta} \, d\omega dv \Bigg|_{r=R} \\ & +  \int_{v_n}^{v_{n+1}} \int_{u_{R,n}}^u \int_{\mathbb{S}^2}   \mathcal{O} ( (r-M)^2 ) ( \slashed{\Delta}  \Omega^k T^m \phi ) \cdot (L \Omega^k T^m \phi ) \cdot v^{1+\delta} \, d\omega du' dv  \\ & +  \int_{v_n}^{v_{n+1}} \int_{u_{R,n}}^u \int_{\mathbb{S}^2}\Bigg|_{r\leq R}   \mathcal{O} ( (r-M)^3 ) (  \Omega^k T^m \phi ) \cdot (L \Omega^k T^m \phi ) \cdot v^{1+\delta} \, d\omega du' dv  \\ & + \sum_{\substack{ k_1 + k_2 = k \\ m_1 + m_2 = m}} \int_{v_n}^{v_{n+1}} \int_{u_{R,n}}^u \int_{\mathbb{S}^2}   \frac{A \cdot D}{r^2} \cdot (L \Omega^{k_1} T^{m_1} \phi ) \cdot \left( \frac{2r}{D} \underline{L} \Omega^{k_2} T^{m_2} \phi \right)  \cdot (L \Omega^k T^{m} \phi ) \cdot v^{1+\delta} \, d\omega du' dv \\ & - \sum_{\substack{ k_1 + k_2 = k \\ m_1 + m_2 = m}}  \int_{v_n}^{v_{n+1}} \int_{u_{R,n}}^u \int_{\mathbb{S}^2} \frac{A \cdot D^2}{2r^3} \cdot ( \Omega^{k_1} \phi ) \cdot \left( \frac{2r}{D} \underline{L} \Omega^{k_2} \phi \right)  \cdot (L \Omega^k \phi ) \cdot v^{1+\delta} \, d\omega du' dv \\ & + \sum_{\substack{ k_1 + k_2 = k \\ m_1 + m_2 = m}} \int_{v_n}^{v_{n+1}} \int_{u_{R,n}}^u \int_{\mathbb{S}^2} \frac{A \cdot D}{r^2} \cdot ( \Omega^{k_1} T^{m_1} \phi ) \cdot (L \Omega^{k_2} T^{m_2} \phi ) \cdot (L \Omega^k T^m \phi ) \cdot v^{1+\delta} \, d\omega du' dv \\ & - \sum_{\substack{ k_1 + k_2 = k \\ m_1 + m_2 = m}}  \int_{v_n}^{v_{n+1}} \int_{u_{R,n}}^u \int_{\mathbb{S}^2}  \frac{A \cdot D^2}{2r^3} \cdot ( \Omega^{k_1} T^{m_1} \phi ) \cdot (\Omega^{k_2} T^{m_2} \phi ) \cdot (L \Omega^k T^m \phi ) \cdot v^{1+\delta} \, d\omega du' dv \\ & + \sum_{\substack{ k_1 + k_2 = k \\ m_1 + m_2 = m}}  \int_{v_n}^{v_{n+1}} \int_{u_{R,n}}^u \int_{\mathbb{S}^2} \frac{A \cdot D}{r} \cdot \langle \slashed{\nabla} \Omega^{k_1} T^{m_1} \phi , \slashed{\nabla} \Omega^{k_2} T^{m_2} \phi \rangle \cdot (L \Omega^k T^m \phi ) \cdot v^{1+\delta} \, d\omega du' dv \\ & + \sum_{\substack{ k_1 + k_2 = k \\ m_1 + m_2 = m \\ k_1 + m_1 > 0}} \int_{v_n}^{v_{n+1}} \int_{u_{R,n}}^u \int_{\mathbb{S}^2} ( \Omega^{k_1} T^{m_1} A (u,v,\psi ) ) \cdot \frac{Dr}{2} \cdot ( \Omega^{k_2} T^{m_2} F^c ) \cdot ( L \Omega^k T^m \phi ) \cdot v^{1+\delta} \, d\omega du' dv \\ \doteq & l_1 + l_2 + l_3 + l_4 + l_5 + l_6 + l_7 + l_8 + l_9 .
\end{align*}
\endgroup
The first term $l_1$ can as follows:
\begingroup
\allowdisplaybreaks
\begin{align*}
l_1 = & \int_{v_n}^{v_{n+1}} \int_{\mathbb{S}^2}  ( L \Omega^k T^m \phi )^2 \cdot v^{1+\delta} \, d\omega dv \Bigg|_{r=R} \\ \lesssim & v_{n+1}^{1+\delta} \int_{v_n}^{v_{n+1}} \int_{\mathbb{S}^2}  ( L \Omega^k T^m \phi )^2 \, d\omega dv \Bigg|_{r=R} \\ \lesssim &  v_{n+1}^{1+\delta} \int_{\mathcal{A}_{v_n}^{v_{n+1}}} ( r- M)^2 ( L \Omega^k T^m \phi )^2 \, d\omega du dv \\ \lesssim & \frac{v_{n+1}^{1+\delta}}{v_n^{3-\delta_1}} \mbox{  if $m=0$, }  \lesssim  \frac{v_{n+1}^{1+\delta}}{v_n^{3+\delta_2}} \mbox{  if $m=1$, } \lesssim  \frac{v_{n+1}^{1+\delta}}{v_n^{2+\delta_2}} \mbox{  if $m=2$, } \lesssim  \frac{v_{n+1}^{1+\delta}}{v_n^{1+\delta_2}} \mbox{  if $m=3$, }
\end{align*}
\endgroup
where we used a standard averaging argument and the Morawetz decay estimates \eqref{dec:mor}, \eqref{dec:mor1}, \eqref{dec:mor2} and \eqref{dec:mor3}.

The ninth term $l_{9}$ can be considered similarly as the rest so we will not study it in detail. For the second term $l_2$ we integrate by parts on $\mathbb{S}^2$ and additionally we integrate by parts with respect to $L$ and we have that:
\begingroup
\allowdisplaybreaks
\begin{align*}
l_2 = & \int_{v_n}^{v_{n+1}} \int_{u_{R,n}}^u \int_{\mathbb{S}^2} \mathcal{O} ( (r-M)^2 ) ( \slashed{\Delta}  \Omega^k T^m \phi ) \cdot (L \Omega^k T^m \phi ) \cdot v^{1+\delta} \, d\omega du' dv \\ = & -  \int_{v_n}^{v_{n+1}} \int_{u_{R,n}}^u \int_{\mathbb{S}^2}  \mathcal{O} ( (r-M)^2 ) \langle \slashed{\nabla}  \Omega^k T^m \phi , \slashed{\nabla} L \Omega^k T^m \phi \rangle \cdot v^{1+\delta} \, d\omega du' dv \\ = & \int_{v_n}^{v_{n+1}} \int_{u_{R,n}}^u \int_{\mathbb{S}^2}  \mathcal{O} ( (r-M)^3 ) |\slashed{\nabla}  \Omega^k T^m \phi |^2 \cdot v^{1+\delta} \, d\omega du' dv \\ & + (1+\delta ) \int_{v_n}^{v_{n+1}} \int_{u_{R,n}}^u \int_{\mathbb{S}^2}  \mathcal{O} ( (r-M)^2 ) |\slashed{\nabla}  \Omega^k T^m \phi |^2 \cdot v^{\delta} \, d\omega du' dv \\ & - v_{n+1}^{1+\delta}  \int_{u_{R,n}}^u \int_{\mathbb{S}^2} \Bigg|_{v = v_{n+1} } \mathcal{O} ( (r-M)^2 ) |\slashed{\nabla}  \Omega^k T^m \phi |^2 \, d\omega du' \\ & - v_{n}^{1+\delta}  \int_{u_{R,n}}^u \int_{\mathbb{S}^2} \Bigg|_{v = v_{n} } \mathcal{O} ( (r-M)^2 ) |\slashed{\nabla}  \Omega^k T^m \phi |^2 \, d\omega du' \\ \lesssim & v_{n+1}^{1+\delta} \int_{v_n}^{v_{n+1}} \int_{u_{R,n}}^u \int_{\mathbb{S}^2} \mathcal{O} ( (r-M)^3 ) |\slashed{\nabla}  \Omega^k T^m \phi |^2 \, d\omega du' dv \\ & + v_{n+1}^{1+\delta} \sup_{v \in [ v_n , v_{n+1} ]} \int_{u_{R,n}}^u \int_{\mathbb{S}^2} \mathcal{O} ( (r-M)^2 ) |\slashed{\nabla}  \Omega^k T^m \phi |^2 \, d\omega du'
\\ \lesssim &C E_0 \epsilon^2 \frac{v_{n+1}^{1+\delta}}{v_n^{3-\delta_1}} \mbox{  if $m=0$,   }  \lesssim C E_0 \epsilon^2 \frac{v_{n+1}^{1+\delta}}{v_n^{3+\delta_2}} \mbox{  if $m=1$, } \\ \lesssim & C E_0 \epsilon^2 \frac{v_{n+1}^{1+\delta}}{v_n^{2+\delta_2}} \mbox{  if $m=2$,   }   \lesssim  C E_0 \epsilon^2 \frac{v_{n+1}^{1+\delta}}{v_n^{1+\delta_2}} \mbox{  if $m=3$, } 
\end{align*}
\endgroup
where we used decay provided by Morawetz decay estimates \eqref{dec:mor}, \eqref{dec:mor1}, \eqref{dec:mor2} and \eqref{dec:mor3}, and the decay of the $T$-fluxes \eqref{dec:ent}, \eqref{dec:entt} and \eqref{dec:enttt}. For the third term $l_3$ we have that
\begingroup
\allowdisplaybreaks
\begin{align*}
l_3 = & \int_{v_n}^{v_{n+1}} \int_{u_{R,n}}^u \int_{\mathbb{S}^2}  \mathcal{O} ( (r-M)^3 ) (  \Omega^k T^m \phi ) \cdot (L \Omega^k T^m \phi ) \cdot v^{1+\delta} \, d\omega du' dv \\ \lesssim & v_{n+1}^{1+\delta} \cdot \left( \int_{v_n}^{v_{n+1}} \int_{u_{R,n}}^u \int_{\mathbb{S}^2}  (r-M)^{3+\beta} (  \Omega^k T^m \phi )^2 \, d\omega du' dv \right)^{1/2} \\ & \times \left( \int_{v_n}^{v_{n+1}} \int_{u_{R,n}}^u \int_{\mathbb{S}^2} (r-M)^{3-\beta} (  L \Omega^k T^m \phi )^2 \, d\omega du' dv \right)^{1/2} \\ \lesssim & v_{n+1}^{1+\delta} \cdot \left( \int_{v_n}^{v_{n+1}} \int_{u_{R,n}}^u \int_{\mathbb{S}^2} (r-M)^{1+\beta} (  \underline{L} \Omega^k T^m \phi )^2 \, d\omega du' dv \right)^{1/2} \\ & \times \left( \int_{v_n}^{v_{n+1}} \int_{u_{R,n}}^u \int_{\mathbb{S}^2}  (r-M)^{3-\beta} (  L \Omega^k T^m \phi )^2 \, d\omega du' dv \right)^{1/2}\\ \lesssim & v_{n+1}^{1+\delta} \cdot \frac{\sqrt{C} \sqrt{E_0} \epsilon}{v_n^{3/2-\delta_1 /2}} \cdot \frac{\sqrt{C} \sqrt{E_0} \epsilon}{v_n^{3/2-\delta_1 /2}}  \lesssim C E_0 \epsilon^2 \cdot \frac{v_{n+1}^{1+\delta}}{v_n^{3-\delta_1}} \mbox{  if $m=0$, } \\ \lesssim & v_{n+1}^{1+\delta} \cdot \frac{\sqrt{C} \sqrt{E_0} \epsilon}{v_n^{3/2+\delta_2 /2}} \cdot \frac{\sqrt{C} \sqrt{E_0} \epsilon}{v_n^{3/2+\delta_2 /2}}  \lesssim C E_0 \epsilon^2 \cdot \frac{v_{n+1}^{1+\delta}}{v_n^{3+\delta_2}} \mbox{  if $m=1$, } \\ \lesssim & v_{n+1}^{1+\delta} \cdot \frac{\sqrt{C} \sqrt{E_0} \epsilon}{v_n^{1+\delta_2 /2}} \cdot \frac{\sqrt{C} \sqrt{E_0} \epsilon}{v_n^{1+\delta_2 /2}}  \lesssim C E_0 \epsilon^2 \cdot \frac{v_{n+1}^{1+\delta}}{v_n^{2+\delta_2}} \mbox{  if $m=2$, } \\ \lesssim & v_{n+1}^{1+\delta} \cdot \frac{\sqrt{C} \sqrt{E_0} \epsilon}{v_n^{1/2+\delta_2 /2}} \cdot \frac{\sqrt{C} \sqrt{E_0} \epsilon}{v_n^{1/2+\delta_2 /2}}  \lesssim  C E_0 \epsilon^2 \frac{v_{n+1}^{1+\delta}}{v_n^{1+\delta_2}} \mbox{  if $m=3$, } 
\end{align*}
\endgroup
where we used Cauchy-Schwarz, Hardy's inequality \eqref{hardy}, and the Morawetz decay estimates \eqref{dec:mor}, \eqref{dec:mor1}, \eqref{dec:mor2} and \eqref{dec:mor3}. For the fourth term we look at the case $m=3$ (since the cases $m=1$ and $m=2$ are either easier or similar) and we have that
\begingroup
\allowdisplaybreaks
\begin{align*}
l_4 & =  \sum_{\substack{ k_1 + k_2 = k \\ m_1 + m_2 = 3 }}  \int_{v_n}^{v_{n+1}} \int_{u_{R,n}}^u \int_{\mathbb{S}^2}  \frac{A \cdot D}{r^2} \cdot (L \Omega^{k_1} T^{m_1} \phi ) \cdot \left( \frac{2r}{D} \underline{L} \Omega^{k_2} T^{m_2} \phi \right)  \cdot (L \Omega^k T^3 \phi ) \cdot v^{1+\delta} \, d\omega du' dv \\ = & \sum_{\substack{ k_1 + k_2 = k  }}  \int_{v_n}^{v_{n+1}} \int_{u_{R,n}}^u \int_{\mathbb{S}^2}   \frac{A \cdot D}{r^2} \cdot (L \Omega^{k_1}  \phi ) \cdot \left( \frac{2r}{D} \underline{L} \Omega^{k_2} T^3 \phi \right)  \cdot (L \Omega^k T^3 \phi ) \cdot v^{1+\delta} \, d\omega du' dv \\ & + \sum_{\substack{ k_1 + k_2 = k}}  \int_{v_n}^{v_{n+1}} \int_{u_{R,n}}^u \int_{\mathbb{S}^2}   \frac{A \cdot D}{r^2} \cdot (L \Omega^{k_1} T \phi ) \cdot \left( \frac{2r}{D} \underline{L} \Omega^{k_2} T^2 \phi \right)  \cdot (L \Omega^k T^3 \phi ) \cdot v^{1+\delta} \, d\omega du' dv \\ & + \sum_{\substack{ k_1 + k_2 = k}}  \int_{v_n}^{v_{n+1}} \int_{u_{R,n}}^u \int_{\mathbb{S}^2}   \frac{A \cdot D}{r^2} \cdot (L \Omega^{k_1} T^2 \phi ) \cdot \left( \frac{2r}{D} \underline{L} \Omega^{k_2} T \phi \right)  \cdot (L \Omega^k T^3 \phi ) \cdot v^{1+\delta} \, d\omega du' dv \\ & + \sum_{\substack{ k_1 + k_2 = k}}  \int_{v_n}^{v_{n+1}} \int_{u_{R,n}}^u \int_{\mathbb{S}^2}  \frac{A \cdot D}{r^2} \cdot (L \Omega^{k_1} T^3 \phi ) \cdot \left( \frac{2r}{D} \underline{L} \Omega^{k_2} \phi \right)  \cdot (L \Omega^k T^3 \phi ) \cdot v^{1+\delta} \, d\omega du' dv  \\ \lesssim &  v_{n+1}^{1+\delta} \sum_{k_1 + k_2 = k} \left( \int_{v_n}^{v_{n+1}} \int_{u_{R,n}}^u \int_{\mathbb{S}^2}  \frac{1}{(r-M)^{1+\delta_2}} (L \Omega^{k_1} \phi )^2  (\underline{L} \Omega^{k_2} T^3 \phi )^2 \, d\omega du' dv \right)^{1/2} \\ & \times \left( \int_{v_n}^{v_{n+1}} \int_{u_{R,n}}^u \int_{\mathbb{S}^2}  (r-M)^{1+\delta_2} (L \Omega^{k} T^3 \phi )^2 \, d\omega du' dv \right)^{1/2} \\ & + v_{n+1}^{1+\delta}\sum_{k_1 + k_2 = k} \left( \int_{v_n}^{v_{n+1}} \int_{u_{R,n}}^u \int_{\mathbb{S}^2}  \frac{1}{(r-M)^{1+\delta_2}} (\underline{L} \Omega^{l_1} T^2 \phi )^2 \, d\omega du' dv \right)^{1/2} \\ & \times \left( \int_{v_n}^{v_{n+1}} \int_{u_{R,n}}^u \int_{\mathbb{S}^2} (r-M)^{1+\delta_2} (L \Omega^{l_2} T^3 \phi )^2 \, d\omega du' dv \right)^{1/2} \\ & + \sqrt{C} \sqrt{E_0} \epsilon v_{n+1}^{1+\delta} \sum_{\substack{ l_1 \leq 5 \\ l_2 \leq 5 }} \left( \int_{v_n}^{v_{n+1}} \int_{u_{R,n}}^u \int_{\mathbb{S}^2}  \frac{1}{(r-M)^{2}} (L \Omega^{l_1} T^2 \phi )^2 \, d\omega du' dv \right)^{1/2} \\ & \times \left( \int_{v_n}^{v_{n+1}} \int_{u_{R,n}}^u \int_{\mathbb{S}^2}  (r-M)^{2} (L \Omega^{l_2} T^3 \phi )^2 \, d\omega du' dv \right)^{1/2}\\ & + \sqrt{C} \sqrt{E_0} \epsilon v_{n+1}^{1+\delta} \sum_{\substack{ l_1 \leq 5 \\ l_2 \leq 5 }} \left( \int_{v_n}^{v_{n+1}} \int_{u_{R,n}}^u \int_{\mathbb{S}^2} \frac{1}{(r-M)^{2}} (L \Omega^{l_1} T^3 \phi )^2 \, d\omega du' dv \right)^{1/2} \\ & \times \left( \int_{v_n}^{v_{n+1}} \int_{u_{R,n}}^u \int_{\mathbb{S}^2} (r-M)^{2} (L \Omega^{l_2} T^3 \phi )^2 \, d\omega du' dv \right)^{1/2} \\ \lesssim & v_{n+1}^{1+\delta} \sum_{k_1 + k_2 = k} \left( \int_{v_n}^{v_{n+1}}  \int_{\mathbb{S}^2} \sup_u (L \Omega^{k_1} \phi )^2 \cdot v^{1+\delta} d\omega dv \cdot \sup_{v \in [v_n , v_{n+1} ]} \int_{u_{R,n}}^u \int_{\mathbb{S}^2} \frac{1}{(r-M)^{1+\delta}}  (\underline{L} \Omega^{k_2} T^3 \phi )^2 \cdot \frac{1}{v^{1+\delta}} \, d\omega du' \right)^{1/2} \\ & \times \left( \int_{v_n}^{v_{n+1}} \int_{u_{R,n}}^u \int_{\mathbb{S}^2} (r-M)^{1+\delta_2} (L \Omega^{k} T^3 \phi )^2 \, d\omega du' dv \right)^{1/2} \\ & + v_{n+1}^{1+\delta} \sum_{k_1 + k_2 = k} \left( \int_{v_n}^{v_{n+1}}  \int_{\mathbb{S}^2} \sup_u (L \Omega^{k_1} T \phi )^2 \cdot v^{1+\delta} \, d\omega dv \cdot \sup_{v \in [v_n , v_{n+1} ]} \int_{u_{R,n}}^u \int_{\mathbb{S}^2} \frac{1}{(r-M)^{1+\delta}}  (\underline{L} \Omega^{k_2} T^2 \phi )^2 \cdot \frac{1}{v^{1+\delta}} \, d\omega du' \right)^{1/2} \\ & \times \left( \int_{v_n}^{v_{n+1}} \int_{u_{R,n}}^u \int_{\mathbb{S}^2} (r-M)^{1+\delta_2} (L \Omega^{k} T^3 \phi )^2 \, d\omega du' dv \right)^{1/2}  \\ & + \sqrt{C} \sqrt{E_0} \epsilon v_{n+1}^{1+\delta} \sum_{\substack{ l_1 \leq 5 \\ l_2 \leq 5 }} \left( \int_{v_n}^{v_{n+1}} \int_{u_{R,n}}^u \int_{\mathbb{S}^2} (r-M)^{2} (L \Omega^{l_1} T^2 \phi )^2  \, d\omega du' dv \right)^{1/2} \\ & \times \left( \int_{v_n}^{v_{n+1}} \int_{u_{R,n}}^u \int_{\mathbb{S}^2} (r-M)^{2} (L \Omega^{l_2} T^3 \phi )^2 \, d\omega du' dv \right)^{1/2}\\ & + \sqrt{C} \sqrt{E_0} \epsilon v_{n+1}^{1+\delta} \sum_{\substack{ l_1 \leq 5 \\ l_2 \leq 5 }} \left( \int_{v_n}^{v_{n+1}} \int_{u_{R,n}}^u \int_{\mathbb{S}^2} (r-M)^{2} (L \Omega^{l_1} T^3 \phi )^2 \, d\omega du' dv \right)^{1/2} \\ & \times \left( \int_{v_n}^{v_{n+1}} \int_{u_{R,n}}^u \int_{\mathbb{S}^2} (r-M)^{2} (L \Omega^{l_2} T^3 \phi )^2 \, d\omega du' dv \right)^{1/2} \\ \lesssim & C^{3/2} E_0^{3/2} \epsilon^3 \frac{v_{n+1}^{1+\delta}}{v_n^{1/2+\delta / 2}} \frac{1}{v_n^{\delta_2/2 - \delta / 2}} \frac{1}{v_n^{1/2+\delta_2 / 2}} + C^{3/2} E_0^{3/2} \epsilon^3 \frac{v_{n+1}^{1+\delta}}{v_n^{1/2+\delta /2}} \frac{1}{v_n^{1/2+\delta_2 / 2 - \delta /2}} \frac{1}{v_n^{1/2+\delta_2 / 2}} \\ & +  C^{3/2} E_0^{3/2} \epsilon^3 \frac{v_{n+1}^{1 +\delta}}{v_n ^{3/2+\delta }}  + C^{3/2} E_0^{3/2} \epsilon^3 \frac{v_{n+1}^{1+\delta}}{v_n ^{1+\delta_2}} \lesssim C^{3/2} E_0^{3/2} \epsilon^3 \frac{v_{n+1}^{1+\delta}}{v_n^{1+\delta_2}} ,
\end{align*}
\endgroup
where we used Sobolev's inequality \eqref{est:sobolev} in all three terms. For the first term in the above expression we used estimate \eqref{est:aux1} for $m=0$ (in the context of the bootstrap argument), the decay estimates \eqref{dec:entttp}, and the Morawetz decay estimate \eqref{dec:mor3}. For the second term in the above expression we used that $LT \phi$ decays like $T^2 \phi$ (hence we use the decay from \eqref{dec:ttpsi}), the boundedness of the $(r-M)^{-p}$-weighted estimate for $T^2$ for $p=2+\delta_2$ (which follows after applying Hardy's inequality \eqref{hardy} and using the boundedness of the corresponding commuted estimate for $p=\delta_2$), and the Morawetz decay estimate \eqref{dec:mor3}. For the third term in the above expression we used the boundedness estimate \eqref{est:boundyt} and the Morawetz decay estimates \eqref{dec:mor2} and \eqref{dec:mor3}. For the fourth term in the above expression we used the boundedness estimate \eqref{est:boundy} and the Morawetz decay estimate \eqref{dec:mor3}. We also note by using the same method we get the following bounds in the $m=0$, $m=1$ and $m=2$ cases:
\begin{equation*}
\mbox{$m=0$:  } \lesssim C^{3/2} E_0^{3/2} \epsilon^3 \frac{v_{n+1}^{1+\delta}}{v_n^{3-\delta_1}} , \mbox{   $m=1$:  } \lesssim C^{3/2} E_0^{3/2} \epsilon^3 \frac{v_{n+1}^{1+\delta}}{v_n^{3+\delta_2}} ,  \mbox{   $m=2$:  } \lesssim C^{3/2} E_0^{3/2} \epsilon^3 \frac{v_{n+1}^{1+\delta}}{v_n^{2+\delta_2}} .
\end{equation*}

For the fifth term of the original expression we look again at the case of $m=3$ (as the cases $m=0$, $m=1$ and $m=2$ are easier or similar) and we have that
\begingroup
\allowdisplaybreaks
\begin{align*}
l_5 & = -\sum_{\substack{ k_1 + k_2 = k \\ m_1 + m_2 = 3 } }  \int_{v_n}^{v_{n+1}} \int_{u_{R,n}}^u \int_{\mathbb{S}^2}    \frac{A \cdot D^2}{2r^3} \cdot ( \Omega^{k_1} T^{m_1} \phi ) \cdot \left( \frac{2r}{D} \underline{L} \Omega^{k_2} T^{m_2} \phi \right)  \cdot (L \Omega^k T^3 \phi ) \cdot v^{1+\delta} \, d\omega du' dv \\ = & - \sum_{\substack{ k_1 + k_2 = k  } }  \int_{v_n}^{v_{n+1}} \int_{u_{R,n}}^u \int_{\mathbb{S}^2}  \frac{A \cdot D^2}{2r^3} \cdot ( \Omega^{k_1}  \phi ) \cdot \left( \frac{2r}{D} \underline{L} \Omega^{k_2} T^3 \phi \right)  \cdot (L \Omega^k T^3 \phi ) \cdot v^{1+\delta} \, d\omega du' dv \\ & - \sum_{\substack{ k_1 + k_2 = k  } }  \int_{v_n}^{v_{n+1}} \int_{u_{R,n}}^u \int_{\mathbb{S}^2}    \frac{A \cdot D^2}{2r^3} \cdot ( \Omega^{k_1} T \phi ) \cdot \left( \frac{2r}{D} \underline{L} \Omega^{k_2} T^2 \phi \right)  \cdot (L \Omega^k T^3 \phi ) \cdot v^{1+\delta} \, d\omega du' dv \\ & - \sum_{\substack{ k_1 + k_2 = k  } }  \int_{v_n}^{v_{n+1}} \int_{u_{R,n}}^u \int_{\mathbb{S}^2} \frac{A \cdot D^2}{2r^3} \cdot ( \Omega^{k_1} T^2 \phi ) \cdot \left( \frac{2r}{D} \underline{L} \Omega^{k_2} T \phi \right)  \cdot (L \Omega^k T^3 \phi ) \cdot v^{1+\delta} \, d\omega du' dv \\ & - \sum_{\substack{ k_1 + k_2 = k  } }  \int_{v_n}^{v_{n+1}} \int_{u_{R,n}}^u \int_{\mathbb{S}^2}  \frac{A \cdot D^2}{2r^3} \cdot ( \Omega^{k_1} T^3 \phi ) \cdot \left( \frac{2r}{D} \underline{L} \Omega^{k_2}  \phi \right)  \cdot (L \Omega^k T^3 \phi ) \cdot v^{1+\delta} \, d\omega du' dv \\ \lesssim & \sqrt{C} \sqrt{E_0} \epsilon \frac{v_{n+1}^{1+\delta}}{v_n^{1-\delta_1 / 2}} \sum_{\substack{ l_1 \leq 5 \\ l_2 \leq 5 }} \left( \int_{v_n}^{v_{n+1}} \int_{u_{R,n}}^u \int_{\mathbb{S}^2}(r-M)^2 (\underline{L} \Omega^{l_1} T^3 \phi )^2 \, d\omega du' dv \right)^{1/2} \\ & \times \left( \int_{v_n}^{v_{n+1}} \int_{u_{R,n}}^u \int_{\mathbb{S}^2}  (r-M)^2 (L \Omega^{l_2} T^3 \phi )^2 \, d\omega du' dv \right)^{1/2} \\ & + \sqrt{C} \sqrt{E_0} \epsilon \frac{v_{n+1}^{1+\delta}}{v_n^{1+\delta_2}} \sum_{\substack{ l_1 \leq 5 \\ l_2 \leq 5 }} \left( \int_{v_n}^{v_{n+1}} \int_{u_{R,n}}^u \int_{\mathbb{S}^2}(r-M)^2 (\underline{L} \Omega^{l_1} T^2 \phi )^2 \, d\omega du' dv \right)^{1/2} \\ & \times \left( \int_{v_n}^{v_{n+1}} \int_{u_{R,n}}^u \int_{\mathbb{S}^2} (r-M)^2 (L \Omega^{l_2} T^3 \phi )^2 \, d\omega du' dv \right)^{1/2} \\ & + \sqrt{C} \sqrt{E_0} \epsilon \frac{v_{n+1}^{1+\delta}}{v_n^{1/2+\delta_2 / 2}} \sum_{\substack{ l_1 \leq 5 \\ l_2 \leq 5 }} \left( \int_{v_n}^{v_{n+1}} \int_{u_{R,n}}^u \int_{\mathbb{S}^2} (r-M)^2 (\underline{L} \Omega^{l_1} T \phi )^2 \, d\omega du' dv \right)^{1/2} \\ & \times \left( \int_{v_n}^{v_{n+1}} \int_{u_{R,n}}^u \int_{\mathbb{S}^2} (r-M)^2 (L \Omega^{l_2} T^3 \phi )^2 \, d\omega du' dv  \right)^{1/2} \\ & + \sqrt{C} \sqrt{E_0} \epsilon v_{n+1}^{1+\delta}\sum_{\substack{ l_1 \leq 5 \\ l_2 \leq 5 }} \left( \int_{v_n}^{v_{n+1}} \int_{u_{R,n}}^u \int_{\mathbb{S}^2} (r-M)^2 (\underline{L} \Omega^{l_1}  \phi )^2 \, d\omega du' dv \right)^{1/2} \\ & \times \left( \int_{v_n}^{v_{n+1}} \int_{u_{R,n}}^u \int_{\mathbb{S}^2}  (r-M)^2 (L \Omega^{l_2} T^3 \phi )^2 \, d\omega du' dv \right)^{1/2} \\ \lesssim & C^{3/2} E_0^{3/2} \epsilon^3 \frac{v_{n+1}^{1+\delta}}{v_n^{2+\delta_2 - \delta_1 / 2}} + C^{3/2} E_0^{3/2} \epsilon^3 \frac{v_{n+1}^{1+\delta}}{v_n^{5/2+ 3\delta_2 / 2 }} + C^{3/2} E_0^{3/2} \epsilon^3 \frac{v_{n+1}^{1+\delta}}{v_n^{9/4+\delta_2 - \delta_1 / 2}} \\ \lesssim & C^{3/2} E_0^{3/2} \epsilon^3 \frac{v_{n+1}^{1+\delta}}{v_n^{2+\delta_2 - \delta_1 / 2}} ,
\end{align*}
\endgroup
 where we used Sobolev's inequality \eqref{est:sobolev}, Cauchy-Schwarz, the decay of the terms $T\phi$, $T^2 \phi$ and $T^3 \phi$ (given by the pointwise decay estimates \eqref{dec:tpsi}, \eqref{dec:ttpsi}, \eqref{dec:tttpsi}), and the Morawetz decay estimates \eqref{dec:mor}, \eqref{dec:mor1}, \eqref{dec:mor2} and \eqref{dec:mor3}. For the $m=0$, $m=1$ and $m=2$ cases we have that:
\begin{equation*}
\mbox{$m=0$:  } \lesssim C^{3/2} E_0^{3/2} \epsilon^3 \frac{v_{n+1}^{1+\delta}}{v_n^{4-3\delta_1 / 2}} , \mbox{   $m=1$:  } \lesssim C^{3/2} E_0^{3/2} \epsilon^3 \frac{v_{n+1}^{1+\delta}}{v_n^{4+\delta_2-\delta_1 / 2}}  \mbox{   $m=2$:  } \lesssim C^{3/2} E_0^{3/2} \epsilon^3 \frac{v_{n+1}^{1+\delta}}{v_n^{3+\delta_2 - \delta_1 / 2}} .
\end{equation*}
 
For the sixth term  we look once again at the case $m=3$ (since the cases $m=0$, $m=1$ and $m=2$ are similar or easier) and we have that
 \begingroup
\allowdisplaybreaks
\begin{align*}
 l_6 & =  \sum_{\substack{k_1 + k_2 = k \\ m_1 + m_2 = 3}}  \int_{v_n}^{v_{n+1}} \int_{u_{R,n}}^u \int_{\mathbb{S}^2} \frac{A \cdot D}{r^2} \cdot ( \Omega^{k_1} T^{m_1} \phi ) \cdot ( L \Omega^{k_2} T^{m_2} \phi )  \cdot (L \Omega^k T^3 \phi ) \cdot v^{1+\delta} \, d\omega du' dv \\ = &  \sum_{\substack{k_1 + k_2 = k}}  \int_{v_n}^{v_{n+1}} \int_{u_{R,n}}^u \int_{\mathbb{S}^2}  \frac{A \cdot D}{r^2} \cdot ( \Omega^{k_1}  \phi ) \cdot ( L \Omega^{k_2} T^3 \phi )  \cdot (L \Omega^k T^3 \phi ) \cdot v^{1+\delta} \, d\omega du' dv \\ & + \sum_{\substack{k_1 + k_2 = k}}  \int_{v_n}^{v_{n+1}} \int_{u_{R,n}}^u \int_{\mathbb{S}^2} \frac{A \cdot D}{r^2} \cdot ( \Omega^{k_1} T \phi ) \cdot ( L \Omega^{k_2} T^2 \phi )  \cdot (L \Omega^k T^3 \phi ) \cdot v^{1+\delta} \, d\omega du' dv \\ & + \sum_{\substack{k_1 + k_2 = k}}  \int_{v_n}^{v_{n+1}} \int_{u_{R,n}}^u \int_{\mathbb{S}^2}  \frac{A \cdot D}{r^2} \cdot ( \Omega^{k_1} T^2 \phi ) \cdot ( L \Omega^{k_2} T \phi )  \cdot (L \Omega^k T^3 \phi ) \cdot v^{1+\delta} \, d\omega du' dv \\ & + \sum_{\substack{k_1 + k_2 = k}}  \int_{v_n}^{v_{n+1}} \int_{u_{R,n}}^u \int_{\mathbb{S}^2} \frac{A \cdot D}{r^2} \cdot ( \Omega^{k_1} T^3 \phi ) \cdot ( L \Omega^{k_2} \phi )  \cdot (L \Omega^k T^3 \phi ) \cdot v^{1+\delta} \, d\omega du' dv \\ \lesssim & \sqrt{C} \sqrt{E_0} \epsilon \frac{v_{n+1}^{1+\delta}}{v_n^{1-\delta_1 / 2}} \sum_{\substack{ l_1 \leq 5 \\ l_2 \leq 5 }} \left( \int_{v_n}^{v_{n+1}} \int_{u_{R,n}}^u \int_{\mathbb{S}^2} (r-M)^2 (L \Omega^{l_1} T^3 \phi )^2 \, d\omega du' dv \right)^{1/2} \\ & \times \left( \int_{v_n}^{v_{n+1}} \int_{u_{R,n}}^u \int_{\mathbb{S}^2} (r-M)^2 (L \Omega^{l_2} T^3 \phi )^2 \, d\omega du' dv \right)^{1/2} \\ & + \sqrt{C} \sqrt{E_0} \epsilon \frac{v_{n+1}^{1+\delta}}{v_n^{1+\delta_2 / 2}} \sum_{\substack{ l_1 \leq 5 \\ l_2 \leq 5 }} \left( \int_{v_n}^{v_{n+1}} \int_{u_{R,n}}^u \int_{\mathbb{S}^2} (r-M)^2 (L \Omega^{l_1} T^2 \phi )^2 \, d\omega du' dv \right)^{1/2} \\ & \times \left( \int_{v_n}^{v_{n+1}} \int_{u_{R,n}}^u \int_{\mathbb{S}^2} (r-M)^2 (L \Omega^{l_2} T^3 \phi )^2 \, d\omega du' dv \right)^{1/2} \\ & + \sqrt{C} \sqrt{E_0} \epsilon \frac{v_{n+1}^{1+\delta}}{v_n^{1/2+\delta_2 / 2}} \sum_{\substack{ l_1 \leq 5 \\ l_2 \leq 5 }} \left( \int_{v_n}^{v_{n+1}} \int_{u_{R,n}}^u \int_{\mathbb{S}^2} (r-M)^2 (L \Omega^{l_1} T \phi )^2 \, d\omega du' dv \right)^{1/2} \\ & \times \left( \int_{v_n}^{v_{n+1}} \int_{u_{R,n}}^u \int_{\mathbb{S}^2}  (r-M)^2 (L \Omega^{l_2} T^3 \phi )^2 \, d\omega du' dv \right)^{1/2} \\ & + \sqrt{C} \sqrt{E_0} \epsilon \frac{v_{n+1}^{1+\delta}}{v_n^{1/4+\delta_2 / 2}} \sum_{\substack{ l_1 \leq 5 \\ l_2 \leq 5 }} \left( \int_{v_n}^{v_{n+1}} \int_{u_{R,n}}^u \int_{\mathbb{S}^2} (r-M)^2 (L \Omega^{l_1}  \phi )^2 \, d\omega du' dv \right)^{1/2} \\ & \times \left( \int_{v_n}^{v_{n+1}} \int_{u_{R,n}}^u \int_{\mathbb{S}^2} (r-M)^2 (L \Omega^{l_2} T^3 \phi )^2 \, d\omega du' dv \right)^{1/2} \\ \lesssim & C^{3/2} E_0^{3/2} \epsilon^3 \frac{v_{n+1}^{1+\delta}}{v_n^{2+\delta_2 - \delta_1 / 2}} + C^{3/2} E_0^{3/2} \epsilon^3 \frac{v_{n+1}^{1+\delta}}{v_n^{5/2+ 3\delta_2 / 2 }} + C^{3/2} E_0^{3/2} \epsilon^3 \frac{v_{n+1}^{1+\delta}}{v_n^{9/4+\delta_2 - \delta_1 / 2}} ,
\end{align*}
\endgroup
 where again we used Sobolev's inequality \eqref{est:sobolev}, Cauchy-Schwarz, the pointwise decay estimates \eqref{dec:psi}, \eqref{dec:tpsi}, \eqref{dec:ttpsi}, \eqref{dec:tttpsi}, and the Morawetz decay estimates \eqref{dec:mor}, \eqref{dec:mor1}, \eqref{dec:mor2}, \eqref{dec:mor3}. For the $m=0$, $m=1$ and $m=2$ cases we have that:
\begin{equation*}
\mbox{$m=0$:  } \lesssim C^{3/2} E_0^{3/2} \epsilon^3 \frac{v_{n+1}^{1+\delta}}{v_n^{4-3\delta_1 / 2}} , \mbox{   $m=1$:  } \lesssim C^{3/2} E_0^{3/2} \epsilon^3 \frac{v_{n+1}^{1+\delta}}{v_n^{4+\delta_2-\delta_1 / 2}},  \mbox{   $m=2$:  } \lesssim C^{3/2} E_0^{3/2} \epsilon^3 \frac{v_{n+1}^{1+\delta}}{v_n^{3+\delta_2 - \delta_1 / 2}} .
\end{equation*}
 
For the seventh term we look once again at the case of $m=3$ (since the cases $m=1$ and $m=2$ are easier or similar) and we have that 
 \begingroup
\allowdisplaybreaks
\begin{align*}
 l_7 & = - \sum_{\substack{k_1 + k_2 = k \\ m_1 + m_2 =3 } }  \int_{v_n}^{v_{n+1}} \int_{u_{R,n}}^u \int_{\mathbb{S}^2} \frac{A \cdot D^2}{2r^3} \cdot ( \Omega^{k_1} \phi ) \cdot (  \Omega^{k_2} \phi )  \cdot (L \Omega^k T^3 \phi ) \cdot v^{1+\delta} \, d\omega du' dv \\ = &  - 2 \sum_{\substack{k_1 + k_2 = k  } }  \int_{v_n}^{v_{n+1}} \int_{u_{R,n}}^u \int_{\mathbb{S}^2}  \frac{A \cdot D^2}{2r^3} \cdot ( \Omega^{k_1} \phi ) \cdot (  \Omega^{k_2} T^3 \phi )  \cdot (L \Omega^k T^3 \phi ) \cdot v^{1+\delta} \, d\omega du' dv \\ & - 2\sum_{\substack{k_1 + k_2 = k  } }  \int_{v_n}^{v_{n+1}} \int_{u_{R,n}}^u \int_{\mathbb{S}^2}  \frac{A \cdot D^2}{2r^3} \cdot ( \Omega^{k_1} T \phi ) \cdot (  \Omega^{k_2} T^2 \phi )  \cdot (L \Omega^k T^3 \phi ) \cdot v^{1+\delta} \, d\omega du' dv \\ \lesssim & \sqrt{C} \sqrt{E_0} \epsilon \frac{v_{n+1}^{1+\delta}}{v_n^{1-\delta_1 / 2}} \sum_{\substack{ l_1 \leq 5 \\ l_2 \leq 5 }} \left( \int_{v_n}^{v_{n+1}} \int_{u_{R,n}}^u \int_{\mathbb{S}^2}  (r-M)^{3-\delta} ( \Omega^{l_1} T^3 \phi )^2 \, d\omega du' dv \right)^{1/2} \\ & \times \left( \int_{v_n}^{v_{n+1}} \int_{u_{R,n}}^u \int_{\mathbb{S}^2} (r-M)^{1+\delta} (L \Omega^{l_2} T^3 \phi )^2 \, d\omega du' dv \right)^{1/2} \\ & + \sqrt{C} \sqrt{E_0} \epsilon \frac{v_{n+1}^{1+\delta}}{v_n^{1+\delta_2 / 2}} \sum_{\substack{ l_1 \leq 5 \\ l_2 \leq 5 }} \left( \int_{v_n}^{v_{n+1}} \int_{u_{R,n}}^u \int_{\mathbb{S}^2} (r-M)^{3-\delta} ( \Omega^{l_1} T^2 \phi )^2 \, d\omega du' dv \right)^{1/2} \\ & \times \left( \int_{v_n}^{v_{n+1}} \int_{u_{R,n}}^u \int_{\mathbb{S}^2} (r-M)^{1+\delta} (L \Omega^{l_2} T^3 \phi )^2 \, d\omega du' dv \right)^{1/2} \\ \lesssim & \sqrt{C} \sqrt{E_0} \epsilon \frac{v_{n+1}^{1+\delta}}{v_n^{1-\delta_1 / 2}} \sum_{\substack{ l_1 \leq 5 \\ l_2 \leq 5 }} \left( \int_{v_n}^{v_{n+1}} \int_{u_{R,n}}^u \int_{\mathbb{S}^2}  (r-M)^{1-\delta} ( \underline{L} \Omega^{l_1} T^3 \phi )^2 \, d\omega du' dv \right)^{1/2} \\ & \times \left( \int_{v_n}^{v_{n+1}} \int_{u_{R,n}}^u \int_{\mathbb{S}^2}  (r-M)^{1+\delta} (L \Omega^{l_2} T^3 \phi )^2 \, d\omega du' dv \right)^{1/2} \\ & + \sqrt{C} \sqrt{E_0} \epsilon \frac{v_{n+1}^{1+\delta}}{v_n^{1+\delta_2 / 2}} \sum_{\substack{ l_1 \leq 5 \\ l_2 \leq 5 }} \left( \int_{v_n}^{v_{n+1}} \int_{u_{R,n}}^u \int_{\mathbb{S}^2} (r-M)^{1-\delta} ( \underline{L} \Omega^{l_1} T^2 \phi )^2 \, d\omega du' dv \right)^{1/2} \\ & \times \left( \int_{v_n}^{v_{n+1}} \int_{u_{R,n}}^u \int_{\mathbb{S}^2} (r-M)^{1+\delta} (L \Omega^{l_2} T^3 \phi )^2 \, d\omega du' dv  \right)^{1/2} \\ \lesssim & C^{3/2} E_0^{3/2} \epsilon^3 \frac{v_{n+1}^{1+\delta}}{v_n^{2+\delta_2 - \delta_1 /2 -\delta / 2}} + C^{3/2} E_0^{3/2} \epsilon^3 \frac{v_{n+1}^{1+\delta}}{v_n^{5/2+\delta_2 - \delta_1 /2 -\delta / 2}} \lesssim C^{3/2} E_0^{3/2} \epsilon^3 \frac{v_{n+1}^{1+\delta}}{v_n^{2+\delta_2 - \delta_1 /2 -\delta / 2}}  ,
\end{align*}
\endgroup
where we used Sobolev's inequality \eqref{est:sobolev}, Hardy's inequality \eqref{hardy}, the Morawetz decay estimate \eqref{dec:mor3}, and the energy decay estimates \eqref{dec:enttt} and \eqref{dec:entt} for $p = \delta$. For the $m=0$, $m=1$ and $m=2$ cases we have that:
\begin{equation*}
\mbox{$m=0$:  } \lesssim C^{3/2} E_0^{3/2} \epsilon^3 \frac{v_{n+1}^{1+\delta}}{v_n^{4-3\delta_1 / 2 - \delta / 2}} , \mbox{   $m=1$:  } \lesssim C^{3/2} E_0^{3/2} \epsilon^3 \frac{v_{n+1}^{1+\delta}}{v_n^{4+\delta_2-\delta_1 / 2 - \delta/2}},  \mbox{   $m=2$:  } \lesssim C^{3/2} E_0^{3/2} \epsilon^3 \frac{v_{n+1}^{1+\delta}}{v_n^{3+\delta_2 - \delta_1 / 2 - \delta / 2}} .
\end{equation*}

For the eighth term we consider once again the case $m=3$ (as the cases $m=0$, $m=1$ and $m=2$ are easier or similar to it) we have that
 \begingroup
\allowdisplaybreaks
\begin{align*}
 l_8 & = \sum_{\substack{ k_1 + k_2 = k \\ m_1 + m_2 = 3 }}  \int_{v_n}^{v_{n+1}} \int_{u_{R,n}}^u \int_{\mathbb{S}^2}  \frac{A \cdot D}{r} \cdot \langle \slashed{\nabla} \Omega^{k_1} T^{m_1} \phi , \slashed{\nabla} \Omega^{k_2} T^{m_2} \phi \rangle \cdot (L \Omega^kT^3 \phi ) \cdot v^{1+\delta} \, d\omega du' dv \\ = & 2 \sum_{\substack{ k_1 + k_2 = k  }}  \int_{v_n}^{v_{n+1}} \int_{u_{R,n}}^u \int_{\mathbb{S}^2} \frac{A \cdot D}{r} \cdot \langle \slashed{\nabla} \Omega^{k_1}  \phi , \slashed{\nabla} \Omega^{k_2} T^3 \phi \rangle \cdot (L \Omega^k T^3 \phi ) \cdot v^{1+\delta} \, d\omega du' dv \\ & + 2 \sum_{\substack{ k_1 + k_2 = k  }}  \int_{v_n}^{v_{n+1}} \int_{u_{R,n}}^u \int_{\mathbb{S}^2} \frac{A \cdot D}{r} \cdot \langle \slashed{\nabla} \Omega^{k_1} T \phi , \slashed{\nabla} \Omega^{k_2} T^2 \phi \rangle \cdot (L \Omega^k T^3 \phi ) \cdot v^{1+\delta} \, d\omega du' dv \\ \lesssim & v_{n+1}^{1+\delta} \sum_{\substack{ l_1 \leq 5 \\ l_2 \leq 5 \\ l_3 \leq 5 }} \left( \int_{v_n}^{v_{n+1}} \int_{u_{R,n}}^u \int_{\mathbb{S}^2} \Bigg{|}_{r \leq R} (r-M)^{3-\delta} \left( \int_{\mathbb{S}^2} ( \Omega^{l_1} \phi )^2 \, d\omega' \right) \cdot  (  \Omega^{l_2 + 1} T^3 \phi )^2 \, d\omega du' dv \right)^{1/2} \\ & \times \left( \int_{v_n}^{v_{n+1}} \int_{u_{R,n}}^u \int_{\mathbb{S}^2} (r-M)^{1+\delta} (L \Omega^{l_3} T^3 \phi )^2 \, d\omega du' dv  \right)^{1/2} \\ & + v_{n+1}^{1+\delta} \sum_{\substack{ l_1 \leq 5 \\ l_2 \leq 5 \\ l_3 \leq 5 }} \left( \int_{v_n}^{v_{n+1}} \int_{u_{R,n}}^u \int_{\mathbb{S}^2} (r-M)^{3-\delta} \left( \int_{\mathbb{S}^2} ( \Omega^{l_1} T^3 \phi )^2 \, d\omega' \right) \cdot  (  \Omega^{l_2 + 1}  \phi )^2 \, d\omega du' dv \right)^{1/2} \\ & \times \left( \int_{v_n}^{v_{n+1}} \int_{u_{R,n}}^u \int_{\mathbb{S}^2} (r-M)^{1+\delta} (L \Omega^{l_3} T^3 \phi )^2 \, d\omega du' dv  \right)^{1/2} \\ & + v_{n+1}^{1+\delta} \sum_{\substack{ l_1 \leq 5 \\ l_2 \leq 5 \\ l_3 \leq 5 }} \left( \int_{v_n}^{v_{n+1}} \int_{u_{R,n}}^u \int_{\mathbb{S}^2} (r-M)^{3-\delta} \left( \int_{\mathbb{S}^2} ( \Omega^{l_1} T \phi )^2 \, d\omega' \right) \cdot  (  \Omega^{l_2 + 1} T^2 \phi )^2 \, d\omega du' dv \right)^{1/2} \\ & \times \left( \int_{v_n}^{v_{n+1}} \int_{u_{R,n}}^u \int_{\mathbb{S}^2} (r-M)^{1+\delta} (L \Omega^{l_3} T^3 \phi )^2 \, d\omega du' dv  \right)^{1/2} \\ & + v_{n+1}^{1+\delta} \sum_{\substack{ l_1 \leq 5 \\ l_2 \leq 5 \\ l_3 \leq 5 }} \left( \int_{v_n}^{v_{n+1}} \int_{u_{R,n}}^u \int_{\mathbb{S}^2} (r-M)^{3-\delta} \left( \int_{\mathbb{S}^2} ( \Omega^{l_1} T^2 \phi )^2 \, d\omega' \right) \cdot  (  \Omega^{l_2 + 1} T \phi )^2 \, d\omega du' dv \right)^{1/2} \\ & \times \left( \int_{v_n}^{v_{n+1}} \int_{u_{R,n}}^u \int_{\mathbb{S}^2} (r-M)^{1+\delta} (L \Omega^{l_3} T^3 \phi )^2 \, d\omega du' dv  \right)^{1/2} \\ \lesssim & C^{3/2} E_0^{3/2} \epsilon^3 \frac{v_{n+1}^{1+\delta}}{v_n^{2+3\delta_2 / 2 - \delta / 2}} + C^{3/2} E_0^{3/2} \epsilon^3 \frac{v_{n+1}^{1+\delta}}{v_n^{5/2 +3\delta_2 / 2 - \delta /2}} +   C^{3/2} E_0^{3/2} \epsilon^3 \frac{v_{n+1}^{1+\delta}}{v_n^{9/4+3\delta_2 / 2 - \delta / 2 }} \\ \lesssim & C^{3/2} E_0^{3/2} \epsilon^3 \frac{v_{n+1}^{1+\delta}}{v_n^{2+3\delta_2 / 2 - \delta / 2}} ,
\end{align*}
\endgroup
where again we used Sobolev's inequality \eqref{est:sobolev}, Cauchy-Schwarz, the Morawetz decay estimate \eqref{dec:mor3} for all terms, the pointwise decay estimates \eqref{dec:psi}, \eqref{dec:tpsi}, \eqref{dec:ttpsi}, \eqref{dec:tttpsi}, and the energy decay estimates \eqref{dec:en0p}, \eqref{dec:entp}, \eqref{dec:enttp} and \eqref{dec:entttp} for $p=\delta$. For the cases $m=0$, $m=1$ and $m=2$ we have that:
\begin{equation*}
\mbox{$m=0$:  } \lesssim C^{3/2} E_0^{3/2} \epsilon^3 \frac{v_{n+1}^{1+\delta}}{v_n^{4-3\delta_1 / 2 - \delta / 2}} , \mbox{   $m=1$:  } \lesssim C^{3/2} E_0^{3/2} \epsilon^3 \frac{v_{n+1}^{1+\delta}}{v_n^{4+\delta_2-\delta_1 / 2 - \delta/2}} , \mbox{   $m=2$:  } \lesssim C^{3/2} E_0^{3/2} \epsilon^3 \frac{v_{n+1}^{1+\delta}}{v_n^{3+\delta_2 - \delta_1 / 2 - \delta / 2}} .
\end{equation*}

All the above estimates give us that in order show the estimates \eqref{est:aux1} we need the following conditions for $\delta$:
\begin{equation*}
\begin{split}
& \mbox{$m=0$:  } 1+ \delta  < 3 - \delta_1  \Rightarrow \delta < 2 - \delta_1 , \\ & \mbox{$m=1$:  } 1+ \delta  < 3 + \delta_2  \Rightarrow \delta < 2 + \delta_2 , \\ & \mbox{$m=2$:  } 1+ \delta  < 2 + \delta_2  \Rightarrow \delta < 1 + \delta_2 ,  \\ & \mbox{$m=3$:  } 1+ \delta  < 1 + \delta_2  \Rightarrow \delta < \delta_2 ,
\end{split}
\end{equation*} 
and $\epsilon$ small enough so that $\epsilon^3 < \epsilon^2$.

\end{proof}
\section{The bootstrap argument}\label{end:boot}
In this section we will present the bootstrap argument and verify the estimates of section \ref{as:boot}. We will prove the following Theorem:
\begin{theorem}\label{close}
Let $\psi$ be a solution of the equation \eqref{nw} with the corresponding data, and assume additionally that the bootstrap assumptions of Section \ref{as:boot} for a given $\epsilon$. We have that the following estimates hold true for $\delta_1 , \delta_2 , \beta_0 > 0$ small enough, for $0 < \beta < \delta_2$ and for $0 < \delta \leq \delta_2$:
\begin{equation}\label{A1'}
 \int_{\mcR_{\tau_1}^{\tau_2}} r^{2} | F_0 |^2 \, d\mu_{\mcR}  \lesssim  \frac{C^2 E_0^2 \epsilon^4}{ (1+\tau_1 )^{3-\delta_1}} \tag{\textbf{A1'}},
 \end{equation} 
\begin{equation}\label{A2'}
\int_{\tau_1}^{\tau_2} \int_{\mcN_v^H} (r-M)^{-p-1} D^2 | F_0 |^2 \, d\omega du dv \lesssim \frac{C^2 E_0^2 \epsilon^4}{ (1+\tau_1 )^{3-\delta_1-p}} \mbox{  for $p\in (0,2-\delta_1 ]$} \tag{\textbf{A2'}},
\end{equation}
\begin{equation}\label{A3'}
\left| \int_{\tau_1}^{\tau_2} \int_{\mcN_v^H} (r-M)^{-p} ( \underline{L} \phi_0 ) \cdot  ( D F_0 )  \, d\omega du dv \right| \lesssim \frac{C^2 E_0^2 \epsilon^4}{ (1+\tau_1 )^{3-\delta_1-p}} \mbox{  for $p\in (2,3-\delta_1 ]$} \tag{\textbf{A3'}},
\end{equation}
\begin{equation}\label{A4'}
\left| \int_{\tau_1}^{\tau_2} \int_{\mcN_v^H}  (r-M)^{-p} (\underline{L}  \Phi_0^H ) \cdot \underline{L} ( r^2  F_0 ) \, d\omega du dv \right| \lesssim \beta_0 \mathcal{L} + \frac{C^2 E_0^2 \epsilon^4}{ (1+\tau_1 )^{1-\delta_1-p}} \mbox{  for $p \in (0,1-\delta_1]$} \tag{\textbf{A4'}} ,
\end{equation}
\begin{equation}\label{A5'}
\begin{split}
 \int_{\mcR_{\tau_1}^{\tau_2}   \setminus \mathcal{A}_{\tau_1}^{\tau_2}} &  r^{p+1} | F_0 |^2 \, d\mu_{\mcR}  \lesssim  \frac{C^2 E_0^2 \epsilon^4}{ (1+\tau_1 )^{3-\delta_1-p}} \mbox{  for $p \in (1,2-\delta_1 ]$, and } \\ & \left( \int_{\tau_1}^{\tau_2} \left( \int_{\mcN_u^I} r^p | F_0 |^2 \, d\omega dv \right)^{1/2} du \right)^2 \lesssim \frac{C^2 E_0^2 \epsilon^4}{ (1+\tau_1 )^{3-\delta_1-p}} \mbox{  for $p \in (2,3-\delta_1 ]$} 
\end{split} 
\tag{\textbf{A5'}},
\end{equation} 
\begin{equation}\label{A6'}
\left| \int_{\tau_1}^{\tau_2} \int_{\mcN_u^I} r^p ( L \Phi_0^I ) \cdot ( L ( r^3  F_0 ) \, d\omega dv du \right| \lesssim \beta_0 \mathcal{L} + \frac{C^2 E_0^2 \epsilon^4}{ (1+\tau_1 )^{1-\delta_1-p}} \mbox{  for $p \in (0,1-\delta_1 ]$} \tag{\textbf{A6'}},
\end{equation}
\begin{equation}\label{B1'}
 \sum_{k \leq 5} \left( \int_{\mathcal{A}_{\tau_1}^{\tau_2}} (r-M)^{-1-\delta} D^2 |\Omega^k F_{\geq 1} |^2 \, d\omega du dv + \int_{\mcR_{\tau_1}^{\tau_2} \setminus \mathcal{A}_{\tau_1}^{\tau_2}} r^{2} | \Omega^k F_{\geq 1} |^2 \, d\mu_{\mcR} \right) \lesssim  \frac{C^2 E_0^2 \epsilon^4}{ (1+\tau_1 )^{3+\delta_2}} \tag{\textbf{B1'}},
 \end{equation} 
 \begin{equation}\label{B2'}
\sum_{k \leq 5} \left| \int_{\tau_1}^{\tau_2} \int_{\mcN_v^H}  (r-M)^{-p} ( \underline{L} \Omega^k \phi_{\geq 1} ) \cdot D  ( \Omega^k F_{\geq 1} ) \, d\omega du dv \right|  \lesssim \beta_0 \mathcal{L} + \frac{C^2 E_0^2 \epsilon^4}{ (1+\tau_1 )^{3+\delta_2-p}} \mbox{  for $p \in (0,2-\delta_1 ]$} \tag{\textbf{B2'}},
\end{equation}
\begin{equation}\label{B3'}
\sum_{k \leq 5} \left| \int_{\tau_1}^{\tau_2} \int_{\mcN_v^H}  (r-M)^{-p} (\underline{L}  \Omega^k \Phi_{\geq 1}^H ) \cdot \underline{L} ( r^2  \Omega^k F_{\geq 1} ) \, d\omega du dv \right| \lesssim \beta_0 \mathcal{L} + \frac{C^2 E_0^2 \epsilon^4}{ (1+\tau_1 )^{1+\delta_2-p}}  \mbox{  for $p \in (0,1+\delta_2]$} \tag{\textbf{B3'}} ,
\end{equation}
\begin{equation}\label{B4'}
 \sum_{k \leq 5} \int_{\mcR_{\tau_1}^{\tau_2} \setminus \mathcal{A}_{\tau_1}^{\tau_2}} r^{p+1} | \Omega^k F_{\geq 1} |^2 \, d\mu_{\mcR}  \lesssim  \frac{C^2 E_0^2 \epsilon^4}{ (1+\tau_1 )^{3+\delta_2-p}} \mbox{  for $p \in (1,2 - \delta_1 ]$} \tag{\textbf{B4'}},
\end{equation} 
\begin{equation}\label{B5'}
\sum_{k \leq 5} \left| \int_{\tau_1}^{\tau_2} \int_{\mcN_u^I} r^p ( L \Omega^k \Phi_{\geq 1}^I ) \cdot ( L ( r^3 \Omega^k F_{\geq 1} ) \, d\omega dv du \right| \lesssim \beta_0 \mathcal{L} + \frac{C^2 E_0^2 \epsilon^4}{ (1+\tau_1 )^{1+\delta_2-p}} \mbox{  for $p \in (0,1+\delta_2 ]$} \tag{\textbf{B5'}},
\end{equation}
\begin{equation}\label{C1'}
\sum_{k \leq 5} \left( \int_{\mathcal{A}_{\tau_1}^{\tau_2}} (r-M)^{-1-\delta} D^2 | \Omega^k T F |^2 \, d\omega du dv +   \int_{\mcR_{\tau_1}^{\tau_2} \setminus \mathcal{A}_{\tau_1}^{\tau_2}} r^{2} | \Omega^k T F |^2 \, d\mu_{\mcR} \right)  \lesssim  \frac{C^2 E_0^2 \epsilon^4}{ (1+\tau_1 )^{3+\delta_2}} \tag{\textbf{C1'}},
\end{equation}
\begin{equation}\label{C2'}
\sum_{k \leq 5} \left| \int_{\tau_1}^{\tau_2} \int_{\mcN_v^H}  (r-M)^{-p} ( \underline{L} \Omega^k T \phi ) \cdot D  ( \Omega^k T F ) \, d\omega du dv \right|  \lesssim \beta_0 \mathcal{L} + \frac{C^2 E_0^2 \epsilon^4}{ (1+\tau_1 )^{3+\delta_2-p}} \mbox{  for $p \in (0,2 - \delta_1 ]$} \tag{\textbf{C2'}},
\end{equation}
\begin{equation}\label{C3'}
\sum_{k \leq 5} \left| \int_{\tau_1}^{\tau_2} \int_{\mcN_v^H}  (r-M)^{-p} (\underline{L}  \Omega^k T \Phi^H ) \cdot \underline{L} ( r^2  \Omega^k T F ) \, d\omega du dv \right| \lesssim \beta_0 \mathcal{L} +  \frac{C^2 E_0^2 \epsilon^4}{ (1+\tau_1 )^{1+\delta_2-p}}  \mbox{  for $p \in (0,1+\delta_2]$} \tag{\textbf{C3'}} ,
\end{equation}
\begin{equation}\label{C4'}
 \sum_{k \leq 5} \int_{\mcR_{\tau_1}^{\tau_2} \setminus \mathcal{A}_{\tau_1}^{\tau_2}} r^{p+1} |  \Omega^k T F |^2 \, d\mu_{\mcR}  \leq  \frac{C^2 E_0^2 \epsilon^4}{ (1+\tau_1 )^{3+\delta_2-p}} \mbox{  for $p \in (1,2 - \delta_1 ]$} \tag{\textbf{C4'}},
\end{equation} 
\begin{equation}\label{C5'}
\sum_{k \leq 5} \left| \int_{\tau_1}^{\tau_2} \int_{\mcN_u^I} r^p ( L \Omega^k T \Phi^I ) \cdot ( L ( r^3 \Omega^k T F ) \, d\omega dv du \right| \lesssim \beta_0 \mathcal{L} + \frac{C^2 E_0^2 \epsilon^4}{ (1+\tau_2 )^{1+\delta_2-p}} \mbox{  for $p \in (0,1+\delta_2 ]$} \tag{\textbf{C5'}},
\end{equation}
\begin{equation}\label{D1'}
\sum_{k \leq 5} \left( \int_{\mathcal{A}_{\tau_1}^{\tau_2}} (r-M)^{-1-\delta} D^2 | \Omega^k T^2 F |^2 \, d\omega du dv +   \int_{\mcR_{\tau_1}^{\tau_2} \setminus \mathcal{A}_{\tau_1}^{\tau_2}} r^{2} | \Omega^k T^2 F |^2 \, d\mu_{\mcR} \right)  \lesssim  \frac{C^2 E_0^2 \epsilon^4}{ (1+\tau_1 )^{2+\delta_2}} \tag{\textbf{D1'}},
\end{equation}
\begin{equation}\label{D2'}
\sum_{k \leq 5} \left| \int_{\tau_1}^{\tau_2} \int_{\mcN_v^H} (r-M)^{-p} ( \underline{L} \Omega^k T^2 \phi ) \cdot D (\Omega^k T^2 F ) \, d\omega du dv \right| \lesssim \beta_0 \mathcal{L} +  \frac{C^2 E_0^2 \epsilon^4}{ (1+\tau_1 )^{2+\delta_2}} \mbox{  for $p \in (0,2-\delta_1 ]$} \tag{\textbf{D2'}},
\end{equation}
\begin{equation}\label{D3'}
\sum_{k \leq 5} \left| \int_{\tau_1}^{\tau_2} \int_{\mcN_v^H}  (r-M)^{-p} (\underline{L}  \Omega^k T^2 \Phi^H ) \cdot \underline{L} ( r^2  \Omega^k T^2 F ) \, d\omega du dv \right| \lesssim \beta_0 \mathcal{L} +  \frac{C^2 E_0^2 \epsilon^4}{ (1+\tau_1 )^{\delta_2-p}}  \mbox{  for $p \in (0,\delta_2]$} \tag{\textbf{D3'}} ,
\end{equation}
\begin{equation}\label{D4'}
 \sum_{k \leq 5} \int_{\mcR_{\tau_1}^{\tau_2} \setminus \mathcal{A}_{\tau_1}^{\tau_2}} r^{p+1} |  \Omega^k T^2 F |^2 \, d\mu_{\mcR}  \lesssim  \frac{C^2 E_0^2 \epsilon^4}{ (1+\tau_1 )^{2+\delta_2-p}} \mbox{  for $p \in (1,2 - \delta_1 ]$} \tag{\textbf{D4'}},
\end{equation} 
\begin{equation}\label{D5'}
\sum_{k \leq 5} \left| \int_{\tau_1}^{\tau_2} \int_{\mcN_u^I} r^p ( L \Omega^k T^2 \Phi^I ) \cdot ( L ( r^3 \Omega^k T^2 F ) \, d\omega dv du \right| \lesssim \beta_0 \mathcal{L} + \frac{C^2 E_0^2 \epsilon^4}{ (1+\tau_2 )^{\delta_2-p}} \mbox{  for $p \in (0,\delta_2 ]$} \tag{\textbf{D5'}},
\end{equation}
\begin{equation}\label{D6'}
\sum_{k \leq 5} \int_{\mcC_{\tau_1}^{\tau_2}} |  \Omega^k T^2 F |^2 \, d\mu_{\mcC} \lesssim \frac{C^2 E_0^2 \epsilon^4}{ (1+\tau_1 )^{3+\delta_2}} \tag{\textbf{D6'}},
\end{equation}
\begin{equation}\label{E1'}
\sum_{k \leq 5} \left( \int_{\mathcal{A}_{\tau_1}^{\tau_2}} (r-M)^{-1-\delta} D^2 | \Omega^k T^3 F |^2 \,  d\omega du dv +  \int_{\mcR_{\tau_1}^{\tau_2} \setminus \mathcal{A}_{\tau_1}^{\tau_2}} r^{2} | \Omega^k T^3 F |^2 \, d\mu_{\mcR} \right) \lesssim  \frac{C^2 E_0^2 \epsilon^4}{ (1+\tau_1 )^{1+\delta_2}} \tag{\textbf{E1'}},
\end{equation}
\begin{equation}\label{E2'}
\sum_{k \leq 5} \left| \int_{\tau_1}^{\tau_2} \int_{\mcN_v^H} (r-M)^{-p} ( \underline{L} \Omega^k T^3 \phi ) \cdot D (\Omega^k T^3 F ) \, d\omega du dv \right| \lesssim \beta_0 \mathcal{L} +  \frac{C^2 E_0^2 \epsilon^4}{ (1+\tau_1 )^{1+\delta_2}} \mbox{  for $p \in (0,1+\delta_2 ]$} \tag{\textbf{E2'}},
\end{equation}
\begin{equation}\label{E3'}
\sum_{k \leq 5} \int_{\mathcal{A}_{\tau_1}^{\tau_2}} (r-M)^{-2} D^2 | \Omega^k T^3 F |^2 \cdot v^{1+\beta} \, d\omega du dv \lesssim C^2 E_0^2 \epsilon^4 \tag{\textbf{E3'}}, 
\end{equation}
\begin{equation}\label{E4'}
 \sum_{k \leq 5} \int_{\mcR_{\tau_1}^{\tau_2}} r^{p+1} |  \Omega^k T^3 F |^2 \, d\mu_{\mcR}  \lesssim  \frac{C^2 E_0^2 \epsilon^4}{ (1+\tau_1 )^{1+\delta_2-p}} \mbox{  for $p \in (0,1+\delta_2]$} \tag{\textbf{E4'}},
\end{equation} 
\begin{equation}\label{E5'}
\sum_{k \leq 5} \int_{\mcC_{\tau_1}^{\tau_2}} | \Omega^k T^3 F |^2 \, d\mu_{\mcC} \lesssim \frac{C^2 E_0^2 \epsilon^4}{ (1+\tau_1 )^{2+\delta_2}} \tag{\textbf{E5'}},
\end{equation}
\begin{equation}\label{E6'}
\sum_{k \leq 5} \left( \int_{\tau_1}^{\tau_2} \left( \int_{\Sigma_{\tau} \cap \left( \mcR_{\tau_1}^{\tau_2} \setminus \mathcal{A}_{\tau_1}^{\tau_2} \right)} r^2 | \Omega^k T^3 F |^2 \, d\mu_{\Sigma} \right)^{1/2} d\tau \right)^2 \lesssim C^2 E_0^2 \epsilon^4  \tag{\textbf{E6'}},
\end{equation}
\begin{equation}\label{F1'}
\sum_{k \leq 5} \left( \int_{\mathcal{A}_{\tau_1}^{\tau_2}} (r-M)^{-1-\delta} D^2 | \Omega^k T^4 F|^2 \, d\omega du dv + \int_{\mcR_{\tau_1}^{\tau_2} \setminus \mathcal{A}_{\tau_1}^{\tau_2}} r^{2} |  \Omega^k T^4 F |^2 \, d\mu_{\mathcal{R}} \right)  \lesssim  \frac{C^2 E_0^2 \epsilon^4}{1+\tau_1} \tag{\textbf{F1'}},
\end{equation}
\begin{equation}\label{F2'}
\sum_{k \leq 5} \left| \int_{\tau_1}^{\tau_2} \int_{\mcN_v^H} (r-M)^{-p} ( \underline{L} \Omega^k T^4 \phi ) \cdot D ( \Omega^k T^4 F ) \, d\omega du dv \right| \lesssim  \frac{C^2 E_0^2 \epsilon^4}{(1+\tau_1 )^{1 - p}} \mbox{  for $p \in (0,1]$}  \tag{\textbf{F2'}},
\end{equation}
\begin{equation}\label{F3'}
\sum_{k \leq 5} \int_{\mathcal{A}_{\tau_1}^{\tau_2}} (r-M)^{-1-\delta} D^2 | \Omega^k T^4 F |^2 \cdot v^{1+\beta} \, d\omega du dv \lesssim C^2 E_0^2 \epsilon^4 \tag{\textbf{F3'}},
\end{equation}
\begin{equation}\label{F4'}
\sum_{k \leq 5} \int_{\mcR_{\tau_1}^{\tau_2} \setminus \mathcal{A}_{\tau_1}^{\tau_2}} r^{p+1} | \Omega^k T^4 F |^2 \, d\mu_{\mcR} \lesssim \frac{C^2 E_0^2 \epsilon^4}{(1+\tau_1 )^{1-p}} \mbox{  for $p \in (0,1]$} \tag{\textbf{F4'}},
\end{equation}
\begin{equation}\label{F5'}
\sum_{k \leq 5} \int_{\mathcal{C}_{\tau_1}^{\tau_2}} | \Omega^k T^4 F |^2 \, d\mu_{\mathcal{C}} \lesssim \frac{C^2 E_0^2 \epsilon^4}{(1+\tau_1 )^2} \tag{\textbf{F5'}},
\end{equation}
\begin{equation}\label{F6'}
\sum_{k \leq 5} \left( \int_{\tau_1}^{\tau_2} \left( \int_{\Sigma_{\tau} \cap \left( \mcR_{\tau_1}^{\tau_2} \setminus \mathcal{A}_{\tau_1}^{\tau_2} \right)} r^{1+\delta} | \Omega^k T^4 F |^2 \, d\mu_{\Sigma} \right)^{1/2} d\tau \right)^2 \lesssim C^2 E_0^2 \epsilon^4  \tag{\textbf{F6'}},
\end{equation}
\begin{equation}\label{G1'}
\begin{split}
\sum_{k \leq 5} & \int_{\mathcal{A}_{\tau_1}^{\tau_2}}  (r-M)^{-1-\delta} D^2 |  \Omega^k T^5 F |^2 \cdot v^{1+\beta} \, d\omega du dv \\ & + \left( \int_{\tau_1}^{\tau_2} \left( \int_{\Sigma_{\tau} \setminus ( \mcN_{\tau}^H \cup \mcN_{\tau}^I )} | \Omega^k T^5 F |^2 d\mu_{\Sigma} \right)^{1/2} d\tau \right)^2 \\ & + \int_{\tau_1}^{\tau_2} \int_{\mcN_{\tau}^I} r^{1+\delta} | \Omega^k T^5 F |^2 \, d\mu_{\mcN^I} d\tau  
\\ & +  \left( \int_{\tau_1}^{\tau_2} \left( \int_{\Sigma_{\tau} \cap \left( \mcR_{\tau_1}^{\tau_2} \setminus \mathcal{A}_{\tau_1}^{\tau_2} \right)} r^{1+\delta} | \Omega^k T^5 F |^2 \, d\mu_{\Sigma} \right)^{1/2} d\tau \right)^2 \lesssim  C^2 E_0^2 \epsilon^4 ,
\end{split}
\tag{\textbf{G1'}} 
\end{equation}
\begin{equation}\label{G2'}
\sum_{k \leq 5} \int_{\mathcal{C}_{\tau_1}^{\tau_2}} | \Omega^k T^5 F |^2 \, d\mu_{\mathcal{C}} \lesssim \frac{C^2 E_0^2 \epsilon^4}{1+\tau_1} \tag{\textbf{G2'}}.
\end{equation}

\end{theorem}  
\begin{remark}
We expand the nonlinear term with $\Omega^k$ and $T^m$ commutations and we have that:
\begin{equation}\label{eq:tcomm}
\begin{split}
\int_{\mcR_{\tau_1}^{\tau_2}} (r-M)^{-p} & ( \underline{L} \Omega^k T^2 \phi ) \cdot D (\Omega^k T^2 F ) \, d\mu_{\mcR}  \\ = & \sum_{\substack{k_1 + k_2 = k \\ m_1 + m_2 = m}}\int_{\mcR_{\tau_1}^{\tau_2}} (r-M)^{-p}  ( \underline{L} \Omega^k T^2 \phi ) \cdot \frac{2A \cdot D}{r^3} \cdot (L \Omega^{k_1} T^{m_1} \phi ) \cdot \left( \frac{2r}{D} \underline{L} \Omega^{k_2} T^{m_2} \phi \right) \, d\omega du dv \\ & - \sum_{\substack{k_1 + k_2 = k \\ m_1 + m_2 = m}}\int_{\mcR_{\tau_1}^{\tau_2}} (r-M)^{-p}  ( \underline{L} \Omega^k T^2 \phi ) \cdot \frac{A \cdot D^2}{r^4} \cdot ( \Omega^{k_1} T^{m_1} \phi ) \cdot \left( \frac{2r}{D} \underline{L} \Omega^{k_2} T^{m_2} \phi \right) \, d\omega du dv \\ & - \sum_{\substack{k_1 + k_2 = k \\ m_1 + m_2 = m}}\int_{\mcR_{\tau_1}^{\tau_2}} (r-M)^{-p}  ( \underline{L} \Omega^k T^2 \phi ) \cdot \frac{2A \cdot D}{r^3} \cdot ( \Omega^{k_1} T^{m_1} \phi ) \cdot ( L \Omega^{k_2} T^{m_2} \phi ) \, d\omega du dv \\ & - \sum_{\substack{k_1 + k_2 = k \\ m_1 + m_2 = m}}\int_{\mcR_{\tau_1}^{\tau_2}} (r-M)^{-p}  ( \underline{L} \Omega^k T^2 \phi ) \cdot \frac{A \cdot D^2}{r^4} \cdot (\Omega^{k_1} T^{m_1} \phi ) \cdot ( \Omega^{k_2} T^{m_2} \phi ) \, d\omega du dv \\ & +  \sum_{\substack{k_1 + k_2 = k \\ m_1 + m_2 = m}}\int_{\mcR_{\tau_1}^{\tau_2}} (r-M)^{-p}  ( \underline{L} \Omega^k T^2 \phi ) \cdot \frac{A \cdot D}{r^2} \cdot \langle \slashed{\nabla} \Omega^{k_1} T^{m_1} \phi , \slashed{\nabla} \Omega^{k_2} T^{m_2} \phi \rangle \, d\omega du dv \\ & +\sum_{\substack{k_1 + k_2 = k \\ m_1 + m_2 = m \\ k_1 + m_1 > 0}} \int_{\mcR_{\tau_1}^{\tau_2}} (r-M)^{-p} ( \underline{L} \Omega^k T^2 \phi ) \cdot D ( \Omega^{k_1} T^{m_1} A ) \cdot (\Omega^{k_2} T^{m_2} F^c ) \, d\mu_{\mcR} .
\end{split}
\end{equation}
We will use the above equation in the bootstrap argument for \eqref{C1'}, \eqref{D1'}, \eqref{E1'} and \eqref{F1'}.
\end{remark}

\begin{proof}
We will prove the estimates stated in the Theorem one by one. 

\eqref{A2'}: In order to estimate the term for the second bootstrap we expand again $F_0$ into its actual terms and we have for any $\tau_1$, $\tau_2$ with $\tau_1 < \tau_2$ that:
 \begingroup
\allowdisplaybreaks
\begin{align*}
\int_{\tau_1}^{\tau_2} & \int_{\mcN_v^H}  (r-M)^{-p-1} D^2 |F_0 |^2 \, d\omega du dv \lesssim   \int_{\tau_1}^{\tau_2} \int_{\mcN_v^H} (r-M)^{-p-1}  A^2_{*} \cdot D^2 \cdot ( L \phi_{*} )^2 \cdot \left( \frac{2r}{D} \underline{L} \phi_{*} \right)^2 \, d\omega du dv  \\ & + \int_{\tau_1}^{\tau_2} \int_{\mcN_v^H} (r-M)^{-p-1}   A^2_{*} \cdot D^4 \cdot \phi_{*}^2  \cdot \left( \frac{2r}{D} \underline{L} \phi_{*} \right)^2 \, d\omega du dv  + \int_{\tau_1}^{\tau_2} \int_{\mcN_v^H} (r-M)^{-p-1}   A^2_{*} \cdot D^2 \cdot \phi_{*}^2  \cdot ( L \phi_{*} )^2 \, d\omega du dv \\ & + \int_{\tau_1}^{\tau_2} \int_{\mcN_v^H} (r-M)^{-p-1}   A^2_{*} \cdot D^4 \cdot \phi_{*}^2 \cdot \phi_{*}^2 \, d\omega du dv   + \int_{\tau_1}^{\tau_2} \int_{\mcN_v^H} (r-M)^{-p-1}   A^2_{*} \cdot D^2 \cdot |\slashed{\nabla} \phi_{*} |^2 \cdot | \slashed{\nabla} \phi_{*} |^2 \, d\omega du dv  ,
\end{align*}
\endgroup
where the subscript $*$ denotes the different angular frequency localizations that always add up to $0$ (when applying the projection to the $0$-th angular frequency to nonlinear terms). For the very first term we have that for any $\tau_1$, $\tau_2$ with $\tau_1 < \tau_2$ that the following holds:
\begin{equation}\label{aux:a1}
\begin{split}
\int_{\tau_1}^{\tau_2} \int_{\mcN_v^H} & (r-M)^{-p-1}  A_{*}^2 \cdot D^2   ( L\phi_{*} )^2 \cdot \left( \frac{2r}{D} \underline{L} \phi_{*} \right)^2 \, d\omega du dv  \\ \lesssim & C E_0 \epsilon^2 \int_{\tau_1}^{\tau_2} \int_{\mcN_v^H} (r-M)^{-p+3} (L \phi_0 )^2 \, d\omega du dv  \lesssim  \frac{C^2 E_0^2 \epsilon^4}{(1+\tau_1 )^{3-\delta_1}} ,
\end{split}
\end{equation}
where we used the pointwise boundedness estimate \eqref{est:boundy}, and the Morawetz decay estimate \eqref{dec:mor} (as $p \in (0,2-\delta_1 ]$). 

For the second term we have that
\begin{equation*}
\begin{split}
\int_{\tau_1}^{\tau_2} \int_{\mcN_v^H} & (r-M)^{-p-1}   A_{*}^2 \cdot D^4 \cdot \phi_{*}^2   \cdot \left( \frac{2r}{D} \underline{L} \phi_{*} \right)^2 \, d\omega du dv \lesssim  \int_{\tau_1}^{\tau_2} \int_{\mcN_v^H} (r-M)^{-p-1}   D^4 \cdot \phi_0^2  \cdot \left( \frac{2r}{D} \underline{L} \phi_{*} \right)^2 \, d\omega du dv \\ \lesssim & C E_0 \epsilon^2  \int_{\tau_1}^{\tau_2} \int_{\mcN_v^H} (r-M)^{-p+7} \cdot \phi_{*}^2 \,  d\omega du dv \lesssim  C E_0 \epsilon^2 \int_{\tau_1}^{\tau_2} \int_{\mcN_v^H} (r-M)^{-p+5} \cdot ( \underline{L} \phi_{*} )^2 \,  d\omega du dv  \\ \lesssim & \frac{C^2 E_0^2 \epsilon^4}{(1+\tau_1 )^{3-\delta_1}}, 
\end{split}
\end{equation*}
where we used the boundedness estimate \eqref{est:boundy}, Hardy's inequality \eqref{hardy}, and the Morawetz decay estimate \eqref{dec:mor}.

For the third term we have that
\begin{equation*}
\begin{split}
\int_{\tau_1}^{\tau_2} \int_{\mcN_v^H} & (r-M)^{-p-1}   A_{*}^2 \cdot D^2 \cdot \phi_{*}^2  \cdot ( L \phi_{*} )^2 \, d\omega du dv \lesssim  \int_{\tau_1}^{\tau_2} \int_{\mcN_v^H} (r-M)^{-p-1}    D^2 \cdot \phi_{*}^2  \cdot ( L \phi_{*} )^2 \, d\omega du dv \\ \lesssim & \frac{C E_0 \epsilon^2}{(1+\tau_1 )^{2-\delta_1}}  \int_{\tau_1}^{\tau_2} \int_{\mcN_v^H} (r-M)^{-p+3}  \cdot \phi_{*}^2  \, d\omega du dv  \lesssim  \frac{C E_0 \epsilon^2}{(1+\tau_1 )^{2-\delta_1}} \int_{\tau_1}^{\tau_2} \int_{\mcN_v^H} (r-M)^{-p+1} \cdot ( \underline{L} \phi_{*} )^2 \,  d\omega du dv  \\ \lesssim & \frac{C^2 E_0^2 \epsilon^4}{(1+\tau_1 )^{5-2\delta_1 - p}}  ,
\end{split}
\end{equation*}
where we used the decay estimate \eqref{dec:tpsi}, Hardy's inequality \eqref{hardy}, and the Morawetz decay estimate \eqref{dec:mor}. As $p \in (0,2-\delta_1 ]$ we have that:
$$ \frac{C^2 E_0^2 \epsilon^4}{(1+\tau_1 )^{5-2\delta_1 - p}} \lesssim \frac{C^2 E_0^2 \epsilon^4}{(1+\tau_1 )^{3-\delta_1 }} . $$

For the fourth term we have that
\begin{equation*}
\begin{split}
 \int_{\tau_1}^{\tau_2} & \int_{\mcN_v^H} (r-M)^{-p-1}   A_{*}^2 \cdot D^4 \cdot \phi_{*}^2 \cdot \phi_{*}^2 \, d\omega du dv \lesssim   \int_{\tau_1}^{\tau_2} \int_{\mcN_v^H} (r-M)^{-p-1}   D^4 \cdot \phi_{*}^2 \cdot \phi_{*}^2 \, d\omega du dv \\ \lesssim & \frac{C E_0 \epsilon^2}{(1+\tau_1 )^{2-\delta_1}} \int_{\tau_1}^{\tau_2} \int_{\mcN_v^H} (r-M)^{-p+7}    \cdot \phi_{*}^2 \, d\omega du dv \lesssim  \frac{C E_0 \epsilon^2}{(1+\tau_1 )^{2-\delta_1}} \int_{\tau_1}^{\tau_2} \int_{\mcN_v^H} (r-M)^{-p+5}    \cdot (\underline{L} \phi_{*} )^2 \, d\omega du dv   \\ \lesssim & \frac{C^2 E_0^2 \epsilon^4}{(1+\tau_1 )^{5-2\delta_1 - p}} \lesssim  \frac{C^2 E_0^2 \epsilon^4}{(1+\tau_1 )^{3-\delta_1 }} ,
\end{split}
\end{equation*}
where once again we used the decay estimate \eqref{dec:psi0}, Hardy's inequality \eqref{hardy}, the Morawetz decay estimate \eqref{dec:mor}, and the fact that $p \in (0,2-\delta_1 ]$.

Finally for the fifth term we have that
\begin{equation*}
\begin{split}
\int_{\tau_1}^{\tau_2} & \int_{\mcN_v^H} (r-M)^{-p-1}   A_{*}^2 \cdot D^2 \cdot |\slashed{\nabla} \phi_{*} |^2 \cdot  | \slashed{\nabla} \phi_{*} |^2 \, d\omega du dv \\ \lesssim & \int_{\tau_1}^{\tau_2} \int_{\mcN_v^H} (r-M)^{-p-1} D^2 \cdot |\slashed{\nabla} \phi_{*} |^2 \cdot | \slashed{\nabla} \phi_{*} |^2 \, d\omega du dv  \lesssim  \frac{C E_0 \epsilon^2}{(1+\tau_1 )^{2+\delta_2}} \int_{\tau_1}^{\tau_2} \int_{\mcN_v^H} (r-M)^{-p+3} \cdot ( \Omega \phi_{*} )^2  \, d\omega du dv \\ \lesssim & \frac{C E_0 \epsilon^2}{(1+\tau_1 )^{2+\delta_2}} \int_{\tau_1}^{\tau_2} \int_{\mcN_v^H} (r-M)^{-p+1} \cdot ( \underline{L} \Omega \phi_{*} )^2  \, d\omega du dv  \lesssim  \frac{C^2 E_0^2 \epsilon^4}{(1+\tau_1 )^{5+2\delta_2 -p}} \lesssim \frac{C^2 E_0^2 \epsilon^4}{(1+\tau_1 )^{3-\delta_1}}, 
\end{split}
\end{equation*}
which is a better estimate than the one desired (as $\delta_2 > 0$), and which we obtained by using the decay estimate \eqref{dec:psi1} for $k=1$, Hardy's inequality \eqref{hardy}, the energy decay estimates \eqref{dec:en1p}, and the fact that $p \in (0,2-\delta_1 ]$.

\eqref{A1'}: This follows from the estimates for $p=1$ in the bootstrap argument for \eqref{A2'}, where it was shown that for $p=1$ we have decay of rate $ \frac{C^2 E_0^2 \epsilon^4}{(1+\tau_1 )^{3-\delta_1}}$ (which is better than required for the purposes of \eqref{A2'}).

\eqref{A3'}: We examine only the term of the nonlinearity involving both $L$ and $\underline{L}$ derivatives (as the rest are similar or easier as demonstrated already by our work in \eqref{A2'}) and we have for $\beta$ small enough and for the subscript $*$ denoting the different angular frequency localizations always adding up to $0$ that:
\begingroup
\allowdisplaybreaks
\begin{align*}
\int_{\tau_1}^{\tau_2} \int_{\mcN_v^H}  (r-M)^{-p} (\underline{L}  \phi_0 ) & \cdot A_{*} ( L \phi_{*} ) \cdot ( \underline{L} \phi_{*} )  \, d\omega du dv \lesssim \int_{\tau_1}^{\tau_2} \int_{\mcN_v^H}  (r-M)^{-p} (\underline{L}  \phi_{*} )^2 \cdot \frac{1}{v^{1+\beta}} \, d\omega du dv \\& + \int_{\tau_1}^{\tau_2} \int_{\mcN_v^H}  (r-M)^{-p} ( L \phi_{*} )^2 \cdot ( \underline{L} \phi_{*} )^2 \cdot v^{1+\beta} \, d\omega du dv  \\ \lesssim & \frac{1}{\tau_1^{\beta}} \sup_{v \in [\tau_1 , \tau_2 ]}\int_{\mcN_v^H}  (r-M)^{-p} (\underline{L}  \phi_0 )^2 \, d\omega du  \\ & + \int_{\tau_1}^{\tau_2} \int_{\mathbb{S}^2} \sup_u ( L \phi_{*} )^2 \cdot v^{1+\beta} \, d\omega dv \cdot \sup_{v \in [\tau_1 , \tau_2 ]} \int_{\mcN_v^H} ( r-M )^{-p} ( \underline{L} \phi_{*} )^2 \, d\omega du \\ \lesssim & \frac{1}{\tau_1^{\beta}} \sup_{v \in [\tau_1 , \tau_2 ]}\int_{\mcN_v^H}  (r-M)^{-p} (\underline{L}  \phi_0 )^2 \, d\omega du  \\ & +  \int_{\tau_1}^{\tau_2} \int_{\mathbb{S}^2} \sup_u ( L \phi_{*} )^2 \cdot v^{1+\beta} \, d\omega dv \cdot \sup_{v \in [\tau_1 , \tau_2 ]} \int_{\mcN_v^H} ( r-M )^{-p} ( \underline{L} \phi_0 )^2 \, d\omega du \\ & + \int_{\tau_1}^{\tau_2} \int_{\mathbb{S}^2} \sup_u ( L \phi_{*} )^2 \cdot v^{1+\beta} \, d\omega dv \cdot \sup_{v \in [\tau_1 , \tau_2 ]} \int_{\mcN_v^H} ( r-M )^{-p+2} ( \underline{L} \Phi^H_{\geq 1} )^2 \, d\omega du \\ \lesssim & \frac{1}{\tau_1^{\beta}} \sup_{v \in [\tau_1 , \tau_2 ]}\int_{\mcN_v^H}  (r-M)^{-p} (\underline{L}  \phi_0 )^2 \, d\omega du  \\ & + C \epsilon \left( \frac{C \epsilon}{\tau_1^{3-\delta_1 - p}} + \frac{C \epsilon}{\tau_1^{3+\delta_2 - p}} \right) , 
\end{align*}
\endgroup
where we used the auxiliary estimates \eqref{est:aux1} for $m=0$ (by taking $\beta < 2-\delta_1$), Hardy's inequality (as $p \in (2, 3-\delta_1 ] \Rightarrow 4-p > 1$) and the energy decay estimates \eqref{dec:en0p}, \eqref{dec:en1p1}.

\eqref{A4'}: Using again the form of $F_0$ and denoting by subscript $*$ the different angular frequency localizations we have that:
 \begingroup
\allowdisplaybreaks
\begin{align*}
\int_{\tau_1}^{\tau_2} \int_{\mcN_v^H}  (r-M)^{-p} (\underline{L}  \Phi_0^H ) & \cdot \underline{L} ( r^2  F_0 ) \, d\omega du dv  = \int_{\tau_1}^{\tau_2} \int_{\mcN_v^H}  (r-M)^{-p} (\underline{L}  \Phi_0^H ) \cdot \frac{A_{*}}{r} \cdot  ( L\phi_{*} ) \cdot \left( \underline{L} \left( \frac{2r}{D} \underline{L} \phi_{*} \right) \right)  \, d\omega du dv \\ & + \int_{\tau_1}^{\tau_2} \int_{\mcN_v^H}  (r-M)^{-p} (\underline{L}  \Phi_0^H ) \cdot \frac{A_{*}}{r} \cdot  ( \underline{L} L\phi_{*} ) \cdot  \left( \frac{2r}{D} \underline{L} \phi_{*} \right)   \, d\omega du dv  \\ & + \int_{\tau_1}^{\tau_2} \int_{\mcN_v^H}  (r-M)^{-p} (\underline{L}  \Phi_0^H ) \cdot \left( \frac{\underline{L}A_{*}}{r} + \frac{A_{*} \cdot D}{2r^2} \right) \cdot  \left[ ( L\phi ) \cdot  \left( \frac{2r}{D} \underline{L} \phi \right) \right]_{*}  \, d\omega du dv \\ & - \int_{\tau_1}^{\tau_2} \int_{\mcN_v^H}  (r-M)^{-p} (\underline{L}  \Phi_0^H ) \cdot \frac{A_{*}}{2r^2} \cdot D \phi_{*} \cdot \left( \underline{L} \left( \frac{2r}{D} \underline{L} \phi_{*} \right) \right)  \, d\omega du dv  \\ & - \int_{\tau_1}^{\tau_2} \int_{\mcN_v^H}  (r-M)^{-p} (\underline{L}  \Phi_0^H ) \cdot \frac{A_{*}}{2r^2} \cdot D ( \underline{L} \phi_{*} ) \cdot  \left( \frac{2r}{D} \underline{L} \phi_{*} \right)  \, d\omega du dv  \\ & + \int_{\tau_1}^{\tau_2} \int_{\mcN_v^H}  (r-M)^{-p} (\underline{L}  \Phi_0^H ) \cdot \left(\frac{\underline{L}A_{*} \cdot D}{2r^2}+\frac{A_{*} \cdot D^2}{2r^3} - \frac{A_{*} \cdot D D'}{4r^2}  \right) \left[ \phi \cdot  \left( \frac{2r}{D} \underline{L} \phi \right) \right]_{*}  \, d\omega du dv \\ &  - \int_{\tau_1}^{\tau_2} \int_{\mcN_v^H}  (r-M)^{-p} (\underline{L}  \Phi_0^H ) \cdot \frac{A_{*}}{r} \cdot ( \underline{L} \phi_{*} ) \cdot  (L \phi_{*} )  \, d\omega du dv \\ &  - \int_{\tau_1}^{\tau_2} \int_{\mcN_v^H}  (r-M)^{-p} (\underline{L}  \Phi_0^H ) \cdot \frac{A_{*}}{r} \cdot \phi_{*}  \cdot  ( \underline{L} L \phi_{*} )  \, d\omega du dv  \\ &  - \int_{\tau_1}^{\tau_2} \int_{\mcN_v^H}  (r-M)^{-p} (\underline{L}  \Phi_0^H ) \cdot \left( \frac{\underline{L}A_{*}}{r} + \frac{A_{*} \cdot D}{2r^2} \right) \cdot \left[ \phi \cdot  (  L \phi ) \right]_{*}  \, d\omega du dv \\ &  - \int_{\tau_1}^{\tau_2} \int_{\mcN_v^H}  (r-M)^{-p} (\underline{L}  \Phi_0^H ) \cdot \frac{A_{*}}{r^2} \cdot 2D  \phi_{*} \cdot  (  \underline{L} \phi_{*} )   \, d\omega du dv \\ &  - \int_{\tau_1}^{\tau_2} \int_{\mcN_v^H}  (r-M)^{-p} (\underline{L}  \Phi_0^H ) \cdot \left( \frac{\underline{L}A_{*} \cdot D}{2r^2} - \frac{A_{*} \cdot D D'}{4r^2} + \frac{A_{*} \cdot D^2}{2r^3} \right) \cdot \left[ \phi^2 \right]_{*} \, d\omega du dv \\ & + \int_{\tau_1}^{\tau_2} \int_{\mcN_v^H}  (r-M)^{-p} (\underline{L}  \Phi_0^H ) \cdot A_{*} \cdot \langle \underline{L} \slashed{\nabla} \phi_{*} , \slashed{\nabla} \phi_{*} \rangle \, d\omega du dv \\ & + \int_{\tau_1}^{\tau_2} \int_{\mcN_v^H}  (r-M)^{-p} (\underline{L}  \Phi_0^H ) \cdot \underline{L} A_{*} \cdot \langle  \slashed{\nabla} \phi_{*} , \slashed{\nabla} \phi_{*} \rangle \, d\omega du dv   .
\end{align*}
\endgroup
For the first term we have that
 \begingroup
\allowdisplaybreaks
\begin{align*}
\Bigg| \int_{\tau_1}^{\tau_2} \int_{\mcN_v^H}  (r-M)^{-p} (\underline{L}  \Phi_0^H ) & \cdot \frac{A_{*}}{r} \cdot  ( L\phi_{*} ) \cdot \left( \underline{L} \left( \frac{2r}{D} \underline{L} \phi_{*} \right) \right)  \, d\omega du dv \Bigg|  \\ \leq &   \int_{\tau_1}^{\tau_2} \int_{\mcN_v^H}  (r-M)^{-p}  \left( \underline{L} \left( \frac{2r}{D} \underline{L} \phi_0g \right) \right)^2 \frac{1}{v^{1+\delta_0}}  \, d\omega du dv \\ & +\int_{\tau_1}^{\tau_2} \int_{\mcN_v^H}  (r-M)^{-p}    ( L\phi_{*} )^2 \cdot \left( \underline{L} \left( \frac{2r}{D} \underline{L} \phi_{*}  \right) \right)^2 v^{1+\delta_0}  \, d\omega du dv \\ \lesssim & \frac{1}{\tau_1^{\delta_0}} \int_{\mcN_v^H}  (r-M)^{-p}  \left( \underline{L} \left( \frac{2r}{D} \underline{L} \phi_0 \right) \right)^2  \, d\omega du  \\ & + \int_{\tau_1}^{\tau_2}  \int_{\mathbb{S}^2} \sup_u ( L\phi_{*} )^2 \cdot v^{1+\delta_0} \, d\omega dv \cdot \sup_{v \in [\tau_1, \tau_2 ]} \int_{\mcN_v^H}  (r-M)^{-p} \left( \underline{L} \left( \frac{2r}{D} \underline{L} \phi_{*} \right) \right)^2  \, d\omega du \\ \lesssim & \frac{1}{\tau_1^{\delta_0}} \int_{\mcN_v^H}  (r-M)^{-p}  \left( \underline{L} \left( \frac{2r}{D} \underline{L} \phi_0 \right) \right)^2  \, d\omega du  \\ & +  C E_0 \epsilon^2 \frac{C E_0 \epsilon^2}{\tau_1^{1-\delta_1 - p}}  ,
\end{align*}
\endgroup
for some $0 < \delta_0 < \delta_2$, where we used the auxiliary estimate \eqref{est:aux1} for $m=0$ and the energy decay estimates \eqref{dec:en0p2}.  

For the second term
\begin{equation*}
\begin{split}
\int_{\tau_1}^{\tau_2} \int_{\mcN_v^H}  (r-M)^{-p} (\underline{L}  \Phi_0^H ) \cdot  & \frac{A_0}{r} \cdot  ( \underline{L} L\phi_0 ) \cdot  \left( \frac{2r}{D} \underline{L} \phi_0 \right)   \, d\omega du dv \mbox{  and  } \\ & \int_{\tau_1}^{\tau_2} \int_{\mcN_v^H}  (r-M)^{-p} (\underline{L}  \Phi_0^H ) \cdot \frac{A_0}{r} \cdot ( \underline{L} L\phi_{\gtrsim 1} ) \cdot \left( \frac{2r}{D} \underline{L} \phi_{\gtrsim 1} \right)  \, d\omega du dv 
\end{split}
\end{equation*}
we can use the equation to expand $\underline{L} L \phi_{*}$, and we get terms with better decay than required.

For the third term we note that due to our assumptions on $A$ we have that
\begin{equation*}
\begin{split}
\Bigg| \int_{\tau_1}^{\tau_2} \int_{\mcN_v^H}  (r-M)^{-p} (\underline{L}  \Phi_0^H ) \cdot & \left( \frac{\underline{L}A_{*}}{r} + \frac{A_{*} \cdot D}{2r^2} \right) \cdot  \left[ ( L\phi ) \cdot  \left( \frac{2r}{D} \underline{L} \phi \right) \right]_{*}  \, d\omega du dv \Bigg| \\ \lesssim & \int_{\tau_1}^{\tau_2} \int_{\mcN_v^H}  (r-M)^{-p} |\underline{L}  \Phi_0^H | \cdot D \cdot  \left| \left[ ( L\phi ) \cdot  \left( \frac{2r}{D} \underline{L} \phi \right) \right]_{*} \right| \, d\omega du dv ,
\end{split}
\end{equation*}
and this can be treated as the terms in the \eqref{A2'} bootstrap.

For the fourth term we have that
 \begingroup
\allowdisplaybreaks
\begin{align*}
\Bigg| \int_{\tau_1}^{\tau_2} \int_{\mcN_v^H}  (r-M)^{-p} (\underline{L}  \Phi_0^H ) & \cdot \frac{A_{*}}{2r^2} \cdot D \phi_{*} \cdot \left( \underline{L} \left( \frac{2r}{D} \underline{L} \phi_{*} \right) \right)  \, d\omega du dv \Bigg| \\ \leq & \beta  \int_{\tau_1}^{\tau_2} \int_{\mcN_v^H}  (r-M)^{-p+1} (\underline{L}  \Phi_0^H )^2 \, d\omega du dv \\ & + \frac{1}{\beta}  \int_{\tau_1}^{\tau_2} \int_{\mcN_v^H}  (r-M)^{-p-1}  D^2 \phi_{*}^2 \cdot \left( \underline{L} \left( \frac{2r}{D} \underline{L} \phi_{*} \right) \right)^2 \, d\omega du dv \\  \lesssim & \beta \int_{\tau_1}^{\tau_2} \int_{\mcN_v^H}  (r-M)^{-p+1} (\underline{L}  \Phi_0^H )^2 \, d\omega du dv \\ & + \frac{1}{\beta}  \frac{C E_0 \epsilon^2}{(1+\tau_1 )^{2-\delta_1}} \int_{\tau_1}^{\tau_2} \int_{\mcN_v^H}  (r-M)^{-p+3}    \left( \underline{L} \left( \frac{2r}{D} \underline{L} \phi_{*} \right) \right)^2 \, d\omega du dv  \\ \lesssim & \beta \int_{\tau_1}^{\tau_2} \int_{\mcN_v^H}  (r-M)^{-p+1} (\underline{L}  \Phi_0^H )^2 \, d\omega du dv \\ & + \frac{1}{\beta}  \frac{C E_0 \epsilon^2}{(1+\tau_1 )^{3-2\delta_1-p}} ,
\end{align*}
\endgroup
and as $\delta_1$ is chosen to be small enough, we have that $3-2\delta_1 - p > 1-\delta_1 -p$ for any $p \in (0,1-\delta_1 ]$, and this gives us the desired estimate by choosing $\beta$ to be also small enough. Note that we used the pointwise decay estimate \eqref{dec:psi} and the energy decay estimates \eqref{dec:enp1}.

For the fifth term we have that
 \begingroup
\allowdisplaybreaks
\begin{align*}
\Bigg| \int_{\tau_1}^{\tau_2} \int_{\mcN_v^H}  (r-M)^{-p} (\underline{L}  \Phi_0^H ) & \cdot \frac{A_0}{2r^2} \cdot D ( \underline{L} \phi_{*} ) \cdot  \left( \frac{2r}{D} \underline{L} \phi_{*} \right)  \, d\omega du dv \Bigg| \\ \leq &  \beta \int_{\tau_1}^{\tau_2} \int_{\mcN_v^H}  (r-M)^{-p+1} (\underline{L}  \Phi_0^H )^2 \, d\omega du dv \\ & + \frac{1}{\beta} \int_{\tau_1}^{\tau_2} \int_{\mcN_v^H}  (r-M)^{-p-1}  \cdot D^2 ( \underline{L} \phi_{*} )^2 \cdot  \left( \frac{2r}{D} \underline{L} \phi_{*} \right)^2  \, d\omega du dv \\ \leq &  \beta \int_{\tau_1}^{\tau_2} \int_{\mcN_v^H}  (r-M)^{-p+1} (\underline{L}  \Phi_0^H )^2 \, d\omega du dv \\ & + \frac{1}{\beta} \int_{\tau_1}^{\tau_2} \int_{\mcN_v^H}  (r-M)^{-p-1}  \cdot D^2 ( \underline{L} \phi_{*} )^2 \cdot \left( \frac{2r}{D} \underline{L} \phi_{*} \right)^2  \, d\omega du dv \\ \leq &  \beta \int_{\tau_1}^{\tau_2} \int_{\mcN_v^H}  (r-M)^{-p+1} (\underline{L}  \Phi_0^H )^2 \, d\omega du dv \\ & + \frac{1}{\beta} C E_0 \epsilon^2 \int_{\tau_1}^{\tau_2} \int_{\mcN_v^H}  (r-M)^{-p+3}  \cdot ( \underline{L} \phi_{*} )^2  \, d\omega du dv \\ \leq &  \beta \int_{\tau_1}^{\tau_2} \int_{\mcN_v^H}  (r-M)^{-p+1} (\underline{L}  \Phi_0^H )^2 \, d\omega du dv \\ & + \frac{1}{\beta} \frac{C^2 E_0^2 \epsilon^4}{(1+\tau_1 )^{3-\delta_1 -p}} ,
\end{align*}
\endgroup
and this gives us better than the desired decay. Note that we used the boundedness estimate \eqref{est:boundy} and the decay estimates \eqref{dec:enp1}.

For the sixth term we note that due to the assumptions on $A$ we have that
\begin{equation*}
\begin{split}
\Bigg| \int_{\tau_1}^{\tau_2} \int_{\mcN_v^H}  (r-M)^{-p} (\underline{L}  \Phi_0^H ) & \cdot \left( \frac{\underline{L}A_{*} \cdot D}{2r^2} - \frac{A_{*} \cdot D D'}{4r^2} + \frac{A_{*} \cdot D^2}{2r^3} \right) \cdot \left[ \phi^2 \right]_{*} \, d\omega du dv \Bigg| \\ \lesssim & \Bigg| \int_{\tau_1}^{\tau_2} \int_{\mcN_v^H}  (r-M)^{-p} |\underline{L}  \Phi_0^H |  \cdot D^{3/2} \cdot \left| \left[ \phi^2 \right]_{*} \right| \, d\omega du dv \Bigg|
\end{split}
\end{equation*}
and the resulting term can be treated as the terms in the \eqref{A2'} bootstrap.

For the seventh term we have that
 \begingroup
\allowdisplaybreaks
\begin{align*}
\Bigg| \int_{\tau_1}^{\tau_2} \int_{\mcN_v^H}  (r-M)^{-p} (\underline{L}  \Phi_0^H ) & \cdot \frac{A_{*}}{r} \cdot ( \underline{L} \phi_{*} ) \cdot  (L \phi_{*} )  \, d\omega du dv \Bigg| \\ \leq & \beta \int_{\tau_1}^{\tau_2} \int_{\mcN_v^H}  (r-M)^{-p+1} (\underline{L}  \Phi_0^H )^2 \, d\omega du dv \\ & + \frac{1}{\beta} \int_{\tau_1}^{\tau_2} \int_{\mcN_v^H}  (r-M)^{-p-1}  ( \underline{L} \phi_{*} )^2 \cdot  (L \phi_{*} )^2  \, d\omega du dv \\ \leq & \beta \int_{\tau_1}^{\tau_2} \int_{\mcN_v^H}  (r-M)^{-p+1} (\underline{L}  \Phi_0^H )^2 \, d\omega du dv \\ & + \frac{1}{\beta} \int_{\tau_1}^{\tau_2} \int_{\mcN_v^H}  (r-M)^{-p-1}  D^2 ( L \phi_{*} )^2 \cdot  \left( \frac{2r}{D} \underline{L} \phi_{*} \right)^2  \, d\omega du dv  \\ \leq & \beta \int_{\tau_1}^{\tau_2} \int_{\mcN_v^H}  (r-M)^{-p+1} (\underline{L}  \Phi_0^H )^2 \, d\omega du dv \\ & + \frac{1}{\beta} C E_0 \epsilon^2 \int_{\tau_1}^{\tau_2} \int_{\mcN_v^H}  (r-M)^{-p+1}   ( L\phi_{*} )^2  \, D d\omega du dv  \\ \leq & \beta \int_{\tau_1}^{\tau_2} \int_{\mcN_v^H}  (r-M)^{-p+1} (\underline{L}  \Phi_0^H )^2 \, d\omega du dv \\ & + \frac{1}{\beta} C E_0 \epsilon^2 \int_{\mathcal{A}_{\tau_1}^{\tau_2}} (r-M)^{-p+3}   ( L \psi_{*} )^2  \, d\mu_{\mathcal{A}} \\ \leq & \beta \int_{\tau_1}^{\tau_2} \int_{\mcN_v^H}  (r-M)^{-p+1} (\underline{L}  \Phi_0^H )^2 \, d\omega du dv \\ & + \frac{1}{\beta} \frac{C^2 E_0^2 \epsilon^4}{(1+\tau_1 )^{3-\delta_1}} ,
\end{align*}
\endgroup
where we used the decay from the Morawetz estimate \eqref{dec:mor} as $p \in (0, 1 -\delta_1 ] \Rightarrow 3-p \in 2+\delta_1$, and this gives us better decay than desired.

For the eighth term:
\begin{equation*}
\begin{split}
\int_{\tau_1}^{\tau_2} \int_{\mcN_v^H}  (r-M)^{-p} (\underline{L}  \Phi_0^H ) \cdot & \frac{A_0}{r} \cdot \phi_{*}  \cdot  ( \underline{L} L \phi_{*} )  \, d\omega du dv \mbox{  and  } \\ & \int_{\tau_1}^{\tau_2} \int_{\mcN_v^H}  (r-M)^{-p} (\underline{L}  \Phi_0^H ) \cdot \frac{A_{*}}{r} \cdot \phi_{*} \cdot  ( \underline{L} L \phi_{*} )  \, d\omega du dv 
\end{split}
\end{equation*}
we can use the equation to expand $\underline{L} L \phi_{*}$, and we get terms with better decay than required.

For the ninth term, due to the assumptions on $A$ we have that
\begin{equation*}
\begin{split}
\Bigg| \int_{\tau_1}^{\tau_2} \int_{\mcN_v^H}  (r-M)^{-p} (\underline{L}  \Phi_0^H ) & \cdot \left( \frac{\underline{L}A_{*}}{r} + \frac{A_{*} \cdot D}{2r^2} \right) \cdot \left[ \phi \cdot  (  L \phi ) \right]_{*}  \, d\omega du dv \Bigg| \\ \lesssim & \int_{\tau_1}^{\tau_2} \int_{\mcN_v^H}  (r-M)^{-p} (\underline{L}  \Phi_0^H )  \cdot D\cdot \left[ \phi \cdot  (  L \phi ) \right]_{*}  \, d\omega du dv ,
\end{split}
\end{equation*}
and the resulting term can be treated as the terms in the \eqref{A2'} bootstrap. 

For the tenth term we have that
\begingroup
\allowdisplaybreaks
\begin{align*}
\Bigg| \int_{\tau_1}^{\tau_2} & \int_{\mcN_v^H}  (r-M)^{-p} (\underline{L}  \Phi_0^H )  \cdot \frac{A_0}{r^2} \cdot 2D  \phi_{*} \cdot  (  \underline{L} \phi_{*} )   \, d\omega du dv  \Bigg| \\ \leq & \beta \int_{\tau_1}^{\tau_2} \int_{\mcN_v^H}  (r-M)^{-p+1} (\underline{L}  \Phi_0^H )^2 \, d\omega du dv  + \frac{1}{\beta}  \int_{\tau_1}^{\tau_2} \int_{\mcN_v^H}  (r-M)^{-p-1}  \cdot D^2  \phi_{*}^2 \cdot  (  \underline{L} \phi_{*} )^2   \, d\omega du dv \\ \leq & \beta \int_{\tau_1}^{\tau_2} \int_{\mcN_v^H}  (r-M)^{-p+1} (\underline{L}  \Phi_0^H )^2 \, d\omega du dv  + \frac{1}{\beta} \frac{C E_0 \epsilon^2}{(1+\tau_1 )^{2-\delta_1}}  \int_{\tau_1}^{\tau_2} \int_{\mcN_v^H}  (r-M)^{-p+3}   (  \underline{L} \phi_{*} )^2   \, d\omega du dv \\ \leq & \beta \int_{\tau_1}^{\tau_2} \int_{\mcN_v^H}  (r-M)^{-p+1} (\underline{L}  \Phi_0^H )^2 \, d\omega du dv  + \frac{1}{\beta} \frac{C^2 E_0^2 \epsilon^4}{(1+\tau_1 )^{5-2\delta_1-p}}  ,
\end{align*}
\endgroup
where we used the pointwise decay estimate \eqref{dec:psi} and the energy decay estimate \eqref{dec:enp}. Since $\delta_1 > 0$ is small enough, we obviously have that $5-2\delta_1 - p > 1-\delta_1 -p$ for $p \in (0,1-\delta_1 ]$, which is better than desired.

For the eleventh term, due to the assumptions on $A$ we have that
\begin{equation*}
\begin{split}
\Bigg| \int_{\tau_1}^{\tau_2} \int_{\mcN_v^H}  (r-M)^{-p} (\underline{L}  \Phi_0^H ) & \cdot \left( \frac{\underline{L}A_{*} \cdot D}{2r^2} - \frac{A_{*} \cdot D D'}{4r^2} + \frac{A_{*} \cdot D^2}{2r^3} \right) \cdot \left[ \phi^2 \right]_0 \, d\omega du dv \Bigg| \\ \lesssim & \int_{\tau_1}^{\tau_2} \int_{\mcN_v^H}  (r-M)^{-p} (\underline{L}  \Phi_0^H ) D^3 \cdot \left[ \phi^2 \right]_{*} \, d\omega du dv ,
\end{split}
\end{equation*}
and the resulting term can be treated similarly to the terms of the \eqref{A2'} bootstrap.

For the twelfth term we note that we have that
 \begingroup
\allowdisplaybreaks
\begin{align*}
\Bigg| \int_{\tau_1}^{\tau_2} & \int_{\mcN_v^H}  (r-M)^{-p} (\underline{L}  \Phi_0^H )  \cdot A_{*} \cdot \langle \underline{L} \slashed{\nabla} \phi_{*} , \slashed{\nabla} \phi_{*} \rangle \, d\omega du dv \Bigg| \\ \lesssim & \int_{\tau_1}^{\tau_2} \int_{\mcN_v^H}  (r-M)^{-p} |\underline{L}  \Phi_0^H |  \cdot | \underline{L} \Omega \phi_{*} | \cdot | \Omega \phi_{*} | \, d\omega du dv  +  \int_{\tau_1}^{\tau_2} \int_{\mcN_v^H}  (r-M)^{-p} |\underline{L}  \Phi_0^H |  \cdot D | \Omega \phi_{*} | \cdot | \Omega \phi_{*} | \, d\omega du dv \\ \lesssim & \beta  \int_{\tau_1}^{\tau_2} \int_{\mcN_v^H}  (r-M)^{-p+1} ( \underline{L}  \Phi_0^H )^2 \, d\omega du dv  + \frac{2}{\beta} \int_{\tau_1}^{\tau_2} \int_{\mcN_v^H}  (r-M)^{-p-1}  \cdot | \underline{L} \Omega \phi_{*} |^2 \cdot | \Omega \phi_{*} |^2 \, d\omega du dv \\ & + \frac{2}{\beta} \int_{\tau_1}^{\tau_2} \int_{\mcN_v^H}  (r-M)^{-p-1}  \cdot D^2 |  \Omega \phi_{*} |^2 \cdot | \Omega \phi_{*} |^2 \, d\omega du dv \lesssim  \beta  \int_{\tau_1}^{\tau_2} \int_{\mcN_v^H}  (r-M)^{-p+1} ( \underline{L}  \Phi_0^H )^2 \, d\omega du dv \\ & + \frac{2}{\beta} \int_{\tau_1}^{\tau_2} \int_{\mcN_v^H}  (r-M)^{-p+3}  \cdot | \frac{2r}{D} \underline{L} \Omega \phi_{*} |^2 \cdot | \Omega \phi_{*} |^2 \, d\omega du dv  + \frac{2}{\beta} \int_{\tau_1}^{\tau_2} \int_{\mcN_v^H}  (r-M)^{-p+3}  \cdot  |  \Omega \phi_{*} |^2 \cdot | \Omega \phi_{*} |^2 \, d\omega du dv \\ \lesssim & \beta  \int_{\tau_1}^{\tau_2} \int_{\mcN_v^H}  (r-M)^{-p+1} ( \underline{L}  \Phi_0^H )^2 \, d\omega du dv \\ & + \frac{2}{\beta} C E_0 \epsilon^2 \int_{\tau_1}^{\tau_2} \int_{\mcN_v^H}  (r-M)^{-p+3}    | \Omega \phi_{*} |^2 \, d\omega du dv  + \frac{2}{\beta} \frac{C E_0 \epsilon^2}{(1+\tau_1 )^{2-\delta_1}} \int_{\tau_1}^{\tau_2} \int_{\mcN_v^H}  (r-M)^{-p+3}   | \Omega \phi_{*} |^2 \, d\omega du dv \\ \lesssim & \beta  \int_{\tau_1}^{\tau_2} \int_{\mcN_v^H}  (r-M)^{-p+1} ( \underline{L}  \Phi_0^H )^2 \, d\omega du dv  + \frac{2}{\beta} \frac{C^2 E_0^2 \epsilon^4}{(1+\tau_1 )^{3-\delta_1}}  + \frac{2}{\beta} \frac{C E_0 \epsilon^2}{(1+\tau_1 )^{5-2\delta_1}} , 
\end{align*}
\endgroup
where we used the Morawetz decay estimate \eqref{dec:mor}, the pointwise decay estimate \eqref{dec:psi1}, and the boundedness estimate \eqref{est:boundy}.

Finally for the thirteenth term we note that
\begin{equation*}
\begin{split}
\Bigg| \int_{\tau_1}^{\tau_2} \int_{\mcN_v^H}  (r-M)^{-p} (\underline{L}  \Phi_0^H ) & \cdot \underline{L} A_{*} \cdot \langle  \slashed{\nabla} \phi_{*} , \slashed{\nabla} \phi_{*} \rangle \, d\omega du dv \Bigg| \\ \lesssim & \int_{\tau_1}^{\tau_2} \int_{\mcN_v^H}  (r-M)^{-p} |\underline{L}  \Phi_0^H | \cdot D \cdot | \Omega \phi_{*} | \cdot | \Omega \phi_{*} |\, d\omega du dv ,
\end{split}
\end{equation*}
and the last term was dealt with in the context of the twelfth term.

\eqref{B2'} Using the form of $F_{\geq 1}$ we have that:
 \begingroup
\allowdisplaybreaks
\begin{align*}
\sum_{k \leq 5} \int_{\tau_1}^{\tau_2} \int_{\mcN_v^H} & (r-M)^{-p} ( \underline{L} \Omega^k \phi_{\geq 1} ) \cdot D  ( \Omega^k F_{\geq 1} ) \, d\omega du dv  \\ = & \sum_{k_1 + k_2 = k} \int_{\tau_1}^{\tau_2} \int_{\mcN_v^H} (r-M)^{-p}  ( \underline{L} \Omega^k \phi_{\geq 1} ) \cdot   \frac{A_{*}}{2r^2} \cdot D \cdot ( L \Omega^{k_1} \phi_{*} ) \cdot \left( \frac{2r}{D} \underline{L} \Omega^{k_2} \phi_{*} \right) \, d\omega du dv  \\ & - \sum_{k_1 + k_2 = k}\int_{\tau_1}^{\tau_2} \int_{\mcN_v^H} (r-M)^{-p}  ( \underline{L} \Omega^k \phi_{\geq 1} ) \cdot  \frac{A_{*}}{4r^3} \cdot D^2 \cdot ( \Omega^{k_1} \phi_{*} ) \cdot \left( \frac{2r}{D} \underline{L} \Omega^{k_2} \phi_{*} \right) \, d\omega du dv \\ & - \sum_{k_1 + k_2 = k} \int_{\tau_1}^{\tau_2} \int_{\mcN_v^H} (r-M)^{-p}  ( \underline{L} \Omega^k \phi_{\geq 1} ) \cdot \frac{A_{*}}{2r^2} \cdot D \cdot (\Omega^{k_1} \phi_{*} )  \cdot ( L \Omega^{k_2} \phi_{*} ) \, d\omega du dv \\ & - \sum_{k_1 + k_2 = k} \int_{\tau_1}^{\tau_2} \int_{\mcN_v^H} (r-M)^{-p}  ( \underline{L} \Omega^k \phi_{\geq 1} ) \cdot  \frac{A_{*}}{4r^3} \cdot D^2 \cdot ( \Omega^{k_1} \phi_{*} )  \cdot ( \Omega^{k_2} \phi_{*} ) \, d\omega du dv \\ & + \sum_{k_1 + k_2 = k}\int_{\tau_1}^{\tau_2} \int_{\mcN_v^H} (r-M)^{-p}  ( \underline{L} \Omega^k \phi_{\geq 1} ) \cdot   \frac{A_{*}}{4r} \cdot D \langle \slashed{\nabla} \Omega^{k_1} \phi_{*} , \slashed{\nabla} \Omega^{k_2} \phi_{*} \rangle \, d\omega du dv \\ & + \sum_{\substack{k_1 + k_2 = k \\ k_1 \geq 1}}\int_{\tau_1}^{\tau_2} \int_{\mcN_v^H} (r-M)^{-p}  ( \underline{L} \Omega^k \phi_{\geq 1} ) \cdot  ( \Omega^{k_1} A_{*} ) \cdot D ( \Omega^{k_2} F^c_{*} ) \, d\omega du dv,
\end{align*}
\endgroup
where we use $*$ as a subscript to denote the different angular frequency localizations that always add up to the case of $\geq 1$.

For the first term we have that:
\begingroup
\allowdisplaybreaks
\begin{align*}
\sum_{k_1 + k_2 = k} \int_{\tau_1}^{\tau_2} \int_{\mcN_v^H} & (r-M)^{-p}  ( \underline{L} \Omega^k \phi_{\geq 1} ) \cdot \frac{A_{*}}{2r^2}  D \cdot ( L \Omega^{k_1} \phi_{*} ) \cdot \left( \frac{2r}{D} \Omega^{k_2} \underline{L} \phi_{*} \right) \, d\omega du dv \\ \lesssim & \beta \int_{\tau_1}^{\tau_2} \int_{\mcN_v^H} (r-M)^{-p+1}  ( \underline{L} \Omega^k \phi_{\geq 1} )^2  \, d\omega du dv \\ & +  \int_{\tau_1}^{\tau_2} \int_{\mcN_v^H}  (r-M)^{1+\delta_1}  ( L \Omega^k \phi_{*} )^2 \cdot ( r-M)^{2-\delta_1 -p} \left( \frac{2r}{D} \underline{L} \phi_0 \right)^2  \, d\omega du dv   \\ \lesssim & \beta \int_{\tau_1}^{\tau_2} \int_{\mcN_v^H} (r-M)^{-p+1}  ( \underline{L} \Omega^k \phi_{\geq 1} )^2  \, d\omega du dv \\ & + \frac{C^2 E_0^2 \epsilon^4}{\tau_1^{3-\delta_1}} \mbox{  if $p \in [1,2-\delta_1 ]$ and we get the bound:} \\ \lesssim & \beta \int_{\tau_1}^{\tau_2} \int_{\mcN_v^H} (r-M)^{-p+1}  ( \underline{L} \Omega^k \phi_{\geq 1} )^2  \, d\omega du dv \\ & + \frac{C^2 E_0^2 \epsilon^4}{\tau_1^{4-3\delta_1}} \mbox{  if $p \in (0,1)$,} 
\end{align*}
\endgroup
where we used the boundedness estimate \eqref{est:boundy}, and the Morawetz decay estimate \eqref{dec:mor} in the case of $p \in [1,2-\delta_1]$, and where we used the auxiliary estimate \eqref{est:auxl} for $q = 1-\delta_1$ and the Morawetz decya estimate \eqref{dec:mor} in the case of $p \in (0,1)$. Note that we got better decay than required in both cases as for $p \in [1,2-\delta_1]$ we have that $3-\delta_1 > 3+\delta_2 -p$, and for $p \in (0,1)$ we have that $4-3\delta_1 > 3 + \delta_2 -p$ due to the smallness of $\delta_1$ and $\delta_2$. For the second term we have that
\begingroup
\allowdisplaybreaks
\begin{align*}
\sum_{k_1+k_2 = k} \int_{\tau_1}^{\tau_2} \int_{\mcN_v^H} & (r-M)^{-p}    ( \underline{L} \Omega^k \phi_{\geq 1} ) \cdot \frac{A_{*}}{4r^3} \cdot D^2 \cdot ( \Omega^{k_1} \phi_{*} )  \cdot \left( \frac{2r}{D} \underline{L} \Omega^{k_2} \phi_{*} \right) \, d\omega du dv \\ \lesssim & \beta \int_{\tau_1}^{\tau_2} \int_{\mcN_v^H} (r-M)^{-p+1}    ( \underline{L} \Omega^k \phi_{\geq 1} )^2 \, d\omega du dv \\ & + \frac{1}{\beta} \sum_{k_1 + k_2 = k} \int_{\tau_1}^{\tau_2} \int_{\mcN_v^H}  (r-M)^{-p-1}     D^4 \cdot ( \Omega^{k_1} \phi_{*} )^2  \cdot \left( \frac{2r}{D} \underline{L} \Omega^{k_2} \phi_{*} \right)^2 \, d\omega du dv  \\ \lesssim & \beta \int_{\tau_1}^{\tau_2} \int_{\mcN_v^H} (r-M)^{-p+1}    ( \underline{L} \Omega^k \phi_{\geq 1} )^2 \, d\omega du dv \\ & + \frac{1}{\beta} \sum_{k_1 + k_2 = k}\int_{\tau_1}^{\tau_2} \int_{\mcN_v^H}  (r-M)^{-p+3}     ( \Omega^{k_1} \phi_{*} )^2  \cdot (\underline{L} \Omega^{k_2} \phi_{*} )^2 \, d\omega du dv \\ \lesssim & \beta \int_{\tau_1}^{\tau_2} \int_{\mcN_v^H} (r-M)^{-p+1}    ( \underline{L} \Omega^k \phi_{\geq 1} )^2 \, d\omega du dv \\ & + \frac{C E_0 \epsilon^2}{(1+\tau_1 )^{2-\delta_1}} \sum_{m \leq 5}\int_{\tau_1}^{\tau_2} \int_{\mcN_v^H}  (r-M)^{-p+3}      (\underline{L} \Omega^m \phi_{*} )^2 \, d\omega du dv \\ \lesssim & \beta \int_{\tau_1}^{\tau_2} \int_{\mcN_v^H} (r-M)^{-p+1}    ( \underline{L} \Omega^k \phi_{\geq 1} )^2 \, d\omega du dv \\ & + \frac{C E_0 \epsilon^2}{(1+\tau_1 )^{2-\delta_1}} \frac{C E_0 \epsilon^2}{(1+\tau_1 )^{3-\delta_1}} ,
\end{align*}
\endgroup
where we used Sobolev's inequality \eqref{est:sobolev}, the pointwise estimate \eqref{dec:psi} and the Morawetz decay estimate \eqref{dec:mor} as $p \in (0,2-\delta_1] \Rightarrow 3-p \geq 1+\delta_1$. The resulting decay is better than desired as by the smallness of $\delta_1$ we have that $5-2\delta_1 > 3+\delta_2-p$ for $p \in (0,2-\delta_1 ]$. For the third term we have that
\begingroup
\allowdisplaybreaks
\begin{align*}
\sum_{k_1 + k_2 = k}\int_{\tau_1}^{\tau_2} \int_{\mcN_v^H} & (r-M)^{-p}   ( \underline{L} \Omega^k \phi_{\geq 1} ) \cdot  \frac{A_{*}}{2r^2} \cdot D ( \Omega^{k_1} \phi_{*} )  \cdot ( L \Omega^{k_2} \phi_{*} ) \, d\omega du dv \\ \lesssim & \beta \int_{\tau_1}^{\tau_2} \int_{\mcN_v^H} (r-M)^{-p+1}   ( \underline{L} \Omega^k \phi_{\geq 1} )^2 \, d\omega du dv \\ & + \frac{1}{\beta} \sum_{k_1 + k_2 = k}\int_{\tau_1}^{\tau_2} \int_{\mcN_v^H} (r-M)^{-p-1}  D^2 ( \Omega^{k_1} \phi_{*} )^2  \cdot ( L \Omega^k \phi_{*} )^2 \, d\omega du dv \\ \lesssim & \beta \int_{\tau_1}^{\tau_2} \int_{\mcN_v^H} (r-M)^{-p}   ( \underline{L} \Omega^k \phi_{\geq 1} )^2 \, d\omega du dv \\ & + \frac{C E_0 \epsilon^2}{(1+\tau_1 )^{2-\delta_1}} \sum_{m \leq 5} \int_{\tau_1}^{\tau_2} \int_{\mcN_v^H} (r-M)^{-p+3}   ( L \Omega^m \phi_{*} )^2 \, d\omega du dv \\ \lesssim & \beta \int_{\tau_1}^{\tau_2} \int_{\mcN_v^H} (r-M)^{-p}   ( \underline{L} \Omega^k \phi_{\geq 1} )^2 \, d\omega du dv \\ & + \frac{C E_0 \epsilon^2}{(1+\tau_1 )^{2-\delta_1}} \frac{C E_0 \epsilon^2}{(1+\tau_1 )^{3-\delta_1}} ,
\end{align*}
\endgroup
where we used the pointwise decay estimate \eqref{dec:psi} and the Morawetz decay estimate \eqref{dec:mor} (as $0 < p \leq 2-\delta_1$) and as before we note that we get better decay than required as $5-2\delta_1 > 3+\delta_2-p$ for $p \in (0,2-\delta_1 ]$. For the fourth term we have that
\begingroup
\allowdisplaybreaks
\begin{align*}
\sum_{k_1 + k_2 = k} & \int_{\tau_1}^{\tau_2}  \int_{\mcN_v^H}  (r-M)^{-p}  ( \underline{L} \Omega^k \phi_{\geq 1} ) \cdot   \frac{A_{*}}{4r^3} \cdot D^2 \cdot ( \Omega^{k_1} \phi_{*} ) \cdot (\Omega^{k_2} \phi_{*} ) \, d\omega du dv \\ \lesssim & \beta \int_{\tau_1}^{\tau_2} \int_{\mcN_v^H}  (r-M)^{-p+1}   ( \underline{L} \Omega^k \phi_{\geq 1} )^2 \, d\omega du dv  + \frac{1}{\beta} \sum_{k_1 + k_2 = k}\int_{\tau_1}^{\tau_2} \int_{\mcN_v^H} (r-M)^{-p-1}   D^4 ( \Omega^{k_1} \phi_{*} )^2  \cdot (  \Omega^{k_2} \phi_{*} ) \, d\omega du dv \\ \lesssim & \beta \int_{\tau_1}^{\tau_2} \int_{\mcN_v^H}  (r-M)^{-p+1}   ( \underline{L} \Omega^k \phi_{\geq 1} )^2 \, d\omega du dv \\ & + \frac{C E_0 \epsilon^2}{(1+\tau_1 )^{2-\delta_1}} \sum_{m \leq 5} \int_{\tau_1}^{\tau_2} \int_{\mcN_v^H} (r-M)^{-p+7}   ( \Omega^m \phi_{*} )^2 \, d\omega du dv  \lesssim  \beta \int_{\tau_1}^{\tau_2} \int_{\mcN_v^H} (r-M)^{-p+1}   ( \underline{L} \Omega^k \phi_{\geq 1} )^2 \, d\omega du dv \\ & + \frac{C E_0 \epsilon^2}{(1+\tau_1 )^{2-\delta_1}} \sum_{m \leq 5} \int_{\tau_1}^{\tau_2} \int_{\mcN_v^H} (r-M)^{-p+5}  ( \Omega^m \underline{L} \phi_{*} )^2 \, d\omega du dv   \lesssim  \beta \int_{\tau_1}^{\tau_2} \int_{\mcN_v^H}  (r-M)^{-p+1}   ( \underline{L} \Omega^k \phi_{\geq 1} )^2 \, d\omega du dv \\ & + \frac{C E_0 \epsilon^2}{(1+\tau_1 )^{2-\delta_1}} \frac{C E_0 \epsilon^2}{(1+\tau_1 )^{3-\delta_1}} ,
\end{align*}
\endgroup
where we used the pointwise decay estimate \eqref{dec:psi}, Hardy's inequality \eqref{hardy}, the Morawetz decay estimate \eqref{dec:mor} (as $0 < p \leq 2-\delta_1$), and we got again better decay than required as $5-2\delta_1 > 3+\delta_2-p$ for $p \in (0,2-\delta_1 ]$. For the fifth term we have that
\begingroup
\allowdisplaybreaks
\begin{align*}
\sum_{k_1 + k_2 = k} & \int_{\tau_1}^{\tau_2}  \int_{\mcN_v^H} (r-M)^{-p}  ( \underline{L} \Omega^k \phi_{\geq 1} ) \cdot   \frac{A_{*}}{4r} \cdot  D \langle \slashed{\nabla} \Omega^{k_1} \phi_{*} , \slashed{\nabla} \Omega^{k_2} \phi_{*} \rangle \, d\omega du dv \\ \lesssim & \beta \int_{\tau_1}^{\tau_2} \int_{\mcN_v^H} (r-M)^{-p+1}  ( \underline{L} \Omega^k \phi_{\geq 1} )^2 \, d\omega du dv  + \frac{1}{\beta} \sum_{k_1 + k_2 = k}\int_{\tau_1}^{\tau_2} \int_{\mcN_v^H} (r-M)^{-p+3}  | \slashed{\nabla} \Omega^{k_1} \phi_{*} |^2 \cdot | \slashed{\nabla} \Omega^{k_2} \phi_{*} |^2 \, d\omega du dv \\ \lesssim & \beta \int_{\tau_1}^{\tau_2} \int_{\mcN_v^H} (r-M)^{-p+1}  ( \underline{L} \Omega^k \phi_{\geq 1} )^2 \, d\omega du dv  + \frac{1}{\beta} \frac{C E_0 \epsilon^2}{\tau_1^{2+\delta_2}} \sum_{ 3 \leq l \leq 5}\int_{\tau_1}^{\tau_2} \int_{\mcN_v^H} (r-M)^{-p+3}  | \slashed{\nabla} \Omega^{l} \phi_{\gtrsim 1} |^2  \, d\omega du dv \\ \lesssim & \beta \int_{\tau_1}^{\tau_2} \int_{\mcN_v^H} (r-M)^{-p+1}  ( \underline{L} \Omega^k \phi_{\geq 1} )^2 \, d\omega du dv  + \frac{C^2 E_0^2 \epsilon^4}{\tau_1^{2+\delta_2 + 3+\delta_2 -p}} , 
\end{align*}
\endgroup
where we used the pointwise decay estimates \eqref{dec:psi} and the energy decay provided by the estimates \eqref{dec:en1p}. Finally the last term
\begin{equation*}
\sum_{\substack{k_1 + k_2 = k \\ k_1 \geq 1}}\int_{\tau_1}^{\tau_2} \int_{\mcN_v^H} (r-M)^{-p}  ( \underline{L} \Omega^k \phi_{\geq 1} ) \cdot  ( \Omega^{k_1} A_{*} ) \cdot D ( \Omega^{k_2} F^c_{*} ) \, d\omega du dv
\end{equation*}
involves terms that can be treated similarly to the previous terms.

\eqref{B1'}: We examine the term close to the horizon as the term away from the horizon can be treated by rather classical methods. Moreover we examine in detail only the terms involving the product of $L$ and $\underline{L}$ derivatives since the rest are similar or easier. These terms are:
\begingroup
\allowdisplaybreaks
\begin{align*}
\int_{\mathcal{A}_{\tau_1}^{\tau_2}} & (r-M)^{-1-\delta} ( L \Omega^k \phi_{\gtrsim 1} )^2 \cdot ( \underline{L} \phi_0 )^2 \, d\omega du dv , \\ \int_{\mathcal{A}_{\tau_1}^{\tau_2}} & (r-M)^{-1-\delta} ( L \phi_0 )^2 \cdot ( \underline{L} \Omega^k \phi_{\gtrsim 1} )^2 \, d\omega du dv ,\\ & \sum_{k_1 + k_2 = k} \int_{\mathcal{A}_{\tau_1}^{\tau_2}}  (r-M)^{-1-\delta} ( L \Omega^{k_1} \phi_{\gtrsim 1} )^2 \cdot ( \underline{L} \Omega^{k_2} \phi_{\gtrsim 1} )^2 \, d\omega du dv  ,  \\ & \int_{\mathcal{A}_{\tau_1}^{\tau_2}}  (r-M)^{-1-\delta} ( L \phi_0 )^2 \cdot ( \underline{L} \phi_0 )^2 \, d\omega du dv 
\end{align*}
\endgroup
For the first term we have that
\begingroup
\allowdisplaybreaks
\begin{align*}
\int_{\mathcal{A}_{\tau_1}^{\tau_2}} & (r-M)^{-1-\delta} ( L \Omega^k \phi_{\gtrsim 1} )^2 \cdot ( \underline{L} \phi_0 )^2 \, d\omega du dv \lesssim \int_{\mathcal{A}_{\tau_1}^{\tau_2}}  (r-M)^{-1-\delta} ( L \Omega^k \phi_{\gtrsim 1} )^2 \cdot \left( \frac{2r}{D} \underline{L} \phi_0 \right)^2 \, d\omega du dv \\ \lesssim & C E_0 \epsilon^2 \int_{\mathcal{A}_{\tau_1}^{\tau_2}}  (r-M)^{3-\delta} ( L \Omega^k \phi_{\gtrsim 1} )^2 \, d\omega du dv \lesssim  \frac{C^2 E_0^2 \epsilon^4}{\tau_1^{3+\delta_2}} ,
\end{align*}
\endgroup
where we used the boundedness estimate \eqref{est:boundy} and the Morawetz decay estimates \eqref{dec:mor11}. For the second term we have that
\begingroup
\allowdisplaybreaks
\begin{align*}
\int_{\mathcal{A}_{\tau_1}^{\tau_2}} & (r-M)^{-1-\delta} ( L  \phi_0 )^2 \cdot  ( \underline{L} \Omega^k \phi_{\gtrsim 1} )^2 \, d\omega du dv \lesssim \int_{\mathcal{A}_{\tau_1}^{\tau_2}}  (r-M)^{1+\delta} ( L  \phi_0 )^2 \cdot D^{1-\delta} \left( \underline{L} \Omega^k \phi_{\gtrsim 1} \right)^2 \, d\omega du dv \\ \lesssim & \frac{C E_0 \epsilon^2}{\tau_1^{3/2-\delta_1 /2}}  \int_{\mathcal{A}_{\tau_1}^{\tau_2}} (r-M)^{1+\delta} ( L  \phi_0 )^2 \, d\omega du dv \lesssim  \frac{C^2 E_0^2 \epsilon^4}{\tau_1^{9/2 - 3 \delta_1 / 2}} \lesssim \frac{C^2 E_0^2 \epsilon^4}{\tau_1^{3+\delta_2}} , 
\end{align*}
\endgroup
where we used the auxiliary estimate \eqref{est:auxl}, the Morawetz decay estimate \eqref{dec:mor}, and the fact that $\delta_1$ and $\delta_2$ are small enough. The third and fourth terms can treated similarly to the first two.

\eqref{B3'}: We use again the form of $F_{\geq 1}$ and we have for subscript $*$ being the angular frequency localization that adds up to $\geq 1$ that:
\begingroup
\allowdisplaybreaks
\begin{align*}
\int_{\tau_1}^{\tau_2} \int_{\mcN_v^H}  & (r-M)^{-p}  (\underline{L}  \Omega^k \Phi_{\geq 1}^H )  \cdot \underline{L} ( r^2  \Omega^k F_{\geq 1} ) \, d\omega du dv \\ = &  \sum_{k_1 + k_2 = k} \int_{\tau_1}^{\tau_2} \int_{\mcN_v^H}  (r-M)^{-p} (\underline{L}  \Omega^k \Phi_{\geq 1}^H ) \cdot \frac{A_{*}}{r} \cdot ( L \Omega^{k_1} \phi_{*} ) \cdot \left( \underline{L} \left( \frac{2r}{D} \underline{L} \Omega^{k_2} \phi_{*} \right) \right) \, d\omega du dv  \\ & + \sum_{k_1 + k_2 = k} \int_{\tau_1}^{\tau_2} \int_{\mcN_v^H}  (r-M)^{-p} (\underline{L} \Omega^k  \Phi_{\geq 1}^H ) \cdot \frac{A_0}{r} \cdot  ( \underline{L} L \Omega^{k_1} \phi_{*} ) \cdot  \left( \frac{2r}{D} \underline{L} \Omega^{k_2} \phi_{*} \right)   \, d\omega du dv \\ & + \int_{\tau_1}^{\tau_2} \int_{\mcN_v^H}  (r-M)^{-p} (\underline{L}\Omega^k   \Phi_{\geq 1}^H ) \cdot \left( \frac{\underline{L}A_{*}}{r} + \frac{A_{*} \cdot D}{2r^2} \right) \cdot  \Omega^k \left[ ( L\phi ) \cdot  \left( \frac{2r}{D} \underline{L} \phi \right) \right]_{*} \, d\omega du dv  \\ & - \sum_{k_1 + k_2 = k} \int_{\tau_1}^{\tau_2} \int_{\mcN_v^H}  (r-M)^{-p} (\underline{L} \Omega^k  \Phi_{\geq 1}^H ) \cdot \frac{A_{*}}{2r^2} \cdot D ( \Omega^{k_1} \phi_{*} ) \cdot \left( \underline{L} \left( \frac{2r}{D} \underline{L} \Omega^{k_2} \phi_{*} \right) \right)  \, d\omega du dv  \\ & - \sum_{k_1 + k_2 = k} \int_{\tau_1}^{\tau_2} \int_{\mcN_v^H}  (r-M)^{-p} (\underline{L} \Omega^k  \Phi_{\geq 1}^H ) \cdot \frac{A_{*}}{2r^2} \cdot D ( \underline{L} \Omega^{k_1} \phi_{*} ) \cdot  \left( \frac{2r}{D} \underline{L} \Omega^{k_2} \phi_{*} \right)  \, d\omega du dv \\ & + \int_{\tau_1}^{\tau_2} \int_{\mcN_v^H}  (r-M)^{-p} (\underline{L}  \Omega^k \Phi_{\geq 1}^H ) \cdot \left(\frac{\underline{L}A_{*} \cdot D}{2r^2}+\frac{A_{*} \cdot D^2}{2r^3} - \frac{A_{*} \cdot D D'}{4r^2}  \right) \Omega^k \left[ \phi \cdot  \left( \frac{2r}{D} \underline{L} \phi \right) \right]_{*}  \, d\omega du dv  \\ &  - \sum_{k_1 + k_2 = k} \int_{\tau_1}^{\tau_2} \int_{\mcN_v^H}  (r-M)^{-p} (\underline{L}  \Omega^k \Phi_{\geq 1}^H ) \cdot \frac{A_{*}}{r} \cdot ( \underline{L} \Omega^{k_1} \phi_{*} ) \cdot  (L \Omega^{k_2} \phi_{*} )  \, d\omega du dv  \\ &  - \sum_{k_1 + k_2 = k} \int_{\tau_1}^{\tau_2} \int_{\mcN_v^H}  (r-M)^{-p} (\underline{L} \Omega^k  \Phi_{\geq 1}^H ) \cdot \frac{A_{*}}{r} \cdot \Omega^{k_1} \phi_{*} \cdot  ( \underline{L} L \Omega^{k_2} \phi_{*} )  \, d\omega du dv \\ &  - \int_{\tau_1}^{\tau_2} \int_{\mcN_v^H}  (r-M)^{-p} (\underline{L} \Omega^k  \Phi_{\geq 1}^H ) \cdot \left( \frac{\underline{L}A_{*}}{r} + \frac{A_{*} \cdot D}{2r^2} \right) \cdot \Omega^k \left[ \phi \cdot  (  L \phi ) \right]_{*}  \, d\omega du dv   \\ &  - \sum_{k_1 + k_2 = k} \int_{\tau_1}^{\tau_2} \int_{\mcN_v^H}  (r-M)^{-p} (\underline{L}  \Omega^k \Phi_{\geq 1}^H ) \cdot \frac{A_{*}}{r^2} \cdot 2D  \Omega^{k_1} \phi_{*} \cdot  (  \underline{L} \Omega^{k_2} \phi_{*} )   \, d\omega du dv \\ &  - \int_{\tau_1}^{\tau_2} \int_{\mcN_v^H}  (r-M)^{-p} (\underline{L}  \Omega^k \Phi_{\geq 1}^H ) \cdot \left( \frac{\underline{L}A_{*} \cdot D}{2r^2} - \frac{A_{*} \cdot D D'}{4r^2} + \frac{A_{*} \cdot D^2}{2r^3} \right) \cdot \Omega^{k} \left[ \phi^2 \right]_{*} \, d\omega du dv  \\ & + \sum_{k_1 + k_2 = k} \int_{\tau_1}^{\tau_2} \int_{\mcN_v^H}  (r-M)^{-p} (\underline{L} \Omega^k  \Phi_{\geq 1}^H ) \cdot A_{*} \cdot \langle \underline{L} \slashed{\nabla} \Omega^{k_1} \phi_{*} , \slashed{\nabla} \Omega^{k_2} \phi_{*} \rangle \, d\omega du dv \\ & + \sum_{k_1 + k_2 = k} \int_{\tau_1}^{\tau_2} \int_{\mcN_v^H}  (r-M)^{-p} (\underline{L} \Omega^k  \Phi_{\geq 1}^H ) \cdot \underline{L} A_{*} \cdot \langle  \slashed{\nabla} \Omega^{k_1} \phi_{*} , \slashed{\nabla} \Omega^{k_2} \phi_{*} \rangle \, d\omega du dv  \\  & + \sum_{\substack{k_1 + k_2 = k \\ k_1 \geq 1}} \int_{\tau_1}^{\tau_2} \int_{\mcN_v^H}  (r-M)^{-p} (\underline{L} \Omega^k  \Phi_{\geq 1}^H ) \cdot \underline{L} \left( r^2 \Omega^k [  r^2 (\Omega^{k_1} A_{*} ) ( \Omega^{k_2} F^c_{*} ) ]  \right) \, d\omega du dv .
\end{align*}
\endgroup

The worst terms from the above are the ones that include $\underline{L} \left( \frac{2r}{D} \underline{L} \phi_0 \right)$, they come from the first term and the last term of the last expression, and their form is the following:
\begin{equation}\label{term:b4bad}
\sum_{k_1 + k_2 = k} \int_{\tau_1}^{\tau_2} \int_{\mcN_v^H}  (r-M)^{-p}  (\underline{L}  \Omega^k \Phi_{\geq 1}^H ) \cdot ( \Omega^{k_1} A_{*} ) \cdot ( L \Omega^{k_2} \phi_{*} ) \cdot \left( \underline{L} \left( \frac{2r}{D} \underline{L}\phi_0 \right) \right) \, d\omega du dv .
\end{equation}
We have for $p = 1+\delta_2$ that:
\begingroup
\allowdisplaybreaks
\begin{align*}
\sum_{k_1 + k_2 = k} \int_{\tau_1}^{\tau_2} \int_{\mcN_v^H}  (r-M)^{-1-\delta_2} & (\underline{L}  \Omega^k \Phi_{\geq 1}^H ) \cdot ( \Omega^{k_1} A_{*} ) \cdot ( L \Omega^{k_2} \phi_{*} ) \cdot \left( \underline{L} \left( \frac{2r}{D} \underline{L}\phi_0 \right) \right) \, d\omega du dv \\ \lesssim & \int_{\tau_1}^{\tau_2} \int_{\mcN_v^H}  (r-M)^{-1-\delta_2}  (\underline{L}  \Omega^k \Phi_{\geq 1}^H )^2 \cdot \frac{1}{v^{1+\beta}} \, d\omega du dv \\ & + \sum_{m \leq 5} \int_{\tau_1}^{\tau_2} \int_{\mcN_v^H}  (r-M)^{-1-\delta_2} ( L \Omega^{m} \phi_{*} )^2 \cdot \left( \underline{L} \left( \frac{2r}{D} \underline{L}\phi_0 \right) \right)^2 \cdot v^{1+\beta} \, d\omega du dv \\ \lesssim &  \frac{1}{\tau_1^{\beta}} \sup_{v \in [\tau_1 , \tau_2 ]} \int_{\mcN_v^H}  (r-M)^{-1-\delta_2}  (\underline{L}  \Omega^k \Phi_{\geq 1}^H )^2  \, d\omega du  \\ & + \sum_{m \leq 5}\int_{\tau_1}^{\tau_2} \int_{\mcN_v^H}  (r-M)^{1+\beta_1 - \delta_2} ( L \Omega^{m} \phi_{*} )^2 \cdot D^{1-\beta_1 /2}\left( \underline{L} \left( \frac{2r}{D} \underline{L}\phi_0 \right) \right)^2 \cdot v^{1+\beta} \, d\omega du dv \\ \lesssim &  \frac{1}{\tau_1^{\beta}} \sup_{v \in [\tau_1 , \tau_2 ]} \int_{\mcN_v^H}  (r-M)^{-1-\delta_2}  (\underline{L}  \Omega^k \Phi_{\geq 1}^H )^2  \, d\omega du  \\ & + C E_0 \epsilon^2 \sum_{m \leq 5} \int_{\tau_1}^{\tau_2} \int_{\mcN_v^H}  (r-M)^{1+\beta_1 - \delta_2} ( L \Omega^{m} \phi_{*} )^2  \cdot v^{1+\beta + \beta_1} \, d\omega du dv \\ \lesssim &  \frac{1}{\tau_1^{\beta}} \sup_{v \in [\tau_1 , \tau_2 ]} \int_{\mcN_v^H}  (r-M)^{-1-\delta_2}  (\underline{L}  \Omega^k \Phi_{\geq 1}^H )^2  \, d\omega du  \\ & + C E_0 \epsilon^2 \int_{\tau_1}^{\tau_2} \int_{\mathbb{S}^2}  \sup_u ( L \Omega^{k} \phi_{\geq 1} )^2  \cdot v^{1+\beta + \beta_1} \, d\omega dv \\ \lesssim &  \frac{1}{\tau_1^{\beta}} \sup_{v \in [\tau_1 , \tau_2 ]} \int_{\mcN_v^H}  (r-M)^{-1-\delta_2}  (\underline{L}  \Omega^k \Phi_{\geq 1}^H )^2  \, d\omega du   + C^2 E_0^2 \epsilon^4 ,
\end{align*} 
\endgroup
where first we used the growth estimate \eqref{est:boundyy}, we used the auxiliary estimate \eqref{est:aux1} for $m=0$, always choosing $\beta_1$ to be bigger than $\delta_2$ and $\beta$ and $\beta_1$ to be small enough such that $\beta + \beta_1 < 2 -\delta_1$ (note that these two conditions can be simultaneously satisfied due to the smallness of $\delta_1$ and $\delta_2$). On the other hand we consider the (out of range) case of $p=0$ and we have that:
\begin{equation*}
\begin{split}
\sum_{k_1 + k_2 = k}\int_{\tau_1}^{\tau_2} \int_{\mcN_v^H} & (\underline{L}  \Omega^k \Phi_{\geq 1}^H ) \cdot ( \Omega^{k_1} A_{*} ) \cdot ( L \Omega^{k_2} \phi_{*} ) \cdot \left( \underline{L} \left( \frac{2r}{D} \underline{L}\phi_0 \right) \right) \, d\omega du dv \\ \lesssim & \beta \int_{\tau_1}^{\tau_2} \int_{\mcN_v^H}  (r-M) (\underline{L}  \Omega^k \Phi_{\geq 1}^H )^2 \, d\omega du dv \\ & + \frac{1}{\beta} \sum_{m \leq 5} \int_{\tau_1}^{\tau_2} \int_{\mcN_v^H} (r-M)^{-1} ( L \Omega^{m} \phi_{*} )^2 \cdot \left( \underline{L} \left( \frac{2r}{D} \underline{L}\phi_0 \right) \right)^2 \, d\omega du dv .
\end{split}
\end{equation*}
The first term of the above expression can be absorbed by the left hand side that has the term of interest in \eqref{B3'} in its right hand side for any $p \in (0,1+\delta_2]$. For the second term we have for any $v \in [\tau_1 , \tau_2 ] $ that:
\begin{equation*}
\begin{split}
\int_{\mcN_v^H} (r-M)^{-1} & ( L \Omega^{k} \phi_{\geq 1} )^2 \cdot \left( \underline{L} \left( \frac{2r}{D} \underline{L}\phi_0 \right) \right)^2 \, d\omega du \\ \lesssim & \int_{\mcN_v^H} (r-M)^2  ( L \Omega^{k} \phi_{\geq 1} )^2 \cdot D^{1/2} \left( \frac{2r}{D} \underline{L} \left( \frac{2r}{D} \underline{L}\phi_0 \right) \right)^2 \, d\omega du \\ \lesssim & C E_0 \epsilon^2 v^{\delta_3} \frac{C E_0 \epsilon^2}{v^{3-\delta_1}} \lesssim \frac{C^2 E_0^2 \epsilon^4}{v^{3-\delta_1-\delta_3}} , 
\end{split}
\end{equation*}
where we used the growth estimate \eqref{est:boundyy} and the Morawetz decay estimate \eqref{dec:mor}, and for $\delta_3 > 0$ small enough we now note that using the last estimate we have that:
\begin{equation*}
\int_{\tau_1}^{\tau_2} \int_{\mcN_v^H} (r-M)^{-1}  ( L \Omega^{k} \phi_{\geq 1} )^2 \cdot \left( \underline{L} \left( \frac{2r}{D} \underline{L}\phi_0 \right) \right)^2 \, d\omega du dv  \lesssim \frac{C^2 E_0^2 \epsilon^4}{(1+\tau_1 )^{2-\delta_1 - \delta_3}} .
\end{equation*}
The required estimate for the term \eqref{term:b4bad} now follows by the estimate that was shown for $p=1+\delta_2$ (which is sharp), the (artificial) $p=0$ estimate (which is better than required as $\delta_1$ and $\delta_3$ are small enough so that $2-\delta_1 - \delta_3 > 1+\delta_2$), and a standard interpolation argument. The rest of the terms are of similar difficulty or even easier and we will not examine them.

\eqref{B5'}: We have for any $k \leq 5$ using the form of $F_{\geq 1}$ and using subscript $*$ to denote the angular frequency localization that adds up to $\geq 1$ that:
\begingroup
\allowdisplaybreaks
\begin{align*}
\int_{\tau_1}^{\tau_2} \int_{\mcN_u^I} r^p ( L \Omega^k \Phi_{\geq 1}^I ) & \cdot ( L ( r^3 \Omega^k F_{\geq 1} ) \, d\omega dv du \simeq_{A_{*}}  \sum_{k_1 + k_2 = k} \int_{\tau_1}^{\tau_2} \int_{\mcN_u^I} r^p ( L \Omega^k \Phi_{\geq 1}^I ) \cdot \frac{2}{r} ( \underline{L} \Omega^{k_1} \phi_{*} ) \cdot \left[ L \left( \frac{2r^2}{D} L \Omega^{k_2} \phi_{*} \right) \right]  \, d\omega dv du \\ & + \sum_{k_1 + k_2 = k} \int_{\tau_1}^{\tau_2} \int_{\mcN_u^I} r^p ( L \Omega^k \Phi_{\geq 1}^I ) \cdot \frac{4r}{D} ( L \Omega^{k_1} \phi_{*} ) \cdot ( L \underline{L} \Omega^{k_2} \phi_{*} ) \, d\omega dv du \\ & - \sum_{k_1 + k_2 = k} \int_{\tau_1}^{\tau_2} \int_{\mcN_u^I} r^p ( L \Omega^k \Phi_{\geq 1}^I ) \cdot 2 ( L \Omega^{k_1} \phi_{*} ) \cdot  ( \underline{L} \Omega^{k_2} \phi_{*}  ) \, d\omega dv du \\ & - \sum_{k_1 + k_2 = k} \int_{\tau_1}^{\tau_2} \int_{\mcN_u^I} r^p ( L \Omega^k \Phi_{\geq 1}^I ) \cdot 2 ( \Omega^{k_1} \phi_{*} ) \cdot  ( L \underline{L} \Omega^{k_2} \phi_{*} ) \, d\omega dv du  \\ & + \sum_{k_1 + k_2 = k} \int_{\tau_1}^{\tau_2} \int_{\mcN_u^I} r^p ( L \Omega^k \Phi_{\geq 1}^I ) \cdot 2 ( L \Omega^{k_1} \phi_{*} ) \cdot ( L  \Omega^{k_2} \phi_{*} ) \, d\omega dv du \\ & + \sum_{k_1 + k_2 = k} \int_{\tau_1}^{\tau_2} \int_{\mcN_u^I} r^p ( L \Omega^k \Phi_{\geq 1}^I ) \cdot \frac{D}{r^2} ( \Omega^{k_1} \phi_{*} ) \cdot  \left[ L \left( \frac{2r^2}{D} L \Omega^{k_2} \phi_{*} \right) \right] \, d\omega dv du \\ & + \int_{\tau_1}^{\tau_2} \int_{\mcN_u^I} r^p ( L \Omega^k \Phi_{\geq 1}^I ) \cdot \mathcal{O}(r^{-1} ) \Omega^{k} ( \phi \cdot  L  \phi )_{*}  \, d\omega dv du \\ & -  \int_{\tau_1}^{\tau_2} \int_{\mcN_u^I} r^p ( L \Omega^k \Phi_{\geq 1}^I ) \cdot \mathcal{O} (r^{-2} ) \Omega^{k} ( \phi^2 )_{*}   \, d\omega dv du \\ & - \sum_{k_1 + k_2 = k} \int_{\tau_1}^{\tau_2} \int_{\mcN_u^I} r^p ( L \Omega^k \Phi_{\geq 1}^I ) \cdot \frac{D}{r} ( \Omega^{k_1} \phi_{*} ) \cdot  ( L \Omega^{k_2} \phi_{*} )  \, d\omega dv du \\ & + \sum_{k_1 + k_2 = k} \int_{\tau_1}^{\tau_2} \int_{\mcN_u^I} r^p ( L \Omega^k \Phi_{\geq 1}^I ) \cdot \mathcal{O} (r) \langle \Omega^{k_1} \slashed{\nabla} \phi_{*} , \slashed{\nabla} L \Omega^{k_2} \phi_{*} \rangle \, d\omega dv du\\ & + \sum_{k_1 + k_2 = k} \int_{\tau_1}^{\tau_2} \int_{\mcN_u^I} r^p ( L \Omega^k \Phi_{\geq 1}^I ) \cdot \frac{D}{2} \langle \Omega^{k_1} \slashed{\nabla} \phi_{*} , \slashed{\nabla} \Omega^{k_2} \phi_{*} \rangle \, d\omega dv du \\ & + \sum_{\substack{k_1 + k_2 = k \\ k_1 \geq 1}}\int_{\tau_1}^{\tau_2} \int_{\mcN_u^I} r^p ( L \Omega^k \Phi_{\geq 1}^I ) \cdot [ ( L ( r^3 ( \Omega^{k_1} A_{*} ) ( \Omega^{k_2} F^c_{*} )  ] \, d\omega dv du .
\end{align*}
\endgroup
From the first term we consider the terms that include $L \left( \frac{2r^2}{D} L \phi_0 \right)$ since the range of $p$ in the commuted estimates for $\phi_0$ at infinity is smaller and we have that
\begingroup
\allowdisplaybreaks
\begin{align*}
\sum_{k_1 + k_2 = k} \int_{\tau_1}^{\tau_2} & \int_{\mcN_u^I} r^p ( L \Omega^k \Phi_{\geq 1}^I )  \cdot \frac{(\Omega^{k_1} A_{*} )}{r} ( \underline{L} \Omega^{k_2} \phi_{*} ) \cdot \left[ L \left( \frac{2r^2}{D} L \phi_0 \right) \right]  \, d\omega dv du \\ \lesssim & \beta  \int_{\tau_1}^{\tau_2} \int_{\mcN_u^I} r^{p-1} ( L \Omega^k \Phi_{\geq 1}^I )^2 \, d\omega dv du \\ & + \frac{1}{\beta} \sum_{m \leq 5}\int_{\tau_1}^{\tau_2} \int_{\mcN_u^I} r^{p+1} \frac{4}{r^2} ( \underline{L} \Omega^m \phi_{*} )^2 \cdot \left( L \left( \frac{2r^2}{D} L \phi_0 \right) \right)^2  \, d\omega dv du \\ \lesssim & \beta  \int_{\tau_1}^{\tau_2} \int_{\mcN_u^I} r^{p-1} ( L \Omega^k \Phi_{\geq 1}^I )^2 \, d\omega dv du \\ & + \frac{1}{\beta} \sum_{m \leq 5}\int_{\tau_1}^{\tau_2} \int_{\mathbb{S}^2} \sup_v ( \underline{L} \Omega^m \phi_{*} )^2 \, d\omega du \cdot \sup_{u \in [\tau_1 , \tau_2 ]} \int_{\mcN_u^I} r^{p-1} \left( L \left( \frac{2r^2}{D} L \phi_0 \right) \right)^2  \, dv  \\ \lesssim & \beta  \int_{\tau_1}^{\tau_2} \int_{\mcN_u^I} r^{p-1} ( L \Omega^k \Phi_{\geq 1}^I )^2 \, d\omega dv du \\ & +\frac{C E_0 \epsilon^2}{\tau_1^{1-\delta_1 - \delta'}}  \sup_{u \in [\tau_1 , \tau_2 ]} \int_{\mcN_u^I} r^{p-1} \left( L \left( \frac{2r^2}{D} L \phi_0 \right) \right)^2  \, dv \\ \lesssim & \beta  \int_{\tau_1}^{\tau_2} \int_{\mcN_u^I} r^{p-1} ( L \Omega^k \Phi_{\geq 1}^I )^2 \, d\omega dv du  +\frac{C E_0 \epsilon^2}{\tau_1^{1-\delta_1 - \delta'}} \frac{C E_0 \epsilon^2}{\tau_1^{\min ( 2-\delta_1-p , 1-\delta_1 )}} , 
\end{align*}
\endgroup
for any $\delta' > 0$, so choosing $\delta'$ small enough gives us better decay than required due to the smallness of $\delta_1$. We used that
\begin{equation}\label{est:infaux}
\sum_{k \leq 5} \int_{\mathbb{S}^2} ( \underline{L} \Omega^k \phi )^2 \, d\omega \lesssim \frac{C E_0 \epsilon^2}{u^{2-\delta_1 - \delta'}} \mbox{  for any $(u,v) \in \mathcal{B}_{\tau_0}^{\infty} / SO(2)$ , } 
\end{equation}
which follows from the pointwise decay estimates \eqref{dec:tpsi} and decay for $L \phi$ which can be obtained by the fundamental theorem of calculus and the energy decay estimates \eqref{dec:enp2} and \eqref{dec:enp3}. Note that it was crucial to use the extra $\frac{1}{r^2}$ term so that for $p \in (0,1+\delta_2] \Rightarrow p-1 < 1-\delta_1$ due to the smallness of $\delta_1$ and $\delta_2$, in order to apply estimate \eqref{dec:en0p1} since the spherically symmetric part has a smaller $p$ range in the commuted estimates than the non-spherically part.

For the second term we write the $\underline{L}$ derivative as a sum of the $T$ and $L$ derivatives and we have that:
\begingroup
\allowdisplaybreaks
\begin{align*}
\sum_{k_1 + k_2 = k} \int_{\tau_1}^{\tau_2} & \int_{\mcN_u^I} r^p ( L \Omega^k \Phi_{\geq 1}^I )  \cdot\frac{4r}{D} ( L \Omega^{k_1}  \phi_{*} ) \cdot ( L \underline{L} \Omega^{k_2} \phi_{*} ) \, d\omega dv du \\ \lesssim & \beta \int_{\tau_1}^{\tau_2} \int_{\mcN_u^I} r^{p-1} ( L \Omega^k \Phi_{\geq 1}^I )^2 \, d\omega dv du \\ & + \frac{1}{\beta} \sum_{k_1 + k_2 = k} \int_{\tau_1}^{\tau_2} \int_{\mcN_u^I} r^{p+3} ( L \Omega^{k_1} \phi_{*} )^2 \cdot ( L \underline{L} \Omega^{k_2} \phi_{*} )^2 \, d\omega dv du \\ \lesssim & \beta \int_{\tau_1}^{\tau_2} \int_{\mcN_u^I} r^{p-1} ( L \Omega^k \Phi_{\geq 1}^I )^2 \, d\omega dv du \\ & + \frac{1}{\beta} \sum_{k_1 + k_2 = k} \int_{\tau_1}^{\tau_2} \int_{\mcN_u^I} r^{p+3} ( L \Omega^{k_1} \phi_{*} )^2 \cdot ( L T \Omega^{k_2} \phi_{*} )^2 \, d\omega dv du \\ & + \frac{1}{\beta} \sum_{k_1 + k_2 = k}\int_{\tau_1}^{\tau_2} \int_{\mcN_u^I} r^{p-1} ( L \Omega^{k_1}\phi_{*} )^2 \cdot \left( L \left( \frac{2r^2}{D}  \Omega^{k_2} L\phi_{*} \right) \right)^2 \, d\omega dv du \\ & + \frac{1}{\beta} \sum_{k_1 + k_2 = k} \int_{\tau_1}^{\tau_2} \int_{\mcN_u^I} r^{p+1} ( L \Omega^{k_1} \phi_{*} )^2 \cdot ( L \Omega^{k_2} \phi_{*} )^2 \, d\omega dv du \\ \lesssim & \beta \int_{\tau_1}^{\tau_2} \int_{\mcN_u^I} r^{p-1} ( L \Omega^k \Phi_{\geq 1}^I )^2 \, d\omega dv du \\ & + \frac{1}{\beta} \sum_{m_1 \leq 5}\left\| \int_{\mathbb{S}^2} ( r^2 L \Omega^{m_1} \phi_{*} )^2 \, d\omega \right\|_{L^{\infty} (\mathcal{B}_{\tau_1}^{\tau_2} / SO(3)} \cdot \sum_{m_2 \leq 5} \int_{\tau_1}^{\tau_2}  \int_{\mcN_u^I} r^{p-1}   ( L T \Omega^{m_2} \phi_{*} )^2 \, d\omega dv du \\ & + \frac{1}{\beta} \sum_{m_1 \leq 5}\int_{\tau_1}^{\tau_2} \int_{\mathbb{S}^2} \sup_v (L \Omega^{m_1} \phi_{*} )^2 \, d\omega du \cdot \sum_{m_2 \leq 5} \sup_{u \in [\tau_1 , \tau_2 ]} \int_{\mcN_u^I} r^{p-1} \left( L \left( \frac{2r^2}{D}  \Omega^{m_2} L\phi_{*} \right) \right)^2 \, d\omega dv du \\ & + \frac{1}{\beta} \sum_{m_1 \leq 5} \int_{\tau_1}^{\tau_2} \int_{\mathbb{S}^2} \sup_v (L \Omega^{m_1} \phi_{*} )^2 \, d\omega du \cdot \sum_{m_2 \leq 5} \sup_{u \in [\tau_1 , \tau_2 ]} \int_{\mcN_u^I} r^{p-1} ( L \Omega^{m_2} \phi_{*} )^2 \, d\omega dv du \\ \lesssim & \beta \int_{\tau_1}^{\tau_2} \int_{\mcN_u^I} r^{p-1} ( L \Omega^k \Phi_{\geq 1}^I )^2 \, d\omega dv du \\ & + C E_0 \epsilon^2 \frac{C E_0 \epsilon^2}{\tau_1^{3-\delta_1 - p}} + \frac{C E_0 \epsilon^2}{\tau_1^{1-\delta_1 - \delta' }} \frac{C E_0 \epsilon^2}{\tau_1^{1-\delta_1 - p}} + \frac{C E_0 \epsilon^2}{\tau_1^{1-\delta_1 - \delta' }} \frac{C E_0 \epsilon^2}{\tau_1^{2-\delta_1 - p}} , 
\end{align*}
\endgroup
for any $\delta' > 0$. We used that $ r^2 L \phi$ is bounded (by the method of characteristics in the spherically symmetric case and by the energy estimates \eqref{dec:en1p1} in the non-spherically symmetric case), the estimate \eqref{est:infaux}, and the energy decay estimates \eqref{dec:enp3} and \eqref{dec:entp}. The fourth term can be treated similarly.

For the third term we have that
\begingroup
\allowdisplaybreaks
\begin{align*}
\sum_{k_1 + k_2 = k}\int_{\tau_1}^{\tau_2} & \int_{\mcN_u^I} r^p ( L \Omega^k \Phi_{\geq 1}^I )  \cdot 2 ( L \Omega^{k_1} \phi_{*} ) \cdot ( \underline{L} \Omega^{k_2} \phi_{*} ) \, d\omega dv du \\ \lesssim & \beta \int_{\tau_1}^{\tau_2} \int_{\mcN_u^I} r^{p-1} ( L \Omega^k \Phi_{\geq 1}^I )^2 \, d\omega dv du \\ & + \frac{1}{\beta} \sum_{k_1 + k_2 = k} \int_{\tau_1}^{\tau_2} \int_{\mcN_u^I} r^{p+1}  ( L \Omega^{k_1} \phi_{*} )^2 \cdot ( \underline{L} \Omega^{k_2} \phi_{*} )^2 \, d\omega dv du\\ \lesssim & \beta \int_{\tau_1}^{\tau_2} \int_{\mcN_u^I} r^{p-1} ( L \Omega^k \Phi_{\geq 1}^I )^2 \, d\omega dv du \\ & + \frac{1}{\beta} \sum_{m_1 \leq 5}\left\| \int_{\mathbb{S}^2} ( r^2 L \Omega^{m_1} \phi_{*} )^2 \, d\omega \right\|_{L^{\infty} (\mathcal{B}_{\tau_1}^{\tau_2} / SO(2)}\cdot \sum_{m_2 \leq 5} \int_{\tau_1}^{\tau_2} \int_{\mcN_u^I} r^{p-3}   ( \underline{L} \Omega^{m_2} \phi_{*} )^2 \, d\omega dv du \\ \lesssim & C E_0 \epsilon^2 \frac{C E_0 \epsilon^2}{\tau_1^{3 - \delta_1}} , 
\end{align*}
\endgroup
where we used again the boundedness of the $r^2 L \phi$ and the Morawetz decay estimate \eqref{dec:mor11} as $p-3 \leq -2+\delta_2 < -1-\eta$ for some $\eta > 0$ as $\delta_2$ is chosen to be small enough. Terms five to nine can be treated similarly to the above. For the tenth term we have that
\begingroup
\allowdisplaybreaks
\begin{align*}
\sum_{k_1 + k_2 = k} \int_{\tau_1}^{\tau_2} & \int_{\mcN_u^I} r^p ( L \Omega^k \Phi_{\geq 1}^I ) \cdot  \mathcal{O} (r) \langle \Omega^{k_1} \slashed{\nabla} \phi_{*} \cdot  \slashed{\nabla} L \Omega^{k_2} \phi_{*} \rangle \, d\omega dv du \\ \lesssim & \beta \int_{\tau_1}^{\tau_2} \int_{\mcN_u^I} r^{p-1} ( L \Omega^k \Phi_{\geq 1}^I )^2 \, d\omega dv du \\ & + \frac{1}{\beta} \sum_{k_1 + k_2 = k} \int_{\tau_1}^{\tau_2} \int_{\mcN_u^I} r^{p+3} | \Omega^{k_1} \slashed{\nabla} \phi_{*} |^2 \cdot  \left| \slashed{\nabla} \frac{2r^2}{D} L \Omega^{k_2} \phi_{*} \right|^2 \, d\omega dv du \\ \lesssim & \beta \int_{\tau_1}^{\tau_2} \int_{\mcN_u^I} r^{p-1} ( L \Omega^k \Phi_{\geq 1}^I )^2 \, d\omega dv du \\ & + \frac{1}{\beta} \sum_{\substack{k_1 + k_2 = k \\ k_1 \leq 2 } } \| \Omega^{k_1 + 1} \phi_{\gtrsim 1} \|^2_{L^{\infty} (\mathcal{B}_{\tau_1}^{\tau_2})}  \int_{\tau_1}^{\tau_2} \int_{\mcN_u^I} r^{p-3}   \left| \slashed{\nabla} \frac{2r^2}{D} L \Omega^{k_2} \phi_{*} \right|^2 \, d\omega dv du \\ & + \frac{1}{\beta} \sum_{\substack{k_1 + k_2 = k \\ k_1 \leq 2 } } \| r^2 L \Omega^{k_1 + 1} \phi_{*} \|^2_{L^{\infty} (\mathcal{B}_{\tau_1}^{\tau_2})}  \int_{\tau_1}^{\tau_2} \int_{\mcN_u^I} r^{p-3}   | \slashed{\nabla}  \Omega^{k_2} \phi_{*} |^2 \, d\omega dv du \\ \lesssim & \beta \int_{\tau_1}^{\tau_2} \int_{\mcN_u^I} r^{p-1} ( L \Omega^k \Phi_{\geq 1}^I )^2 \, d\omega dv du + \frac{C E_0 \epsilon^2}{\tau_1^{2+\delta_2}} \frac{C E_0 \epsilon^2}{\tau_1^{1+\delta_2 -p}} + C E_0 \epsilon^2 \frac{C E_0 \epsilon^2}{\tau_1^{3+\delta_2 -p}} ,
\end{align*}
\endgroup
where we used the pointwise decay estimate \eqref{dec:psi}, the boundedness of $r^2 L \phi$, and the energy decay estimates \eqref{dec:en1p2} and \eqref{dec:en1p3}. Term eleven can be treated in a similar way, while the last term twelve can be expanded and can be easily seen to comprise of terms similar to the above. 

\eqref{C2'}: We have for any $k \leq 5$ that 
\begingroup
\allowdisplaybreaks
\begin{align*}
\int_{\tau_1}^{\tau_2} \int_{\mcN_v^H} & (r-M )^{-p} ( \underline{L} \Omega^k T \phi ) \cdot D ( \Omega^k T F ) \, d\omega du dv \\ = & \sum_{k_1 + k_2 = k} \int_{\tau_1}^{\tau_2} \int_{\mcN_v^H}  (r-M )^{-p} ( \underline{L} \Omega^k T \phi ) \cdot \frac{A}{2r^2} \cdot D ( L T \Omega^{k_1} \phi ) \cdot \left( \frac{2r}{D} \underline{L} \Omega^{k_2} \phi \right) \, d\omega du dv \\ & + \sum_{k_1 + k_2 = k} \int_{\tau_1}^{\tau_2} \int_{\mcN_v^H}  (r-M )^{-p} ( \underline{L} \Omega^k T \phi ) \cdot \frac{A}{2r^2} \cdot D  ( L \Omega^{k_1} \phi ) \cdot \left( \frac{2r}{D} \underline{L} \Omega^{k_2} T \phi \right) \, d\omega du dv \\ & - \sum_{k_1 + k_2 = k} \int_{\tau_1}^{\tau_2} \int_{\mcN_v^H}  (r-M )^{-p} ( \underline{L} \Omega^k T \phi ) \cdot \frac{A}{4r^3} \cdot D^2 ( \Omega^{k_1} T \phi ) \cdot \left( \frac{2r}{D} \underline{L} \Omega^{k_2}  \phi \right) \, d\omega du dv \\ & - \sum_{k_1 + k_2 = k} \int_{\tau_1}^{\tau_2} \int_{\mcN_v^H}  (r-M )^{-p} ( \underline{L} \Omega^k T \phi ) \cdot \frac{A}{4r^3} \cdot D^2 ( \Omega^{k_1}  \phi ) \cdot \left( \frac{2r}{D} \underline{L} \Omega^{k_2}  T \phi \right) \, d\omega du dv \\ & - \sum_{k_1 + k_2 = k} \int_{\tau_1}^{\tau_2} \int_{\mcN_v^H}  (r-M )^{-p} ( \underline{L} \Omega^k T \phi ) \cdot \frac{A}{2r^2} \cdot D ( \Omega^{k_1} T \phi ) \cdot ( L \Omega^{k_2}  \phi ) \, d\omega du dv \\ & - \sum_{k_1 + k_2 = k} \int_{\tau_1}^{\tau_2} \int_{\mcN_v^H}  (r-M )^{-p} ( \underline{L} \Omega^k T \phi ) \cdot \frac{A}{2r^2} \cdot D ( \Omega^{k_1}  \phi ) \cdot ( L \Omega^{k_2} T \phi ) \, d\omega du dv \\ & - \sum_{k_1 + k_2 = k} \int_{\tau_1}^{\tau_2} \int_{\mcN_v^H}  (r-M )^{-p} ( \underline{L} \Omega^k T \phi ) \cdot \frac{A}{4r^3} \cdot 2D^2 ( \Omega^{k_1} T \phi ) \cdot (  \Omega^{k_2}  \phi ) \, d\omega du dv \\ & + \sum_{k_1 + k_2 = k} \int_{\tau_1}^{\tau_2} \int_{\mcN_v^H}  (r-M )^{-p} ( \underline{L} \Omega^k T \phi ) \cdot \frac{A}{4r} \cdot D \langle \slashed{\nabla} \Omega^{k_1} T \phi , \slashed{\nabla} \Omega^{k_2} \phi \rangle \, d\omega du dv \\ & + \sum_{k_1 + k_2 = k , k_1 > 0} \int_{\tau_1}^{\tau_2} \int_{\mcN_v^H}  (r-M )^{-p} ( \underline{L} \Omega^k T \phi ) \cdot  ( \Omega^{k_1} A ) \cdot ( \Omega^{k_2} T F^c ) \, d\omega du dv  \\ & + \sum_{k_1 + k_2 = k} \int_{\tau_1}^{\tau_2} \int_{\mcN_v^H}  (r-M )^{-p} ( \underline{L} \Omega^k T \phi ) \cdot  ( \Omega^{k_1} T A ) \cdot ( \Omega^{k_2} F^c ) \, d\omega du dv . 
\end{align*}
\endgroup

For the first term we have for any $k \leq 5$ that:
\begingroup
\allowdisplaybreaks
\begin{align*}
\sum_{k_1 + k_2 = k} \int_{\tau_1}^{\tau_2} \int_{\mcN_v^H} & (r-M )^{-p} ( \underline{L} \Omega^k T \phi ) \cdot \frac{A}{2r^2} \cdot D ( L T \Omega^{k_1} \phi ) \cdot \left( \frac{2r}{D} \underline{L} \Omega^{k_2} \phi \right) \, d\omega du dv \\ \lesssim &  \beta \int_{\tau_1}^{\tau_2} \int_{\mcN_v^H}  (r-M )^{-p+1} ( \underline{L} \Omega^k T \phi )^2 \, d\omega du dv \\ & + \frac{1}{\beta} \sum_{k_1 + k_2 = k} \int_{\tau_1}^{\tau_2} \int_{\mcN_v^H}  (r-M )^{-p-1} D^2 ( L T \Omega^{k_1} \phi )^2 \cdot \left( \frac{2r}{D} \underline{L} \Omega^{k_2} \phi \right)^2  \, d\omega du dv  \\ \lesssim &  \beta \int_{\tau_1}^{\tau_2} \int_{\mcN_v^H}  (r-M )^{-p+1} ( \underline{L} \Omega^k T \phi )^2  \, d\omega du dv \\ & + C E_0 \epsilon^2 \sum_{l \leq 5} \int_{\tau_1}^{\tau_2} \int_{\mathbb{S}^2} (r-M)^{3-p} ( L T \Omega^l \phi )^2 \, d\omega du dv   \\ \lesssim &  \beta \int_{\tau_1}^{\tau_2} \int_{\mcN_v^H}  (r-M )^{-p+1} ( \underline{L} \Omega^k T \phi )^2  \, d\omega du dv  + C E_0 \epsilon^2 \frac{C E_0 \epsilon^2}{\tau_1^{3+\delta_2}} ,
\end{align*}
\endgroup
where we used the boundedness estimate \eqref{est:boundy} and the Morawetz decay estimates \eqref{dec:mor1} as $p\in (0,2-\delta_1 ]$. For the second term we have that:
\begingroup
\allowdisplaybreaks
\begin{align*}
\sum_{k_1 + k_2 = k} \int_{\tau_1}^{\tau_2} \int_{\mcN_v^H} & (r-M )^{-p} ( \underline{L} \Omega^k T \phi ) \cdot \frac{A}{2r^2} \cdot D ( L  \Omega^{k_1} \phi ) \cdot \left( \frac{2r}{D} \underline{L} \Omega^{k_2} T \phi \right) \, d\omega du dv \\ \lesssim &  \int_{\tau_1}^{\tau_2} \int_{\mcN_v^H}  (r-M )^{-p+1} ( \underline{L} \Omega^k T \phi )^2 \cdot \frac{1}{v^{1+\beta}} \, d\omega du dv \\ & + \sum_{k_1 + k_2 = k} \int_{\tau_1}^{\tau_2} \int_{\mcN_v^H}  (r-M )^{-p} D^2 ( L  \Omega^{k_1} \phi )^2 \cdot \left( \frac{2r}{D} \underline{L} \Omega^{k_2} T \phi \right)^2 \cdot v^{1+\beta} \, d\omega du dv  \\ \lesssim &  \int_{\tau_1}^{\tau_2} \int_{\mcN_v^H}  (r-M )^{-p} ( \underline{L} \Omega^k T \phi )^2 \cdot \frac{1}{v^{1+\beta}} \, d\omega du dv \\ & + \sum_{k_1 + k_2 = k} \int_{\tau_1}^{\tau_2} \int_{\mathbb{S}^2}  \sup_u ( L  \Omega^{l_1} \phi )^2 \cdot v^{1+\beta} \, d\omega dv \cdot \sup_{v \in [\tau_1 , \tau_2 ]} \int_{\mcN_v^H}  (r-M )^{-p} ( \underline{L} \Omega^{l_2} T \phi )^2  \, d\omega du \\ \lesssim &  \frac{1}{\tau_1^{\beta}} \int_{\mcN_v^H}  (r-M )^{-p} ( \underline{L} \Omega^k T \phi )^2 \cdot \frac{1}{v^{1+\beta}} \, d\omega du  + C E_0 \epsilon^2 \frac{C E_0 \epsilon^2}{(1+\tau_1 )^{3+\delta_2 -p}} ,
\end{align*}
\endgroup
where we used Sobolev's inequality \eqref{est:sobolev}, the auxiliary estimate \eqref{est:aux1} for $m=0$ and the energy decay estimates \eqref{dec:entp}.

The rest of the terms can be treated now as in the case of \eqref{A2'} by using the energy decay estimates \eqref{dec:entp}.

\eqref{C1'}: We deal with the term close to the horizon as the term away from the horizon can be treated in a rather classical manner. As in \eqref{B1'} we examine in detail only the terms involving $L$ and $\underline{L}$ derivatives since the rest are either similar or easier. We consider the terms:
\begin{equation*}
\begin{split}
\sum_{k_1 + k_2 = k} \int_{\mathcal{A}_{\tau_1}^{\tau_2}} & (r-M)^{-1-\delta} ( L \Omega^{k_1} T \phi )^2 ( \underline{L} \Omega^{k_2} \phi )^2 \, d\omega du dv \mbox{  and  } \\ & \sum_{k_1 + k_2 = k} \int_{\mathcal{A}_{\tau_1}^{\tau_2}} (r-M)^{-1-\delta} ( L \Omega^{k_1} \phi )^2 ( \underline{L} \Omega^{k_2} T \phi )^2 \, d\omega du dv .
\end{split}
\end{equation*}
For the first term we have that
\begingroup
\allowdisplaybreaks
\begin{align*}
\sum_{k_1 + k_2 = k} \int_{\mathcal{A}_{\tau_1}^{\tau_2}} & (r-M)^{-1-\delta} ( L \Omega^{k_1} T \phi )^2 ( \underline{L} \Omega^{k_2} \phi )^2 \, d\omega du dv \\ \lesssim & \sum_{k_1 + k_2 = k} \int_{\mathcal{A}_{\tau_1}^{\tau_2}}  (r-M)^{1+\delta} ( L \Omega^{k_1} T \phi )^2\frac{1}{(r-M)^{2+2\delta}} (\underline{L} \Omega^{k_2} \phi )^2 \, d\omega du dv \\ \lesssim & \frac{C E_0 \epsilon^2}{\tau_1^{3/2-\delta_1 / 2}} \sum_{l \leq 5} \int_{\mathcal{A}_{\tau_1}^{\tau_2}}  (r-M)^{1+\delta} ( L \Omega^l T \phi )^2 \, d\omega du dv \\ \lesssim & \frac{C^2 E_0^2 \epsilon^4}{\tau_1^{9/2 -  \delta_1 / 2 + \delta_2 }} \lesssim \frac{C^2 E_0^2 \epsilon^4}{\tau_1^{3+\delta_2}} ,
\end{align*}
\endgroup
where we used Sobolev's inequality, the auxiliary estimate \eqref{est:auxl}, the Morawetz decay estimate \eqref{dec:mor1}, and the fact that $\delta_1$ and $\delta_2$ are small enough.

For the second term we argue similarly this time using the auxiliary estimate \eqref{est:auxlt}, the Morawetz decay estimate \eqref{dec:mor}, and the fact that $\delta_1$ and $\delta_2$ are small enough and we have that:
\begingroup
\allowdisplaybreaks
\begin{align*}
\sum_{k_1 + k_2 = k} \int_{\mathcal{A}_{\tau_1}^{\tau_2}} & (r-M)^{-1-\delta} ( L \Omega^{k_1}  \phi )^2 ( \underline{L} \Omega^{k_2} T \phi )^2 \, d\omega du dv \\ \lesssim & \sum_{k_1 + k_2 = k} \int_{\mathcal{A}_{\tau_1}^{\tau_2}}  (r-M)^{1+\delta} ( L \Omega^{k_1} \phi )^2\frac{1}{(r-M)^{2+2\delta}} (\underline{L} \Omega^{k_2} T \phi )^2 \, d\omega du dv \\ \lesssim & \frac{C E_0 \epsilon^2}{\tau_1^{3/2+\delta_2 / 2}} \sum_{l \leq 5} \int_{\mathcal{A}_{\tau_1}^{\tau_2}}  (r-M)^{1+\delta} ( L \Omega^l  \phi )^2 \, d\omega du dv \\ \lesssim & \frac{C^2 E_0^2 \epsilon^4}{\tau_1^{9/2 +\delta_2 / 2 - \delta_1 }} \lesssim \frac{C^2 E_0^2 \epsilon^4}{\tau_1^{3+\delta_2}} .
\end{align*}
\endgroup

\eqref{C3'}: Once again we expand $F$ and due to the commutations with $T$ and $\Omega^k$ for $k \leq 5$, we have the following:
\begingroup
\allowdisplaybreaks
\begin{align*}
\int_{\tau_1}^{\tau_2} \int_{\mcN_v^H} &  (r-M)^{-p}  (\underline{L}  T \Omega^k \Phi^H )  \cdot \underline{L} ( r^2  T \Omega^k F ) \, d\omega du dv \\ = & \sum_{k_1 + k_2 = k} \int_{\tau_1}^{\tau_2} \int_{\mcN_v^H}  (r-M)^{-p} (\underline{L}  T \Omega^k \Phi^H )  \cdot \frac{A}{r} \cdot ( L T\Omega^{k_1} \phi ) \cdot \left( \underline{L} \left( \frac{2r}{D} \underline{L} \Omega^{k_2} \phi \right) \right) \, d\omega du dv \\ & +  \sum_{k_1 + k_2 = k} \int_{\tau_1}^{\tau_2} \int_{\mcN_v^H}  (r-M)^{-p} (\underline{L}  T \Omega^k \Phi^H ) \cdot \frac{A}{r} \cdot ( L \Omega^{k_1} \phi ) \cdot \left( \underline{L} \left( \frac{2r}{D} \underline{L} \Omega^{k_2} T \phi \right) \right) \, d\omega du dv \\ & + \sum_{k_1 + k_2 = k} \int_{\tau_1}^{\tau_2} \int_{\mcN_v^H}  (r-M)^{-p} (\underline{L}  T \Omega^k \Phi^H )  \cdot \frac{A}{r} \cdot ( \underline{L} L T\Omega^{k_1} \phi ) \cdot  \left( \frac{2r}{D} \underline{L} \Omega^{k_2} \phi \right)  \, d\omega du dv \\ & + \sum_{k_1 + k_2 = k} \int_{\tau_1}^{\tau_2} \int_{\mcN_v^H}  (r-M)^{-p} (\underline{L}  T \Omega^k \Phi^H )  \cdot \frac{A}{r} \cdot ( \underline{L} L \Omega^{k_1} \phi ) \cdot  \left( \frac{2r}{D} \underline{L} T \Omega^{k_2} \phi \right)  \, d\omega du dv \\ & + \sum_{k_1 + k_2 = k} \int_{\tau_1}^{\tau_2} \int_{\mcN_v^H}  (r-M)^{-p} (\underline{L}  T \Omega^k \Phi^H )  \cdot \frac{A}{2r^2} \cdot D (T \Omega^{k_1} \phi ) \cdot  \left( \underline{L} \left( \frac{2r}{D} \underline{L} \Omega^{k_2} \phi \right) \right)  \, d\omega du dv \\ & + \sum_{k_1 + k_2 = k} \int_{\tau_1}^{\tau_2} \int_{\mcN_v^H}  (r-M)^{-p} (\underline{L}  T \Omega^k \Phi^H )  \cdot \frac{A}{2r^2} \cdot D ( \Omega^{k_1} \phi ) \cdot  \left( \underline{L} \left( \frac{2r}{D} \underline{L} T \Omega^{k_2} \phi \right) \right)  \, d\omega du dv \\ & + \sum_{k_1 + k_2 = k} \int_{\tau_1}^{\tau_2} \int_{\mcN_v^H}  (r-M)^{-p} (\underline{L}  T \Omega^k \Phi^H )  \cdot \frac{A}{2r^2} \cdot D (\underline{L} T \Omega^{k_1} \phi ) \cdot  \left( \frac{2r}{D} \underline{L} \Omega^{k_2} \phi \right) \, d\omega du dv  \\ & + \sum_{k_1 + k_2 = k} \int_{\tau_1}^{\tau_2} \int_{\mcN_v^H}  (r-M)^{-p} (\underline{L}  T \Omega^k \Phi^H )  \cdot \frac{A}{2r^2} \cdot D (\underline{L}  \Omega^{k_1} \phi ) \cdot  \left( \frac{2r}{D} \underline{L} T \Omega^{k_2} \phi \right) \, d\omega du dv \\ & - \sum_{k_1 + k_2 = k} \int_{\tau_1}^{\tau_2} \int_{\mcN_v^H}  (r-M)^{-p} (\underline{L}  T \Omega^k \Phi^H )  \cdot \frac{A}{r} \cdot (\underline{L}  \Omega^{k_1} \phi ) \cdot  ( L T \Omega^{k_2} \phi ) \, d\omega du dv \\ & - \sum_{k_1 + k_2 = k} \int_{\tau_1}^{\tau_2} \int_{\mcN_v^H}  (r-M)^{-p} (\underline{L}  T \Omega^k \Phi^H )  \cdot \frac{A}{r} \cdot (\underline{L}  T \Omega^{k_1} \phi ) \cdot  ( L \Omega^{k_2} \phi ) \, d\omega du dv \\ & - \sum_{k_1 + k_2 = k} \int_{\tau_1}^{\tau_2} \int_{\mcN_v^H}  (r-M)^{-p} (\underline{L}  T \Omega^k \Phi^H )  \cdot \frac{A}{r} \cdot ( \Omega^{k_1} \phi ) \cdot  ( \underline{L} L T \Omega^{k_2} \phi ) \, d\omega du dv \\ & - \sum_{k_1 + k_2 = k} \int_{\tau_1}^{\tau_2} \int_{\mcN_v^H}  (r-M)^{-p} (\underline{L}  T \Omega^k \Phi^H )  \cdot \frac{A}{r} \cdot ( T \Omega^{k_1} \phi ) \cdot  ( \underline{L} L \Omega^{k_2} \phi ) \, d\omega du dv \\ & - \sum_{k_1 + k_2 = k} \int_{\tau_1}^{\tau_2} \int_{\mcN_v^H}  (r-M)^{-p} (\underline{L}  T \Omega^k \Phi^H )  \cdot \frac{A}{r^2} \cdot 2D ( T \Omega^{k_1} \phi ) \cdot  ( \underline{L}  \Omega^{k_2} \phi ) \, d\omega du dv \\ & - \sum_{k_1 + k_2 = k} \int_{\tau_1}^{\tau_2} \int_{\mcN_v^H}  (r-M)^{-p} (\underline{L}  T \Omega^k \Phi^H )  \cdot \frac{A}{r^2} \cdot 2D ( \Omega^{k_1} \phi ) \cdot  ( \underline{L} T \Omega^{k_2} \phi ) \, d\omega du dv \\ & + \sum_{k_1 + k_2 = k} \int_{\tau_1}^{\tau_2} \int_{\mcN_v^H}  (r-M)^{-p} (\underline{L}  T \Omega^k \Phi^H )  \cdot A \cdot \langle  \underline{L} \Omega^{k_1}\slashed{\nabla} \phi , \Omega^{k_2} \slashed{\nabla} T \phi \rangle \, d\omega du dv \\ & + \sum_{k_1 + k_2 = k} \int_{\tau_1}^{\tau_2} \int_{\mcN_v^H}  (r-M)^{-p} (\underline{L}  T \Omega^k \Phi^H )  \cdot A \cdot \langle \underline{L} \Omega^{k_1} \slashed{\nabla} \phi ,  \Omega^{k_2} T \slashed{\nabla} \phi \rangle \, d\omega du dv \\ & + F_1 (\Omega A, F_{rest , \underline{L} , T} ) + F_2 ( \Omega T A, F_{rest , \underline{L}} )  + F_3 ( \Omega \underline{L}( A,r) , F_{rest , T} ) + F_4 ( \Omega T \underline{L} (A,r) , F_{rest} ) .
\end{align*}
\endgroup
where $F_1$ includes all terms where some angular derivatives fall on $A$, $F_2$ includes terms where some angular derivative and the $T$ derivative fall on $A$, $F_3$ includes all terms where some angular derivatives and the $\underline{L}$ derivative fall on $A$ or a term involving $r$, and the $F_4$ includes all terms where some angular derivatives, and the $\underline{L}$ and $T$ derivatives fall on $A$. All these terms are similar or easier than the rest so we will not examine them in detail.

For the second term we have that for any $k \leq 5$ that
\begin{equation*}
\begin{split}
\sum_{k_1 + k_2 = k} \int_{\tau_1}^{\tau_2} \int_{\mcN_v^H} & (r-M)^{-p} (\underline{L}  T \Omega^k \Phi^H )   \cdot \frac{A}{r} \cdot ( L \Omega^{k_1} \phi ) \cdot \left( \underline{L} \left( \frac{2r}{D} \underline{L} \Omega^{k_2} T \phi \right) \right) \, d\omega du dv \\ \lesssim & \int_{\tau_1}^{\tau_2} \int_{\mcN_v^H}  (r-M)^{-p} (\underline{L}  T \Omega^k \Phi^H )^2 \frac{1}{v^{1+\beta}}  \, d\omega du dv \\ & + \sum_{k_1 + k_2 = k} \int_{\tau_1}^{\tau_2} \int_{\mcN_v^H}  (r-M)^{-p} ( L \Omega^{k_1} \phi )^2 \cdot \left( \underline{L} \left( \frac{2r}{D} \underline{L} \Omega^{k_2} T \phi \right) \right)^2 \cdot v^{1+\beta} \, d\omega du dv \\ \lesssim & \frac{1}{\tau_1^{\beta}} \sup_{v \in [\tau_1 , \tau_2]}  \int_{\mcN_v^H}  (r-M)^{-p} (\underline{L}  T \Omega^k \Phi^H )^2   \, d\omega du \\ & + \sum_{\substack{l_1 \leq 5 \\ l_2 \leq 5}} \int_{\tau_1}^{\tau_2} \int_{\mathbb{S}^2} \sup_u ( L  \Omega^{l_1} \phi )^2 \cdot v^{1+\beta} \, d\omega dv \cdot \sup_{v \in [\tau_1 , \tau_2 ]}\int_{\mcN_v^H} (r-M)^{-p}  \left( \underline{L} \left( \frac{2r}{D} \underline{L} \Omega^{l_2} \phi \right) \right)^2 \, d\omega du \\ \lesssim & \bar{\delta} \sup_{v \in [\tau_1 , \tau_2]}  \int_{\mcN_v^H}  (r-M)^{-p} (\underline{L}  T \Omega^k \Phi^H )^2   \, d\omega du \\ & + C E_0 \epsilon^2 \cdot \frac{C E_0 \epsilon^2}{(1+\tau_1 )^{1+\delta_2 -p}} ,
\end{split}
\end{equation*}
where we used Sobolev's inequality \eqref{est:sobolev}, the auxiliary estimate \eqref{est:aux1} for $m=0$, and the decay provided by the energy estimates \eqref{dec:entp}.

For the first term the above process cannot work, as the spherically symmetric part of the term with two $\underline{L}$ derivatives admits $(r-M)^{-p}$-weighted estimates for $p \in (0,1-\delta_1]$, and not for $p \in (1-\delta_1 , 1+\delta_2 ]$. For this purpose we break the first term in the following two parts:
\begingroup
\allowdisplaybreaks
\begin{align*}
\sum_{k_1 + k_2 = k} \int_{\tau_1}^{\tau_2} \int_{\mcN_v^H} & (r-M)^{-1-\delta_2}  (\underline{L}  \Omega^k T \Phi^H ) \cdot \frac{A}{r} \cdot ( L \Omega^{k_1} T \phi ) \cdot \left( \underline{L} \left( \frac{2r}{D} \underline{L} \Omega^{k_2} \phi \right) \right) \, d\omega du dv \\ = & \int_{\tau_1}^{\tau_2} \int_{\mcN_v^H}  (r-M)^{-1-\delta_2}  (\underline{L}  \Omega^k T \Phi^H ) \cdot \frac{A}{r} \cdot ( L \Omega^k T \phi ) \cdot \left( \underline{L} \left( \frac{2r}{D} \underline{L} \phi_0 \right) \right) \, d\omega du dv \\ & + \sum_{k_1 + k_2 = k} \int_{\tau_1}^{\tau_2} \int_{\mcN_v^H}  (r-M)^{-1-\delta_2}  (\underline{L}  \Omega^k T \Phi^H ) \cdot \frac{A}{r} \cdot ( L \Omega^{k_1} T \phi ) \cdot \left( \underline{L} \left( \frac{2r}{D} \underline{L} \Omega^{k_2} \phi_{\geq 1} \right) \right) \, d\omega du dv 
\end{align*}
\endgroup
The last term of the above expression as the second term that was treated above (as for this one the term with the two $\underline{L}$ derivatives admits $(r-M)^{-p}$ estimates for up to $p = 1+\delta_2$). For the other term instead we have for $p = 1+\delta_2$ that:
\begingroup
\allowdisplaybreaks
\begin{align*}
\int_{\tau_1}^{\tau_2} \int_{\mcN_v^H}  (r-M)^{-1-\delta_2} & (\underline{L}  \Omega^k T \Phi^H ) \cdot \frac{A}{r} \cdot ( L \Omega^{k} T \phi ) \cdot \left( \underline{L} \left( \frac{2r}{D} \underline{L} \phi_0 \right) \right) \, d\omega du dv \\ \lesssim & \int_{\tau_1}^{\tau_2} \int_{\mcN_v^H}  (r-M)^{-1-\delta_2}  (\underline{L}  \Omega^k T \Phi^H )^2 \cdot \frac{1}{v^{1+\beta}} \, d\omega du dv \\ & +  \int_{\tau_1}^{\tau_2} \int_{\mcN_v^H}  (r-M)^{-1-\delta_2} ( L \Omega^{k} T \phi )^2 \cdot \left( \underline{L} \left( \frac{2r}{D} \underline{L} \phi_0 \right) \right)^2 \cdot v^{1+\beta} \, d\omega du dv \\ \lesssim &  \frac{1}{\tau_1^{\beta}} \sup_{v \in [\tau_1 , \tau_2 ]} \int_{\mcN_v^H}  (r-M)^{-1-\delta_2}  (\underline{L}  \Omega^k T \Phi^H )^2  \, d\omega du  \\ & + \int_{\tau_1}^{\tau_2} \int_{\mcN_v^H}  (r-M)^{1+\beta_1 - \delta_2} ( L \Omega^{k} T \phi )^2 \cdot D^{1-\beta_1 /2}\left( \underline{L} \left( \frac{2r}{D} \underline{L} \phi_0 \right) \right)^2 \cdot v^{1+\beta} \, d\omega du dv \\ \lesssim &  \frac{1}{\tau_1^{\beta}} \sup_{v \in [\tau_1 , \tau_2 ]} \int_{\mcN_v^H}  (r-M)^{-1-\delta_2}  (\underline{L}  \Omega^k T \Phi^H )^2  \, d\omega du  \\ & + C E_0 \epsilon^2 \int_{\tau_1}^{\tau_2} \int_{\mcN_v^H}  (r-M)^{1+\beta_1 - \delta_2} ( L \Omega^k T \phi )^2  \cdot v^{1+\beta + \beta_1} \, d\omega du dv \\ \lesssim &  \frac{1}{\tau_1^{\beta}} \sup_{v \in [\tau_1 , \tau_2 ]} \int_{\mcN_v^H}  (r-M)^{-1-\delta_2}  (\underline{L}  \Omega^k T \Phi^H )^2  \, d\omega du  \\ & + C E_0 \epsilon^2 \int_{\tau_1}^{\tau_2}  \int_{\mathbb{S}^2} \sup_u ( L \Omega^k T \phi )^2  \cdot v^{1+\beta + \beta_1} \, d\omega dv \\ \lesssim &  \frac{1}{\tau_1^{\beta}} \sup_{v \in [\tau_1 , \tau_2 ]} \int_{\mcN_v^H}  (r-M)^{-1-\delta_2}  (\underline{L}  \Omega^k T \Phi^H )^2  \, d\omega du  \\ & + C^2 E_0^2 \epsilon^4 ,
\end{align*} 
\endgroup
where we used the growth estimate \eqref{est:boundyy}, the auxiliary estimate \eqref{est:aux1} for $m=0$, always choosing $\beta_1$ to be bigger than $\delta_2$ and $\beta$ and $\beta_1$ to be small enough such that $\beta + \beta_1 < 2 -\delta_1$ (note that these two conditions can be simultaneously satisfied due to the smallness of $\delta_1$ and $\delta_2$). On the other hand we consider the (out of range) case of $p=0$ and we have that:
\begin{equation*}
\begin{split}
\int_{\tau_1}^{\tau_2} \int_{\mcN_v^H} & (\underline{L}  \Omega^k T \Phi^H ) \cdot \frac{A}{r} \cdot ( L \Omega^{k} T\phi ) \cdot \left( \underline{L} \left( \frac{2r}{D} \underline{L}\phi_0 \right) \right) \, d\omega du dv \\ \lesssim & \beta \int_{\tau_1}^{\tau_2} \int_{\mcN_v^H}  (r-M) (\underline{L}  \Omega^k T \Phi^H )^2 \, d\omega du dv \\ & + \frac{1}{\beta}  \int_{\tau_1}^{\tau_2} \int_{\mcN_v^H} (r-M)^{-1} ( L \Omega^{k} T \phi )^2 \cdot \left( \underline{L} \left( \frac{2r}{D} \underline{L} \phi_0 \right) \right)^2 \, d\omega du dv .
\end{split}
\end{equation*}
The first term of the above expression can be absorbed by the left hand side that has the term of interest in \eqref{B3'} in its right hand side for any $p \in (0,1+\delta_2]$. For the second term we have for any $v \in [\tau_1 , \tau_2 ] $ that:
\begin{equation*}
\begin{split}
\int_{\mcN_v^H} (r-M)^{-1} & ( L \Omega^{k} T \phi )^2 \cdot \left( \underline{L} \left( \frac{2r}{D} \underline{L}  \phi_0 \right) \right)^2 \, d\omega du \\ \lesssim & \sum_{k_1 + k_2 = k} \int_{\mcN_v^H} (r-M)^2  ( L \Omega^{k} T \phi_{\geq 1} )^2 \cdot D \left( \frac{2r}{D} \underline{L} \left( \frac{2r}{D} \underline{L} \phi_0 \right) \right)^2 \, d\omega du \\ \lesssim & C E_0 \epsilon^2 v^{\delta_3} \frac{C E_0 \epsilon^2}{v^{3+\delta_2}} \lesssim \frac{C^2 E_0^2 \epsilon^4}{v^{3+\delta_2-\delta_3}} , 
\end{split}
\end{equation*}
where we used the growth estimate \eqref{est:boundyy} and the Morawetz decay estimate \eqref{dec:mor1}, and for $\delta_3 > 0$ small enough we now note that using the last estimate we have that:
\begin{equation*}
\int_{\tau_1}^{\tau_2} \int_{\mcN_v^H} (r-M)^{-1}  ( L \Omega^{k} T \phi_{\geq 1} )^2 \cdot \left( \underline{L} \left( \frac{2r}{D} \underline{L} \phi_0 \right) \right)^2 \, d\omega du dv  \lesssim \frac{C^2 E_0^2 \epsilon^4}{(1+\tau_1 )^{2+\delta_2 - \delta_3}} .
\end{equation*}
The required estimate for the second term of \eqref{B3'} now follows by the estimate that was shown for $p=1+\delta_2$ (which is sharp), the (artificial) $p=0$ estimate (which is better than required as $\delta_3$ is small enough), and a standard interpolation argument.

For the fifth term we have for any $k \leq 5$ that:
\begingroup
\allowdisplaybreaks
\begin{align*}
\sum_{k_1 + k_2 = k} & \int_{\tau_1}^{\tau_2}  \int_{\mcN_v^H}   (r-M)^{-p} (\underline{L}  T \Omega^k \Phi^H )  \cdot \frac{A}{2r^2} \cdot D (T \Omega^{k_1} \phi ) \cdot  \left( \underline{L} \left( \frac{2r}{D} \underline{L} \Omega^{k_2} \phi \right) \right)  \, d\omega du dv \\ \leq & \beta \int_{\tau_1}^{\tau_2} \int_{\mcN_v^H}  (r-M)^{-p+1} (\underline{L}  T \Omega^k \Phi^H )^2 \, d\omega dudv \\ & + \frac{1}{\beta} \sum_{k_1 + k_2 = k} \int_{\tau_1}^{\tau_2} \int_{\mcN_v^H}  (r-M)^{-p-1}  D^2 (T \Omega^{k_1} \phi )^2 \cdot  \left( \underline{L} \left( \frac{2r}{D} \underline{L} \Omega^{k_2} \phi \right) \right)^2  \, d\omega du dv \\ \lesssim & \beta \int_{\tau_1}^{\tau_2} \int_{\mcN_v^H}  (r-M)^{-p+1} (\underline{L}  T \Omega^k \Phi^H )^2 \, d\omega dudv \\ & + \frac{1}{\beta} \sum_{k_1 + k_2 = k} \| T \Omega^{k_1} \phi \|^2_{L^{\infty} (\mathcal{A}_{\tau_1}^{\tau_2}}  \int_{\tau_1}^{\tau_2} \int_{\mcN_v^H}  (r-M)^{-p+3}    \left( \underline{L} \left( \frac{2r}{D} \underline{L} \Omega^{k_2} \phi \right) \right)^2  \, d\omega du dv \mbox{  if $k_1 \leq 3$} \\ & \mbox{  or  } + \frac{1}{\beta} \sum_{k_1 + k_2 = k} \left( \sup_{u , v \in [\tau_1 , \tau_2 ]} \int_{\mathbb{S}^2} (T \Omega^{k_1} \phi )^2 (u,v,\omega) \, d\omega \right) \cdot \int_{\tau_1}^{\tau_2} \int_{\mcN_v^H}  (r-M)^{-p+3}      \left( \underline{L} \left( \frac{2r}{D} \underline{L} \Omega^{k_2 +2} \phi \right) \right)^2  \, d\omega du dv \\ & \mbox{  if $k_1 \geq 4$} \\ \lesssim & \beta \int_{\tau_1}^{\tau_2} \int_{\mcN_v^H}  (r-M)^{-p+1} (\underline{L}  T \Omega^k \Phi^H )^2 \, d\omega dudv  + \frac{C E_0 \epsilon^2}{(1+\tau_1 )^{2+\delta_2}} \frac{C E_0 \epsilon^2}{(1+\tau_1 )^{1-\delta_1 -p}}  ,
\end{align*}
\endgroup
where we used Sobolev's inequality \eqref{est:sobolev}, the pointwise estimate \eqref{dec:tpsi}, and the hierarchy of energy estimates \eqref{dec:en0p2}. The above decay is better than required.

The sixth term (again considered for any $k \leq 5$)
$$ \sum_{k_1 + k_2 = k} \int_{\tau_1}^{\tau_2} \int_{\mcN_v^H}  (r-M)^{-p} (\underline{L}  T \Omega^k \Phi^H )  \cdot \frac{A}{2r^2} \cdot D (\Omega^{k_1} \phi ) \cdot  \left( \underline{L} \left( \frac{2r}{D} \underline{L} \Omega^{k_2} T \phi \right) \right)  \, d\omega du dv $$
can be treated similarly to the fifth term, by using the pointwise decay \eqref{dec:psi} in the place of \eqref{dec:tpsi}, and the hierarchy of energy estimate \eqref{dec:ent} in the place of \eqref{dec:en0p2}.

For the seventh term we have for any $k \leq 5$ that:
\begingroup
\allowdisplaybreaks
\begin{align*}
 \sum_{k_1 + k_2 = k} \int_{\tau_1}^{\tau_2} & \int_{\mcN_v^H}  (r-M)^{-p} (\underline{L}  T \Omega^k \Phi^H )  \cdot \frac{A}{2r^2} \cdot D (\underline{L} T \Omega^{k_1} \phi ) \cdot  \left( \frac{2r}{D} \underline{L} \Omega^{k_2} \phi \right) \, d\omega du dv \\ \leq & \beta \int_{\tau_1}^{\tau_2} \int_{\mcN_v^H}  (r-M)^{-p+1} (\underline{L}  T \Omega^k \Phi^H )^2 \, d\omega du dv \\ & +  \frac{1}{\beta} \sum_{k_1 + k_2 = k} \int_{\tau_1}^{\tau_2} \int_{\mcN_v^H}  (r-M)^{-p-1}  D^2 (\underline{L} T \Omega^{k_1} \phi )^2 \cdot  \left( \frac{2r}{D} \underline{L} \Omega^{k_2} \phi \right)^2 \, d\omega du dv \\ \lesssim & \beta \int_{\tau_1}^{\tau_2} \int_{\mcN_v^H}  (r-M)^{-p+1} (\underline{L}  T \Omega^k \Phi^H )^2 \, d\omega du dv \\ & +  \frac{1}{\beta} \sum_{k_1 + k_2 = k} \left\|\frac{2r}{D} \underline{L} \Omega^{k_2} \phi \right\|^2_{L^{\infty} (\mathcal{A}_{\tau_1}^{\tau_2})} \int_{\tau_1}^{\tau_2} \int_{\mcN_v^H}  (r-M)^{-p+3} (\underline{L} T \Omega^{k_1} \phi )^2 \, d\omega du dv \mbox{  if $k_2 \leq 3$,} \\ & \mbox{  or  } + \frac{1}{\beta} \sum_{k_1 + k_2 = k} \sup_{u , v \in [\tau_1 , \tau_2 ]} \int_{\mathbb{S}^2} \left( \frac{2r}{D} \underline{L} \Omega^{k_2} \phi \right)^2 (u,v,\omega) \, d\omega \cdot \int_{\tau_1}^{\tau_2} \int_{\mcN_v^H}  (r-M)^{-p+3} (\underline{L} T \Omega^{k_1} \phi )^2 \, d\omega du dv \\ & \mbox{  if $k_2 \geq 4$}  \\ \lesssim & \beta \int_{\tau_1}^{\tau_2} \int_{\mcN_v^H}  (r-M)^{-p+1} (\underline{L}  T \Omega^k \Phi^H )^2 \, d\omega du dv  + C E_0 \epsilon^2 \frac{C E_0 \epsilon^2}{\tau_1^{3+\delta_2}} ,
\end{align*}
\endgroup
where we used Sobolev's inequality \eqref{est:sobolev}, the pointwise bound \eqref{est:boundy} and the Morawetz decay estimate \eqref{dec:mor1} (as $p \in (0,1+\delta_2 ]$).

For the eighth term we have for any $k \leq 5$ that:
\begingroup
\allowdisplaybreaks
\begin{align*}
\sum_{k_1 + k_2 = k} \int_{\tau_1}^{\tau_2} & \int_{\mcN_v^H}  (r-M)^{-p} (\underline{L}  T \Omega^k \Phi^H )  \cdot \frac{A}{2r^2} \cdot D (\underline{L} \Omega^{k_1} \phi ) \cdot  \left( \frac{2r}{D} \underline{L} \Omega^{k_2} T \phi \right) \, d\omega du dv \\ \leq &  \beta \int_{\tau_1}^{\tau_2}  \int_{\mcN_v^H}  (r-M)^{-p+1} (\underline{L}  T \Omega^k \Phi^H )^2 \, d\omega du dv \\ & + \frac{1}{\beta}\sum_{k_1 + k_2 = k} \int_{\tau_1}^{\tau_2} \int_{\mcN_v^H}  (r-M)^{-p-1}  D^2  (\underline{L} \Omega^{k_1} \phi )^2 \cdot  \left( \frac{2r}{D} \underline{L} \Omega^{k_2} T \phi \right)^2 \, d\omega du dv \\ \lesssim & \beta \int_{\tau_1}^{\tau_2}  \int_{\mcN_v^H}  (r-M)^{-p+1} (\underline{L}  T \Omega^k \Phi^H )^2  \, d\omega du dv \\ & + \left\| \frac{2r}{D}\underline{L} \Omega^{k_1} \phi \right\|^2_{L^{\infty} (\mathcal{A}_{\tau_1}^{\tau_2}} \sum_{k_1 + k_2 = k} \int_{\tau_1}^{\tau_2} \int_{\mcN_v^H}  (r-M)^{-p+3}     ( \underline{L} \Omega^{k_2} T \phi )^2 \, d\omega du dv \mbox{  if $k_1 \leq 3$,} \\ & \mbox{  or } +  \sup_{u , v \in [\tau_1 , \tau_2 ]}\int_{\mathbb{S}^2} \left( \frac{2r}{D}\underline{L} \Omega^{k_1} \phi \right)^2 (u,v,\omega) \, d\omega \cdot \sum_{k_1 + k_2 = k} \int_{\tau_1}^{\tau_2} \int_{\mcN_v^H}  (r-M)^{-p+3}     (  \underline{L} \Omega^{k_2} T \phi )^2 \, d\omega du dv  \\ \lesssim & \beta \int_{\tau_1}^{\tau_2}  \int_{\mcN_v^H}  (r-M)^{-p+1} (\underline{L}  T \Omega^k \Phi^H )^2 \, d\omega du dv  + C E_0 \epsilon^2 \frac{C E_0 \epsilon^2}{(1+\tau_1 ^{3-\delta_2}},
\end{align*}
\endgroup
where we used Sobolev's inequality \eqref{est:sobolev}, the boundedness estimate \eqref{est:boundy} and the Morawetz decay estimate \eqref{dec:mor1}.

For the ninth term we have for any $k \leq 5$ that:
\begingroup
\allowdisplaybreaks
\begin{align*}
\sum_{k_1 + k_2 = k} \int_{\tau_1}^{\tau_2} \int_{\mcN_v^H} & (r-M)^{-p} (\underline{L}  T \Omega^k \Phi^H )  \cdot \frac{A}{r} \cdot (\underline{L}  \Omega^{k_1} \phi ) \cdot  ( L T \Omega^{k_2} \phi ) \, d\omega du dv \\ \lesssim & \beta \int_{\tau_1}^{\tau_2} \int_{\mcN_v^H}  (r-M)^{-p+1} (\underline{L}  T \Omega^k \Phi^H )^2 \, d\omega du dv \\ & + \frac{1}{\beta}  \sum_{k_1 + k_2 = k} \int_{\tau_1}^{\tau_2} \int_{\mcN_v^H}  (r-M)^{-p-1}  (\underline{L}  \Omega^{k_1} \phi )^2 \cdot  ( L T \Omega^{k_2} \phi )^2 \, d\omega du dv \\ \lesssim & \beta \int_{\tau_1}^{\tau_2} \int_{\mcN_v^H}  (r-M)^{-p} (\underline{L}  T \Omega^k \Phi^H )^2 \, d\omega du dv \\ & + \frac{1}{\beta}  \sum_{k_1 + k_2 = k} \int_{\tau_1}^{\tau_2} \int_{\mcN_v^H}  (r-M)^{-p+3}  \left( \frac{2r}{D} \underline{L}  \Omega^{k_1} \phi \right)^2 \cdot  ( L T \Omega^{k_2} \phi )^2 \, d\omega du dv  \\ \lesssim & \beta \int_{\tau_1}^{\tau_2} \int_{\mcN_v^H}  (r-M)^{-p} (\underline{L}  T \Omega^k \Phi^H )^2 \, d\omega du dv \\ & + \left\| \frac{2r}{D} \underline{L} \Omega^{k_1} \phi \right\|_{L^{\infty}( \mathcal{A}_{\tau_1}^{\tau_2} )}^2   \sum_{k_1 + k_2 = k} \int_{\tau_1}^{\tau_2} \int_{\mcN_v^H}  (r-M)^{-p+3}    ( L T \Omega^{k_2} \phi )^2 \, d\omega du dv \mbox{  if $k_1 \leq 3$, or} \\ & + \sup_{u , v \in [\tau_1 , \tau_2 ]} \int_{\mathbb{S}^2} \left( \frac{2r}{D} \underline{L} \Omega^{k_1} \right)^2 (u,v,\omega) \, d\omega   \sum_{k_1 + k_2 = k} \int_{\tau_1}^{\tau_2} \int_{\mcN_v^H}  (r-M)^{-p+3}    ( L T \Omega^{k_2 +2} \phi )^2 \, d\omega du dv \mbox{  if $k_1 \geq 4$} \\ \lesssim & \beta \int_{\tau_1}^{\tau_2} \int_{\mcN_v^H}  (r-M)^{-p} (\underline{L}  T \Omega^k \Phi^H )^2 \, d\omega du dv  + C E_0 \epsilon^2 \frac{C E_0 \epsilon^2}{(1+\tau_1 )^{3+\delta_2}} ,
\end{align*}
\endgroup
where we used Sobolev's inequality \eqref{est:sobolev}, the pointwise estimate \eqref{est:boundy}, and the Morawetz decay estimate \eqref{dec:mor1} as $p \in (0,1+\delta_2]$.

For the tenth term we have for any $k \leq 5$ that:
\begingroup
\allowdisplaybreaks
\begin{align*}
\sum_{k_1 + k_2 = k} \int_{\tau_1}^{\tau_2} \int_{\mcN_v^H} & (r-M)^{-p} (\underline{L}  T \Omega^k \Phi^H )  \cdot \frac{A}{r} \cdot (\underline{L}  T \Omega^{k_1} \phi ) \cdot  ( L  \Omega^{k_2} \phi ) \, d\omega du dv \\ \lesssim & \int_{\tau_1}^{\tau_2} \int_{\mcN_v^H} (r-M)^{-p} (\underline{L}  T \Omega^k \Phi^H )^2 \cdot \frac{1}{v^{1+\beta}} \, d\omega du dv \\ & + \sum_{k_1 + k_2 = k} \int_{\tau_1}^{\tau_2} \int_{\mcN_v^H} (r-M)^{-p}  (\underline{L}  T \Omega^{k_1} \phi )^2 \cdot  ( L  \Omega^{k_2} \phi )^2 \cdot v^{1+\beta} \, d\omega du dv \\ \lesssim & \frac{1}{\tau_1^{\beta}}  \sup_{v \in [\tau_1 , \tau_2 ]} \int_{\mcN_v^H} (r-M)^{-p} (\underline{L}  T \Omega^k \Phi^H )^2 \, d\omega du \\ & + \sum_{\substack{l_1 \leq 5 \\ l_2 \leq 5}} \int_{\tau_1}^{\tau_2} \int_{\mathbb{S}^2} ( L  \Omega^{l_1} \phi )^2 \cdot v^{1+\beta} \, d\omega du dv \cdot \sup_{v \in [\tau_1 , \tau_2 ] } \int_{\mcN_v^H} (r-M)^{-p}  (\underline{L}  T \Omega^{l_2} \phi )^2 \, d\omega du \\ \lesssim & \frac{1}{\tau_1^{\beta}}  \sup_{v \in [\tau_1 , \tau_2 ]} \int_{\mcN_v^H} (r-M)^{-p} (\underline{L}  T \Omega^k \Phi^H )^2 \, d\omega du  + C E_0 \epsilon^2 \frac{C E_0 \epsilon^2}{(1+\tau_1 )^{3+\delta_2 -p}} ,
\end{align*}
\endgroup
where we used Sobolev's inequality \eqref{est:sobolev}, the auxiliary estimate \eqref{est:aux1} for $m=0$, the energy decay estimates \eqref{dec:entp}. The decay obtained is better than required.

For the eleventh term we have for any $k \leq 5$ that:
\begingroup
\allowdisplaybreaks
\begin{align*}
 \sum_{k_1 + k_2 = k} \int_{\tau_1}^{\tau_2} & \int_{\mcN_v^H}  (r-M)^{-p} (\underline{L}  T \Omega^k \Phi^H )  \cdot \frac{A}{r} \cdot ( \Omega^{k_1} \phi ) \cdot  ( \underline{L} L T \Omega^{k_2} \phi ) \, d\omega du dv \\ \lesssim &  \int_{\tau_1}^{\tau_2}  \int_{\mcN_v^H}  (r-M)^{-p} (\underline{L}  T \Omega^k \Phi^H )^2 \cdot \frac{1}{v^{1+\beta}} \, d\omega du dv \\ & +  \sum_{k_1 + k_2 = k} \int_{\tau_1}^{\tau_2} \int_{\mcN_v^H}  (r-M)^{-p}  ( \Omega^{k_1} \phi )^2 \cdot  ( \underline{L} L T \Omega^{k_2} \phi )^2 \cdot v^{1+\beta} \, d\omega du dv \\  \lesssim & \frac{1}{\tau_1^{\beta}} \sup_{v \in [\tau_1 , \tau_2 ]} \int_{\mcN_v^H}  (r-M)^{-p} (\underline{L}  T \Omega^k \Phi^H )^2 \, d\omega du dv \\ & +   \sum_{k_1 + k_2 = k} \int_{\tau_1}^{\tau_2} \int_{\mcN_v^H}  (r-M)^{-p}  ( \Omega^{k_1} \phi )^2 \cdot  ( \underline{L} T^2 \Omega^{k_2} \phi )^2 \cdot v^{1+\beta} \, d\omega du dv \\ & + \sum_{k_1 + k_2 = k} \int_{\tau_1}^{\tau_2} \int_{\mcN_v^H}  (r-M)^{-p}  ( \Omega^{k_1} \phi )^2 \cdot D^2  \left( \underline{L} \left( \frac{2r}{D} \underline{L}  \Omega^{k_2} T \phi \right) \right)^2 \cdot v^{1+\beta} \, d\omega du dv \\ & +  \sum_{k_1 + k_2 = k} \int_{\tau_1}^{\tau_2} \int_{\mcN_v^H}  (r-M)^{-p}  ( \Omega^{k_1} \phi )^2 \cdot  D ( \underline{L}  \Omega^{k_2} T \phi )^2 \cdot v^{1+\beta} \, d\omega   du dv \\  \lesssim & \frac{1}{\tau_1^{\beta}} \int_{\tau_1}^{\tau_2}  \int_{\mcN_v^H}  (r-M)^{-p} (\underline{L}  T \Omega^k \Phi^H )^2 \, d\omega du dv \\ & +  \frac{C E_0 \epsilon^2}{\tau_1^{1-\delta_1-\beta}} \sum_{l \leq 5} \int_{\tau_1}^{\tau_2} \int_{\mcN_v^H}  (r-M)^{-p}   ( \underline{L} T^2 \Omega^{l} \phi )^2 \, d\omega du dv \\ & +  \frac{C E_0 \epsilon^2}{\tau_1^{1-\delta_1 - \beta}} \sum_{l \leq 5} \int_{\tau_1}^{\tau_2} \int_{\mcN_v^H}  (r-M)^{-p+4}   \left( \underline{L} \left( \frac{2r}{D} \underline{L}  \Omega^l T \phi \right) \right)^2 \, d\omega du dv \\ & +  \frac{C E_0 \epsilon^2}{\tau_1^{1-\delta_1 -\beta}} \sum_{l \leq 5} \int_{\tau_1}^{\tau_2} \int_{\mcN_v^H}  (r-M)^{-p+2}    ( \underline{L}  \Omega^l T \phi )^2 \, d\omega   du dv \\ \lesssim & \frac{C^2 E_0^2 \epsilon^4}{\tau_1^{1-\delta_1 -\beta + 2 +\delta_2 -p-1}} + \frac{C^2 E_0^2 \epsilon^4}{\tau_1^{1-\delta_1 -\beta}} + \frac{C^2 E_0^2 \epsilon^4}{\tau_1^{1-\delta_1 -\beta + 3 +\delta_2 -p}}  ,
\end{align*}
\endgroup
where we used Sobolev's inequality, the pointwise estimate \eqref{dec:psi}, and the energy decay estimates \eqref{dec:enttp}, \eqref{dec:entp1} and \eqref{dec:entp}. Note that as we choose $\beta$ to be small enough, the obtained decay is better than the one required.

For the twelfth term we have that:
\begingroup
\allowdisplaybreaks
\begin{align*}
 \sum_{k_1 + k_2 = k} \int_{\tau_1}^{\tau_2} & \int_{\mcN_v^H}  (r-M)^{-p} (\underline{L}  T \Omega^k \Phi^H )  \cdot \frac{A}{r} \cdot ( \Omega^{k_1} T \phi ) \cdot  ( \underline{L} L  \Omega^{k_2} \phi ) \, d\omega du dv \\ \lesssim &  \beta \int_{\tau_1}^{\tau_2}  \int_{\mcN_v^H}  (r-M)^{-p+1} (\underline{L}  T \Omega^k \Phi^H )^2  \, d\omega du dv \\ & +  \frac{1}{\beta} \sum_{k_1 + k_2 = k} \int_{\tau_1}^{\tau_2} \int_{\mcN_v^H}  (r-M)^{-p-1}  ( \Omega^{k_1} T \phi )^2 \cdot  ( \underline{L} L  \Omega^{k_2} \phi )^2 \, d\omega du dv  \\ \lesssim &  \beta \int_{\tau_1}^{\tau_2}  \int_{\mcN_v^H}  (r-M)^{-p+1} (\underline{L}  T \Omega^k \Phi^H )^2  \, d\omega du dv \\ & +  \frac{1}{\beta} \sum_{k_1 + k_2 = k} \int_{\tau_1}^{\tau_2} \int_{\mcN_v^H}  (r-M)^{-p-1}  ( \Omega^{k_1} T \phi )^2 \cdot  ( \underline{L} T  \Omega^{k_2} \phi )^2 \, d\omega du dv \\ & +  \frac{1}{\beta} \sum_{k_1 + k_2 = k} \int_{\tau_1}^{\tau_2} \int_{\mcN_v^H}  (r-M)^{-p-1}  D^2 ( \Omega^{k_1} T \phi )^2 \cdot  \left( \underline{L} \left( \frac{2r}{D}  \underline{L} \Omega^{k_2} \phi \right) \right)^2 \, d\omega du dv \\ & +  \frac{1}{\beta} \sum_{k_1 + k_2 = k} \int_{\tau_1}^{\tau_2} \int_{\mcN_v^H}  (r-M)^{-p-1}  D ( \Omega^{k_1} T \phi )^2 \cdot  ( \underline{L}   \Omega^{k_2} \phi )^2 \, d\omega du dv .
\end{align*}
\endgroup
We look separately at the last three terms of the last expression. For the first one by using Sobolev's inequality \eqref{est:sobolev}, the pointwise decay estimate \eqref{dec:tpsi}, Hardy's inequality \eqref{hardy} and the energy decay estimates \eqref{dec:entp1} we have that
\begingroup
\allowdisplaybreaks
\begin{align*}
\sum_{k_1 + k_2 = k} & \int_{\tau_1}^{\tau_2} \int_{\mcN_v^H}  (r-M)^{-p-1}  ( \Omega^{k_1} T \phi )^2 \cdot  ( \underline{L} T  \Omega^{k_2} \phi )^2 \, d\omega du dv \\ \lesssim & \frac{C E_0 \epsilon^2}{\tau_1^{2+\delta_2}} \sum_{l \leq 5}\int_{\tau_1}^{\tau_2} \int_{\mcN_v^H}  (r-M)^{-p-1}   ( \underline{L} T  \Omega^{k_2} \phi )^2 \, d\omega du dv \\ \lesssim & \frac{C E_0 \epsilon^2}{\tau_1^{2+\delta_2}} \sum_{l \leq 5}\int_{\tau_1}^{\tau_2} \int_{\mcN_v^H}  (r-M)^{-p+1}   \left( \underline{L} \left( \frac{2r}{D} \underline{L} T  \Omega^{k_2} \phi \right) \right)^2 \, d\omega du dv  \lesssim  \frac{C^2 E_0^2 \epsilon^4}{\tau_1^{2+\delta_2 + 1 +\delta_2 - p}} .
\end{align*}
\endgroup
For the second one by using Sobolev's inequality \eqref{est:sobolev}, the pointwise decay estimate \eqref{dec:tpsi} and the energy decay estimates \eqref{dec:enp1} (as $(r-M)^{-p-1} D^2 \simeq (r-M)^{-p+3}$ so although $p \in (0,1+\delta_2]$ we can use the estimates \eqref{dec:enp1} as $\delta_2$ is small) we have that it is bounded by:
$$ \frac{C^2 E_0^2 \epsilon^4}{\tau_1^{2+\delta_2 }} . $$
For the third one by by using Sobolev's inequality \eqref{est:sobolev}, the pointwise decay estimate \eqref{dec:tpsi} and the energy decay estimates \eqref{dec:enp} we have that for any $v$ it holds that:
\begingroup
\allowdisplaybreaks
\begin{align*}
\sum_{k_1 + k_2 = k} & \int_{\mcN_v^H}  (r-M)^{-p-1}   D ( \Omega^{k_1} T \phi )^2 \cdot  ( \underline{L}   \Omega^{k_2} \phi )^2 \, d\omega du dv \\ \lesssim & \frac{C E_0 \epsilon^2}{v^{2+\delta_2}} \sum_{l \leq 5}  \int_{\mcN_v^H}  (r-M)^{-p+1}  ( \underline{L}   \Omega^{l} \phi )^2 \, d\omega du dv \\ \lesssim & \frac{C^2 E_0^2 \epsilon^4}{v^{2+ \delta_2 + 3-\delta_1 -p +1 }} \mbox{  if $-p+1 \leq 0$, and }
 \lesssim  \frac{C^2 E_0^2 \epsilon^4}{v^{2+ \delta_2 + 3-\delta_1 }} \mbox{  otherwise. }
\end{align*}
\endgroup
Integrating in $[\tau_1 , \tau_2]$ gives us better decay than required.

The thirteenth and fourteenth terms can be shown to have better decay than required by using the decay for the bulk term involving angular derivatives of the energy decay estimates \eqref{dec:enp}, \eqref{dec:enp1}, \eqref{dec:entp} and \eqref{dec:entp1} (depending on which part is hit by the largest number of angular derivatives, as if a term has six angular derivatives we are forced to use the aforementioned energy decay estimates, as it was done for the other bootstraps).

\eqref{D1'}: We examine first the term close to the horizon for any $k \leq 5$ and we have that:
\begingroup
\allowdisplaybreaks
\begin{align*}
\int_{\mathcal{A}_{\tau_1}^{\tau_2}}  (r-M)^{-1-\delta} D^2 & | \Omega^k T^2 F |^2 \, d\omega du dv \lesssim \sum_{\substack{k_1 + k_2 = k \\ m_1 + m_2 = 2}} \int_{\mathcal{A}_{\tau_1}^{\tau_2} }\frac{1}{(r-M)^{1+\delta}} (L \Omega^{k_1} T^{m_1} \phi )^2 \cdot ( \underline{L} \Omega^{k_2} T^{m_2} \phi )^2 \,  d\omega du dv \\ & + \sum_{\substack{k_1 + k_2 = k \\ m_1 + m_2 = 2}} \int_{\mathcal{A}_{\tau_1}^{\tau_2} } (r-M)^{3-\delta} ( \Omega^{k_1} T^{m_1} \phi )^2 \cdot ( \underline{L} \Omega^{k_2} T^{m_2} \phi )^2  \,  d\omega du dv \\ & + \sum_{\substack{k_1 + k_2 = k \\ m_1 + m_2 = 2}} \int_{\mathcal{A}_{\tau_1}^{\tau_2} } (r-M)^{3-\delta} ( \Omega^{k_1} T^{m_1} \phi )^2 \cdot ( L \Omega^{k_2} T^{m_2} \phi )^2  \,  d\omega du dv \\ & + \sum_{\substack{k_1 + k_2 = k \\ m_1 + m_2 = 2}} \int_{\mathcal{A}_{\tau_1}^{\tau_2} } (r-M)^{7-\delta} ( \Omega^{k_1} T^{m_1} \phi )^2 \cdot (  \Omega^{k_2} T^{m_2} \phi )^2   \,  d\omega du dv \\ & + \sum_{\substack{k_1 + k_2 = k \\ m_1 + m_2 = 2}} \int_{\mathcal{A}_{\tau_1}^{\tau_2} } (r-M)^{3-\delta} | \slashed{\nabla} \Omega^{k_1} T^{m_1} \phi |^2 \cdot | \slashed{\nabla}  \Omega^{k_2} T^{m_2} \phi |^2  \,  d\omega du dv .
\end{align*}
\endgroup
For the first term for any $k \leq 5$, we have for the case of $m_1 = 0$ and $m_2 = 2$, for any $v$, and since $\delta \leq \delta_2$ that:
\begin{equation*}
\begin{split}
\sum_{k_1 + k_2 = k} \int_{\mcN_v^H }\frac{1}{(r-M)^{1+\delta}} & (L \Omega^{k_1}  \phi )^2 \cdot ( \underline{L} \Omega^{k_2} T^2 \phi )^2 \,  d\omega du \\ \lesssim & \frac{C E_0 \epsilon^2}{v^{2+\delta_2}} \sum_{l \leq 5}\int_{\mcN_v^H }\frac{1}{(r-M)^{1+\delta}} ( \underline{L} \Omega^{k_2} T^{l} \phi )^2 \,  d\omega du \lesssim \frac{C^2 E_0^2 \epsilon^4}{v^{3+\delta_2}} ,
\end{split}
\end{equation*}
where we used Sobolev's inequality \eqref{est:sobolev}, the pointwise decay estimates \eqref{dec:tpsi} and the energy decay estimates \eqref{dec:entt} for $p = 1+\delta \leq 1+ \delta_2$. This implies that:
\begin{equation*}
\sum_{k_1 + k_2 = k} \int_{\mathcal{A}_{\tau_1}^{\tau_2}}\frac{1}{(r-M)^{1+\delta}}  (L \Omega^{k_1} \phi )^2 \cdot ( \underline{L} \Omega^{k_2} T^2 \phi )^2 \,  d\omega du dv \lesssim \int_{\tau_1}^{\tau_2} \frac{C^2 E_0^2 \epsilon^4}{v^{3+\delta_2}} \, dv \lesssim \frac{C^2 E_0^2 \epsilon^4}{\tau_1^{2+\delta_2}} .
\end{equation*}
For the case of $m_1 = 1$ and $m_2 = 1$ we have for any $v$ (and again since $\delta \leq \delta_2$) that:
\begin{equation*}
\begin{split}
\sum_{k_1 + k_2 = k} \int_{\mcN_v^H }\frac{1}{(r-M)^{1+\delta}} & (L \Omega^{k_1} T \phi )^2 \cdot ( \underline{L} \Omega^{k_2} T \phi )^2 \,  d\omega du \\ \lesssim & \frac{C E_0 \epsilon^2}{v^{1+\delta_2}} \sum_{l \leq 5} \int_{\mcN_v^H }\frac{1}{(r-M)^{1+\delta}}  ( \underline{L} \Omega^{l} T \phi )^2 \,  d\omega du \lesssim \frac{C^2 E_0^2 \epsilon^4}{v^{3+\delta_2}} ,
\end{split}
\end{equation*}
where we used Sobolev's inequality \eqref{est:sobolev}, the pointwise decay estimates \eqref{dec:ttpsi} and the energy decay estimates \eqref{dec:ent} for $p = 1+\delta \leq 1+ \delta_2$. This implies that:
\begin{equation*}
\sum_{k_1 + k_2 = k} \int_{\mathcal{A}_{\tau_1}^{\tau_2}}\frac{1}{(r-M)^{1+\delta}}  (L \Omega^{k_1} T\phi )^2 \cdot ( \underline{L} \Omega^{k_2} T \phi )^2 \,  d\omega du dv \lesssim \int_{\tau_1}^{\tau_2} \frac{C^2 E_0^2 \epsilon^4}{v^{3+\delta_2}} \, dv \lesssim \frac{C^2 E_0^2 \epsilon^4}{\tau_1^{2+\delta_2}} .
\end{equation*}
For the case $m_1 =2$ and $m_2 = 0$ we have for any $v$ (and again since $\delta \leq \delta_2$) that:
\begin{equation*}
\begin{split}
\sum_{k_1 + k_2 = k} \int_{\mathcal{A}_{\tau_1}^{\tau_2}}\frac{1}{(r-M)^{1+\delta}} & (L \Omega^{k_1} T^2 \phi )^2 \cdot ( \underline{L} \Omega^{k_2}  \phi )^2 \,  d\omega du dv \\ \lesssim & \sum_{k_1 + k_2 = k} \int_{\mathcal{A}_{\tau_1}^{\tau_2}} (r-M)^{3-\delta} (L \Omega^{k_1} T^2 \phi )^2 \cdot \left( \frac{2r}{D} \underline{L} \Omega^{k_2}  \phi \right)^2 \,  d\omega du dv \\ \lesssim & C E_0 \epsilon^2 \sum_{l \leq 5}\int_{\mathcal{A}_{\tau_1}^{\tau_2}} (r-M)^{3-\delta} (L \Omega^{k_1} T^2 \phi )^2  \,  d\omega du dv \lesssim \frac{C^2 E_0^2 \epsilon^4}{\tau_1^{2+\delta_2}} ,
\end{split}
\end{equation*}
where we used Sobolev's inequality \eqref{est:sobolev}, the pointwise boundedness estimates \eqref{est:boundy} and the Morawetz decay estimate \eqref{dec:mor2}.

\eqref{D2'}: We use \eqref{eq:tcomm} for $m=2$. The last term of \eqref{eq:tcomm} for $m=2$ is similar to the rest so we will not examine it in detail. For the first term we have for $k \leq 5$ that:
\begin{equation}\label{aux:tt1}
\begin{split}
\sum_{\substack{k_1 + k_2 = k \\ m_1 + m_2 = 2}}\int_{\mcR_{\tau_1}^{\tau_2}} (r-M)^{-p} & ( \underline{L} \Omega^k T^2 \phi ) \cdot \frac{2A \cdot D}{r^3} \cdot (L \Omega^{k_1} T^{m_1} \phi ) \cdot \left( \frac{2r}{D} \underline{L} \Omega^{k_2} T^{m_2} \phi \right) \, d\omega du dv \\ = & \sum_{\substack{k_1 + k_2 = k }}\int_{\mcR_{\tau_1}^{\tau_2}} (r-M)^{-p} ( \underline{L} \Omega^k T^2 \phi ) \cdot \frac{2A \cdot D}{r^3} \cdot (L \Omega^{k_1}\phi ) \cdot \left( \frac{2r}{D} \underline{L} \Omega^{k_2} T^2 \phi \right) \, d\omega du dv \\ & + \sum_{\substack{k_1 + k_2 = k }}\int_{\mcR_{\tau_1}^{\tau_2}} (r-M)^{-p} ( \underline{L} \Omega^k T^2 \phi ) \cdot \frac{2A \cdot D}{r^3} \cdot (L \Omega^{k_1} T \phi ) \cdot \left( \frac{2r}{D} \underline{L} \Omega^{k_2} T \phi \right) \, d\omega du dv \\ & +\sum_{\substack{k_1 + k_2 = k }}\int_{\mcR_{\tau_1}^{\tau_2}} (r-M)^{-p} ( \underline{L} \Omega^k T^2 \phi ) \cdot \frac{2A \cdot D}{r^3} \cdot (L \Omega^{k_1} T^2 \phi ) \cdot \left( \frac{2r}{D} \underline{L} \Omega^{k_2} \phi \right) \, d\omega du dv .
\end{split}
\end{equation}
For the first term of the last expression after using Cauchy-Schwarz, we get the following term:
\begin{equation*}
\sum_{k_1 + k_2 = k} \Bigg( \int_{\tau_1}^{\tau_2}  \Big( \int_{\mcN_v^H} (r-M)^{-p} (L \Omega^{k_1} \phi )^2 \cdot ( \underline{L} \Omega^{k_2} T^2 \phi )^2 \, d\omega du \Big)^{1/2} dv \Bigg)^2 .
\end{equation*}
We look at the endpoint cases $p=2$ and $p=0$. In both situations we use that $L \phi$ is integrable in $v$ (using the estimates \eqref{dec:tpsi}). For $p=2$ we get after applying Sobolev's inequality \eqref{est:sobolev} that:
\begin{equation*}
\begin{split}
\sum_{k_1 + k_2 = k} \Bigg( \int_{\tau_1}^{\tau_2}  & \Big( \int_{\mcN_v^H} (r-M)^{-2} (L \Omega^{k_1} \phi )^2 \cdot ( \underline{L} \Omega^{k_2} T^2 \phi )^2 \, d\omega du \Big)^{1/2} dv \Bigg)^2 \\ \lesssim & \sum_{k_1 + k_2 \leq k, k_1 \leq 3} \Bigg( \int_{\tau_1}^{\tau_2} \frac{\sqrt{C E_0} \epsilon}{v^{1+\delta_2/2}}  \Big( \int_{\mcN_v^H} (r-M)^{-2-\delta_2}( \underline{L} \Omega^{k_2} T^2 \phi )^2 \, d\omega du \Big)^{1/2} dv \Bigg)^2 \\ & + \sum_{k_1 + k_2 \leq k, k_1 \geq 4} \Bigg( \int_{\tau_1}^{\tau_2} \frac{\sqrt{C E_0} \epsilon}{v^{1+\delta_2/2}}  \Big( \int_{\mcN_v^H} (r-M)^{-2-\delta_2}( \underline{L} \Omega^{k_2 + 2} T^2 \phi )^2 \, d\omega du \Big)^{1/2} dv \Bigg)^2 \\ \lesssim & \Bigg( \int_{\tau_1}^{\tau_2} \frac{C E_0 \epsilon^2}{v^{1+\delta_2}}  \, dv  \Bigg)^2 \lesssim \frac{C^2 E_0^2 \epsilon^4}{\tau_1^{2\delta_2}} .
\end{split}
\end{equation*} 
On the other hand, using again Sobolev's inequality \eqref{est:sobolev} and the energy decay estimates \eqref{dec:entt} we get that:
\begin{equation*}
\begin{split}
\sum_{k_1 + k_2 = k} \Bigg( \int_{\tau_1}^{\tau_2}  & \Big( \int_{\mcN_v^H} (L \Omega^{k_1} \phi )^2 \cdot ( \underline{L} \Omega^{k_2} T^2 \phi )^2 \, d\omega du \Big)^{1/2} dv \Bigg)^2 \\ \lesssim & \sum_{k_1 + k_2 \leq k, k_1 \leq 3} \Bigg( \int_{\tau_1}^{\tau_2} \frac{\sqrt{C E_0} \epsilon}{v^{1+\delta_2/2}}  \Big( \int_{\mcN_v^H} ( \underline{L} \Omega^{k_2} T^2 \phi )^2 \, d\omega du \Big)^{1/2} dv \Bigg)^2 \\ & + \sum_{k_1 + k_2 \leq k, k_1 \geq 4} \Bigg( \int_{\tau_1}^{\tau_2} \frac{\sqrt{C E_0} \epsilon}{v^{1+\delta_2/2}}  \Big( \int_{\mcN_v^H} ( \underline{L} \Omega^{k_2 + 2} T^2 \phi )^2 \, d\omega du \Big)^{1/2} dv \Bigg)^2 \\ \lesssim & \Bigg( \int_{\tau_1}^{\tau_2} \frac{C E_0 \epsilon^2}{v^{2+\delta_2}}  \, dv   \Bigg)^2 \lesssim \frac{C^2 E_0^2 \epsilon^4}{(1+\tau_1 )^{2+2\delta_2}} \lesssim \frac{C^2 E_0^2 \epsilon^4}{(1+\tau_1 )^{2+\delta_2}} .
\end{split}
\end{equation*} 
The rest of the estimates for $p \in (0,2+\delta_2 )$ follow by interpolation.

For the second term of \eqref{aux:tt1} we have to control again a similar term to the previous ones after using Cauchy-Schwarz. For $p=2$ we have after using Sobolev's inequality \eqref{est:sobolev}, the pointwise decay estimate \eqref{dec:ttpsi} and the energy decay estimates \eqref{dec:ent} that:
\begin{equation*}
\begin{split}
\sum_{k_1 + k_2 = k} \Bigg( \int_{\tau_1}^{\tau_2}  & \Big( \int_{\mcN_v^H} (r-M)^{-2} (L \Omega^{k_1} T\phi )^2 \cdot ( \underline{L} \Omega^{k_2} T \phi )^2 \, d\omega du \Big)^{1/2} dv \Bigg)^2 \\ \lesssim & \sum_{k_1 + k_2 \leq k, k_1 \leq 3} \Bigg( \int_{\tau_1}^{\tau_2} \frac{\sqrt{C E_0} \epsilon}{v^{1/2+\delta_2/2}}  \Big( \int_{\mcN_v^H} (r-M)^{-2}( \underline{L} \Omega^{k_2} T \phi )^2 \, d\omega du \Big)^{1/2} dv \Bigg)^2 \\ & + \sum_{k_1 + k_2 \leq k, k_1 \geq 4} \Bigg( \int_{\tau_1}^{\tau_2} \frac{\sqrt{C E_0} \epsilon}{v^{1/2+\delta_2/2}}  \Big( \int_{\mcN_v^H} (r-M)^{-2}( \underline{L} \Omega^{k_2 + 2} T \phi )^2 \, d\omega du \Big)^{1/2} dv \Bigg)^2 \\ \lesssim & \Bigg( \int_{\tau_1}^{\tau_2} \frac{C E_0 \epsilon^2}{v^{1+\delta_2}}  \, dv  \Bigg)^2 \lesssim \frac{C^2 E_0^2 \epsilon^4}{\tau_1^{2\delta_2}} ,
\end{split}
\end{equation*}
and for the case of $p=0$ using again the same tools as before we have that:
\begin{equation*}
\begin{split}
\sum_{k_1 + k_2 = k} \Bigg( \int_{\tau_1}^{\tau_2}  & \Big( \int_{\mcN_v^H} (L \Omega^{k_1} T\phi )^2 \cdot ( \underline{L} \Omega^{k_2} T \phi )^2 \, d\omega du \Big)^{1/2} dv \Bigg)^2 \\ \lesssim & \sum_{k_1 + k_2 \leq k, k_1 \leq 3} \Bigg( \int_{\tau_1}^{\tau_2} \frac{\sqrt{C E_0} \epsilon}{v^{1/2+\delta_2/2}}  \Big( \int_{\mcN_v^H} ( \underline{L} \Omega^{k_2} T \phi )^2 \, d\omega du \Big)^{1/2} dv \Bigg)^2 \\ & + \sum_{k_1 + k_2 \leq k, k_1 \geq 4} \Bigg( \int_{\tau_1}^{\tau_2} \frac{\sqrt{C E_0} \epsilon}{v^{1/2+\delta_2/2}}  \Big( \int_{\mcN_v^H} ( \underline{L} \Omega^{k_2 + 2} T \phi )^2 \, d\omega du \Big)^{1/2} dv \Bigg)^2 \\ \lesssim & \Bigg( \int_{\tau_1}^{\tau_2} \frac{C E_0 \epsilon^2}{v^{2+\delta_2}}  \, dv \Bigg)^2 \lesssim \frac{C^2 E_0^2 \epsilon^4}{(1+\tau_1 )^{2+2\delta_2}} ,
\end{split}
\end{equation*}
and finally as before the rest of the $p$ estimates follow by interpolation.

Finally for the third term of \eqref{aux:tt1} we have that
\begin{equation*}
\begin{split}
\sum_{\substack{k_1 + k_2 = k }}\int_{\mcR_{\tau_1}^{\tau_2}} & (r-M)^{-p} ( \underline{L} \Omega^k T^2 \phi ) \cdot \frac{2A \cdot D}{r^3} \cdot (L \Omega^{k_1} T^2 \phi ) \cdot \left( \frac{2r}{D} \underline{L} \Omega^{k_2} \phi \right) \, d\omega du dv \\ \lesssim & \int_{\mcR_{\tau_1}^{\tau_2}} (r-M)^{-p}  ( \underline{L} \Omega^k T^2 \phi )^2 \cdot \frac{1}{v^{1+\beta}}  \, d\omega du dv \\ & + \sum_{\substack{k_1 + k_2 = k }}\int_{\mcR_{\tau_1}^{\tau_2}} (r-M)^{-p} D^2 (L \Omega^{k_1} T^2 \phi )^2 \cdot \left( \frac{2r}{D} \underline{L} \Omega^{k_2} \phi \right)^2 \cdot v^{1+\beta} \, d\omega du dv \\ \lesssim & \frac{1}{\tau_1^{\beta}} \sup_{v \in [\tau_1 , \tau_2 ]} \int_{\mcN_v^H} (r-M)^{-p}  ( \underline{L} \Omega^k T^2 \phi )^2 \cdot \frac{1}{v^{1+\beta}} \, d\omega du \\ &  + \sum_{\substack{k_1 + k_2 = k }} \int_{\tau_1}^{\tau_2} \int_{\mathbb{S}^2} (L \Omega^{k_1 + 2} T^2 \phi )^2 \cdot v^{1+\beta} \, d\omega dv \cdot \sup_{v \in [\tau_1 , \tau_2 ]} \int_{\mcN_v^H} (r-M)^{-p} ( \underline{L} \Omega^{k_2 } \phi )^2 d\omega du \mbox{  if $k \leq 3$, or} \\ \lesssim & \frac{1}{\tau_1^{\beta}} \sup_{v \in [\tau_1 , \tau_2 ]} \int_{\mcN_v^H} (r-M)^{-p}  ( \underline{L} \Omega^k T^2 \phi )^2 \cdot \frac{1}{v^{1+\beta}} \, d\omega du \\ &  + \sum_{\substack{k_1 + k_2 = k }} \int_{\tau_1}^{\tau_2} \int_{\mathbb{S}^2} (L \Omega^{k_1 } T^2 \phi )^2 \cdot v^{1+\beta} \, d\omega dv \cdot \sup_{v \in [\tau_1 , \tau_2 ]} \int_{\mcN_v^H} (r-M)^{-p} ( \underline{L} \Omega^{k_2 + 2} \phi )^2 d\omega du \mbox{  if $k \geq 4$}  \\ \lesssim & \frac{1}{\tau_1^{\beta}} \sup_{v \in [\tau_1 , \tau_2 ]} \int_{\mcN_v^H} (r-M)^{-p}  ( \underline{L} \Omega^k T^2 \phi )^2 \cdot \frac{1}{v^{1+\beta}} \, d\omega du \\ & + C E_0 \epsilon^2 \sup_{v \ in [\tau_1 , \tau_2 ]} \int_{\mcN_v^H} (r-M)^{-p} ( \underline{L} \Omega^{m} \phi )^2 d\omega du  ,
\end{split}
\end{equation*}
where we used Sobolev's inequality and the auxiliary estimate \eqref{est:aux1} for $m=2$. In the last term of the last expression for the case $p=2+\delta_2$ we use Hardy's inequality \eqref{hardy} and we have that
\begin{equation*}
\begin{split}
C E_0 \epsilon^2 \sum_{m \leq 5} \int_{\mcN_v^H} (r-M)^{-2} & ( \underline{L} \Omega^{m} \phi )^2 d\omega du \\ \lesssim & C E_0 \epsilon^2 \sum_{m \leq 5} \int_{\mcN_v^H} \left( \underline{L} \left( \frac{2r}{D}  \underline{L} \Omega^{m} \phi \right) \right)^2 d\omega du  \lesssim C E_0 \epsilon^2 \frac{C E_0 \epsilon^2}{\tau_1^{1-\delta_1}} , 
\end{split}
\end{equation*}
by using the energy decay estimates \eqref{dec:en0}. On the other hand for the same term for $p=0$ we have that
\begin{equation*}
C E_0 \epsilon^2 \sum_{m \leq 5} \int_{\mcN_v^H}   ( \underline{L} \Omega^{m} \phi )^2 d\omega du  \lesssim C E_0 \epsilon^2 \frac{C E_0 \epsilon^2}{\tau_1^{3-\delta_1}} , 
\end{equation*}
where we used the energy decay estimates \eqref{dec:en0}. Both of the last two estimates are better than desired, and the rest of the $p$ range follows by interpolation.

\eqref{D3'}: This can be done similarly to \eqref{A3'}, \eqref{B3'} and \eqref{C3'}.

\eqref{A6'}, \eqref{C5'} and \eqref{D5'}: All these terms can be treated similarly to the \eqref{B5'} term. 

\eqref{E1'}: We deal first with the term close to the horizon and we have that:
\begin{equation*}
\begin{split}
\int_{\mathcal{A}_{\tau_1}^{\tau_2}}  (r-M)^{-1-\delta} & D^2 | \Omega^k T^3 F |^2 \, d\omega du dv \lesssim \sum_{\substack{k_1 + k_2 = k \\ m_1 + m_2 = 3}} \int_{\mathcal{A}_{\tau_1}^{\tau_2} }\frac{1}{(r-M)^{1+\delta}} (L \Omega^{k_1} T^{m_1} \phi )^2 \cdot ( \underline{L} \Omega^{k_2} T^{m_2} \phi )^2 \,  d\omega du dv \\ & + \sum_{\substack{k_1 + k_2 = k \\ m_1 + m_2 = 3}} \int_{\mathcal{A}_{\tau_1}^{\tau_2} } (r-M)^{3-\delta} ( \Omega^{k_1} T^{m_1} \phi )^2 \cdot ( \underline{L} \Omega^{k_2} T^{m_2} \phi )^2  \,  d\omega du dv \\ & + \sum_{\substack{k_1 + k_2 = k \\ m_1 + m_2 = 3}} \int_{\mathcal{A}_{\tau_1}^{\tau_2} } (r-M)^{3-\delta} ( \Omega^{k_1} T^{m_1} \phi )^2 \cdot ( L \Omega^{k_2} T^{m_2} \phi )^2  \,  d\omega du dv \\ & + \sum_{\substack{k_1 + k_2 = k \\ m_1 + m_2 = 3}} \int_{\mathcal{A}_{\tau_1}^{\tau_2} } (r-M)^{7-\delta} ( \Omega^{k_1} T^{m_1} \phi )^2 \cdot (  \Omega^{k_2} T^{m_2} \phi )^2   \,  d\omega du dv \\ & + \sum_{\substack{k_1 + k_2 = k \\ m_1 + m_2 = 3}} \int_{\mathcal{A}_{\tau_1}^{\tau_2} } (r-M)^{3-\delta} | \slashed{\nabla} \Omega^{k_1} T^{m_1} \phi |^2 \cdot | \slashed{\nabla}  \Omega^{k_2} T^{m_2} \phi |^2  \,  d\omega du dv .
\end{split}
\end{equation*}
For the first term of the above estimate we have for $m_1 =0$ and $m_2 = 3$, for $\delta \leq \delta_2$ and for any $v$ that:
\begin{equation*}
\begin{split}
\sum_{\substack{k_1 + k_2 = k }} \int_{\mcN_v^H} & \frac{1}{(r-M)^{1+\delta}} (L \Omega^{k_1} \phi )^2 \cdot ( \underline{L} \Omega^{k_2} T^3 \phi )^2  \,  d\omega du  \\ \lesssim &  \frac{C E_0 \epsilon^2}{v^{2+\delta_2}} \sum_{\substack{l \leq 5  }} \int_{\mcN_v^H}  \frac{1}{(r-M)^{1+\delta}}  ( \underline{L} \Omega^{l} T^3 \phi )^2  \,  d\omega du  \lesssim  \frac{C^2 E_0^2 \epsilon^4}{v^{2+\delta_2}} ,
\end{split}
\end{equation*}
where we used Sobolev's inequality \eqref{est:sobolev}, the pointwise decay estimate \eqref{dec:tpsi}, and the energy decay estimates \eqref{dec:enttt} for $p = 1+\delta \leq 1+\delta_2$. This implies that
\begin{equation*}
\sum_{\substack{k_1 + k_2 = k }} \int_{\mcN_v^H}  \frac{1}{(r-M)^{1+\delta}} (L \Omega^{k_1} \phi )^2 \cdot ( \underline{L} \Omega^{k_2} T^3 \phi )^2  \,  d\omega du dv \lesssim \int_{\tau_1}^{\tau_2} \frac{C^2 E_0^2 \epsilon^4}{v^{2+\delta_2}} \, dv \lesssim \frac{C^2 E_0^2 \epsilon^4}{\tau_1^{1+\delta_2}} .
\end{equation*}
For the case $m_1 = 1$ and $m_2 =2$ for any $v$ we have that (again as $\delta \leq \delta_2$):
\begin{equation*}
\begin{split}
\sum_{\substack{k_1 + k_2 = k }} \int_{\mcN_v^H} & \frac{1}{(r-M)^{1+\delta}} (L \Omega^{k_1} T\phi )^2 \cdot ( \underline{L} \Omega^{k_2} T^2 \phi )^2  \,  d\omega du \\ \lesssim & \frac{C E_0 \epsilon^2}{v^{1+\delta_2}} \sum_{\substack{l \leq 5 }} \int_{\mcN_v^H} \frac{1}{(r-M)^{1+\delta}}  ( \underline{L} \Omega^{l} T^2 \phi )^2  \,  d\omega du \lesssim \frac{C^2 E_0^2 \epsilon^2}{v^{2+\delta_2}} ,
\end{split}
\end{equation*}
where we used Sobolev's inequality \eqref{est:sobolev}, the pointwise decay estimates \eqref{dec:ttpsi}, and the energy decay estimates \eqref{dec:entt}. This implies that
\begin{equation*}
\sum_{\substack{k_1 + k_2 = k }} \int_{\mathcal{A}_{\tau_1}^{\tau_2}}  \frac{1}{(r-M)^{1+\delta}} (L \Omega^{k_1} T\phi )^2 \cdot ( \underline{L} \Omega^{k_2} T^2 \phi )^2  \,  d\omega du dv \lesssim \int_{\tau_1}^{\tau_2} \frac{C^2 E_0^2 \epsilon^4}{v^{2+\delta_2}} \, dv \lesssim \frac{C^2 E_0^2 \epsilon^4}{\tau_1^{1+\delta_2}} .
\end{equation*}
For the case $m_1 =2$ and $m_1 =1$ for any $v$ we have that (again as $\delta \leq \delta_2$):
\begin{equation*}
\begin{split}
\sum_{\substack{k_1 + k_2 = k }} \int_{\mcN_v^H} & \frac{1}{(r-M)^{1+\delta}} (L \Omega^{k_1} T^2 \phi )^2 \cdot ( \underline{L} \Omega^{k_2} T \phi )^2  \,  d\omega du \\ \lesssim & \frac{C E_0 \epsilon^2}{v^{1/4+\delta_2}} \sum_{\substack{l \leq 5 }} \int_{\mcN_v^H} \frac{1}{(r-M)^{1+\delta}}  ( \underline{L} \Omega^{l} T \phi )^2  \,  d\omega du \lesssim \frac{C^2 E_0^2 \epsilon^2}{v^{9/4+\delta_2}} ,
\end{split}
\end{equation*}
where we used Sobolev's inequality \eqref{est:sobolev}, the pointwise decay estimates \eqref{dec:ttpsi}, and the energy decay estimates \eqref{dec:entt}. This implies that
\begin{equation*}
\sum_{\substack{k_1 + k_2 = k }} \int_{\mathcal{A}_{\tau_1}^{\tau_2}}  \frac{1}{(r-M)^{1+\delta}} (L \Omega^{k_1} T^2 \phi )^2 \cdot ( \underline{L} \Omega^{k_2} T \phi )^2  \,  d\omega du dv \lesssim \int_{\tau_1}^{\tau_2} \frac{C^2 E_0^2 \epsilon^4}{v^{9/4+\delta_2}} \, dv \lesssim \frac{C^2 E_0^2 \epsilon^4}{\tau_1^{5/4+\delta_2}} ,
\end{equation*}
the last estimate being better than desired. Finally for $m_1 = 3$ and $m_2 = 0$ we have that
\begin{equation*}
\begin{split}
\sum_{\substack{k_1 + k_2 = k }} \int_{\mathcal{A}_{\tau_1}^{\tau_2}} & \frac{1}{(r-M)^{1+\delta}} (L \Omega^{k_1} T^3 \phi )^2 \cdot ( \underline{L} \Omega^{k_2}  \phi )^2  \,  d\omega du dv \\ \lesssim & \sum_{\substack{k_1 + k_2 = k }} \int_{\mathcal{A}_{\tau_1}^{\tau_2}}  (r-M)^{3-\delta} (L \Omega^{k_1} T^3 \phi )^2 \cdot \left( \frac{2r}{D} \underline{L} \Omega^{k_2}  \phi \right)^2  \,  d\omega du dv \\ \lesssim & C E_0 \epsilon^2 \sum_{l \leq 5} \int_{\mathcal{A}_{\tau_1}^{\tau_2}}  (r-M)^{3-\delta} (L \Omega^{k_1} T^3 \phi )^2   \,  d\omega du dv \lesssim \frac{C^2 E_0^2 \epsilon^4}{\tau_1^{1+\delta_2}} ,
\end{split}
\end{equation*}
where we used Sobolev's inequality \eqref{est:sobolev}, the boundedness estimate \eqref{est:boundy}, and the Morawetz decay estimate \eqref{dec:mor3}.

\eqref{E2'}: Now we use \eqref{eq:tcomm} for $m=3$ and we have for any $k\leq 5$ for the first term that:
\begin{equation}\label{aux:t3}
\begin{split}
\sum_{\substack{k_1 + k_2 = k \\ m_1 + m_2 = 3}}\int_{\mcR_{\tau_1}^{\tau_2}} (r-M)^{-p} & ( \underline{L} \Omega^k T^3 \phi ) \cdot \frac{2A \cdot D}{r^3} \cdot (L \Omega^{k_1} T^{m_1} \phi ) \cdot \left( \frac{2r}{D} \underline{L} \Omega^{k_2} T^{m_2} \phi \right) \, d\omega du dv \\ = & \sum_{\substack{k_1 + k_2 = k }}\int_{\mcR_{\tau_1}^{\tau_2}} (r-M)^{-p} ( \underline{L} \Omega^k T^3 \phi ) \cdot \frac{2A \cdot D}{r^3} \cdot (L \Omega^{k_1}\phi ) \cdot \left( \frac{2r}{D} \underline{L} \Omega^{k_2} T^3 \phi \right) \, d\omega du dv \\ & + \sum_{\substack{k_1 + k_2 = k }}\int_{\mcR_{\tau_1}^{\tau_2}} (r-M)^{-p} ( \underline{L} \Omega^k T^3 \phi ) \cdot \frac{2A \cdot D}{r^3} \cdot (L \Omega^{k_1} T \phi ) \cdot \left( \frac{2r}{D} \underline{L} \Omega^{k_2} T^2 \phi \right) \, d\omega du dv \\ & +\sum_{\substack{k_1 + k_2 = k }}\int_{\mcR_{\tau_1}^{\tau_2}} (r-M)^{-p} ( \underline{L} \Omega^k T^3 \phi ) \cdot \frac{2A \cdot D}{r^3} \cdot (L \Omega^{k_1} T^2 \phi ) \cdot \left( \frac{2r}{D} \underline{L} \Omega^{k_2} T \phi \right) \, d\omega du dv \\ & + \sum_{\substack{k_1 + k_2 = k }}\int_{\mcR_{\tau_1}^{\tau_2}} (r-M)^{-p} ( \underline{L} \Omega^k T^3 \phi ) \cdot \frac{2A \cdot D}{r^3} \cdot (L \Omega^{k_1} T^3 \phi ) \cdot \left( \frac{2r}{D} \underline{L} \Omega^{k_2}  \phi \right) \, d\omega du dv  .
\end{split}
\end{equation}
For the first term of the last expression after applying Cauchy-Shcwarz we get the following term for $p=1+\delta_2$:
\begin{equation*}
\begin{split}
\sum_{k_1 + k_2 = k} \Bigg( \int_{\tau_1}^{\tau_2}  & \Big( \int_{\mcN_v^H} (r-M)^{-1-\delta_2} (L \Omega^{k_1} \phi )^2 \cdot ( \underline{L} \Omega^{k_2} T^3 \phi )^2 \, d\omega du \Big)^{1/2} dv \Bigg)^2 \\ \lesssim & \sum_{k_1 + k_2 \leq k, k_1 \leq 3} \Bigg( \int_{\tau_1}^{\tau_2} \frac{\sqrt{C E_0} \epsilon}{v^{1+\delta_2/2}}  \Big( \int_{\mcN_v^H} (r-M)^{-1-\delta_2}( \underline{L} \Omega^{k_2} T^3 \phi )^2 \, d\omega du \Big)^{1/2} dv \Bigg)^2 \\ & + \sum_{k_1 + k_2 \leq k, k_1 \geq 4} \Bigg( \int_{\tau_1}^{\tau_2} \frac{\sqrt{C E_0} \epsilon}{v^{1+\delta_2/2}}  \Big( \int_{\mcN_v^H} (r-M)^{-1-\delta_2}( \underline{L} \Omega^{k_2 + 2} T^3 \phi )^2 \, d\omega du \Big)^{1/2} dv \Bigg)^2 \\ \lesssim & \Bigg( \int_{\tau_1}^{\tau_2} \frac{C E_0 \epsilon^2}{v^{1+\delta_2/2}}  \, dv  \Bigg)^2 \lesssim C^2 E_0^2 \epsilon^4 ,
\end{split}
\end{equation*}
where we used Sobolev's inequality \eqref{est:sobolev}, the pointwise decay estimates \eqref{dec:tpsi}, and the energy decay estimates \eqref{dec:enttt}.  On the other hand using the same estimates we have for the $p=0$ case that:
\begin{equation*}
\begin{split}
\sum_{k_1 + k_2 = k} \Bigg( \int_{\tau_1}^{\tau_2}  & \Big( \int_{\mcN_v^H}  (L \Omega^{k_1} \phi )^2 \cdot ( \underline{L} \Omega^{k_2} T^3 \phi )^2 \, d\omega du \Big)^{1/2} dv \Bigg)^2 \\ \lesssim & \sum_{k_1 + k_2 \leq k, k_1 \leq 3} \Bigg( \int_{\tau_1}^{\tau_2} \frac{\sqrt{C E_0} \epsilon}{v^{1+\delta_2/2}}  \Big( \int_{\mcN_v^H} ( \underline{L} \Omega^{k_2} T^3 \phi )^2 \, d\omega du \Big)^{1/2} dv \Bigg)^2 \\ & + \sum_{k_1 + k_2 \leq k, k_1 \geq 4} \Bigg( \int_{\tau_1}^{\tau_2} \frac{\sqrt{C E_0} \epsilon}{v^{1+\delta_2/2}}  \Big( \int_{\mcN_v^H} ( \underline{L} \Omega^{k_2 + 2} T^3 \phi )^2 \, d\omega du \Big)^{1/2} dv \Bigg)^2 \\ \lesssim & \Bigg( \int_{\tau_1}^{\tau_2} \frac{C E_0 \epsilon^2}{v^{3/2+\delta_2}}  \, dv  \Bigg)^2 \lesssim \frac{C^2 E_0^2 \epsilon^4}{(1+\tau_1 )^{1+2\delta_1}} \lesssim \frac{C^2 E_0^2 \epsilon^4}{(1+\tau_1 )^{1+\delta_1}}  ,
\end{split}
\end{equation*}
and the rest of the $p$ range follows by interpolation.

For the second term of \eqref{aux:t3} we work similarly and for $p=1+\delta_2$ we get the term:
\begin{equation*}
\begin{split}
\sum_{k_1 + k_2 = k} \Bigg( \int_{\tau_1}^{\tau_2}  & \Big( \int_{\mcN_v^H} (r-M)^{-1-\delta_2} (L \Omega^{k_1} T \phi )^2 \cdot ( \underline{L} \Omega^{k_2} T^2 \phi )^2 \, d\omega du \Big)^{1/2} dv \Bigg)^2 \\ \lesssim & \sum_{k_1 + k_2 \leq k, k_1 \leq 3} \Bigg( \int_{\tau_1}^{\tau_2} \frac{\sqrt{C E_0} \epsilon}{v^{1/2+\delta_2/2}}  \Big( \int_{\mcN_v^H} (r-M)^{-1-\delta_2}( \underline{L} \Omega^{k_2} T^2 \phi )^2 \, d\omega du \Big)^{1/2} dv \Bigg)^2 \\ & + \sum_{k_1 + k_2 \leq k, k_1 \geq 4} \Bigg( \int_{\tau_1}^{\tau_2} \frac{\sqrt{C E_0} \epsilon}{v^{1/2+\delta_2/2}}  \Big( \int_{\mcN_v^H} (r-M)^{-1-\delta_2}( \underline{L} \Omega^{k_2 + 2} T^2 \phi )^2 \, d\omega du \Big)^{1/2} dv \Bigg)^2 \\ \lesssim & \Bigg( \int_{\tau_1}^{\tau_2} \frac{C E_0 \epsilon^2}{v^{1+\delta_2/2}}  \, dv  \Bigg)^2 \lesssim C^2 E_0^2 \epsilon^4 ,
\end{split}
\end{equation*}
where we used Sobolev's inequality, the pointwise decay estimate \eqref{dec:ttpsi} and the energy decay estimates \eqref{dec:entt}. Using the same estimates we get for $p=0$ that:
\begin{equation*}
\begin{split}
\sum_{k_1 + k_2 = k} \Bigg( \int_{\tau_1}^{\tau_2}  & \Big( \int_{\mcN_v^H}  (L \Omega^{k_1} T \phi )^2 \cdot ( \underline{L} \Omega^{k_2} T^2 \phi )^2 \, d\omega du \Big)^{1/2} dv \Bigg)^2 \\ \lesssim & \sum_{k_1 + k_2 \leq k, k_1 \leq 3} \Bigg( \int_{\tau_1}^{\tau_2} \frac{\sqrt{C E_0} \epsilon}{v^{1/2+\delta_2/2}}  \Big( \int_{\mcN_v^H} ( \underline{L} \Omega^{k_2} T^2 \phi )^2 \, d\omega du \Big)^{1/2} dv \Bigg)^2 \\ & + \sum_{k_1 + k_2 \leq k, k_1 \geq 4} \Bigg( \int_{\tau_1}^{\tau_2} \frac{\sqrt{C E_0} \epsilon}{v^{1/2+\delta_2/2}}  \Big( \int_{\mcN_v^H} ( \underline{L} \Omega^{k_2 + 2} T^2 \phi )^2 \, d\omega du \Big)^{1/2} dv \Bigg)^2 \\ \lesssim & \Bigg( \int_{\tau_1}^{\tau_2} \frac{C E_0 \epsilon^2}{v^{3/2+\delta_2}}  \, dv  \Bigg)^2 \lesssim \frac{C^2 E_0^2 \epsilon^4}{(1+\tau_1 )^{1+2\delta_1}} \lesssim \frac{C^2 E_0^2 \epsilon^4}{(1+\tau_1 )^{1+\delta_1}}  ,
\end{split}
\end{equation*}
and the rest of the $p$ range follows by interpolation.

For the third term of \eqref{aux:t3} we work similarly using Cauchy-Schwarz, Sobolev's inequality \eqref{est:sobolev}, the pointwise decay estimate \eqref{dec:tttpsi} and the energy decay estimates \eqref{dec:ent} and we have for $p=1+\delta_2$ that:
\begin{equation*}
\begin{split}
\sum_{k_1 + k_2 = k} \Bigg( \int_{\tau_1}^{\tau_2}  & \Big( \int_{\mcN_v^H} (r-M)^{-1-\delta_2} (L \Omega^{k_1} T^2 \phi )^2 \cdot ( \underline{L} \Omega^{k_2} T \phi )^2 \, d\omega du \Big)^{1/2} dv \Bigg)^2 \\ \lesssim & \sum_{k_1 + k_2 \leq k, k_1 \leq 3} \Bigg( \int_{\tau_1}^{\tau_2} \frac{\sqrt{C E_0} \epsilon}{v^{1/4+\delta_2/2}}  \Big( \int_{\mcN_v^H} (r-M)^{-1-\delta_2}( \underline{L} \Omega^{k_2} T \phi )^2 \, d\omega du \Big)^{1/2} dv \Bigg)^2 \\ & + \sum_{k_1 + k_2 \leq k, k_1 \geq 4} \Bigg( \int_{\tau_1}^{\tau_2} \frac{\sqrt{C E_0} \epsilon}{v^{1/4+\delta_2/2}}  \Big( \int_{\mcN_v^H} (r-M)^{-1-\delta_2}( \underline{L} \Omega^{k_2 + 2} T^2 \phi )^2 \, d\omega du \Big)^{1/2} dv \Bigg)^2 \\ \lesssim & \Bigg( \int_{\tau_1}^{\tau_2} \frac{C E_0 \epsilon^2}{v^{5/4+\delta_2/2}}  \, dv  \Bigg)^2 \lesssim \frac{C^2 E_0^2 \epsilon^4}{(1+\tau_1 )^{1/4+\delta_2 / 2}} ,
\end{split}
\end{equation*}
which is better than desired. By using the same estimates for $p=0$ we have that:
\begin{equation*}
\begin{split}
\sum_{k_1 + k_2 = k} \Bigg( \int_{\tau_1}^{\tau_2}  & \Big( \int_{\mcN_v^H}  (L \Omega^{k_1} T^2 \phi )^2 \cdot ( \underline{L} \Omega^{k_2} T \phi )^2 \, d\omega du \Big)^{1/2} dv \Bigg)^2 \\ \lesssim & \sum_{k_1 + k_2 \leq k, k_1 \leq 3} \Bigg( \int_{\tau_1}^{\tau_2} \frac{\sqrt{C E_0} \epsilon}{v^{1/4+\delta_2/2}}  \Big( \int_{\mcN_v^H} ( \underline{L} \Omega^{k_2} T \phi )^2 \, d\omega du \Big)^{1/2} dv \Bigg)^2 \\ & + \sum_{k_1 + k_2 \leq k, k_1 \geq 4} \Bigg( \int_{\tau_1}^{\tau_2} \frac{\sqrt{C E_0} \epsilon}{v^{1/4+\delta_2/2}}  \Big( \int_{\mcN_v^H} ( \underline{L} \Omega^{k_2 + 2} T^2 \phi )^2 \, d\omega du \Big)^{1/2} dv \Bigg)^2 \\ \lesssim & \Bigg( \int_{\tau_1}^{\tau_2} \frac{C E_0 \epsilon^2}{v^{7/4+\delta_2}}  \, dv  \Bigg)^2 \lesssim \frac{C^2 E_0^2 \epsilon^4}{(1+\tau_1 )^{3/2+2\delta_1}} \lesssim \frac{C^2 E_0^2 \epsilon^4}{(1+\tau_1 )^{3/2+\delta_1}}  ,
\end{split}
\end{equation*}
and the rest of the $p$ range follows by interpolation.

For the term 
$$ \sum_{\substack{k_1 + k_2 = k }}\int_{\mcR_{\tau_1}^{\tau_2}} (r-M)^{-p} ( \underline{L} \Omega^k T^3 \phi ) \cdot \frac{2A \cdot D}{r^3} \cdot (L \Omega^{k_1} T^3 \phi ) \cdot \left( \frac{2r}{D} \underline{L} \Omega^{k_2}  \phi \right) \, d\omega du dv $$ 
we work similarly as for the third term of \eqref{aux:tt1}, now using the auxiliary estimate \eqref{est:aux1} for $m=3$.

\eqref{E3'}: We examine in detail only the term involving the $L$ and $\underline{L}$ derivatives (since the rest are either better os similar) which is bounded by:
$$ \sum_{\substack{k_1 + k_2 = k \\ m_1 + m_2 = 3}} \int_{\mathcal{A}_{\tau_1}^{\tau_2}} (r-M)^{-2} (L \Omega^{k_1} T^{m_1} \phi )^2 \cdot (\underline{L} \Omega^{k_2} T^{m_2} \phi )^2 \cdot v^{1+\beta} \, d\omega du dv , $$
and for which we have that
\begingroup
\allowdisplaybreaks
\begin{align*}
\sum_{\substack{k_1 + k_2 = k \\ m_1 + m_2 = 3}} \int_{\mathcal{A}_{\tau_1}^{\tau_2}} (r-M)^{-2} & (L \Omega^{k_1} T^{m_1} \phi )^2 \cdot (\underline{L} \Omega^{k_2} T^{m_2} \phi )^2 \cdot v^{1+\beta} \, d\omega du dv \\ \lesssim & \sum_{\substack{l_1 \leq 5 , l_2 \leq 5 \\ m_1 + m_2 = 3}} \int_{\tau_1}^{\tau_2} \int_{\mathbb{S}^2} \sup_u (L \Omega^{l_1} T^{m_1} \phi )^2 \cdot v^{1+\beta} \, d\omega dv \cdot \sup_{v \in [\tau_1 , \tau_2 ]} \int_{\mcN_v^H} (r-M)^{-2} ( \underline{L} \Omega^{l_2} T^{m_2} \phi )^2 \, d\omega du \\ \lesssim & C^2 E_0^2 \epsilon^4 , 
\end{align*}
\endgroup
where we used the auxiliary estimates of \eqref{est:aux1} for $m=0,1,2,3$ (as $\beta < \delta_2$), the energy decay estimates \eqref{dec:enp}, \eqref{dec:entp}, \eqref{dec:enttp} for $p=2$, and the energy boundedness estimate \eqref{dec:entttp1}.

\eqref{F1'}: We examine the part close to the horizon and for any $k \leq 5$ we have that:
\begin{equation*}
\begin{split}
\int_{\mathcal{A}_{\tau_1}^{\tau_2} } \frac{1}{(r-M)^{1+\delta}} & D^2 | \Omega^k T^4 F |^2 \, d\omega du dv \\ \lesssim & \sum_{\substack{k_1 + k_2 = k \\ m_1 + m_2 = 4}} \int_{\mathcal{A}_{\tau_1}^{\tau_2} }\frac{1}{(r-M)^{1+\delta}} (L \Omega^{k_1} T^{m_1} \phi )^2 \cdot ( \underline{L} \Omega^{k_2} T^{m_2} \phi )^2  \,  d\omega du dv \\ & + \sum_{\substack{k_1 + k_2 = k \\ m_1 + m_2 = 4}} \int_{\mathcal{A}_{\tau_1}^{\tau_2} } (r-M)^{3-\delta} ( \Omega^{k_1} T^{m_1} \phi )^2 \cdot ( \underline{L} \Omega^{k_2} T^{m_2} \phi )^2  \,  d\omega du dv \\ & + \sum_{\substack{k_1 + k_2 = k \\ m_1 + m_2 = 4}} \int_{\mathcal{A}_{\tau_1}^{\tau_2} } (r-M)^{3-\delta} ( \Omega^{k_1} T^{m_1} \phi )^2 \cdot ( L \Omega^{k_2} T^{m_2} \phi )^2   \,  d\omega du dv \\ & + \sum_{\substack{k_1 + k_2 = k \\ m_1 + m_2 = 4}} \int_{\mathcal{A}_{\tau_1}^{\tau_2} } (r-M)^{7-\delta} ( \Omega^{k_1} T^{m_1} \phi )^2 \cdot (  \Omega^{k_2} T^{m_2} \phi )^2   \,  d\omega du dv \\ & + \sum_{\substack{k_1 + k_2 = k \\ m_1 + m_2 = 4}} \int_{\mathcal{A}_{\tau_1}^{\tau_2} } (r-M)^{3-\delta} | \slashed{\nabla} \Omega^{k_1} T^{m_1} \phi |^2 \cdot | \slashed{\nabla}  \Omega^{k_2} T^{m_2} \phi |^2  \,  d\omega du dv .
\end{split}
\end{equation*}
For the first term of the above expression for $m_1 = 0$ and $m_2 = 4$ we have for any $v$ and since $\delta \leq \delta_2$ that:
\begin{equation*}
\begin{split}
\sum_{\substack{k_1 + k_2 = k }} \int_{\mcN_v^H } & \frac{1}{(r-M)^{1+\delta}} (L \Omega^{k_1}  \phi )^2 \cdot ( \underline{L} \Omega^{k_2} T^4 \phi )^2  \,  d\omega du \\ \lesssim & \frac{C E_0 \epsilon^2}{v^{2+\delta_2}} \sum_{l \leq 5}  \int_{\mcN_v^H }  \frac{1}{(r-M)^{1+\delta}}  ( \underline{L} \Omega^{l} T^4 \phi )^2  \,  d\omega du  \lesssim  \frac{C^2 E_0^2 \epsilon^4}{v^{2+\delta_2}} ,
\end{split}
\end{equation*}
where we used Sobolev's inequality \eqref{est:sobolev}, the pointwise decay estimate \eqref{dec:tpsi}, and the energy decay estimates \eqref{dec:entttt}. This implies that:
\begin{equation*}
\sum_{\substack{k_1 + k_2 = k }} \int_{\mathcal{A}_{\tau_1}^{\tau_2} }  \frac{1}{(r-M)^{1+\delta}} (L \Omega^{k_1}  \phi )^2 \cdot ( \underline{L} \Omega^{k_2} T^4 \phi )^2  \,  d\omega du dv \lesssim \int_{\tau_1}^{\tau_2} \frac{C^2 E_0^2 \epsilon^4}{v^{2+\delta_2}} \, dv \lesssim \frac{C^2 E_0^2 \epsilon^4}{\tau_1^{1+\delta_2}} \lesssim \frac{C^2 E_0^2 \epsilon^4}{\tau_1} .
\end{equation*}
For the case $m_1 = 1 $ and $m_2 = 3$ we have that:
\begin{equation*}
\begin{split}
\sum_{\substack{k_1 + k_2 = k }} \int_{\mathcal{A}_{\tau_1}^{\tau_2} } & \frac{1}{(r-M)^{1+\delta}} (L \Omega^{k_1} T \phi )^2 \cdot ( \underline{L} \Omega^{k_2} T^3 \phi )^2  \,  d\omega du dv \\ = & \sum_{\substack{k_1 + k_2 = k }} \int_{\mathcal{A}_{\tau_1}^{\tau_2} }  \frac{1}{(r-M)^{1+\delta}} (L \Omega^{k_1} T \phi )^2 \cdot ( \underline{L} \Omega^{k_2} T^3 \phi )^2 \cdot \frac{v^{1+\beta}}{v^{1+\beta}}  \,  d\omega du dv \\ \lesssim & \sum_{\substack{ l_1 \leq 5 \\ l_2 \leq 5}} \int_{\tau_1}^{\tau_2} \int_{\mathbb{S}^2} \sup_u ( L \Omega^{l_1} T \phi )^2 \cdot v^{1+\beta} \, d\omega dv \cdot \sup_{v \in [\tau_1 , \tau_2 ]} \int_{\mcN_v^H}\frac{1}{(r-M)^{1+\delta}} ( \underline{L} \Omega^{k_2} T^3 \phi )^2 \cdot \frac{1}{v^{1+\beta}}  \,  d\omega du \\ \lesssim & C E_0 \epsilon^2 \cdot \frac{1}{\tau_1^{1+\beta}} \sup_{v \in [\tau_1 , \tau_2 ]} \int_{\mcN_v^H}\frac{1}{(r-M)^{1+\delta}} ( \underline{L} \Omega^{k_2} T^3 \phi )^2 \,  d\omega du \\ \lesssim & \frac{C^2 E_0^2 \epsilon^4}{\tau_1^{1+\beta}} \lesssim \frac{C^2 E_0^2 \epsilon^4}{\tau_1} ,
\end{split}
\end{equation*}
where we used Sobolev's inequality \eqref{est:sobolev}, the auxiliary estimate \eqref{est:aux1} for $m=1$, and the energy decay estimates \eqref{dec:enttt}.

For the case $m_1 = 2 $ and $m_2 = 2$ we work as in the previous case now using the auxiliary estimate \eqref{est:aux1} for $m=2$.  In this case we get even better decay of rate $\frac{C^2 E_0^2 \epsilon^4}{\tau_1^2}$. For $m_1 = 3$ and $m_2 = 1$ we have that:
\begin{equation*}
\begin{split}
\sum_{\substack{k_1 + k_2 = k }} \int_{\mathcal{A}_{\tau_1}^{\tau_2} } & \frac{1}{(r-M)^{1+\delta}} (L \Omega^{k_1} T^3 \phi )^2 \cdot ( \underline{L} \Omega^{k_2} T \phi )^2  \,  d\omega du dv \\ \lesssim & \sum_{\substack{k_1 + k_2 = k }} \int_{\mathcal{A}_{\tau_1}^{\tau_2} } (r-M)^{3-\delta} (L \Omega^{k_1} T^3 \phi )^2 \cdot \left( \frac{2r}{D} \underline{L} \Omega^{k_2} T \phi \right)^2  \,  d\omega du dv \\ \lesssim & C E_0 \epsilon^2 \sum_{l \leq 5} \int_{\mathcal{A}_{\tau_1}^{\tau_2} } (r-M)^{3-\delta} (L \Omega^{l} T^3 \phi )^2 \, d\omega du dv  \lesssim  \frac{C^2 E_0^2 \epsilon^4}{\tau_1^{1+\delta_2}} , 
\end{split}
\end{equation*}
where we used Sobolev's inequality \eqref{est:sobolev}, the boundedness estimate \eqref{est:boundyt}, and the Morawetz decay estimates \eqref{dec:mor3}, as $\delta \leq \delta_2$ is small enough. For the case $m_1 = 4$ and $m_2 = 0$ we have that:
\begin{equation*}
\begin{split}
\sum_{\substack{k_1 + k_2 = k }} \int_{\mathcal{A}_{\tau_1}^{\tau_2} } & \frac{1}{(r-M)^{1+\delta}} (L \Omega^{k_1} T^4 \phi )^2 \cdot ( \underline{L} \Omega^{k_2} \phi )^2  \,  d\omega du dv \\ \lesssim & \sum_{\substack{k_1 + k_2 = k }} \int_{\mathcal{A}_{\tau_1}^{\tau_2} } (r-M)^{3-\delta} (L \Omega^{k_1} T^4 \phi )^2 \cdot \left( \frac{2r}{D} \underline{L} \Omega^{k_2} \phi \right)^2  \,  d\omega du dv \\ \lesssim & C E_0 \epsilon^2 \sum_{l \leq 5} \int_{\mathcal{A}_{\tau_1}^{\tau_2} } (r-M)^{3-\delta} (L \Omega^{l} T^4 \phi )^2 \, d\omega du dv  \lesssim  \frac{C^2 E_0^2 \epsilon^4}{\tau_1} , 
\end{split}
\end{equation*}
where we used Sobolev's inequality \eqref{est:sobolev}, the boundedness estimate \eqref{est:boundy}, and the Morawetz decay estimates \eqref{dec:mor4}, as $\delta \leq \delta_2$ is small enough.

\eqref{F2'}: We examine the part close to the horizon and for any $k \leq 5$ we have that:
\begin{equation*}
\begin{split}
\int_{\mathcal{A}_{\tau_1}^{\tau_2} } \frac{1}{(r-M)^{1+\delta}} & D^2 | \Omega^k T^4 F |^2 \cdot v^{1+\beta} \, d\omega du dv \\ \lesssim & \sum_{\substack{k_1 + k_2 = k \\ m_1 + m_2 = 4}} \int_{\mathcal{A}_{\tau_1}^{\tau_2} }\frac{1}{(r-M)^{1+\delta}} (L \Omega^{k_1} T^{m_1} \phi )^2 \cdot ( \underline{L} \Omega^{k_2} T^{m_2} \phi )^2\cdot v^{1+\beta}  \,  d\omega du dv \\ & + \sum_{\substack{k_1 + k_2 = k \\ m_1 + m_2 = 4}} \int_{\mathcal{A}_{\tau_1}^{\tau_2} } (r-M)^{3-\delta} ( \Omega^{k_1} T^{m_1} \phi )^2 \cdot ( \underline{L} \Omega^{k_2} T^{m_2} \phi )^2\cdot v^{1+\beta}  \,  d\omega du dv \\ & + \sum_{\substack{k_1 + k_2 = k \\ m_1 + m_2 = 4}} \int_{\mathcal{A}_{\tau_1}^{\tau_2} } (r-M)^{3-\delta} ( \Omega^{k_1} T^{m_1} \phi )^2 \cdot ( L \Omega^{k_2} T^{m_2} \phi )^2 \cdot v^{1+\beta}  \,  d\omega du dv \\ & + \sum_{\substack{k_1 + k_2 = k \\ m_1 + m_2 = 4}} \int_{\mathcal{A}_{\tau_1}^{\tau_2} } (r-M)^{7-\delta} ( \Omega^{k_1} T^{m_1} \phi )^2 \cdot (  \Omega^{k_2} T^{m_2} \phi )^2 \cdot v^{1+\beta}  \,  d\omega du dv \\ & + \sum_{\substack{k_1 + k_2 = k \\ m_1 + m_2 = 4}} \int_{\mathcal{A}_{\tau_1}^{\tau_2} } (r-M)^{3-\delta} | \slashed{\nabla} \Omega^{k_1} T^{m_1} \phi |^2 \cdot | \slashed{\nabla}  \Omega^{k_2} T^{m_2} \phi |^2 \cdot v^{1+\beta}  \,  d\omega du dv .
\end{split}
\end{equation*}
For the first term of the last expression when $m_1 = 4$ and $m_2 =0$ we use the auxiliary estimate \eqref{est:auxl} and we have for $\beta$ small enough that:
$$ \int_{\mathbb{S}^2} \frac{1}{(r-M)^{2+2\delta}} ( \underline{L} \Omega^l \phi )^2 (u,v,\omega )  \, d\omega \cdot v^{1+\beta} \lesssim C E_0 \epsilon^2 , $$
which implies after using Sobolev's inequality, and the Morawetz boundedness estimate \eqref{dec:mor4}, that we have for any $k \leq 5$ and for any $v$ that:
\begin{equation*}
\begin{split}
\sum_{\substack{k_1 + k_2 = k }} \int_{\mathcal{A}_{\tau_1}^{\tau_2} } & \frac{1}{(r-M)^{1+\delta}} (L \Omega^{k_1} T^4 \phi )^2 \cdot ( \underline{L} \Omega^{k_2}  \phi )^2 \cdot v^{1+\beta} \,  d\omega du dv  \\ \lesssim & C E_0 \epsilon^2 \sum_{l \leq 5} \int_{\mathcal{A}_{\tau_1}^{\tau_2} }  (r-M)^{1+\delta} (L \Omega^{k_1} T^4 \phi )^2  \,  d\omega du dv \lesssim C^2 E_0^2 \epsilon^4 .
\end{split}
\end{equation*}
For $m_1 = 1$, $m_2 = 3$ we use Sobolev's inequality \eqref{est:sobolev}, the auxiliary estimate \eqref{est:aux1} and the energy boundedness estimate \eqref{dec:enttt} for $p=1 +\delta$ as we have that $\delta \leq \delta_2$, and we have for any $k \leq 5$:
\begin{equation*}
\begin{split}
\sum_{\substack{k_1 + k_2 = k }} \int_{\mathcal{A}_{\tau_1}^{\tau_2} } & \frac{1}{(r-M)^{1+\delta}} (L \Omega^{k_1} T \phi )^2 \cdot ( \underline{L} \Omega^{k_2} T^3 \phi )^2 \cdot v^{1+\beta} \,  d\omega du dv  \\ \lesssim & \sum_{\substack{l_1 \leq 5 \\  l_2 \leq 5}} \int_{\tau_1}^{\tau_2} \int_{\mathbb{S}^2} \sup_u (L \Omega^{l_1} T \phi )^2 \cdot v^{1+\beta} \, d\omega dv \cdot \sup_{v \in [\tau_1 , \tau_2 ]} \int_{\mcN_v^H}  \frac{1}{(r-M)^{1+\delta}}  ( \underline{L} \Omega^{k_2} T^3 \phi )^2 \cdot v^{1+\beta} \,  d\omega du \\ \lesssim & C^2 E_0^2 \epsilon^4  .
\end{split}
\end{equation*}
For $m_1 = 2$ and $m_2 = 2$ we use Sobolev's inequality \eqref{est:sobolev}, the auxiliary estimate \eqref{est:aux1} for $m=2$ and the energy decay estimate \eqref{dec:entt} for $p=2$ and we have for any $k \leq 5$ that:
\begin{equation*}
\begin{split}
\sum_{\substack{k_1 + k_2 = k }} \int_{\mathcal{A}_{\tau_1}^{\tau_2} } & \frac{1}{(r-M)^{1+\delta}} (L \Omega^{k_1} T^2 \phi )^2 \cdot ( \underline{L} \Omega^{k_2} T^2 \phi )^2 \,  d\omega du dv \\ \lesssim & \sum_{\substack{l_1 \leq 5 \\ l_2 \leq 5 }} \int_{\tau_1}^{\tau_2} \int_{\mathbb{S}^2} \sup_u (L \Omega^{l_1} T^2 \phi )^2 \cdot v^{1+\beta} \,  d\omega dv \cdot \sup_{v \in [\tau_1 , \tau_2 ]}   \int_{\mcN_v^H} \frac{1}{(r-M)^{1+\delta}} ( \underline{L} \Omega^{l_2} T^2 \phi )^2  \,  d\omega du  \\ \lesssim & \frac{C^2 E_0^2 \epsilon^4}{\tau_1^{1+\delta_2 - \delta}} .
\end{split}
\end{equation*}
For $m_1 = 3$ and $m_2 =1$ we work as in the case above where we use the auxiliary estimate \eqref{est:aux1} for $m=3$. In the end we have that
$$ \sum_{\substack{k_1 + k_2 = k }} \int_{\mathcal{A}_{\tau_1}^{\tau_2} } \frac{1}{(r-M)^{1+\delta}} (L \Omega^{k_1} T^3 \phi )^2 \cdot ( \underline{L} \Omega^{k_2} T \phi )^2 \,  d\omega du dv \lesssim \frac{C^2 E_0^2 \epsilon^4}{\tau_1^{3+\delta_2 - \delta}} . $$
For $m_1 = 0$ and $m_2 =4$ we use Sobolev's inequality \eqref{est:sobolev}, the pointwise decay estimates \eqref{dec:tpsi} and \eqref{est:boundyt}, and the energy boundedness estimates \eqref{dec:enttttp1} (since $\delta \leq \delta_1$) and we have for any $v$ that:
\begin{equation*}
\begin{split}
\sum_{\substack{k_1 + k_2 = k }} \int_{\mcN_v^H } & \frac{1}{(r-M)^{1+\delta}} (L \Omega^{k_1}  \phi )^2 \cdot ( \underline{L} \Omega^{k_2} T^4 \phi )^2 \cdot v^{1+\beta} \,  d\omega du  \\ \lesssim & \frac{C E_0 \epsilon^2}{v^{1+\delta_2-\beta}}  \sum_{\substack{l_1 \leq 5 , l_2 \leq 5}} \int_{\mathcal{A}_{\tau_1}^{\tau_2} } \frac{1}{(r-M)^{1+\delta}}  (\underline{L} \Omega^{l} T^{4} \phi )^2  \,  d\omega du  \\ \lesssim &  \frac{C^2 E_0^2 \epsilon^4}{v^{1+\delta_2-\beta}} ,
\end{split}
\end{equation*}
and by choosing $\beta < \delta_2$, we get in the end that:
\begin{equation*}
\begin{split}
\sum_{\substack{k_1 + k_2 = k }} \int_{\mathcal{A}_{\tau_1}^{\tau_2} } & \frac{1}{(r-M)^{1+\delta}} (L \Omega^{k_1}  \phi )^2 \cdot ( \underline{L} \Omega^{k_2} T^4 \phi )^2 \cdot v^{1+\beta} \,  d\omega du dv \\ \lesssim & C^2 E_0^2 \epsilon^4 \int_{\tau_1}^{\tau_2} \frac{1}{v^{1+\delta_2-\beta}} \lesssim \frac{C^2 E_0^2 \epsilon^4}{\tau_1^{\delta_2 - \beta}} \lesssim C^2 E_0^2 \epsilon^4 . 
\end{split}
\end{equation*}

\eqref{F3'}: The case $p=1$ is being covered by the bootstrap \eqref{F2'}. We consider the other endpoint case $p=0$ and by using the same methods as in the case of the \eqref{F2'} bootstrap we get for any $k \leq 5$ that
\begin{equation*}
\int_{\mathcal{A}_{\tau_1}^{\tau_2}} D^2 | \Omega^k T^4 F|^2 \, d\omega du dv \lesssim \frac{C^2 E_0^2 \epsilon^4}{\tau_1}.
\end{equation*}
The rest of the $p$ range  -- $p \in (0,1)$ -- follows now by interpolation.

\eqref{A5'}, \eqref{B4'}, \eqref{C4'}, \eqref{D4'}, \eqref{E4'} and \eqref{F4'}: All these terms can be treated in a rather classical manner and we will not examine them here in detail, apart from the second term of \eqref{A5'}, i.e. the estimate:
$$ \left( \int_{\tau_1}^{\tau_2} \left( \int_{\mcN_u^I} r^{p+2} | F_0 |^2 \, d\omega dv \right)^{1/2} du \right)^2 \lesssim \frac{C E_0 \epsilon^2}{ (1+\tau_1 )^{3-\delta_1-p}} \mbox{  for $p \in (2,3-\delta_1 ]$} . $$
We write once again the nonlinearity $F$ in terms of $\phi$, and we note that close to infinity we have that:
$$ | F|^2 \sim \frac{1}{r^4} ( L \phi )^2 ( \underline{L} \phi )^2 + \frac{1}{r^6} \phi^2 ( \underline{L} \phi )^2 + \frac{1}{r^6} \phi^2 ( L \phi )^2 + \frac{1}{r^8} \phi^4 + \frac{1}{r^4} | \slashed{\nabla} \phi |^2 . $$
We note that as $L \phi$ decays with respect to $r$ as $r^{-2}$ at infinity, while $\phi$ and $\underline{L} \phi$ are just bounded, the term with the worst $r$ decay of $F$ is the one involving the product $\phi \cdot ( \underline{L} \phi)$, and more specifically as we are considering $F_0$ we examine in detail only the term involving the product $\phi_0 \cdot ( \underline{L} \phi_0 )$ which has also the worst $u$ decay. Noticing that $\underline{L} \phi_0$ decays at infinity with rate $C^{1/2} E_0^{1/2} \epsilon u^{-1-\delta/2 - \eta / 2}$ we have that:
\begingroup
\allowdisplaybreaks
\begin{align*}
\int_{\mcN_u^I} r^{p-4} \phi_0^2 \cdot (\underline{L} \phi_0 )^2 \, d\omega dv \lesssim & \frac{C E_0 \epsilon^2}{u^{2+\delta_2 - \eta }} \int_{\mcN_u^I} r^{p-4} \phi_0^2 \, d\omega dv \\ \lesssim & \frac{C E_0 \epsilon^2}{u^{2+\delta_2 - \eta }} \int_{\mcN_u^I} r^{p-2} ( L\phi_0 )^2 \, d\omega dv \lesssim \frac{C^2 E_0^2 \epsilon^4}{u^{7+ \delta_2 - \delta_1 - \eta -p}} ,
\end{align*}
\endgroup
where we used Hardy's inequality \eqref{hardy1} after noticing that $p-4 \in ( -2 , -1-\delta_1 ]$. The result is better than required due to the smallness of $\delta_1$, $\delta_2$ and $\eta$. The rest of the terms can be treated similarly.

\eqref{D6'}, \eqref{E5'} and \eqref{F5'}: These follow from the interior decay for $\int_{\mathbb{S}^2} ( \Omega^k T^m \psi )^2 \, d\omega$ for $k \leq 5$ and $m \leq 4$ (via the use of the elliptic estimates \eqref{est:elliptic}) and the Morawetz decay estimates \eqref{dec:mor}, \eqref{dec:mor1}, \eqref{dec:mor2}, \eqref{dec:mor3} and \eqref{dec:mor4}.

\eqref{G1'}: We examine the part close to the horizon and for any $k \leq 5$ we have that:
\begin{equation*}
\begin{split}
\int_{\mathcal{A}_{\tau_1}^{\tau_2} } \frac{1}{(r-M)^{1+\delta}} & D^2 | \Omega^k T^5 F |^2 \cdot v^{1+\beta} \, d\omega du dv \\ \lesssim & \sum_{\substack{k_1 + k_2 = k \\ m_1 + m_2 = 5}} \int_{\mathcal{A}_{\tau_1}^{\tau_2} }\frac{1}{(r-M)^{1+\delta}} (L \Omega^{k_1} T^{m_1} \phi )^2 \cdot ( \underline{L} \Omega^{k_2} T^{m_2} \phi )^2\cdot v^{1+\beta}  \,  d\omega du dv \\ & + \sum_{\substack{k_1 + k_2 = k \\ m_1 + m_2 = 5}} \int_{\mathcal{A}_{\tau_1}^{\tau_2} } (r-M)^{3-\delta} ( \Omega^{k_1} T^{m_1} \phi )^2 \cdot ( \underline{L} \Omega^{k_2} T^{m_2} \phi )^2\cdot v^{1+\beta}  \,  d\omega du dv \\ & + \sum_{\substack{k_1 + k_2 = k \\ m_1 + m_2 = 5}} \int_{\mathcal{A}_{\tau_1}^{\tau_2} } (r-M)^{3-\delta} ( \Omega^{k_1} T^{m_1} \phi )^2 \cdot ( L \Omega^{k_2} T^{m_2} \phi )^2 \cdot v^{1+\beta}  \,  d\omega du dv \\ & + \sum_{\substack{k_1 + k_2 = k \\ m_1 + m_2 = 5}} \int_{\mathcal{A}_{\tau_1}^{\tau_2} } (r-M)^{7-\delta} ( \Omega^{k_1} T^{m_1} \phi )^2 \cdot (  \Omega^{k_2} T^{m_2} \phi )^2 \cdot v^{1+\beta}  \,  d\omega du dv \\ & + \sum_{\substack{k_1 + k_2 = k \\ m_1 + m_2 = 5}} \int_{\mathcal{A}_{\tau_1}^{\tau_2} } (r-M)^{3-\delta} | \slashed{\nabla} \Omega^{k_1} T^{m_1} \phi |^2 \cdot | \slashed{\nabla}  \Omega^{k_2} T^{m_2} \phi |^2 \cdot v^{1+\beta}  \,  d\omega du dv .
\end{split}
\end{equation*}
For the case $m_1 = 0$ and $m_2 = 5$ we have that:
\begin{equation*}
\begin{split}
 \sum_{\substack{k_1 + k_2 = k }} \int_{\mathcal{A}_{\tau_1}^{\tau_2} } & \frac{1}{(r-M)^{1+\delta}} (L \Omega^{k_1}  \phi )^2 \cdot ( \underline{L} \Omega^{k_2} T^5 \phi )^2\cdot v^{1+\beta}  \,  d\omega du dv \\ \lesssim & \sum_{\substack{ l_1 \leq 5 \\ l_2 \leq 5 }} \int_{\tau_1}^{\tau_2} \int_{\mathbb{S}^2} \sup_u ( L \Omega^{l_1} \phi )^2 \cdot v^{1+\beta} \, d\omega dv \cdot \sup_{v \in [\tau_1 , \tau_2 ]} \int_{\mcN_v^H} (\underline{L} \Omega^l T^5 \phi )^2 \, d\omega du \\ \lesssim & C^2 E_0^2 \epsilon^4 , 
\end{split}
\end{equation*}
where we used Sobolev's inequality \eqref{est:sobolev}, the auxiliary estimate \eqref{est:aux1} for $m=0$ (or the pointwise decay estimates \eqref{dec:tpsi}), and the energy boundedness estimates \eqref{dec:enttttt}. For the cases $m_1 = 1$ and $m_2 = 4$, $m_1 =2 $ and $m_2 = 3$, $m_1 = 3$ and $m_2 = 2$, we argue as above but now using the auxiliary estimates \eqref{est:aux1} for $m=1$, $m=2$ and $m=3$ respectively.

For the case $m_1 = 5$ and $m_2 = 0$ we use the additional estimate \eqref{est:auxl} and we have that:
\begin{equation*}
\begin{split}
\sum_{\substack{k_1 + k_2 = k }} \int_{\mathcal{A}_{\tau_1}^{\tau_2} } & \frac{1}{(r-M)^{1+\delta}} (L \Omega^{k_1} T^5 \phi )^2 \cdot ( \underline{L} \Omega^{k_2}  \phi )^2 \cdot v^{1+\beta}  \,  d\omega du dv \\ \lesssim & C E_0 \epsilon^2 \sum_{l \leq 5} \int_{\mathcal{A}_{\tau_1}^{\tau_2} } (r-M)^{1+\delta} (L \Omega^{l} T^5 \phi )^2 \,  d\omega du dv \\ \lesssim & C^2 E_0^2 \epsilon^4 ,
\end{split}
\end{equation*}
where moreover we used Sobolev's inequality \eqref{est:sobolev} and the Morawetz boundedness estimate \eqref{dec:mor5}. For the case $m_1 = 4$ and $m_2 = 1$ we argue similarly using now the auxiliary estimate \eqref{est:auxlt} and we have that:
\begin{equation*}
\begin{split}
\sum_{\substack{k_1 + k_2 = k }} \int_{\mathcal{A}_{\tau_1}^{\tau_2} } & \frac{1}{(r-M)^{1+\delta}} (L \Omega^{k_1} T^5 \phi )^2 \cdot ( \underline{L} \Omega^{k_2}  \phi )^2 \cdot v^{1+\beta}  \,  d\omega du dv \\ \lesssim & C E_0 \epsilon^2 \sum_{l \leq 5} \int_{\mathcal{A}_{\tau_1}^{\tau_2} } (r-M)^{1+\delta} (L \Omega^{l} T^5 \phi )^2 \,  d\omega du dv \\ \lesssim & \frac{C^2 E_0^2 \epsilon^4}{\tau_1} ,
\end{split}
\end{equation*}
where moreover we used Sobolev's inequality \eqref{est:sobolev} and the Morawetz decay estimate \eqref{dec:mor4}.

For the term which is away from the horizon and away from infinity
$$ \left( \int_{\tau_1}^{\tau_2} \left( \int_{\Sigma_{\tau} \setminus ( \mcN_{\tau}^H \cup \mcN_{\tau}^I)} \Omega^k T^5 F |^2 \, d\mu_{\Sigma} \right)^{1/2} d\tau \right)^2 $$
we note that the inner integral can be bounded by the following by use of the elliptic estimates \eqref{est:elliptic}:
\begin{equation}\label{est:intt5}
\begin{split}
\sum_{\substack{ 1 \leq l_1 \leq 5 \\  l_2 \leq 5 \\ l_3 \leq 5 }}  & \sup_{\Sigma_{\tau} \setminus ( \mcN_{\tau}^H \cup \mcN_{\tau}^I)} \left[ \int_{\mathbb{S}^2} ( \Omega^{l_1} \psi )^2 d\omega + \int_{\mathbb{S}^2} ( \Omega^{l_2} T \psi )^2 \, d\omega \right] \cdot \int_{\Sigma_{\tau}} J^T [ \Omega^{l_3} T^5 \psi ] \cdot \textbf{n}_{\Sigma} \, d\mu_{\Sigma} \\ & + \sum_{\substack{ l_1 \leq 5 \\ l_2 \leq 5}}    \sup_{\Sigma_{\tau} \setminus ( \mcN_{\tau}^H \cup \mcN_{\tau}^I)}  \int_{\mathbb{S}^2} ( \Omega^{l_1} T^2 \psi )^2 d\omega \cdot \int_{\Sigma_{\tau}} J^T [ \Omega^{l_2} T^4 \psi ] \cdot \textbf{n}_{\Sigma} \, d\mu_{\Sigma} \\ & + \sum_{\substack{ l_1 \leq 5 \\ l_2 \leq 5}}    \sup_{\Sigma_{\tau} \setminus ( \mcN_{\tau}^H \cup \mcN_{\tau}^I)}  \int_{\mathbb{S}^2} ( \Omega^{l_1} T^3 \psi )^2 d\omega \cdot \int_{\Sigma_{\tau}} J^T [ \Omega^{l_2} T^3 \psi ] \cdot \textbf{n}_{\Sigma} \, d\mu_{\Sigma}  \\ & + \sum_{\substack{ l_1 \leq 5 \\ 1 \leq l_2 \leq 5}}    \sup_{\Sigma_{\tau} \setminus ( \mcN_{\tau}^H \cup \mcN_{\tau}^I)}  \int_{\mathbb{S}^2} ( \Omega^{l_1} T^5 \psi )^2 d\omega \cdot \int_{\Sigma_{\tau}} J^T [ \Omega^{l_2}  \psi ] \cdot \textbf{n}_{\Sigma} \, d\mu_{\Sigma}\\ \lesssim & \frac{C^2 E_0^2 \epsilon^4}{\tau^{2+2\delta_2}} , 
\end{split}
\end{equation}
using Sobolev's inequality \eqref{est:sobolev}, the pointwise decay and boundedness estimates in the interior \eqref{dec:psiint}, \eqref{dec:tpsiint}, \eqref{dec:ttpsiint}, \eqref{dec:tttpsiint}, \eqref{dec:tttttpsiint}, and the decay and boundedness of the $T$-fluxes for $\psi$, $T^3 \psi$, $T^4 \psi$ and $T^5 \psi$ given by \eqref{dec:en}, \eqref{dec:enttt}, \eqref{dec:entttt}, \eqref{dec:enttttt}.

\eqref{G2'}: This estimate follows now directly by using the last computation \eqref{est:intt5} and integrating it in $\tau$.

\end{proof}

The proof of the main Theorem \ref{thm:main} now follows by a standard bootstrap argument and choosing $\epsilon$ to be small enough and smaller than $\epsilon'$ (from Theorem \ref{thm:boundy0}), $\epsilon''$ (from Theorem \ref{thm:boundyy}) and $\epsilon_0$ (from Theorem \ref{thm:aux}).

\section{Asymptotic instabilities on the horizon}\label{insta}
In this section we will show that the estimate \eqref{est:boundyy} is sharp on the horizon for $\delta = 1$. We will work similarly as in Section 11 of \cite{yannis1}, to show that $\frac{2r}{D}  \underline{L} \left( \frac{2r}{D} \underline{L} \phi \right)$ diverges to infinity on $\mathcal{H}^{+}$ growing like $v$. Moreover we observe that $\frac{2r}{D} \underline{L} \phi$ exhibits no decay along $\mathcal{H}^{+}$. We note that both of these instabilities (that have no analogues in the sub-extremal setting) come from the spherically symmetric part of the wave. 

\begin{theorem}\label{thm:inst}
For a solution of \eqref{eq:thm} with data as in the main Theorem \ref{thm:main} that was obtained in the previous sections, we have that:
\begin{equation}\label{bdd:yhor}
 \left. \frac{2r}{D} \underline{L} \phi (v,\omega) \right|_{\mathcal{H}^{+}} - \left. \frac{2r}{D} \underline{L} \phi (v_0 ,\omega) \right|_{\mathcal{H}^{+}} \simeq C_1 \epsilon^2 , 
 \end{equation}
 and
 \begin{equation}\label{bdd:yyhor}
  \left. \frac{2r}{D} \underline{L} \left( \frac{2r}{D} \underline{L} \phi \right)  (v,\omega) \right|_{\mathcal{H}^{+}} \simeq C_1 \epsilon v  \mbox{  for all $v \geq v_l \geq \tau_0$ where $v_l$ is large enough. }
 \end{equation} 
\end{theorem}
\begin{remark}
In the above Theorem, we use the terminology
$$ f_1 \simeq f_2 $$
for some functions $f_1$, $f_2$, means that there are constants $c$, $C$ such that
$$ c f_2 \lesssim f_1 \lesssim C f_2 . $$
\end{remark}
\begin{proof}
For estimate \eqref{bdd:yhor} we use equation \eqref{eq:yder} for $\phi_0$ and we evaluate it on the horizon $r=M$, from which we get for any $v \geq v_0 = \tau_0$ that:
\begingroup
\allowdisplaybreaks
\begin{align*}
\left. \frac{2r}{D} \underline{L} \phi_0 ( v, \omega ) \right|_{\mathcal{H}^{+}} = & \left. \frac{2r}{D} \underline{L} \phi_0 ( v_0 , \omega ) \right|_{\mathcal{H}^{+}} + \int_{v_0}^v \left[ \frac{A}{r} ( L \phi ) \cdot \left( \frac{2r}{D} \underline{L} \phi \right) \right]_0 \, dv' \\ \simeq & \left. \frac{2r}{D} \underline{L} \phi_0 ( v_0 , \omega ) \right|_{\mathcal{H}^{+}} + \left. \int_{v_0}^v ( L \phi ) \cdot \left( \frac{2r}{D} \underline{L} \phi \right) \, dv'\right|_{\mathcal{H}^{+}}  .
\end{align*}
\endgroup
Note that as $\frac{2r}{D} \underline{L} \phi$ is bounded by $C^{1/2} E_0^{1/2} \epsilon$ and $L\phi$ has integrable decay in $v$, we note that the second term can be bounded by:
$$ C E_0 \epsilon^2 $$ 
and as $\epsilon$ is small enough the size of $\left. \frac{2r}{D} \underline{L} \phi_0 ( v, \omega ) \right|_{\mathcal{H}^{+}}$ is comparable to the size of $\left. \frac{2r}{D} \underline{L} \phi_0 ( v_0 , \omega ) \right|_{\mathcal{H}^{+}}$.

For estimate \eqref{bdd:yyhor} we use equation \eqref{null:eqyy} for $\phi_0$, evaluate it on the horizon $r=M$, and then integrate it in $v$. On the other hand as we have that:
\begin{equation*}
\begin{split}
 \left. \frac{2r}{D} \underline{L} \left( \frac{2r}{D} \underline{L} \phi \right)  (v,\omega) \right|_{\mathcal{H}^{+}} \simeq &  \left. \frac{2r}{D} \underline{L} \left( \frac{2r}{D} \underline{L} \phi \right)  (v_0 ,\omega) \right|_{\mathcal{H}^{+}}  \\ & + \left.\left. \int_{v_0}^v \frac{2r}{D} \underline{L} \phi_0 ( v, \omega ) \, dv'\right|_{\mathcal{H}^{+}}  + \int_{v_0}^v ( \phi_0 + \mathcal{N} ) \, dv' \right|_{\mathcal{H}^{+}}  ,
\end{split}
\end{equation*}
where $\mathcal{N}$ involves all the nonlinear terms from \eqref{null:eqyy}, we note that the major contribution of the above expression comes from the term involving $\frac{2r}{D} \phi_0$ which from estimate \eqref{bdd:yhor} gives us that:
$$ \left. \int_{v_0}^v \frac{2r}{D} \underline{L} \phi_0 ( v, \omega ) \, dv'\right|_{\mathcal{H}^{+}} \simeq C_1 \epsilon ( v - v_0 ) \simeq \bar{C_1} \epsilon v , $$
for some constant $\bar{C_1}$. We finally observe that as the nonlinear terms $\mathcal{N}$ can be bounded by $C E_0 \epsilon^2 v^{-\eta}$ for some $\eta > 0$ and as $\phi_0$ can be bounded by $\frac{C^{1/2} E_0^{1/2} \epsilon}{v^{1-\delta_1 /2}}$ the last two terms can be bounded by $( c_1 \epsilon + c_2 \epsilon^2 ) v^{\eta'}$ for some constant $c_1$, $c_2$ and some $\eta' \in (0,1)$, hence for $v$ large they can be neglected due to the term involving $\frac{2r}{D} \underline{L} \phi_0$. This finishes the proof of estimate \eqref{bdd:yyhor}.
\end{proof}

\begin{remark}
It is worth noticing that if we consider the standard null form $g^{\alpha \beta} \cdot \partial_{\alpha} \psi \cdot \partial_{\beta} \psi$ in equation \eqref{nw}, then the following quantity is actually \textit{conserved} on the horizon:
$$ H^{NL}_0 [ \psi ] (v) \doteq \left. \int_{\mathbb{S}^2} \left[ e^{-\psi ( v,M , \omega ) } \partial_r \psi ( v,M,\omega ) + \frac{1}{M} \left( 1 - e^{-\psi (v,M,\omega) } \right) \right] \, d\omega \right|_{\mathcal{H}^{+}} . $$
Similar quantities are conserved on the horizon for the more general nonlinearities of equation \eqref{nw}, the derivation and the investigation of the properties of such quantities will be pursued in future work of the authors of the present paper.
\end{remark}

\section{Remarks on other nonlinearities}\label{non:other}
Due to the relation of the $r^p$-weighted estimates at infinity which can be derived for sub-extremal black holes as it was done in \cite{paper1} in the linear case, and the $(r-M)^{-p}$-weighted estimates at the horizon, our method is robust enough to deal with nonlinearities at infinity with growing weights in $r$. A model example can be the following nonlinear problem:
\begin{equation}\label{eq:sub-ex}
\Box_{g_{sub}} \psi = \chi_{ \{ r \geq R > R_h \}} \left[ \frac{1}{r} ( L \phi ) ( \underline{L} \phi ) + \frac{1}{r} | \slashed{\nabla} \phi |^2 \right]  ,
\end{equation}
for $\Box_{g_{sub}}$ the d'Alembertian operator on a sub-extremal Reissner--Nordstr\"{o}m spacetimes, for $\chi_{ \{ r \geq R > R_h \} }$ a cut-off function supported in the region $\{ r \geq R > R_h \}$ where $R_h$ is the value of the radial variable $r$ on the event horizon, for $\phi = r\psi$, and for small enough data on a spacelike-null hypersurface $\Sigma_{\tau_0}$. Note that the aforementioned nonlinearity has an extra $r$ weight on the $(L \phi) \cdot (\underline{L} \phi)$ term compared to the classical null form. To deal with such a problem and defining $\mathcal{B}_{\tau_0}^{\infty}$ as in the extremal case, we will have to separate the spherically symmetric part of the solution from the non-spherically symmetric one, derive the same range of $r^p$-weighted estimates as for Theorem \ref{thm:main} where the corresponding bootstraps will need the following estimates
\begin{equation}\label{est:auxsub}
\int_{u_0}^U \int_{\mathbb{S}^2} \sup_{v \in [v_{R,n} , V]} ( \underline{L} T^m \Omega^k \phi )^2 \cdot u^{1+\delta} \, d\omega du \lesssim \epsilon^2 ,
\end{equation}
for all $U$, for any $0 < \delta < 2-\delta_1$ if $m=0$, for any $0 < \delta < 2+\delta_2$ if $m=1$, for any $0 < \delta < 1+\delta_2$ if $m=2$, for any $0 < \delta < \delta_2$ if $m=3$, for any $u_{R,n}$, $U$, $V$ in the region $\mathcal{B}_{\tau_0}^{\infty}$ (where $v_{R,n}$ is on the hypersurface $r=R$), for $m\in \{0,1,2,3\}$, and for any $k \leq 5$, which are the analogues of estimates \eqref{est:aux1}, the boundedness estimate
$$ | r^2 L \phi | (u,v,\omega) \lesssim \epsilon , $$
in $\mathcal{B}_{\tau_0}^{\infty}$ which is the analoge of \eqref{est:boundy0}, and the growth estimate 
$$ r^{1+\delta} | L ( r^2 L \phi ) | (u,v,\omega ) \lesssim \epsilon u^{\delta} \mbox{  for any $\delta \in ( 0,1]$}, $$
in $\mathcal{B}_{\tau_0}^{\infty}$ which is the analogue of \eqref{est:boundyy}.

Finally we note that due to the robustness of our methods we plan to investigate in future work how to derive \textit{precise asymptotics} for solutions of nonlinear wave equations satisfying the null condition both on extremal and sub-extremal Reissner--Nordstr\"{o}m black holes spacetimes.

\appendix
\addtocontents{toc}{\protect\setcounter{tocdepth}{2}}
\section{}
\subsection{The Couch-Torrence conformal isometry}\label{app:iso}
An extremal Reissner--Nordstr\"{o}m spacetime of mass $M$ admits a conformal isometry called the Couch-Torrence first introduced in \cite{couch} that in ingoing Eddington-Finkelstein coordinates is given by
$$ \bar{\Phi} ( v,r,\omega ) \doteq \left( u = v , r' = M + M^2 (r-M)^{-1} , \omega \right) , $$
and through it $\mathcal{H}^{+}$ is mapped onto $\mathcal{I}^{+}$.

\subsection{The d'Alembertian in different coordinates}
The nonlinearity of \eqref{nw} can be written as follows in double null coordinates (where $\phi = r \psi$):
\begin{equation}\label{null:nonlin}
\begin{split}
g^{\alpha \beta} \cdot \partial_{\alpha} \psi \cdot \partial_{\beta} \psi = & \frac{4}{D} \cdot (L \psi ) \cdot (\underline{L} \psi ) + | \slashed{\nabla} \psi |^2 \\ = & \frac{4}{Dr^2} \cdot (L \phi ) \cdot (\underline{L} \phi ) - \frac{2}{r^3} \cdot \phi \cdot (\underline{L} \phi ) \\ & + \frac{2}{r^3} \cdot \phi \cdot (L \phi ) - \frac{D}{r^4} \cdot \phi^2 + \frac{1}{r^2} | \slashed{\nabla} \phi |^2 \\ = & \frac{2}{r^3} \cdot (L \phi ) \cdot \left( \frac{2r}{D}\underline{L} \phi \right) - \frac{D}{r^4} \cdot \phi \cdot \left( \frac{2r}{D}\underline{L} \phi \right) \\ & + \frac{2}{r^3} \cdot \phi \cdot (L \phi ) - \frac{ D}{r^4} \cdot \phi^2 + \frac{1}{r^2} | \slashed{\nabla} \phi |^2 .
\end{split}
\end{equation}

Equation \eqref{nw} can then be written in terms of $\phi = r \psi$ as follows in double null coordinates:
\begin{equation}
\begin{split}\label{null:eq}
4L \underline{L} \phi = & D\slashed{\Delta} \phi  - \frac{D \cdot D'}{r} \phi  \\ & + \frac{2A \cdot D}{r^2} \cdot (L \phi ) \cdot \left( \frac{2r}{D}\underline{L} \phi \right) - \frac{A \cdot D^2}{r^3} \cdot \phi \cdot \left( \frac{2r}{D}\underline{L} \phi \right) \\ & + \frac{2A \cdot D}{r^2} \cdot \phi \cdot (L \phi ) - \frac{A \cdot D^2}{r^3} \cdot \phi^2 + \frac{A \cdot D}{r} | \slashed{\nabla} \phi |^2 .
\end{split}
\end{equation}

\subsection{Basic inequalities}
We record some basic inequalities. The first is the Sobolev inequality on the sphere from which we have that for any smooth function $f$:
\begin{equation}\label{est:sobolev}
\int_{\mathbb{S}^2} f^2 \, d\omega \lesssim \sum_{k \leq 2}\int_{\mathbb{S}^2} ( \Omega^k f )^2 \, d\omega .
\end{equation}
The second one is Hardy's inequality, which close to the horizon it has the following form for a smooth function $f$ and for any $s \neq 1$ and for any $M \leq r_1 < r_2 < \infty$:
\begin{equation}\label{hardy}
\int_{u_{r_2} (v)}^{u_{r_1} (v)} (r-M)^{-s} f^2 \, du \lesssim \frac{1}{(1+s)^2} \int_{u_{r_2} (v)}^{u_{r_1} (v)} (r-M)^{-s-2} (\underline{L} f )^2 \, du  + 2 (r_1 - M )^{-s-1} f^2 ( u_{r_1} (v) , v) ,
\end{equation}
and close to infinity it gives us that for any $M < r_1 < r_2 \leq \infty$:
\begin{equation}\label{hardy1}
\int_{v_{r_1} (u)}^{v_{r_2} (u)} r^s f^2 \, dv \lesssim \frac{1}{(1+s)^2} \int_{v_{r_1} (u)}^{v_{r_2} (u)} r^{s+2} (L f )^2 \, dv + 2 r^{s+1} f^2 (u , v_{r_2} (u)) ,
\end{equation}
where in the case of the horizon if $r_1 = M$ then the last term is considered as:
$$ 2 \lim_{u \rightarrow \infty} (r - M )^{-s-1} f^2 ( u_{r} (v) , $$
and in the case of infinity if $r_2 = \infty$ then the last term is considered as:
$$  2 \lim_{v \rightarrow \infty} r^{s+1} f^2 (u , v_{r} (u))  . $$
For proofs of these inequalities see \cite{paper4}.
\subsection{Elliptic estimates}
We record as well the following basic elliptic estimate from \cite{aretakis2}:
\begin{equation}\label{est:elliptic}
\begin{split}
\int_{\Sigma_{\tau} \cap \{ r\geq r_0 > M \}} (\partial_a \partial_b \psi )^2 & d\mu_{\Sigma} \lesssim   \int_{\Sigma_{\tau} \cap \{ r\geq r_0 \}} J^T_{\mu} [\psi ] \cdot \textbf{n}_{\Sigma} d\mu_{\Sigma} \\ & + \int_{\Sigma_{\tau} \cap \{ r\geq r_0 \}} J^T_{\mu} [T\psi ] \cdot \textbf{n}_{\Sigma} d\mu_{\Sigma} + \int_{\Sigma_{\tau} \cap \{ r\geq r_0 \}} |F|^2 d\mu_{\Sigma} , 
\end{split}
\end{equation}
for any fixed $r_0 > M$, and any $\partial_a , \partial_b \in \{ L , \underline{L} , \partial_{\theta}, \partial_{\sigma} \}.$
\subsection{Additional norms}\label{norm:add}
For any smooth function $f : \mathcal{M} \rightarrow \mathbb{R}$ we define for any $\tau \in [ \tau_0 , \infty)$ the norm:
\begingroup
\allowdisplaybreaks
\begin{align*}
E^{\tau} [ f ] \doteq \int_{\Sigma_{\tau}} f^2 \, d\mu_{\Sigma_{\tau}} + & \sum_{\substack{k \leq 5 \\ l \leq 5} } \int_{\Sigma_{\tau}} J^T [ \Omega^k T^l f ] \cdot \textbf{n}_{\Sigma_{\tau}} \, d\mu_{\Sigma_{\tau}}  \\ & + \int_{\mcN_{\tau}^H} (r-M)^{-3+\delta_1} ( \underline{L} f_0 )^2 \, d\omega du + \int_{\mcN_{\tau}^H} (r-M)^{-1+\delta_1} \left( \underline{L}  \left( \frac{2r}{D} \underline{L} f_0 \right)  \right)^2 \, d\omega du   \\ & + \int_{\mcN_{\tau}^I} r^{3-\delta_1} ( L f_0 )^2 \, d\omega dv + \int_{\mcN_{\tau}^I} r^{1-\delta_1} \left( L \left( \frac{2r^2}{D} L f_0 \right) \right)^2 \, d\omega dv \\ & + \sum_{k \leq 5} \Bigg[ \int_{\mcN_{\tau}^H} (r-M)^{-3-\delta_2} ( \underline{L} \Omega^k f_{\geq 1} )^2 \, d\omega du + \int_{\mcN_{\tau}^H} (r-M)^{-1-\delta_2} \left( \underline{L} \left( \frac{2r}{D} \underline{L} \Omega^k f_{\geq 1} \right) \right)^2 \, d\omega du   \\ & + \int_{\mcN_{\tau}^I} r^{3+\delta_2} ( L \Omega^k f_{\geq 1} )^2 \, d\omega dv + \int_{\mcN_{\tau}^I} r^{1+\delta_2} 
\left( L \left( \frac{2r^2}{D} L \Omega^k f_{\geq 1} \right)   \right)^2 \, d\omega dv \\ & + \int_{\mcN_{\tau}^H} (r-M)^{-3-\delta_2} ( \underline{L} \Omega^k T f )^2 \, d\omega du + \int_{\mcN_{\tau}^H} (r-M)^{-1-\delta_2} \left( \underline{L} \left( \frac{2r}{D} \underline{L} \Omega^k T f \right) \right)^2 \, d\omega du   \\ & + \int_{\mcN_{\tau}^I} r^{3+\delta_2} ( L\Omega^k T f )^2 \, d\omega dv + \int_{\mcN_{\tau}^I} r^{1+\delta_2} \left(L \left( \frac{2r^2}{D} L \Omega^k T f \right)  \right)^2 \, d\omega dv \\ & + \int_{\mcN_{\tau}^H} (r-M)^{-2-\delta_2} ( \underline{L} \Omega^k T^2 f )^2 \, d\omega du + \int_{\mcN_{\tau}^H} (r-M)^{-\delta_2} \left( \underline{L} \left( \frac{2r}{D} \underline{L} \Omega^k T^2 f \right) \right)^2 \, d\omega du   \\ & + \int_{\mcN_{\tau}^I} r^{2+\delta_2} ( L\Omega^k T^2 f )^2 \, d\omega dv + \int_{\mcN_{\tau}^I} r^{\delta_2} \left(L \left( \frac{2r^2}{D} L \Omega^k T^2 f \right)  \right)^2 \, d\omega dv \\ & + \int_{\mcN_{\tau}^H} (r-M)^{-2} ( \underline{L} \Omega^k T^3 f )^2 \, d\omega du  + \int_{\mcN_{\tau}^I} r^2 ( L\Omega^k T^3 f )^2 \, d\omega dv \\ & + \int_{\mcN_{\tau}^H} (r-M)^{-1-\delta_2} ( \underline{L} \Omega^k T^4 f )^2 \, d\omega du  + \int_{\mcN_{\tau}^I} r^{1+\delta_2} ( L\Omega^k T^4 f )^2 \, d\omega dv \\ & + \int_{\mcN_{\tau}^H} (r-M)^{-1-\delta_2} ( \underline{L} \Omega^k T^5 f )^2 \, d\omega du  + \int_{\mcN_{\tau}^I} r^{1+\delta_2} ( L\Omega^k T^5 f )^2 \, d\omega dv .
\end{align*}
\endgroup
Moreover we also define the standard Sobolev norms for any $s \in \mathbb{N}$ as:
$$ \| f \|_{H^s_{\tau}} \doteq \sum_{| \alpha | \leq s} \left( \int_{\Sigma_{\tau}} ( \partial^{\alpha} f )^2 \, d\mu_{\Sigma_{\tau}} \right)^{1/2}\  \text{ and } \ \ \| f \|_{\widetilde{H}^s_{\tau}} \doteq \sum_{| \alpha | \leq s} \left( \int_{\Sigma^{int}_{\tau}} ( \partial^{\alpha} f )^2 \, d\mu_{\Sigma^{int}_{\tau}} \right)^{1/2} , $$
where $\partial \in \{T , Y , \partial_{\theta}, \partial_{\varphi} \}$.

\small

Department of Mathematics, University of California, Los Angeles, CA 90095, United States, yannis@math.ucla.edu

\smallskip

Princeton University, Department of Mathematics, Fine Hall, Washington Road, Princeton, NJ 08544, United States, aretakis@math.princeton.edu

\smallskip

Department of Mathematics, University of Toronto, 40 St George Street, Toronto, ON, Canada, aretakis@math.toronto.edu

\smallskip

Department of Pure Mathematics and Mathematical Statistics, University of Cambridge, Wilberforce Road, Cambridge CB3 0WB, United Kingdom, dg405@cam.ac.uk 
\end{document}